\documentclass[10pt]{amsart}
\usepackage[all,cmtip]{xy}
\usepackage{latexsym}                                                          
\usepackage{amssymb}
\usepackage{epsfig}
\usepackage{psfrag}
\usepackage{amsthm}
\usepackage{amscd}
\usepackage{amsmath}
\usepackage{amsfonts}
\usepackage{graphics}
\usepackage[all]{xy}
\usepackage{comment}
\newtheorem*{namedtheorem}{\theoremname}
\newcommand{\theoremname} {testing}
\newenvironment{named}[1]{\renewcommand{\theoremname}{#1}\begin{namedtheorem}}
{\end{namedtheorem}}

\newtheorem{thm}{Theorem}[section]

\newtheorem{remark}[thm]{Remark}

\newtheorem{problem*}{Problem}

\newtheorem{prop}[thm]{Proposition}

\newtheorem{lemma}[thm]{Lemma}

\newtheorem{condition}[thm]{Condition}

\newcommand{\bbZ}{\mathbb{Z}}

\newcommand{\bbR}{\mathbb{R}}
\newcommand{\bbC}{\mathbb{C}}

\newcommand{\into}{\longrightarrow}

\renewcommand{\dim}{\textup{dim}}
\newcommand{\Int}{\textup{Int\,}}
\renewcommand{\index}{\textup{index\,}}

\DeclareMathOperator{\dom}{dom}

\title[Morse complexifications]{Complexifications of Morse functions and the directed Donaldson-Fukaya category}
\author{Joe Johns}
\address{Columbia University, Mathematics Department, 
2990 Broadway, MC4403, New York, NY, 10027}
\email{jjohns@math.columbia.edu}
\begin{document}

%\vspace{-1 cm}

\begin{abstract} Let $N$ be a closed four dimensional manifold which admits a self-indexing Morse function $f: N \into \bbR$ with only 3 critical values $0$,\,$2$,\,$4$, and a unique maximum and minimum. Let $g$ be a Riemannian metric on $N$
such that $(f,g)$ is Morse-Smale. We construct from $(N,f,g)$ a certain six dimensional exact symplectic manifold $M$, together with some exact Lagrangian spheres $V_4$,\,$V_2^j$,\,$V_0$ in $M$, 
$j=1, \ldots, k$. These spheres correspond to the critical points $x_4$,\,$x_2^j$,\,$x_0$ of $f$, where the subscript indicates the Morse index. (In a companion paper we explain how 
$(M, V_4,\{V_2^j\},V_0)$ is a model for  the regular fiber and vanishing spheres of 
the complexification of $f$, viewed as a Lefschetz fibration on the disk cotangent bundle 
$D(T^*N)$.) Our main result is a computation of the Lagrangian Floer homology groups 
$$HF(V_4,V_2^j),\, HF(V_2^j, V_0), \, HF(V_4,V_0)$$ 
and the triangle product
$$\mu_2: HF(V_4,V_2^j) \otimes HF(V_2^j,V_0) \into HF(V_4,V_0)$$
in terms of the Morse theory of $(N,f,g)$. 
The outcome is that the directed Donaldson-Fukaya category 
of $(M, V_4, \{V_2^j\}, V_0)$ is isomorphic to the flow category of $(N,f,g)$.
%as sketched in \cite{J}. 
\end{abstract}

\maketitle
%\vspace{-1 cm}
\tableofcontents
%\addtocontents{toc}
%\addtocontents{toc}{\textbf}
\vspace{-1cm}
\section{Introduction}\label{intro}

This paper is a natural sequel to \cite{JA}, although it does not rely on the results there. 
%The reader may also find \cite{JB} to be a helpful companion paper; there, we sketch the proof of the main theorem in this paper at an intuitive level, and we give some outlook on potential applications. 
%\\\newlie 
In that paper we consider the following problem. Suppose that $N$ is a real analytic manifold and $f: N \into \bbR$ is a real analytic Morse function. Then,
in local charts on $N$, $f$ is represented by some convergent power series with real coefficients; if we complexify these local power series to get complex analytic power series (with the same coefficients) we obtain a complex analytic map on the disk bundle of $T^*N$ of some small radius, %$\epsilon>0$,
$$f_\bbC:  D(T^*N) \into \bbC,$$
called the \emph{complexification} of $f$. In favorable circumstances, $f_{\bbC}$ can be regarded as a symplectic Lefschetz fibration.
 (Obviously this description 
is a little imprecise, but for us $f_\bbC$ will only be used for motivation;
%it is interesting to ask what the correct precise version is; 
one possibility for a precise version is given in \cite{HK}.) 

\begin{named}{Problem}\notag%\label{complexify} 
Describe the generic fiber of $f_\bbC$ as a symplectic manifold $M$, and describe its vanishing spheres 
as Lagrangian spheres in $M$.
\end{named}
 Implicit in this problem is the more informal question:  
%\vspace{0.15cm}
\emph{To what extent is the Morse theory of $f$ reflected in the complex geometry of $f_\bbC$?}
%\vspace{0.15cm}
Problems of this type have been considered for some time in various fields, see for example \cite{K}, \cite{Lempert}, \cite{AMP}, \cite{A'C}. In fact we
draw a great deal of inspiration from the last mentioned paper, which solves the above problem when
$f: \bbR^2 \into \bbR$ is a real polynomial. 
\\
\newline 
%Our approach   is as follows. Let $g$ be a Riemannian metric such that $(f,g)$ is Morse-Smale. 
Our approach in \cite{JA} is as follows. We first  take  a Riemannian metric $g$ such that $(f,g)$ is Morse-Smale and  we only assume $(N,f,g)$ is smooth, not necessarily real analytic. Then, given the corresponding handle decomposition of $N$, we construct an exact symplectic manifold $M$ of dimension $2\,\dim N-2$ together with some exact Lagrangian spheres $V_1, \ldots, V_m \subset M,$ one for each critical point of $f$. (See \S \ref{example} for the construction of $(M, V_1, \ldots, V_m)$ in the simple case $N = \bbR P^2$; see also \S \ref{fibersketch}, \ref{VCsimplesketch} for a sketch of the four dimensional case.)
Given $(M, V_1, \ldots, V_m)$, there is a unique (up to deformation) symplectic Lefschetz fibration $\pi :  E \into \bbC$, with regular fiber $M= \pi^{-1}(b)$ and vanishing spheres  $V_1, \ldots, V_m$ 
(see \cite[\S 16e]{S08}).
%, with respect to suitably chosen vanishing paths $\gamma_1, \ldots, \gamma_m$ in $\bbC$  joining $b$ to the critical values.
%Contrary to what the reader may have hoped, we will not prove that $(E,\pi)$ is isomorphic to
%$(D(T^*N), f_\bbC)$. Instead we will prove 
Theorem $A$ below shows that $(E,\pi)$ and $(D(T^*N), f_\bbC)$ have the same 
salient features. In this sense $(E,\pi)$ is a good model for $(D(T^*N), f_\bbC)$ and $(M, V_1, \ldots, V_m)$ is very likely a correct answer to the above problem. In any case, for this paper and other applications 
(see \cite{JA} and below) 
we need not use $f_\bbC$; instead we will always use $(E,\pi)$.

\begin{named}{Theorem A}[\cite{JA}] Assume $N$ is a smooth closed manifold and  $f:  N \into \bbR$ is self-indexing Morse function with either two, three, or four critical values: $\{0,n\}$, $\{0,n,2n\}$, or $\{0,n,n+1,2n+1\}$. Let $(M,V_1,  \ldots, V_m)$ be the data from the construction we discussed above (which depends in addition on a Riemannian metric $g$ on $N$ such that $(f,g)$ is Morse-Smale).  Let $\pi: E \into \bbC$ be the corresponding symplectic Lefschetz fibration 
with fiber $M$ and vanishing spheres $(V_1, \ldots, V_m)$.
%(which is unique up to deformation equivalence). % with respect to certain vanishing paths.
Then, %$N$ embeds in $E$ as an exact Lagrangian submanifold. Moreover,  
\begin{itemize}
%\item $N \subset E$ is a homotopy equivalence.
\item  $N$ embeds in $E$ as an exact Lagrangian submanifold;
\item all the critical points of $\pi$ lie on $N$, and in fact $Crit(\pi) = Crit(f)$; and,
%\item $\pi(N)$ is a closed subinterval of $\bbR$; and,
\item $\pi|N = f: N \into \bbR$ (up to reparameterizing $N$ and $\bbR$ by diffeomorphisms).
\end{itemize}
\end{named}

The hypotheses on $f$ ensure that the construction of $M$ is relatively easy. If there are five or more critical values then constructing $M$ becomes more complicated, so that is postponed for later treatment. 
In \cite{JA} we also sketch a proof that
$E$ is homotopy equivalent to $N$, and we explain why 
$E$ is expected to be conformally exact symplectomorphic to $D(T^*N)$.
 %In that case $E$ is exact symplectomorphic to $D(T^*\bbR P^2)$ and $\pi = f_\bbC$,  and we conjecture that this  holds in general.
\\
\newline %Except for conceptual motivation Theorem $A$ plays no role in this paper, so we can forget about it if we wish. 
In this paper we consider the simplest nontrivial case, where  $f$ is self-indexing and takes only three critical values $0$,\,$n$,\,$2n$, with a unique maximum and minimum. Furthermore, we focus on the case $\dim N=4$, firstly for the sake of concreteness  and secondly because the stock of examples is most rich in that dimension (see \cite{JA} and the references there). (Everything in this paper can be done in an arbitrary dimension $\dim N =2n$ in a completely analogous way, see  \S  \ref{0,n,2n} for a sketch.) Thus, we assume $f$ has critical points $x_4$,\,$x_2^j$,\,$x_0$, $j=1, \ldots, k$, where the subscript indicates the Morse index, and we choose a Riemannian metric $g$ such that $(f,g)$ is Morse-Smale. 
\\
\newline Using $(N,f,g)$ we will construct $M$ and  $V_4$,$V_2^j$,$V_0$, 
$j=1, \ldots, k$ in a self-contained way (see  \S  \ref{fiber} and  \S  \ref{bigsectionvanishingcycles}).  Then the purpose of this paper is to compute, 
for each $j$, the Lagrangian Floer homology groups  
\begin{gather}
 HF(V_4,V_2^j),\, HF(V_2^j, V_0), \,HF(V_4,V_0)  \label{FloerV}
\end{gather}
%$HF(V_4,V_2^j)$, $HF(V_2^j, V_0)$, $HF(V_4,V_0)$ 
and the triangle product (defined by counting holomorphic triangles in $M$ with boundary on $V_4,V_2^j,V_0$): 
\begin{gather}\label{mu2}
 \mu_2: HF(V_4,V_2^j) \otimes HF(V_2^j,V_0) \into HF(V_4,V_0).
\end{gather}
\\
To carry out these calculations we do not need to know that 
 $M$ and  $V_4$, $V_2^j$, $V_0$ can be viewed as the fiber and vanishing spheres
of a Lefschetz fibration satisfying the conditions in Theorem $A$.
%(so we can basically forget about Theorem $A$ later on). 
Nevertheless, this viewpoint  is very helpful for understanding \emph{why} 
we would want to compute (\ref{FloerV}) and (\ref{mu2}). 
Before discussing that, we 
first explain what the answer is; not surprisingly, it can be expressed
 nicely in terms of the Morse theory of $(N,f,g)$. 
Given $x,y \in Crit(f)$, let $Flow^{\,\circ}(x,y)$ be the space of unparameterized $(f,g)$-gradient trajectories from $x$ to $y$, and let $Flow(x,y)$ denote its compactification, which is obtained by allowing broken trajectories  
(possibly broken many times). This can be viewed a manifold with corners (see \cite{HWZ}) and for any $x,y, z \in Crit(f)$ with decreasing Morse indices, $Flow(x,y) \times Flow(y,z)$  embeds into $Flow(x,z)$ as a boundary face.
%(Both these spaces will be empty unless the Morse index decreases along $x, y, z$.)
The \emph{flow category} of $(N,f,g)$, which first appeared in \cite{CJS},
%\footnote{This construction first appeared in \cite{CJS}. They showed that
%the underlying topological category $\mathcal C$, with morphisms $Flow(x,y)$, is such that the geometric realization$B \mathcal{C}$ is homeomorphic to $N$, provided $(f,g)$ is Morse-Smale. This suggests that our homological version of the flow category is a much finer invariant than the Morse-homology of $(f,g)$.}  of $(N,f,g)$ %, see \cite{CJS}, i
is defined as follows:
\begin{itemize}
\item  The objects are the critical points of $f$.
\item The morphism space from $x$ to $y$ is $H_*(Flow(x,y))$.  
\item Composition $\mu^{Flow}: H_*(Flow(x,y)) \otimes H_*(Flow(y,z)) \into H_*(Flow(x,z))$ is induced by the inclusion $Flow(x,y) \times Flow(y,z)\subset  Flow(x,z)$ combined with the K\"{u}nneth isomorphism.
\end{itemize}

In our case it is not difficult to explicitly compute  the flow category of 
$(N,f,g)$ (see \S \ref{sectionFlow}). The main result of this paper of this 
paper is then as follows.

\begin{named}{Theorem B} Let $(N,f,g)$ and  $(M, V_4, \{V_2^j\}, V_0)$  be as 
above. Then  
%We can comp (\ref{FloerV}) and (\ref{mu2}) as follows: %$HF(V_4,V_2^j) \cong H_*(Flow(x_4, x_2^j))$, $HF(V_2^j, V_0) \cong H_*(Flow(x_2^j, x_0))$, and $HF(V_4,V_0) \cong H_*(Flow(x_4, x_0))$; 
\begin{gather*} 
HF(V_4,V_2^j) \cong H_*(Flow(x_4, x_2^j)), \phantom{bbb} HF(V_2^j, V_0) \cong H_*(Flow(x_2^j, x_0)), \\ 
HF(V_4,V_0) \cong H_*(Flow(x_4, x_0)),
\end{gather*}
and, under this correspondence, the triangle product 
$$\mu_2: HF(V_4,V_2^j) \otimes HF(V_2^j,V_0) \into HF(V_4,V_0)$$
coincides with the composition in the Flow category
$$\mu^{Flow}: H_*(Flow(x_4, x_2^j))\otimes H_*(Flow(x_2^j, x_0)) \into H_*(Flow(x_4, x_0)).$$
\end{named}
See  \S \ref{example} for an explanation of this theorem in the case $N = \bbR P^2$; see also \S \ref{FukFlowsketch} for a sketch of the four dimensional case.
\\
\newline Let us now return to the question of why we would want to compute 
(\ref{FloerV}) and (\ref{mu2}), and why Theorem $B$ is a useful answer.
To this end, we make a brief digression and
 explain  some of the results in Seidel's recent book 
\cite{S08}. 
%keeping in mind the Lefschetz fibration $(E,\pi)$ from  Theorem $A$. 
One of the main results  \cite[Corollary 18.25 plus Proposition 18.14]{S08} %(and \cite[Proposition 18.14]{S08}) 
says that, for any Lefschetz fibration $\pi: E \into \bbC$, the Lefschetz thimbles $\Delta_1, \ldots , \Delta_m \subset E$  in a certain sense \emph{generate} the whole Fukaya category $Fuk(E)$.
%(actually, the derived category of $Fuk(E)$). 
In other words, each exact Lagrangian submanifold
$L \subset E$ can be expressed as a certain algebraic combination of 
$\Delta_1, \ldots , \Delta_m \subset E$ (intuitively, $L$ is geometrically obtained from $\Delta_1, \ldots , \Delta_m \subset E$ by  surgery theoretic 
operations; see \cite{JA}). We will refer to this informally as the \emph{Seidel decomposition} of $L$.
%More precisely, every closed exact Lagrangian $L \subset E$ can be expressed in terms of $\Delta_1, \ldots, \Delta_n$ by repeatedly forming mapping cones of certain morphisms. (Implicitly, this takes place in a context where ``mapping cone'' makes sense, namely the so-called \emph{derived} Fukaya category of $E$; this is a formal expansion which makes it into a triangulated category.)  
This generating property highlights the importance of the sub-category 
with objects $\Delta_1, \ldots , \Delta_m$ and morphisms $CF(\Delta_i, \Delta_j)$.
Now, the Lefschetz thimbles only intersect one another 
in a fixed regular fiber $M$ at points in the
corresponding vanishing spheres $V_1, \ldots,  V_m \subset M$. Therefore, 
we expect $CF(\Delta_i, \Delta_j)$ (which is computed in $E$)  to be equal to
$CF(V_i,V_j)$ (which is computed in $M$). However, because $\Delta_1, \ldots, \Delta_m$ are not closed one needs to  choose some perturbation convention in the definition of  $CF(\Delta_i, \Delta_j)$. The standard choice involves a natural ordering on  $\Delta_1, \ldots , \Delta_m$ given by  the counter clock-wise ordering of the vanishing paths (see \cite[\S 2]{SH=HH} for further informal discussion about this). The upshot is that, in fact, $\Delta_1, \ldots, \Delta_m$ 
form a \emph{directed} sub-category, with morphisms 
\begin{gather} \label{directed}
CF(\Delta_i, \Delta_j) \cong   \left \{ \begin{matrix} CF(V_i, V_j) & \hbox{ if } i > j  \\
                                                                      \mathbb K  & \hbox{ if } i=j \\
                                                                      0  & \hbox{ if } i<j \end{matrix}
                                                                      \right.
\end{gather}
where $\mathbb K$ is the base ring and we think of $CF(\Delta_i, \Delta_i) \cong \mathbb K$ as generated by an identity element $e$. So, we can forget about the Lefschetz thimbles if we wish, and define the \emph{directed Fukaya category} of $(M, V_1,\ldots  ,V_m)$ to be the $A_\infty$ category with objects  
$V_1,\ldots  ,V_m$ and morphisms $Hom(V_i,V_j)$ given by the right-hand side of (\ref{directed}); this is denoted $Fuk^{\rightarrow}(M, V_1,\ldots  ,V_m)$, and it is also known as the Seidel-Fukaya category (see \cite{SI00, SH=HH,  S08} 
 for more about this). 
%What we've said so far indicates that $Fuk^{\rightarrow}(M, V_1,\ldots  ,V_m)$ contains a 
%a good deal of information about $Fuk(E)$ (see  \cite[Corollary 18.25]{S08}). We will return to this at the end of this introduction in the context of more specific potential applications
\\
\newline Thus, what we are computing in  Theorem $B$ is the
directed Fukaya category of $(M, V_4, \{V_2^j\}, V_0)$, but at the level of homology; we call this 
the \emph{directed Donaldson-Fukaya category}, and denote it 
$H(Fuk^{\rightarrow}(M, V_4, \{V_2^j\}, V_0))$. 
(Here there are only three levels in the ordering, corresponding to the three critical values  $0,2,4$; thus all $V_2^1, \ldots,  V_2^k$ are on equal footing.)
So we can re-formulate Theorem $B$ as follows:

\begin{named}{Theorem $B\,'$} The directed Donaldson-Fukaya category  
%$H(Fuk^{\rightarrow}(M, V_4, \{V_2^j\}, V_0))$
of $(M,V_4, \{V_2^j\}, V_0)$ 
is isomorphic to  the flow category of  $(N,f,g)$.
\end{named}
This result was essentially conjectured by Seidel in \cite[\S 8]{SII00}. For an arbitrary Lefschetz fibration it is not realistic to hope for an 
explicit computation of  (\ref{FloerV}) 
and (\ref{mu2})
because the regular fiber and  vanishing spheres are often hard to describe. 
But in our case
%for our very special Lefschetz fibration $(E,\pi)$, 
we have a simple explicit model $(M, V_4, V_2^j, V_0)$
for these, and the vanishing spheres $V_4, V_2^j, V_0 \subset M$ 
are also highly symmetric; this is what makes Theorem 
$B$ or $B'$ feasible.
\begin{remark} \label{chainlevel}
In our  situation, the directed Donaldson-Fukaya category, given by (\ref{FloerV}) and (\ref{mu2}), is actually very close to the chain level version.
% $Fuk^{\rightarrow}(M, V_4, \{V_2^j\}, V_0)$. 
This is because 
%there are only three levels in the directed $A_\infty$ category, so 
there are only two products: $\mu_1$ (the Floer differential), and $\mu_2$; all higher order products $\mu_3, \mu_4, \ldots $ are zero, because there are only three levels in 
the directed category.
%In any case, for practical purposes in applications (see below) it will be enough to have just the homology level information (\ref{FloerV}) and (\ref{mu2}).
 \end{remark}
Thus,  Theorem $B$ or $B'$ has a certain amount of intrinsic interest, as it 
provides a simple description of an important invariant attached to
some natural Lefschetz fibrations on  cotagent bundles 
(assuming $E \cong D(T^*N)$ in Theorem $A$ for simplicity). 
The answer  moreover reflects 
the expected close relationship between the complex geometry of 
$(E,\pi)$ and the Morse theory of $(N,f,g)$ (recall $\pi$ is 
a model for the complexification  of $f$).
In the larger scheme it fits in with many other results 
%re is a long history of results of this kind, 
% result falls into step with many others 
relating Floer theoretic invariants of cotangent bundles $T^*N$ to 
 more classical invariants of the base $N$, as in  
\cite{FloerMorse, FO, AbS, V, N, NZ}.
%and we will see in amoment that Theorem $B'$ has a particularly close 
%relationship to the last  mentioned paper. 
\\
\newline On a more practical level there are a couple of reasons why this 
explicit answer in terms of the flow category is of interest as well. 
The most immediate application we have in mind is to use Seidel's decomposition
(which we discussed above) in combination with Theorems $A$ and $B$ 
to study Lagrangian submanifolds in $T^*N$. One basic goal is to prove 
%for certain $N$ 
that any closed exact Lagrangian $L \subset T^*N$ is Floer theoretically equivalent to $N$.  This means in particular that $HF(L,L) \cong HF(N,N)$, so that $H^*(L) \cong H^*(N)$, and
$HF(L, T_x^*N) \cong HF(N, T_x^*N)$, so that $deg(L \into N) = \pm 1$.
Of course, results of this kind have been obtained for arbitrary spin manifolds $N$ in \cite{FSS, FSSB} and \cite{N} (for simply-connected $N$). We want to  consider a slightly different approach along the lines of the quiver-theoretic approach for the case $N = S^n$ in \cite{S04}. Because this approach is  
more explicit it is expected to  yield somewhat more refined results. 
For example, we  should be able  to remove one significant assumption on $L$, 
namely that it has vanishing Maslov class $\mu_L \in H^1(L)$. 
\\
\newline Here is a concrete example which illustrates how Theorem $B$ makes
Seidel's decomposition very explicit. Take $N = \bbC P^2$ with its standard handle-decomposition, corresponding to a Morse-Smale pair $(f,g)$, where $f$ has three critical points $x_0, x_2, x_4$, with Morse indices $0,2,4$.
Then the flow category of $(N,f,g)$ is relatively easy to  compute (see \S \ref{sectionFlow});
%--the answer is given by (\ref{quiver}) below)\
it is given by the following quiver with relations: 
%\vspace{-0.5cm}
\begin{gather}\label{Flowquiver}
\xymatrixcolsep{5pc}\xymatrix{x_4 \ar@/^0.25pc/[r]^{a_1} \ar@/_0.25pc/[r]_{a_0} \ar@/^1.5pc/[rr]^{c_1} \ar@/_1.5pc/[rr]_{c_0} & x_2 \ar@/^0.25pc/[r]^{b_1} \ar@/_0.25pc/[r]_{b_0} & x_0} \\
b_1a_1 =0, \,  b_0a_0 = c_0, \,  b_0a_1 - b_1a_0 = c_1 \notag
\end{gather}
%\vspace{-0.5cm}
By Theorem $A$, there is a Lefschetz fibration $(E,\pi)$ corresponding to $(N,f,g)$, which models the complexification of $f$ on $D(T^*N)$. By construction, $\pi$ comes with  an explicit regular fiber $M$ and vanishing spheres $V_0, V_2, V_4 \subset M$. 
Theorem $B$ or $B'$ says that the directed Donaldson Fukaya category of $(M, V_4, V_2, V_0)$ is represented 
by the exact same quiver (\ref{Flowquiver}), except we replace the labels $x_4, x_2, x_0$ on the vertices by $V_4, V_2, V_0$:
\begin{gather}\label{Fukquiver}
\xymatrixcolsep{5pc}\xymatrix{V_4 \ar@/^0.25pc/[r]^{a_1} \ar@/_0.25pc/[r]_{a_0} \ar@/^1.5pc/[rr]^{c_1} \ar@/_1.5pc/[rr]_{c_0} & V_2 \ar@/^0.25pc/[r]^{b_1} \ar@/_0.25pc/[r]_{b_0} & V_0} \\
b_1a_1 =0, \,  b_0a_0 = c_0, \,  b_0a_1 - b_1a_0 = c_1 \notag
\end{gather}
Now, let $L \subset T^*N$ be any closed exact Lagrangian submanifold.  Theorem $A$ also says  we have an exact Lagrangian embedding $N \subset E$. Consequently there is an exact Weinstein embedding $D(T^*N) \subset E$. By rescaling $L \leadsto \epsilon L$ by some small $\epsilon >0$, we get an exact Lagrangian embedding $L \subset E$. Now the Seidel decomposition applied to $L \subset E$ says that $L$ can be expressed in terms of the Lefschetz thimbles of $\pi$, say $\Delta_4, \Delta_2, \Delta_0 \subset E$. But as we discussed earlier, this is equivalent to saying 
that $L$ can be expressed in terms of the
directed Fukaya category of the corresponding vanishing spheres $V_4, V_2,V_0$, given by the quiver (\ref{Fukquiver}). In our  example this has the following concrete meaning:  
$L$  is represented by a certain \emph{quiver representation} of (\ref{Fukquiver}):
\begin{gather}\label{rep}
\xymatrixcolsep{5pc}\xymatrix{W_2 \ar@/^0.25pc/[r]^{A_1} \ar@/_0.25pc/[r]_{A_0} \ar@/^1.6pc/[rr]^{C_1} \ar@/_1.5pc/[rr]_{C_0} & W_1 \ar@/^0.25pc/[r]^{B_1} \ar@/_0.25pc/[r]_{B_0} & W_0} \\
B_1A_1 =0, \, B_0A_0 = C_0, \, B_0A_1 - B_1A_0= C_1 \notag
\end{gather}
%This quiver represents $L$ in the sense that all Floer theoretic questions about $L$ can in principle be answered with knowledge of this
Here, a quiver representation, such as (\ref{rep}), is just a choice of  
vector-spaces  $W_4$, $W_2$, $W_0$ at each vertex, and a choice of linear maps $A_0, A_1, B_0, B_1,C_0, C_1$ satisfying the given 
relations (the particular quiver representation (\ref{rep}) 
corresponding to $L$ 
is determined by $W_i = HF(L, \Delta_i)$ and the triangle products
$W_i \otimes HF(\Delta_i, \Delta_j) \into W_j$).
%In (\ref{rep}), these data come from further Floer theoretic computations, namely:%$$W_i = HF(L, \Delta_i), i=0,2,4$$
%and  $A_0, A_1, B_0, B_1,C_0, C_1$ are determined by the triangle products
%\begin{gather}\label{othertriangle}
%HF(L, \Delta_i) \otimes HF(\Delta_i, \Delta_j) \into  HF(L, \Delta_j).\notag \end{gather}
To show  $L$ is Floer theoretically equivalent to $N$ in  $T^*N$ is equivalent to showing that 
the representation (\ref{rep}) is necessarily isomorphic to the representation 
$$W_4 = W_2 = W_0 = \bbC, \, A_0= B_0 = C_0 = id, \, A_1= B_1 =C_1 =0.$$ 
(Of course, this is the representation corresponding to $N \subset T^*N$.)
The analogous problem for $N = S^n$ was solved in \cite{S04}. 
%Work on this and related problems is currently in progress.
\begin{comment}
To round off the discussion we mention one other much more tentative 
 potential applications; these stem from the flow category's close relationship to the topology of $N$ and the category of constructible sheaves on $N$ 
respectively. First, we recall that in \cite{CJS} the authors consider a topological version of the flow category, say $Flow_{top}(N,f,g)$, where one has morphism spaces $Flow(x,y)$ rather than $H_*(Flow(x,y))$. In that paper they showed that $Flow_{top}(N,f,g)$ 
is a strong enough invariant to recover the homeomorphism type of $N$.
Algebraically, Sullivan  conjectures that a chain 
level version of $Flow(N,f,g)$ (which is a DG category) 
can be viewed as an ``$A_\infty$ Frobenius category'' 
(an algebraic structure which is yet to be defined precisely); moreover he
conjectures that this object should recover the homeomorphism type of 
$N$ as well \cite{Sul}. Thus, using some variant of Theorem $B'$, 
one might hope to extract from the symplectic topology of $T^*N$
a certain amount of topological information about the base $N$ 
(like the homeomorphism type). Note, however, that   
the directed Fukaya category really depends on the Lefschetz fibration 
(and not just $T^*N$), though one might be able to do something analogous to
 \cite{SH=HH} to remove this dependence.
\end{comment}
\\
\newline %Now let us turn to a more feasible application of Theorem $B'$.
To round off the discussion we mention one other potential 
application which stems from the close relationship
between the  Flow category of $N$ 
and the category of constructible sheaves on $N$.
%we discuss the relationship between 
Our aim is to describe the relationship between two
 recent successful approaches for analyzing the Fukaya 
category of cotagent bundles (see \cite{FSSB} for a detailed comparison).
One approach, due to Fukaya, Seidel, and Smith  \cite{FSS}, is based on Lefschetz fibrations and Picard-Lefschetz theory, 
as in \cite{S08}. The other approach, due to Nadler and Zaslow  \cite{NZ, N},
 relates Lagrangian submanifolds in $T^*N$ to constructible sheaves on $N$,
  and is based on  the characteristic cycle construction of 
Kashiwara-Shapira \cite{KS}.
(There is also a third approach in progress, based on a
Leray-Serre type spectral sequence, see \cite{FSSB}, and above we 
discussed a fourth quiver-theoretic approach which expands on \cite{S04}.)
Now, if $N$ is real analytic, 
the  flow category of $N$ 
(more precisely, the derived category of the chain level version)
is expected to be equivalent to the derived category
of constructible sheaves on $N$ (constructible with respect to the stratification by unstable manifolds of $f$), see \cite[Remark 7.1]{SII00}. 
This conjecture of Seidel 
combined with  Theorem $B'$ (extended slightly to the chain
level as in remark \ref{chainlevel}) leads to a direct way of comparing 
the two viewpoints \cite{FSS} and \cite{NZ, N} described above. 
More precisely, we have the following 
(commutative) diagram, where each arrow is either an isomorphism or an 
equivalence (below  the $D$'s are for derived categories).
\begin{gather*}
\xymatrixcolsep{5pc}\xymatrix{ D(Fuk_S(T^*N)) \ar@{-->}[r]\ar[d]^{\text{Seidel decomposition}} & D(Fuk_{NZ}(T^*N)) \\
D(Fuk^\rightarrow(M,\{V_i\})) \ar[d]^{\text{Theorem } B'}   \\
 D(Flow(N,f,g))\ar[r]_{\text{Seidel conjecture}} & D(Constr(N)) \ar[uu]_{\text{Nadler-Zaslow}} }
\end{gather*}
%Here the $D$'s are for derived categories. 
Here, $Fuk_{NZ}(T^*N)$ and $Fuk_S(T^*N)$  denote two versions of the 
Fukaya category of $T^*N$, due to Nadler-Zaslow and Seidel respectively.
The difference lies in the choice of noncompact Lagrangians they allow
(see \cite{FSSB}).  It is believed that their derived categories
are equivalent, but to the author's knowledge no proof has been nailed down
yet (this equivalence is 
of interest in particular for applications to the homological
 mirror symmetry conjecture \cite{ZA,ZB}).
One consequence of the above diagram is this desired equivalence 
(the top dotted arrow), though it may be easier and more enlightening to 
find a direct correspondence.  Assuming the existence of 
such a natural correspondence, and replacing the dotted arrow by that, 
it is also interesting to ask if the above diagram commutes; the diagram
represents  a parallelism between Picard-Lefschetz theory and the 
characteristic cycle functor.

\subsection*{Overview and organization} In this paper our focus is on giving a precise proof of Theorem B (see Proposition \ref{generators} and Theorem \ref{fuk=flow}). The crucial %fact which makes the computation possible 
point is that $V_4$,$V_2^j$,$V_0$ have a certain rotational symmetry, and this allows us to compute the moduli space of holomorphic triangles explicitly.
On the other hand, this means we  cannot perturb $V_4$,$V_2^j$,$V_0$ to make them transverse, as one usually does to compute Floer homology, and therefore the whole calculation must be done in the context of Morse-Bott Floer homology. In that theory
the  Lagrangian submanifolds are allowed to intersect in a Morse-Bott
fashion along submanifolds, and the Floer complex is generated by
cycles in the intersection components. We use Morse cycles as in
\cite{PSS}, \cite{Schwarz}, \cite{Fr}, since that is the cleanest
known approach. In the interest of brevity we do not treat signs and
gradings for Morse-Bott Floer homology in this paper. Consequently,
Theorem B is limited for the moment to $\bbZ/2$ coefficients and ungraded Floer groups.
\\
\newline
%\subsection*{Organization} 
The paper is basically split into three
parts. The first part (\S \ref{example}--\ref{outline}) 
outlines the main lines of argument in two warm-up sections:
\S \ref{example} treats the simple case $N = \bbR P^2$, and \S \ref{outline}
 outlines the paper in detail in the case $N = \bbC P^2$. 
In \S \ref{sectionFlow} we compute the flow category of $(N,f,g)$
under the assumptions of Theorem $B$.
 From these (especially \S \ref{outline}) the reader can get a very good idea
of the proof of Theorem $B$, including all the pitfalls one has to watch for.
 The rest of the paper gives the complete details.
The second part (\S \ref{fiber}--\ref{sectiontriangleproduct}) contains  
the main line of  argument. Along the way we refer to
%In \S \ref{fiber} we construct the fiber $M$. In \S 
%\ref{bigsectionvanishingcycles} we construct the
%vanishing spheres $V_4$,$V_2^j$,$V_0$ in $M$.
%In \S \ref{sectionFloer} we compute the Floer homology groups
%In \S \ref{triangleproduct} we compute the triangle product
%prove Theorem $B$. 
%Along the way we refer to two later sections \S
%\ref{localstrips} and \ref{localtriangles} which are the backbone of this paper from atechnical point of view.
the third part (\S \ref{localstrips}--\ref{MorseBott}), which is
where the main technical work
is done: In  \S \ref{localstrips}  we prove 
that the relevant moduli spaces of holomorphic strips are 
in 1-1 correspondence with the gradient flow
lines of certain functions;  %under certain conditions %; of course this is 
%used to compute the Floer homology groups.
in  \S  \ref{localtriangles} we explicitly describe  the moduli space of holomorphic triangles (inside a Weinstein neighborhood $D(T^*V_2^j) \subset M$ 
of each $V_2^j$).
%; this is the material we use to compute the triangle product.
In  \S  \ref{MorseBott} we provide an appendix which explains the definition and
main features of Morse-Bott Floer homology in some special cases
sufficient for our purposes.
 The paper is basically written in the intended order of reading, with
the exception of the last section, which should be
referred to as necessary. 
Some Floer theory notation and conventions are in sections \ref{conventions}, \ref{notation}. We use the notation $\nu^*K \subset T^*N$ for the conormal bundle of $K \subset N$.
\\
\newline \subsection*{Acknowledgements} The results of this paper are based on my Ph.D. thesis, carried out at the University of Chicago from 2003-2006 under the supervision of Paul Seidel. Of course, this paper owes much to him. Thanks also go to Fr\'{e}d\'{e}ric Bourgeois, Kenji Fukaya,  Kaoru Ono, and Helmut Hofer for help with Fredholm theory 
in Morse-Bott Floer homology, to Felix Schm\"{a}schke for
pointing out several notational mistakes in an earlier draft, 
and to Urs Frauenfelder for patiently answering many questions about 
Morse-Bott homology. An extra big
thanks goes to Peter Albers for frequent sanity checks and many
helpful discussions about Floer theory. 

\section{A simple example: $N = \bbR P^2$} \label{example}
To illustrate the main ideas of this paper, we first consider the two dimensional example $N = \bbR P^2$, where we take the self-indexing standard Morse function $f: \bbR P^2 \into \bbR$ which has 3 critical points $x_2,x_1,x_0$, with the standard flow lines as in figure \ref{RP2Morsefn}. In this section we will explain several things in this simple example. First, we explain  
a natural construction of the complexification of $f$ on the disk cotangent-bundle of $D(T^*\bbR P^2)$ from elementary algebraic geometry; it is quite different from the construction we give for $(E, \pi)$ in \cite[Theorem A]{JA}, but it gives a useful additional viewpoint. (In most examples this approach would not work.) After that we explain the  construction of $(M,V_2, V_1, V_0)$ (see figure \ref{RP2fiberVC}) using the topological method we will use in this paper (and in \cite{JA}). We also briefly explain how  the construction relates to Morse theory, and we explain roughly why the vanishing spheres are defined as they are. 
We note that $(M,V_2, V_1, V_0)$ is isomorphic to the regular fiber and vanishing spheres of algebreo-geometrically constructed Lefschetz fibration from before. Finally, we sketch the proof of Theorem $B$ in this simple low-dimensional case by inspecting $(M,V_2, V_1, V_0)$ and the flow lines of $f$ on $\bbR P^2$.

\subsection{An algebreo-geometric construction of the complexification} %of $f$ on $D(T^*N)$.}

The complexification $f_\bbC: D(T^*\bbR P^2) \into \bbC$ can be viewed as the restriction of a 
well-known Lefschetz pencil  in algebraic geometry.
 %(In the general cases treated in \cite{J}, this is of course approached in a more abstract way.) 
Namely, in  \cite[p. 19]{A} (or \cite[p.7,27]{SII00}, or \cite[p. 39]{AS}) one finds the example of a Lefschetz pencil on $\bbC P^2$ formed by two homogeneous degree 2 polynomials $\sigma_0, \sigma_1$ with real coefficients.
The base locus $B = \{\sigma_0 = \sigma_1 =0\}$ consists of four points. 
If we delete a neighborhood $U$ of the fiber at infinity $\{\sigma_0=0\}$ 
(we assume $\sigma_0$ has no zeros on $\bbR P^2$) we will in particular delete this base locus from every fiber. Thus we obtain a Lefschetz fibration 
$$\pi = \sigma_1/\sigma_0 : \bbC P^2 \setminus U \into \bbC.$$
Because  $\sigma_0, \sigma_1$ have degree 2, the regular fiber of the pencil is a 2-sphere.
When we delete a neighborhood of the base locus the regular fiber becomes 
$S^2$ with four small disks removed; this is the regular fiber of  $\pi = \sigma_1/\sigma_0$ on $\bbC P^2 \setminus U$. There are three singular fibers each of which consists of 
a pair of  2-spheres which touch at a single point; each of the two 2-spheres contains two of the four base points. The three singular fibers correspond to the three ways of splitting up the four base into pairs.
Thus, there are three vanishing spheres in the regular fiber, which divide the four holes in the fiber into pairs in all three possible ways, see figure \ref{pencilfiber} (see also \cite[p. 19]{A} or \cite[p. 39]{AS}).
\begin{figure}
\begin{center}
\includegraphics[width=2in]{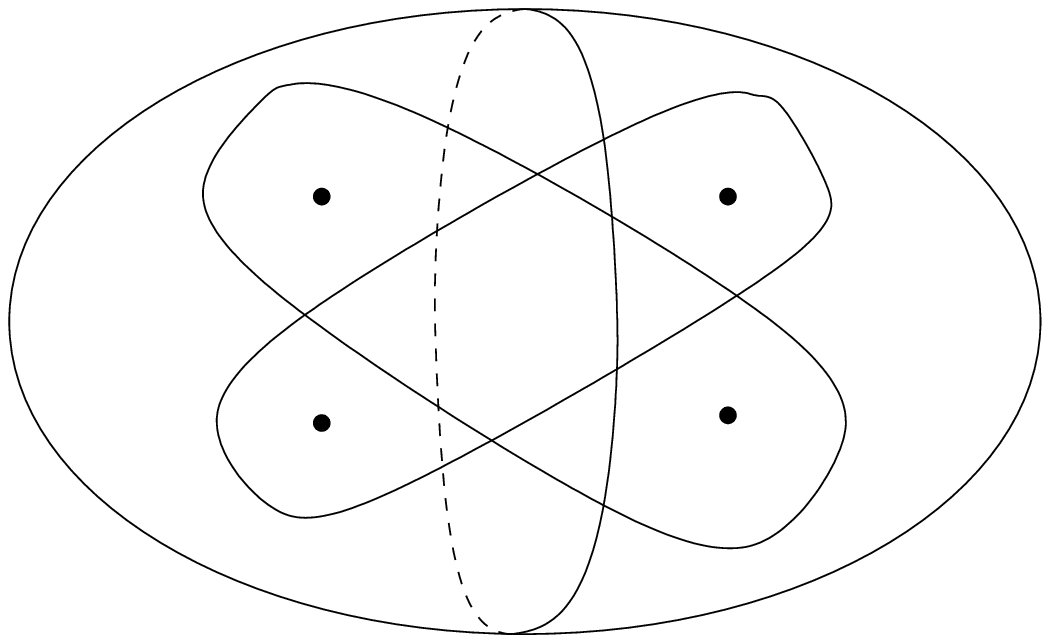} 

\caption{The regular fiber of the pencil $\pi = \sigma_1/\sigma_0$ (with the four base points deleted)
and the three vanishing spheres.}
%(We superimpose the two crosses.)}
\label{pencilfiber}
\end{center}
\end{figure}
To see that $\pi$ is a complexification, consider
$$f = \pi|\bbR P^2: \bbR P^2 \into \bbR.$$ If $\sigma_0$, $\sigma_1$ are chosen suitably then the critical points of $\pi$ lie on $\bbR P^2$ and $f$ isotopic to the standard Morse function with three critical points. Furthermore one can show using the Liouville flow that there is a Weinstein tubular neighborhood 
of $\bbR P^2 \subset \bbC P^2 $ which fills out all of $ \bbC P^2 \setminus U$; thus $ \bbC P^2 \setminus U \cong D(T^*\bbR P^2)$ and $\pi$ is the complexification of $f$.

\subsection{Topological construction of  $M$ }

Take a Riemannian metric $g$ on $\bbR P^2$ such that  $(f,g)$ is a Morse-Smale pair, where $f$ has 3 critical points $x_2,x_1,x_0$, with the standard flow lines as in figure \ref{RP2Morsefn}.
\begin{figure}
\begin{center}
\includegraphics[width=2in]{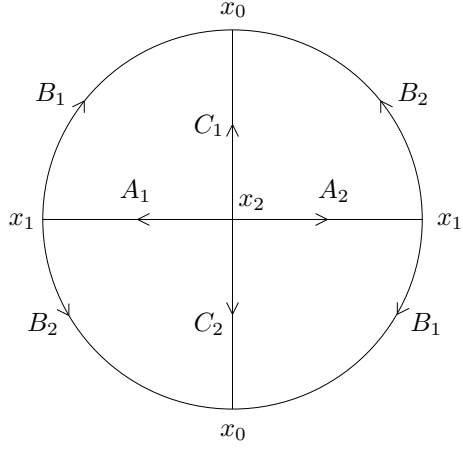} 
%\begin{comment} %this one is the newer re-sized labels
\put(-87,105){$C_1$}
\put(-87,30){$C_2$}
\put(-147,117){$B_1$}
\put(-10,117){$B_2$}
\put(-150,30){$B_2$}
\put(-5,30){$B_1$}
\put(-115,80){$A_1$}
\put(-40,80){$A_2$}
\put(-77,150){$x_0$}
\put(-77,-10){$x_0$}
\put(-70,77){$x_2$}
\put(-157,70){$x_1$}
\put(5,70){$x_1$}
%\end{comment}
%\put(-90,110){$C_1$}
%\put(-124,50){$C_2$}
%\put(-170,117){$A_1$}
%\put(-60,117){$A_2$}
%\put(-220,160){$B_1$}
%\put(-8,160){$B_2$}
%\put(-220,50){$B_2$}
%\put(-8,50){$B_1$}
%\put(-105,112){$x_2$}
%\put(-113,220){$x_0$}
%\put(-113,-7){$x_0$}
%\put(-230,107){$x_1$}
%\put(3,107){$x_1$}
\caption{The flow lines of $(f,g)$ on $\bbR P^2$.}
\label{RP2Morsefn}
\end{center}
\end{figure}
To define $M$, let
$V_0 = V_1 = S^1$.
Then take the unit disk bundles with respect to the standard metrics on the cotangent bundles,
$D(T^*V_0)$ and $D(T^*V_1)$. The short answer to constructing $M$ is to take 
two copies of $S^0$, say $K_0 \subset V_0$ and $K_1 \subset V_1$ and plumb
$D(T^*V_0)$ and $D(T^*V_1)$ together so that $V_0$ and $V_1$ meet along $S^0 \cong K_0 \cong K_1$.
As usual, plumbing means we identify neighborhoods of $K_0$ and $K_1$ in 
$D(T^*V_0)$ and $D(T^*V_1)$ by identifying the fiber direction in one with the  base direction in the other. More precisely, we take tubular neighborhoods of $K_0, K_1$, say 
$S^0 \times D^1 \subset V_0$,  $S^0 \times D^1 \subset V_1$
and we  trivialize the disk bundles $D(T^*V_0)$ and $D(T^*V_1)$ over these neighborhoods
$$D(T^*V_0)|_{(S^0 \times D^1)}, \, D(T^*V_1)|_{(S^0 \times D^1)} \cong S^0 \times D^1 \times D^1.$$
Then we glue $D(T^*V_0)$ to $D(T^*V_1)$ along $D(T^*V_0)|_{S^0 \times D^1}$ and 
$D(T^*V_1)|_{S^0 \times D^1}$ using the map 
$(x,y) \mapsto (-y,x): D^1 \times D^1 \into D^1 \times D^1.$
(We have a minus sign to ensure this map is symplectic.)
\\
\newline Depending on the orientations of the  identifications 
\begin{gather}\label{ident}
D(T^*V_0)|_{S^0 \times D^1}, \, D(T^*V_1)|_{S^0 \times D^1} \cong S^1 \times D^1 \times D^1,\end{gather}
there may be twists (like a M\"obius strip) in $D(T^*V_0)$ or $D(T^*V_1)$ after we plumb them together. For $N = \bbR P^2$ it turns out that there should be \emph{no twists} in $D(T^*V_0)$ or $D(T^*V_1)$. (See below for how (\ref{ident}) is determined by Morse theory.)
Thus $M$ is homeomorphic to $S^2$ with four small open disks removed (see figure \ref{RP2fiberVC}). We say homeomorphic, rather than symplectomorphic,  because $M$ has some corners. To remedy that, one can alternatively construct $M$ so that it has smooth contact type boundary by 
attaching two 1-handles to  the boundary $D(T^*V_0)$ using Weinstein's technique \cite{W}.
In this case the 1-handles should be attached to $\partial D(T^*V_0)$ along 
the boundary of the disk conormal bundle $\partial D(\nu^*K_0) \subset \partial D(T^*V_0)$. 
\\
 \newline Let us now explain the connection to Morse theory. Above, $V_0$ and $V_1$ represent the vanishing spheres of $f_\bbC$ corresponding to $x_0$ and $x_1$. Let us assume that $f$ is self-indexing so that it has
two regular level sets $f^{-1}(1/2)$ and $f^{-1}(3/2)$. The first step in relating the 
Picard-Lefschetz data to the Morse theory data is to identify
$$V_0 = f^{-1}(1/2) \cong S^1.$$
%This corresponds to choosing the basepoint $b =1/2$ for $f_\bbC: D(T^*N) \into \bbC$ and the vanishing path $[0,1/2]$ for $x_0$; in that case the unstable manifold
%$f^{-1}([0,1/2])$ coincides with the Lefschetz thimble over   $[0,1/2]$, and so the vanishing cycle is  $V_0 = f^{-1}(1/2)$.
Now $(f,g)$ determines the standard handle-decomposition of $\bbR P^2$ with three handles of index $0,1,2$. The 1-handle has an attaching sphere 
$$K_0 = S^0 \subset S^1 \cong f^{-1}(1/2)= V_0.$$ 
Moreover, we have  a tubular neighborhood of $K_0 \subset V_0$
$$\phi: S^0 \times [-1,1] \into S^1$$
determined up to isotopy by the framing we use to attach the 1-handle.
Then $\phi$ in turn determines an exact symplectic identification
\begin{gather}
 \widehat \phi: D(T^*V_0)_{(S^0 \times [-1,1])} \into S^0 \times D^1 \times D^1, \label{plumbingtriv}
\end{gather}
where $D^1 \times D^1 \subset (\bbR^2, dy \wedge dx)$, and  the coordinates are $(x,y) \in \bbR^2$.
On the other hand $V_1 = S^1 \subset \bbR^2$ has two standard $S^0$'s 
$$K_+^1 = \{ (\pm 1, 0 ) \} \text { and } K_-^1 =  \{ ( 0, \pm 1) \}.$$
($K_-^1$  will be used in the construction of $M$; $K_+^1$ will be used in the construction of $V_2$.)
We take the canonical orientation preserving tubular neighborhood of $K_-^1$, 
$S^0 \times [-1,1] \subset S^1$ and the corresponding
exact symplectic trivialization
$$D(T^*V_1)|_{(S^0 \times [-1,1])} \cong S^0 \times D^1 \times D^1$$
which we use to plumb $D(T^*V_0)$ and $D(T^*V_1)$ together.
The main point we wanted to make was, first, that the $S^0$ corresponds precisely to the
intersection of the unstable manifold $U(x_1)$ with the level set 
$V_0 = f^{-1}(1/2)$, and second that the framing of the corresponding 1-handle determines
the trivialization (\ref{plumbingtriv}). (On the other hand, the corresponding trivialization 
for $D(T^*V_1)$ is always the same; this is analogous to the set up in a handle attachment.)

\subsection{Construction of the vanishing spheres}

We have already defined $V_0, V_1$. We now explain how to define $V_2$ (see figure \ref{RP2fiberVC}).
The method we present here is the same as the one that we use in general. Let $\Phi$ denote the time $\pi/2$ geodesic flow on $D(T^*V_1)$ (which is Hamiltonian). Consider 
the disk conormal bundle $D(\nu^*K_+^1) \subset D(T^*V_1)$ and set $H = \Phi(D(\nu^*K_+^1))$. Since $\Phi$ fixes points in the zero section, and moves covectors $v$ of length 1 a distance $\pi/2$ in the direction of $v$, we have $\partial H = \partial D(\nu^* K_-^1)$ and $H \cap V_1 = K_+^1$. 
Now tweak $\Phi$ slightly to get a new Hamiltonian diffeomorphism $\widetilde \Phi$ such that
$\widetilde H = \widetilde \Phi(D(\nu^*K_+^1))$ agrees with $D(\nu^* K_-^1)$ in a neighborhood of 
$\partial D(\nu^* K_-^1)$.
Then we define 
$$V_4 = (V_0 \setminus D(\nu^*K_-^1)) \cup \widetilde H.$$ 
See figure \ref{RP2fiberVC}.
\begin{figure}
\begin{center}
\includegraphics[width=5in]{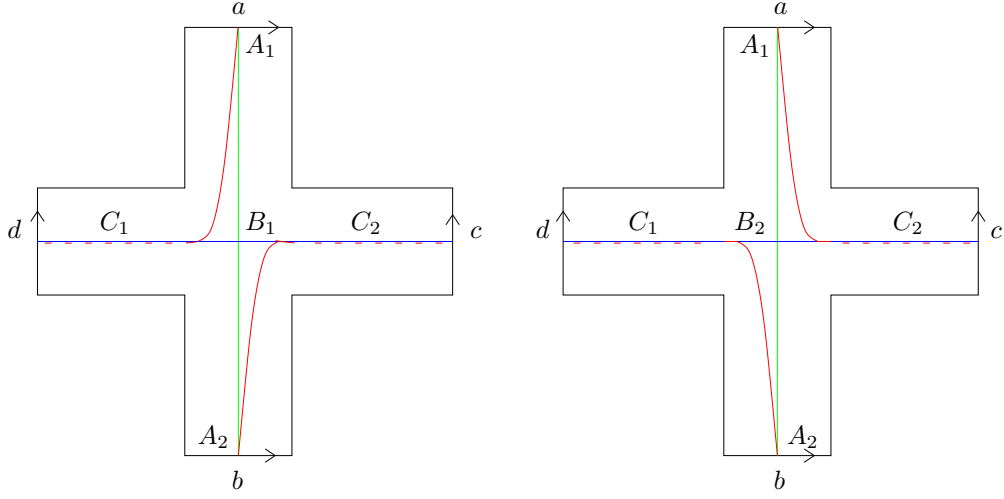} 
%\begin{comment}
\put(-80,170){$a$}
\put(-285,170){$a$}
\put(-80,-10){$b$}
\put(-285,-10){$b$}
\put(-170, 85){$d$}
\put(-195, 85){$c$}
\put(2, 85){$c$}
\put(-370, 85){$d$}
\put(-93,155){$A_1$}
\put(-280,155){$A_1$}
\put(-75,7){$A_2$}
\put(-298,7){$A_2$}
\put(-135, 88){$C_1$}
\put(-95, 88){$B_2$}
\put(-240, 88){$C_2$}
\put(-280, 88){$B_1$}
\put(-35, 88){$C_2$}
\put(-335, 88){$C_1$}
%\end{comment}
\caption{The fiber $M$ and vanishing spheres $V_2$ (red) $V_1$ (green), $V_0$ (blue) for $(\bbR P^2,f,g)$. The edges labeled by letters ($a,b,c,d$) are identified according to the indicated orientations. The other labels correspond to generators in Floer homology
$A_1, A_2 \in HF(V_2, V_1)$, $ B_1, B_2 \in HF(V_1,V_0)$, $C_1, C_2 \in HF(V_2, V_0)$.}
%(We superimpose the two crosses.)}
\label{RP2fiberVC}
\end{center}
\end{figure}
Notice that $V_4$ is obtained by surgery on $V_0$ along the framed attaching sphere $K_0 \subset V_0$,
just as $f^{-1}(3/2)$ is obtained by surgery on $f^{-1}(1/2)$. This is not a coincidence, as we now explain. 
\\
\newline In this simple example it is worth giving a brief sketch of where the construction of fiber and vanishing spheres come from (see \cite{JA} for details). One should consider $M$ as being a model for the regular fiber of $f_\bbC: D(T^*N) \into \bbC$
at the base point $b = 1/2$. And one should think of $(V_0,V_1, V_2)$ as being a model for the vanishing spheres in $M$ relative to the vanishing paths $\gamma_0, \gamma_1, \gamma_2$, as in figure \ref{Fukvanishingpaths}. Namely, $\gamma_0$ parameterizes $[0,1/2]$, $\gamma_1$
parameterizes $[1/2,1]$, and $\gamma_2$ goes around the critical value $1$, by a half loop in the lower half plane, and then continues along the interval $[3/2,2]$.
 \begin{figure}
\begin{center}
\includegraphics[width=3in]{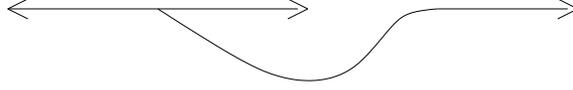} 

\caption{The three vanishing paths $\gamma_0, \gamma_1, \gamma_2$.}
%(We superimpose the two crosses.)}
\label{Fukvanishingpaths}
\end{center}
\end{figure}
The fact that $\gamma_0$ and $\gamma_1$ are straight line segments is reflected in the simplicity of $V_0$ and $V_1$. Also the identification of $V_0$ with $f^{-1}(1/2)$ makes sense in view of the fact the stable manifold  of $x_0$ over $[0,1/2]$ is all of $f^{-1}([0,1/2])$ and so 
$f^{-1}([0,1/2])$ coincides with the Lefschetz thimble over $[0,1/2]$.
On the other hand the half loop in $\gamma_2$ accounts for the twist in $V_2$. Indeed, if instead of $b= 1/2$ we had the base point at $b' = 3/2$, with $\gamma_2'$ parameterizing $[3/2,2]$,
then the corresponding vanishing sphere  $V_2'$ would be identified with $f^{-1}(3/2)\subset f_\bbC^{-1}(3/2)$ and the Lefschetz thimble 
would be the unstable manifold $f^{-1}([3/2,1])$. Thus, $V_2' \subset f_\bbC^{-1}(3/2)$ would look similar to $V_0$. But because $V_2 \subset f_\bbC^{-1}(1/2)$ is actually obtained from $V_2'  \subset f_\bbC^{-1}(3/2)$ by parallel transport along the half loop, $V_2$ is twisted by something like a ``half Dehn twist''. 
\\
\newline We conclude this section by pointing out that since $V_2$,$V_1$,$V_0$ divide the four holes of $M$ into pairs in all three possible ways, the fiber and vanishing spheres we have constructed are equivalent to those of the algebreo-geometric example above.

\subsection{Comparing the Flow category and the directed Fukaya category} % of $(M,V_2,V_1,V_0)$}
In the case $N =\bbR P^2$ Theorem $B$ was explained already in \cite[p.\,9,27]{SII00}. The flow category of $(\bbR P^2, f,g)$ has three objects $x_2,x_1,x_0$ and the space of  morphisms from $x_i$ to $x_j$, for every $i >j$, is generated by two elements (both in degree 0). For example, from figure \ref{RP2Morsefn} one sees that $Flow(x_2,x_0)$ is diffeomorphic to the disjoint union of two closed intervals $I_1,I_2$, where the gradient flow lines corresponding to $I_1$ (resp. $I_2$) fill out the top half (resp. bottom half) of the disk in figure \ref{RP2Morsefn}. 
%Thus the generators of $H_*(Flow(x_4,x_0))$ correspond to the two components of $f^{-1}(-1/2) \setminus U(x_1) \cap f^{-1}(-1/2)$. 
Let $A_1, A_2$, $B_1,B_2$, $C_1,C_2$ denote the generators as labeled in figure \ref{RP2Morsefn}. 
Thus the boundary of $I_1$ corresponds to the 
broken trajectories $B_2 \circ  A_2 $ and $B_1\circ A_1$, and
similarly $\partial I_2 \cong \{B_1\circ  A_2, \, B_2 \circ A_1\}$. 
Thus $B_1\circ  A_2$ is homologous to $B_2 \circ A_1$ in $Flow(x_2,x_0)$, and similarly  $B_2 \circ  A_2 \backsim B_1\circ A_1$. Therefore  the flow category can be described as the following  quiver with relations: 

$$\xymatrixcolsep{5pc}\xymatrix{x_2 \ar@/^0.5pc/[r]^{A_1} \ar@/_0.5pc/[r]_{A_2} \ar@/^2.5pc/[rr]^{C_1} \ar@/_2.5pc/[rr]_{C_2} & x_1 \ar@/^0.5pc/[r]^{B_1} \ar@/_0.5pc/[r]_{B_2} & x_0}
$$

$$B_2 A_2 = B_1A_1=C_1, \,\, B_1 A_2 = B_2  A_1 =C_2.$$

On the other hand, the directed Donaldson-Fukaya category of $(M,V_2,V_1,V_0)$ has objects  $V_2,V_1,V_0$ and the Floer homology groups are also each generated by 
two elements (of degree 0). (Indeed, the intersection of any two $V_i$'s is either a pair of points or a pair of intervals; to be more formal one should isotope $V_2$ so that the two intervals become two points.)  By counting triangles in $M$ with boundary on $V_2,V_1,V_0$
one finds precisely the same quiver with relations as above.
For example, in figure \ref{RP2fiberVC} one has four such triangles; the right-most one
has vertices corresponding to $A_1, B_2, C_2$, and it gives rise to the relation
$B_2A_1 = C_2$ (see also \cite[p. 39]{AS}).

\section{Computing the flow category} \label{sectionFlow}

%\section{Computing the flow category when $N$ is a 4-manifold without
%  1- and 3-handles} \label{sectionFlow}

Fix a closed 4-manifold $N$ \label{N} which admits %a handle decomposition without 1- and 3- handles. This means that $N$ admits
a self-indexing Morse function $f: N \into [0,4],$ with critical values $0,2,4$ and with  a unique maximum and minimum. Let $g$ be  Riemannian metric such that $(f,g)$ \label{(f,g)} is Morse-Smale.
 %\cite[pages~?]{KI}, \cite{KII}.collar
In this section we compute the flow category 
of $(N,f,g)$ in this case (see the introduction for the definition). %Consider the handle decomposition of $N$ induced by $(f,g)$.
%There is a 0-handle, $k$ 2-handles and a 4-handle.
\\
\newline Let $L_- =f^{-1}(1) \cong S^3 \text{ and } L_+ =f^{-1}(3) \cong S^3.$ 
For each $j=1,\ldots,k$ we have two knots
$$K_-^j = U(x_2^j) \cap L_- \cong S^1 \text{ and } K_+^j = S(x_2^j) \cap L_+ \cong S^1.$$
Here $U(x)$ is the unstable manifold and $S(x)$ is the stable manifold.
The handle decomposition of $N$ coming from $(f,g)$ has $k$ 2-handles with attaching spheres  $ K^1_-$, $\ldots,$ $ K^k_-$. Let 
$\phi_j: S^1 \times D^2 \into U_j$
 denote the attaching map for the $j$th 2-handle where $U_j \subset L_-$ is a tubular neighborhood of $K_-^j$ in $L_{-}$.
Set 
$U= \cup_{j=1}^{j=k}U_j$, $K_{\pm} = \cup_j K_{\pm}^j.$ 
The first step in the computation is to note the following result.
\begin{lemma} \label{lemmaFlow} There are diffeomorphisms
$Flow(x_4,x_2^j) \cong K_+^j$,  $Flow(x_2^j,x_0) \cong  K_-^j$, and 
$Flow(x_4,x_0) \cong L_- \setminus \Int(U)$. 
\end{lemma}

\begin{proof} For the first two diffeomorphisms we use the map which associates 
to a trajectory $\gamma$ the corresponding intersection point in $\gamma \cap L_- \in L_-$ or $\gamma \cap L_+ \in L_+$. For the second diffeomorphism, note first that 
$$Flow^\circ(x_4, x_0) \cong L_+ \setminus K_+ \cong L_- \setminus K_-$$ by the same map.
Now take the compactification $Flow(x,y)$ and delete an open collar neighborhood of the boundary
to obtain an diffeomorphic manifold 
$$\widetilde {Flow}(x,y) \cong Flow(x,y), \text{ with } \widetilde {Flow}(x,y) \subset Flow^\circ(x,y).$$
Now consider $L_- \setminus \Int(U) \subset L_-\setminus K_-$. Under the correspondence
$$ Flow^\circ(x,y) \cong L_-\setminus K_-,$$ 
$L_- \setminus \Int(U)$ also corresponds to $F(x,y)$ minus some collar neighborhood of the boundary, just like $\widetilde {Flow}(x,y)$. But any two complements of a collar neighborhood are diffeomorphic, so 
$ Flow(x,y) \cong \widetilde{Flow}(x,y) \cong L_- \setminus \Int(U).$ 
\end{proof}
%(The last isomorphism is obtained by deleting a collar neighborhood of the boundary of $Flow(x_4,x_0)$.)
For $H_*( K_{\pm}^j)$ take the generators $[K_{\pm}^j]$, $[p_{\pm}^j]$, where 
$p_{\pm}^j \in K_{\pm}^j$. 
For $H_*(\partial U_j)$ take the generators 
$[\partial U_j], \,\lambda_j,\, \mu_j,\, [q_j],$ 
where
$q_j \in \partial U_j$ and $\lambda_j,\, \mu_j \in H_1(\partial U_j)$ 
satisfy 
$$\ell k(\lambda_j, K_-^j) =1, \, \ell k(\mu_j,K_-^j)=0.$$
Here, $\ell k(\cdot,\cdot)$ is the linking number. Then $H_*( L_- \setminus \Int(U))$ is generated by $$[\partial U_j], \,\lambda_j,\, \mu_j, \, [q_j],\, j=1,\ldots,k, $$ and the relations are:
\begin{gather} \label{homologyrelations}
 \Sigma_j [\partial U_j] =0, \, [q_1] = \ldots =[q_k],\\
 \lambda_j = \Sigma_{i\neq j} \ell k(K_-^j,K_i^-)\mu_i, \,  j =1,\ldots, k.
\notag
\end{gather}
(To see the last relation, take a Seifert surface $\Sigma_j$ bounding $K_j$, and cut away a neighborhood of $K_j$ and each $K_i$, $ i \neq j$.) 
%Let $m_j\in \bbZ$ denote the framing coefficient of $\phi_j$. 
%This can be characterized as follows. 
\\
\newline Recall that each $2-$handle is a copy of $D^2 \times D^2$, and to attach it to $D^4$ we glue 
$$S^1 \times D^2 = \partial D^2\times D^2 \subset D^2 \times D^2 \text { to } U_j \subset L_-$$
using $\phi_j$. Now take $\widetilde \lambda,\widetilde \mu \in H_1(S^1 \times \partial D^2)$ to be
the homology classes of $S^1 \times \{q\}$ and $\{p\} \times \partial D^2$ respectively,
where $p \in S^1$,  $q\in \partial D^2$.
%so that 
%$$\ell k(\lambda, S^1 \times \{0\}) =1, \ell k(\mu, S^1 \times \{0\}) =0$$
%the linking numbers of $\lambda, \mu$ with $S^1 \times \{0\}$ are  1 and 0, respectively.
%Similarly, suppose the  linking numbers of $\lambda_j, \mu_j \in H_1(\partial U_j)$ with $K_j$ are respectively 1 and 0.
Then $ \phi_j$ induces a map 
$$(\phi_j)_*: H_1(S^1 \times \partial D^2) \into  H_1(\partial U_j)$$ which satisfies
\begin{gather}\label{phi_jformula}
 (\phi_j)_*(\widetilde \lambda)= \lambda_j + m_j\mu_j,  \, (\phi_j)_*(\widetilde \mu) = \mu_j
\end{gather} 
for some $m_j \in \bbZ$ (called the framing coefficient). Denote the composition in the flow category by
\begin{gather*}
\mu^{Flow}: H_*(Flow(x_4,x_2^j)) \otimes H_*(Flow(x_2^j,x_0))\into H_*(Flow(x_4,x_0)).
\end{gather*}
\begin{prop}%[\cite{JB}] 
\label{propFlow} In terms of the identifications in lemma \ref{lemmaFlow}, $\mu^{Flow}$ is given by:
$$\mu^{Flow}([K_+^j], [p_-^j]) = \mu_j, \, \mu^{Flow}([p_+^j], [K_-^j]) = \lambda_j + m_j\mu_j$$
$$\mu^{Flow}([K_+^j], [K_-^j]) = [\partial U_j], \, 
\mu^{Flow}([p_+^j], [p_-^j]) = [q_j].$$
(All other products are zero.)  In view of (\ref{homologyrelations}), the relations among the products are therefore
$$ \Sigma_j \mu^{Flow}([K_+^j], [K_-^j]) =0, \, \mu^{Flow}([p_+^1], [p_-^1]) = \ldots =\mu^{Flow}([p_+^k],[p_-^k]),$$
$$\mu^{Flow}([p_+^j], [K_-^j]) = m_j\mu^{Flow}([K_+^j], [p_-^j]) + \Sigma_{i\neq j}\, L_{ji}\mu^{Flow}([K_+^i], [p_i^-]),$$
where $L_{ji}= \ell k(K_-^j, K_-^i)$ and $m_j$ is the framing coefficient of $K_-^j$.
\end{prop}

\begin{proof} The $\mu_2([p_+^j],[p_-^j])$ product is very simple, so we omit that. The products $\mu^{Flow}([K_+^j], [K_-^j])$, $\mu^{Flow}([p_+^j], [K_-^j])$ and $\mu^{Flow}([K_+^j], [p_-^j])$ are cycles in $Flow(x_4,x_0)$, which  are respectively represented by the following submanifolds 
$$C= \{\gamma \in Flow(x_4, x_0) : \gamma \cap L_+ \in K_+^j, \gamma \cap L_- \in K_-^j\}.$$
$$C'= \{\gamma \in Flow(x_4, x_0) : \gamma \cap L_+ =  p_+^j, \gamma \cap L_- \in K_-^j\}.$$
$$C''= \{\gamma \in Flow(x_4, x_0) : \gamma \cap L_+ \in K_+^j, \gamma \cap L_- = p_-^j\}.$$
(We remind the reader that each $\gamma \in Flow(x_4, x_0)$ could be either broken (at some $x_2^j$)
or not broken; in either case $\gamma \cap L_{\pm}$ makes sense.)
%Let's focus on $C'$, $C''$ can be treated similarly. 
\\
\newline Let's look at how we can represent these submanifolds inside the fixed handle $H_j = D^2 \times D^2$. 
The parts of $L_{-}$ and $L_{+}$ in   $H_j$ are respectively
$$H_- = \partial D^2 \times D^2 \text{ and }H_+ =  D^2 \times \partial D^2,$$
and $K_+^j$, $K_-^j$ correspond respectively to 
$$K_-=  \partial D^2 \times \{0\} \text{ and }K_+=   \{0\}\times \partial D^2.$$
%(Note that, in $N$, $\phi_j$ identifies $K_-$ with $K_j \subset L_-$.)
Now if we view $C \cap H_j$, $C' \cap H_j$ and  $C'' \cap H_j$ as 
families of broken trajectories  in $H_j$ then 
$$C \cap L_+ = K_+, \, C \cap L_- = K_-$$ 
$$C' \cap L_+ = \{p_+\}, \, C' \cap L_- = K_-$$ 
$$C'' \cap L_+ = K_+, \, C'' \cap L_- = \{p_-\}$$ 
where $p_{\pm} \in K_{\pm}$ corresponds to $p_{\pm}^j \in K_{\pm}^j$.
%and  $p_- \in K_-$ is such that $\phi_j(p_-) = p_-^j$.
We represent the negative gradient flow in $H_j$ by the map
$$\Phi:  (D^2 \setminus \{0\}) \times \partial D^2 \into \partial D^2 \times (D^2 \setminus \{0\}), \phantom{bbb} \Phi(x,y) = (\frac{x}{|x|}, |x| y).$$
Then $\Phi$ maps $H_+ \setminus K_+$ diffeomorphically onto $H_- \setminus K_-$
and it fixes $\partial D^2 \times \partial D^2$ point-wise, since that is in both $H_+$ and $H_-$. Let $0 < \epsilon < 1$, and let $D^2_\epsilon \subset D^2$ 
denote the smaller disk of radius $\epsilon$. Set 
$$T_+ = T_- = \partial D^2 \times \partial D^2,$$
$$\lambda_- = \partial D^2 \times \{q_-\} \text{ for some }   q_- \in  \partial D^2_\epsilon,$$
$$\mu_- = \{r_- \} \times \partial D^2_\epsilon \text{ for some } r_- \in \partial D^2.$$
Similarly, set
$$\lambda_+ = \{r_+\} \times \partial D^2 \text{ for } r_+ = \epsilon r_-\in  \partial D^2_\epsilon,$$
 $$\mu_+  =  \partial D^2_\epsilon \times \{q_+ \} \text{ for } q_+ =  \frac{q_-}{|q_-|} \in \partial D^2.$$
Then 
\begin{gather} \label{PhiRelations}
\Phi(T_+) = T_-, \, \Phi(\lambda_{+}) = \mu_-, \, \Phi(\mu_{+}) = \lambda_-.
\end{gather}
Now, returning to $C, C',C''$, we apply the map which retracts 
$Flow(x_4, x_0)$ onto 
$$\widetilde{Flow}(x_4,x_0) \subset Flow^\circ(x_4,x_0)$$
to obtain new submanifolds $\widetilde C,\widetilde C',\widetilde C''$
which are homologous in $Flow(x_4, x_0)$ to  $C, C', C''$ respectively 
(here $\widetilde{Flow}(x_4,x_0)$ is as in the last lemma).
If $\widetilde{Flow}(x_4,x_0)$ is chosen suitably 
%(which means we would change $U$ and the maps $\phi_j: S^1 \times D^2 \into S^3$ by an isotopy) 
then $\widetilde C,\widetilde C',\widetilde C''$
will satisfy
$$\widetilde C \cap L_+ = T_+, \, \widetilde C \cap L_- = T_-,$$ 
$$\widetilde C' \cap L_+ = \mu_+, \, \widetilde C' \cap L_- = \lambda_-,$$ 
$$\widetilde  C'' \cap L_+ = \lambda_+, \, \widetilde  C'' \cap L_- = \mu_-.$$ 
(Note that $\Phi(\lambda_{+}) = \mu_-, \, \Phi(\mu_{+}) = \lambda_-$ is consistent with the fact that 
$\widetilde C$ is a collection of \emph{unbroken} gradient trajectories, so $\Phi(\widetilde C \cap L_+)$ should be equal to $\widetilde C \cap L_-$, and similarly for $\widetilde C'$, $\widetilde C''$.)
\\
\newline Now we return to what happens in $N$. Recall we wish to represent our cycles 
$C, C',C''$ as cycles in $H_*(\partial U_j)$, using the identification
$$\widetilde{Flow}(x_4,x_0) \cong S^3 \setminus (\cup_j \Int(U_j)).$$
We have already represented $\widetilde C,\widetilde C',\widetilde C''$ as submanifolds in $H_j \cap L_+$ and $H_j \cap L_-$. In fact 
since $\widetilde C,\widetilde C',\widetilde C''$ are families of unbroken trajectories, one only needs the representation in  $H_j \cap L_-$ since one can use the gradient flow to recover the representation in  $H_j \cap L_+$. To view $T_-, \lambda_-, \mu_- \subset H_j \cap L_-$ as cycles in $\partial U_j$, we 
first replace $\lambda_-$ and $\mu_-$ respectively by the homologous cycles 
$\lambda,   \mu  \subset \partial D^2 \times D^2$ from before.
Then we simply apply the map 
$$(\phi_j)_* : H_1(\partial D^2 \times  \partial D^2) \into H_*(\partial U_j ) $$ 
to each of $T_-, \lambda, \mu$ and  (\ref{phi_jformula}) implies 
$$(\phi_j)_*(T_-) = [\partial U_j],\,  (\phi_j)_* (\widetilde \lambda) = \lambda_j + m_j \mu_j, 
(\phi_j)_* (\widetilde \mu) = \mu_j.$$
Since  $\widetilde C,\widetilde C',\widetilde C''$ correspond respectively to $T_-, \lambda, \mu$, the formula for $\mu_2$ follows.
\end{proof}

\section{An outline of the paper in the case $N = \bbC P^2$}\label{outline}

In this section we will give a fairly detailed outline of the 
the main lines of argument in this paper. We focus on the  technical 
aspects for the most part, but we suppress details about Morse-Bott Floer 
homology. %At the end  we will also provide a detailed 
%summary of the contents of the remaining sections in the paper; these
%will carry out the same arguments in detail, but in a slightly different order.
\begin{comment}
 This will explain the main characters in this story and how they interact. We focus on
 the technical aspects for  the most part; by the end of \S 
\ref{contnmapsketch} we are set up for the main calculation, which proves 
Theorem $B$. However, we do not provide the technical 
version of this calculation;  we leave that to the proof of Theorem 
\ref{fuk=flow}. Instead, in \S \ref{FukFlowsketch}, we sketch the 
calculation at an intuitive level, 
ignoring all the technical difficulties; this should serve as a better 
illustration of the main shape of the argument. 
At the end of the section we will also provide a detailed summary of the
contents of the remaining sections, which will carry out the same arguments
in detail, but in a slightly different order.
\end{comment}
\\
\newline Fix a closed four manifold $N$ and let $(f,g)$ be a Morse function and 
Riemannian metric on $N$ such that the gradient flow is Morse-Smale. To keep
 the notation simple in this section we assume that $f$ has just three critical 
points, 
say $x_0, x_2, x_4$, with  Morse index $0,2,4$. (For example, we could take 
$N=\bbC P^2$ with its standard handle-decomposition.) 
\\
\newline Below we will sketch the 
construction of  the fiber $M$, based on $(N,f,g)$, and  we will construct
the vanishing spheres $V_4$, $V_2$, $V_0$ in $M$. Then we will
sketch the computation of  the Floer homology groups and the triangle
product.  The crucial ingredient in
the paper is the following: We will see that 
(parts of) $V_4$, $V_2$, $V_0$ have a nice rotational symmetry, and we
will leverage this to find a simple explicit description of
all the holomorphic triangles in $M$ (with respect to some natural
almost complex structure). The demands of this one argument are responsible for
most  of the  technical aspects of this paper: 
First, it necessitates  the use of Morse-Bott
Floer homology (as we mentioned at the end of the introduction). 
Second, while we start with a simple version of the vanishing spheres,
we must repeatedly replace these 
by more complicated versions which are 
exact isotopic to the original ones; each version  resolves a particular
technical problem. (There are four versions in all, but always the 
crucial rotational symmetry will be preserved.)
%To explain this we have organized our summary around this sequence of 
%modifications.
Here is a very rough outline of the whole argument, arranged according to the 
subsections below:
% Each new version is introduced to solve some technical problem; 
%Second,
%it requires us to replace  $V_4$, $V_2$, $V_0$ several times 
%into a sequence of  fairly complicated configurations
%(but always the crucial rotational symmetry will be preserved,
%of course). 
%(but always the crucial rotational symmetry will be preserved,
%of course). 
%To explain the whole argument in a digestable way we will
%present a sequence of approximations to the correct argument, which gently
%increase in complexity. In each approximation certain problems will arise, 
%and we will correct these by making the  modifications to  $V_4$, $V_2$,
%$V_0$ (but always the crucial rotational symmetry will be preserved,
%of course).  
\begin{itemize}
\item[\S \ref{fibersketch}] We sketch the construction of the fiber $M$.
\vspace{0.1cm}
\item[\S \ref{VCsimplesketch}] (Vanishing spheres version I)  
We define a simple version of the vanishing spheres  in $M$ which we denote
$L_0$, $L_2$, $L_4$. (Figure \ref{IntrofigsimpleVCdim2} below shows 
$L_0$, $L_2$, $L_4$ intersected with a certain 2 dimensional slice in $M$.)
 \vspace{0.1cm}
\item[\S \ref{firsttriangle}] We sketch the core argument of the paper:
 We classify the holomorphic 
triangles in a certain Weinstein neighborhood of $L_2$, 
say $D(T^*L_2) \subset M$, using the rotational symmetry of $L_0$, $L_2$, $L_4$.
\vspace{0.1cm}
\item[\S \ref{VCsketch}] (Vanishing spheres version II) 
We point out the main technical problem with the argument in 
\S \ref{firsttriangle} and we correct it by replacing 
$L_0$, $L_2$, $L_4$ by some new exact isotopic versions 
$V_0$, $V_2$, $V_4$ (see figure \ref{CorrectionGammaIIIbig}).
\vspace{0.1cm}
\item[\S \ref{Floersketch}] (Vanishing spheres version III)
We sketch the computation of the  Floer homology groups. To do this we have 
to modify $V_0$, $V_2$, $V_4$ very slightly: we locate certain graphs of exact 
1-forms $df$, which are subsets of $V_0$, $V_2$, $V_4$,  and we replace these by $\frac{1}{n} df$ for some large fixed 
$n \geq 1$; then we patch these back into  $V_0$, $V_2$, $V_4$. The basic shape of 
 $V_0$, $V_2$, $V_4$ is unchanged, and we keep the same notation.
\vspace{0.1cm}
\item[\S \ref{secondtriangle}] (Vanishing spheres version IV)
We complete our sketch of the classification of holomorphic triangles in $M$. To do this we replace  $V_0$, $V_2$, $V_4$ one final time by new 
exact isotopic versions denoted $\widetilde V_0$, $\widetilde V_2$, $\widetilde V_4$. The main point of  $\widetilde V_0$, $\widetilde V_2$, $\widetilde V_4$
is to shrink the size of the four triangles in figure \ref{CorrectionGammaIIIbig} so that their areas are small (see figure \ref{CorrectionGammaIIbig}). 
This ensures that any holomorphic triangle in $M$ 
must lie in the region  $D(T^*L_2)$ from \S \ref{firsttriangle}.
In addition,  $\widetilde V_0$, $\widetilde V_2$, $\widetilde V_4$
make the boundary of the four triangles in figure \ref{CorrectionGammaIIbig} real analytic, which fixes one last problem in \S \ref{firsttriangle}.

\vspace{0.1cm}
\item[\S \ref{contnmapsketch}] %(Relating versions III and IV)
%We explain how to 
%combine the calculations in \S \ref{Floersketch} (with  $V_0$, $V_2$, $V_4$)
%\S \ref{secondtriangle} (with  $\widetilde V_0$, $\widetilde V_2$, 
%$\widetilde V_4$). To do this 
We compute certain continuation maps, 
$$\phi: HF(V_i, V_j) \into HF(\widetilde V_i, \widetilde V_j).$$
This relates the Floer groups  $HF(V_i, V_j)$  from \S \ref{Floersketch}  to 
the Floer groups $HF(\widetilde V_i, \widetilde V_j)$, which are 
compatible with our classification of triangles in \S \ref{secondtriangle}.
The groups $HF(V_i, V_j)$ have some particularly nice generators and relations 
which correspond perfectly to the ones we found in our flow category calculation in \S \ref{sectionFlow}. We therefore prove that $\phi$
fixes these generators. Once this is known
%With this set-up it 
it is easy to compute the triangle product for
 $\widetilde V_0$, $\widetilde V_2$, $\widetilde V_4$ (using our explicit 
classification of holomorphic triangles), \emph{and} we can compare easily with the Flow category calculation (using  the above generators and relations for
$HF(\widetilde V_i, \widetilde V_j)$).
\vspace{0.1cm}
\item[\S \ref{FukFlowsketch}] We change gears and give an intuitive argument 
for why we expect Theorem $B$ to be true; this ignores all the 
 technical difficulties we have attended to up until this point. 
We work in the full generality of Theorem $B$ (so $f$ has any number of index 2 critical points $x_2^j$), but we use the simple vanishing spheres 
$L_0$, $L_2^j$, $L_4$. This illustrates the essential structure
of the actual calculation, which can be found in the 
proof of Theorem \ref{fuk=flow}. 
\end{itemize}  

\subsection{Construction of the fiber  $M$}  \label{fibersketch}

Set $L_0 = L_2 = S^3$ and take the disk cotangent bundles  $D(T^*L_0)$
and $D(T^*L_2)$ with respect to some metrics. Let
$K_0 \subset L_0$ and $K_2 \subset L_2$ be two embedded copies of
$S^1$ (i.e. knots) with chosen trivializations of their normal bundles.
Then, roughly speaking, to construct $M$ we will do a version of  
plumbing  (see \cite{GS}) where we glue  $D(T^*L_0)$ and $D(T^*L_2)$ 
together along neighborhoods of
$K_0 \subset D(T^*L_1)$ and $K_2 \subset D(T^*L_2)$ in such a way that
$K_0$ is identified with $K_2$.
%the normal bundle of $K_0$ is
%identified with the conormal bundle of $K_2$, and vice-versa.
%In particular, and more precisely,   
Moreover, there is a tubular neighborhood
of $K_0$ in $L_0$, say $U_0 \cong S^1 \times D^2$, which is identified with the disk conormal bundle $D(\nu^*K_2) \subset D(T^*L_2)$, and vice-versa.
%(Actually, if we construct $M$ in this way, the boundary is not smooth; 
%we explain how to smooth the boundary to a contact type hypersurface 
%in \S \ref{}.)
Earlier in \S \ref{example} we saw the plumbing construction in the two 
dimensional case, %$\dim D(T^*L_0)=2 = \dim D(T^*L_1)$, $K_0 \cong K_2 \cong S^0$;
see  figure \ref{RP2fiberVC}.  For a schematic picture of the 
higher dimensional case   near the
plumbing region see figure \ref{figureM_0} in \S \ref{fiber}. 
(In \S \ref{fiber} we  call the
plumbing $M_0$, and $M$ will denote a slightly larger space where
the boundary has been smoothed.)
\\
\newline To specify $K_0 \subset L_0$ we look at the 
handle decomposition of $N$ induced by $(f,g)$; this 
determines a knot $K_0 \subset S^3 = L_0$, 
and a parameterization of a tubular neighborhood of $K_0$ 
(determined up to isotopy), say
$S^1 \times D^2 \cong U_0 \subset S^3$. 
%We set $L_0 = S^3$, $K_0 = K$, and $U_0 = U$. 
Define $K_2 \subset L_2$ to be the
following  unknot in $L_2$ (whose normal bundle has an obvious trivialization): 
$$K_2 = K_-^2= \{(0,0, x_3, x_4) | x_3^2 + x_4^2 =1\}.$$
When we define the Lagrangian submanifolds we will use the 
following related knot: 
$$K_+^2 = \{ (x_1, x_2,0,0) | x_1^2 + x_2^2 =1\}.$$

\subsection{Construction of the vanishing spheres $L_0$, $L_2$, $L_4$ (version I)}\label{VCsimplesketch}
%A little further below, 
%we will construct the vanishing spheres $V_4$, $V_2$, $V_0$,
%which are certain exact Lagrangian spheres in $M$.
We describe a simple version of the vanishing spheres,
which we denote $L_0$, $L_2$, $L_4$ (only $L_4$ is new). 
%which will be exact isotopic to $V_4$, $V_2$, $V_0$. 
%(We remark that $L_0$, $L_2$, $L_4$ are the versions we use in \cite{JA}.)
 %Both $L_0$ and $L_2$ are already defined, so we just have to
Roughly speaking, $L_4$ is defined as the (Morse-Bott) Lagrangian surgery of 
$L_0$ and $L_2$. More precisely, recall that when $D(T^*L_2)$ is plumbed onto
$D(T^*L_0)$, $D(\nu^*K_2) = D(\nu^*K_-^2)$ is identified with a tubular neighborhood
of $K_0$ in $L_0$, say $U_0$. To simplify things, assume $D(T^*L_2)$ 
is the unit disk bundle with respect to the standard round metric.
% = D_1^{g_{S^3}}(T^*S^3)$ for each $j$. 
Now let $\Phi: D(T^*L_2) \into  D(T^*L_2)$ denote the time $\pi/2$ 
geodesic flow, which is Hamiltonian.  Recall that 
we also have the complementary unknot $K_+^2 \subset L_2$ 
and consider its disk conormal bundle $D(\nu^*K_+^2)$. 
The effect of $\Phi$ on $D(\nu^*K_+^2)$ is to fix vectors of zero
length (i.e. points in $K_+^2$) and map the unit vectors $S(T^*K_+^2)$
diffeomorphically onto $S(T^*K_-^2)$, while vectors of intermediate
length interpolate between these extremes.  Now tweak $\Phi$ slightly to get a new Hamiltonian diffeomorphism
$\widetilde \Phi$ so that $\widetilde \Phi(D(\nu^*K_+^2))$ agrees with
$D(\nu^*K_-^2) $ in a small neighborhood of  $\partial
D(\nu^*K_-^2)=S(T^*K_-^2)$. Figure \ref{IntrofigsimpleVCdim2} below 
depicts the  analogous two dimensional case, 
where $\widetilde \Phi(D(\nu^*K_+^2))$ 
corresponds to the two red curves. Then we define
$$H = \widetilde \Phi(D(\nu^*K_+^2)) \text{ and } L_4 = (L_0 \setminus
U_0) \cup H.$$ Thus $L_4$ is obtained from $L_0$ by a surgery along
$K_0$. In \S \ref{example} we made an analogous definition,
where $V_2$ was the Lagrangian surgery of $V_0$ and $V_1$;
see  figure \ref{RP2fiberVC}.

\subsection{Classification of holomorphic triangles in $D(T^*L_2^j)$} 
\label{firsttriangle} 
We first explain how to visualize the parts of $L_0$, $L_2$, $L_4$
in $D(T^*L_2)$ using their rotational symmetry. For each $e \in K_+^2$ and 
$f \in K_-^2$, consider the great circle $K_{ef} \subset L_2$ passing through $e,f$. There is a natural
embedding
$T^*K_{ef} \subset T^*L_3$ based on the standard identification $T^*S^3 = TS^3$.
The rotational symmetry of $L_4$, $L_2$, $L_0$ in $D(T^*L_2)$ can be characterized by noting that
$L_4$, $L_2$, $L_0$ intersect each slice $D(T^*K_{ef})$  in exactly the same way as $e,f$ vary.
In figure \ref{IntrofigsimpleVCdim2} we have identified $D(T^*K_{ef})$ with
$\bbR /2\pi\bbZ \times [-1,1]$ and the intersection of
$L_4$, $L_2$, $L_0$ with $D(T^*K_{ef})$ are indicated respectively by the two red curves, 
the horizontal green line, and the two vertical blue lines.
\begin{figure}
\begin{center}
\includegraphics[width=3in]{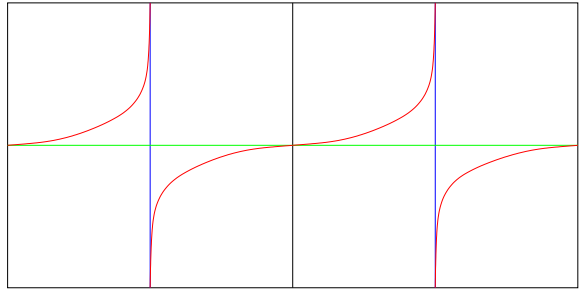}
%\put(-125,125){$\widehat{\psi}$}
%\put(-80,70){$\widehat{\psi}$}
\caption{Simple version of vanishing spheres
%$\widehat \Gamma_4$,  $\widehat \Gamma_2$, $\widehat \Gamma_0$ 
%$\Gamma_4^{'}$, $ \Gamma_2^{'}$,  $\Gamma_0^{'}$
$\Gamma_4^s$ (red), $ \Gamma_2^s$ (green),  $\Gamma_0^s$ (blue)
in $\bbR/2\pi\bbZ \times [-1,1]$. } 
\label{IntrofigsimpleVCdim2}
\end{center}
\end{figure}
In the figure the green curve of course corresponds to $K_{ef} \subset
L_2$. The 
two intersection points of the blue and green curves
correspond to the points $\pm f$, and the  two intersection points of
the red and green curves correspond to $\pm e$.
\\
\newline We now sketch how to give a complete description of the
moduli space of all holomorphic triangles (with finite symplectic area) 
in $D(T^*L_2)$ with boundary on
$L_0 \cap D(T^*L_2)$, $L_2 \cap D(T^*L_2)$, $L_4 \cap D(T^*L_2)$, with respect to some suitable almost complex
structure. Later, it will not be too hard to arrange things so that
any holomorphic triangle in $M$ must necessarily lie in $D(T^*L_2)$.
For the rest of \S \ref{outline}, we will often omit the phrase 
"with finite symplectic area" when talking about  
holomorphic strips and triangles.
\\
\newline First, we identify $T^*L_2=T^*S^3$ with $\{z \in \bbC^4: \Sigma_j z_j^2 =1\}$ as exact symplectic manifolds and
let $J_\bbC$ denote the complex structure which comes from that identification.
%For each $e \in K_-$, $f \in K_+$ we construct a $J_\bbC-$holomorphic triangle in $D(T^*K_{ef}) \subset D(T^*L_2)$ with two of its vertices at $e$ and $f$.
In figure \ref{IntrofigsimpleVCdim2} we see the there are four obvious
triangles in $D(T^*K_{ef})$ corresponding to $\pm e$ and $\pm f$.  Let
$V$ denote the disk in $\bbC$ 
with three boundary punctures removed.
Then, using the Riemann mapping theorem, 
it is easy to construct a $J_\bbC-$holomorphic map
$$w_{ef}: V \into D(T^*K_{ef}) \subset D(T^*L_2)$$
with image equal to the triangle with the two base vertices
corresponding to  $e$, $f$ for  each $e$, $f$. Our goal is to  prove that any
$J_\bbC$-holomorphic triangle $w: V \into D(T^*L_2)$ must coincide  with
one of these standard ones. 
Note that this is intuitively plausible on the grounds that holomorphic 
disks are minimal surfaces (however, we will pursue another line of argument).
%and there may in fact exist an argument using the theory of minimal surfaces.
\\
\newline To exploit the rotational symmetry of $L_0 \cap
D(T^*L_2)$, $L_2 \cap D(T^*L_2)$, $L_4 \cap D(T^*L_2)$, we introduce a
certain  $J_\bbC$-holomorphic map on the target, 
$$P : T^*L_2 \into \bbC, \phantom{123}  P(z_1,z_2,z_3,z_4) = z_1^2 + z_2^2 - z_3^2-z_4^2.$$ 
Here, $P$ is invariant under the rotational symmetry
of $T^*L_2$ which moves the slices $T^*K_{ef}$ into one another.
Because of this invariance, $P$ maps the three Lagrangians $L_0 \cap
D(T^*L_2)$, $L_2 \cap D(T^*L_2)$, $L_4 \cap D(T^*L_2)$ onto three curves
in the plane which bound a triangle $T \subset \bbC$. (Each of the four triangles in figure \ref{IntrofigsimpleVCdim2} is mapped by $P$ onto $T$.)
The crucial property of $P$ is that 
$P|_{P^{-1}(T)}$ can be holomorphically trivialized
over $T$ with fiber $T^*K_+^2 \times T^*K_-^2$.
Now let $w: V \into D(T^*L_2)$ be any $J_\bbC$-holomorphic triangle.
By applying the maximum principle to $P \circ w$, one sees that 
$P(w(V)) = T$. From this it follows that any such
$w$  can be viewed as a holomorphic section of $P|_{P^{-1}(T)}$,
$$s: T \into P^{-1}(T) \cong T \times (T^*K_+^2 \times T^*K_-^2),$$ 
where the three boundary conditions of $w$
all correspond to the same Lagrangian boundary condition for $s$ in the fiber, 
namely
$K_+^2 \times K_-^2 \subset T^*K_+^2 \times T^*K_-^2$. Each standard holomorphic
triangle $w_{ef}$ of course corresponds to the constant map with value $(e,f)$.  But it is easy  to see that
the energy of
\emph{any} holomorphic section must be zero, hence constant. 
(This follows from Stokes' theorem, because the
canonical one form $\theta$ on $T^*K_2^{\pm}$ satisfies $\theta|_{K_2^{\pm}}=0$.)
Thus $w = w_{ef}$ for some $e$, $f$.  This yields
the classification of holomorphic triangles in $D(T^*L_2)$ we wanted.
%With this detailed knowledge of the  holomorphic triangles in hand, it is %not hard to compute the triangle product.

\subsection{Construction of  $V_4$, $V_2$, $V_0$ (version II): Main correction to \S \ref{firsttriangle} } 
  \label{VCsketch} % the triangle argument in 

The most immediate problem with the above argument in \S \ref{firsttriangle} 
is that 
$P$ is singular along $K_-^2$ and $K_+^2$ (indeed, $P$ is the complexification of a Morse-Bott function on $L_2 = S^3$ with maximum at $K_+^2$ and minimum at $K_-^2$). This is a problem because each $\pm e \in K_-^2$,  $\pm f \in K_+^2$ correspond to the bottom vertices of the four basic triangles
$w_{\pm e, \pm f}$ in $D(T^*K_{ef})$ in figure \ref{IntrofigsimpleVCdim2},
and we recall that each of these four triangles are mapped onto $T$; thus,
$P$ cannot quite be trivialized over the whole of $T$, because it is 
singular over the two vertices of $T$ corresponding to $\pm e$ and $\pm f$. 
%(The reader may wonder of it is possible to squeak by with  
%trivializing $P$ only over $T^*$, the triangle with its vertices deleted, since we only work with holomorphic triangles $w: V  \into D(T^*L_2)$ which exclude the vertices (in particular $P(w(V)) = T^*$). 
(Incidentally, we must trivialize $P$ over the whole of $T$, and not just $T^*$, the triangle with its vertices deleted. This is because we  need to work with the continuous extensions 
of holomorphic triangles $w$ to the closed disk 
$\overline w : D^2 \into D(T^*L_2)$ 
and the corresponding sections $\overline s: T \into P^{-1}(T)$; this is 
necessary in order to conclude that the symplectic area of $s$ is finite, which we need for the Stoke's theorem argument at the end of \S \ref{firsttriangle}.)
\\
\newline To avoid this difficulty we deform $L_4$, $L_2$, $L_0$ to some new exact Lagrangian spheres $V_4$, $V_2$, $V_0$. For now we focus on defining the parts of 
$V_4$, $V_2$, $V_0$ in $D(T^*L_2)$.
Let $\Gamma_4^s$, $\Gamma_2^s$, $\Gamma_0^s$ denote the three curves in figure 
\ref{IntrofigsimpleVCdim2} corresponding to 
 $L_4$, $L_2$, $L_0$ ($s$ is for simple). Now we choose some new curves  $\Gamma_4$, $\Gamma_2$, $\Gamma_0$ in $\bbR/2 \pi \bbZ \times [-1,1]$ (with $\Gamma_2 =\Gamma_2^s$)
as in figure \ref{CorrectionGammaIIIbig} (the dotted rectangles are not relevant for now), which are exact isotopic to 
the old ones (because the signed area between them is zero). 
\begin{figure}
\begin{center}
\includegraphics[width=5in]{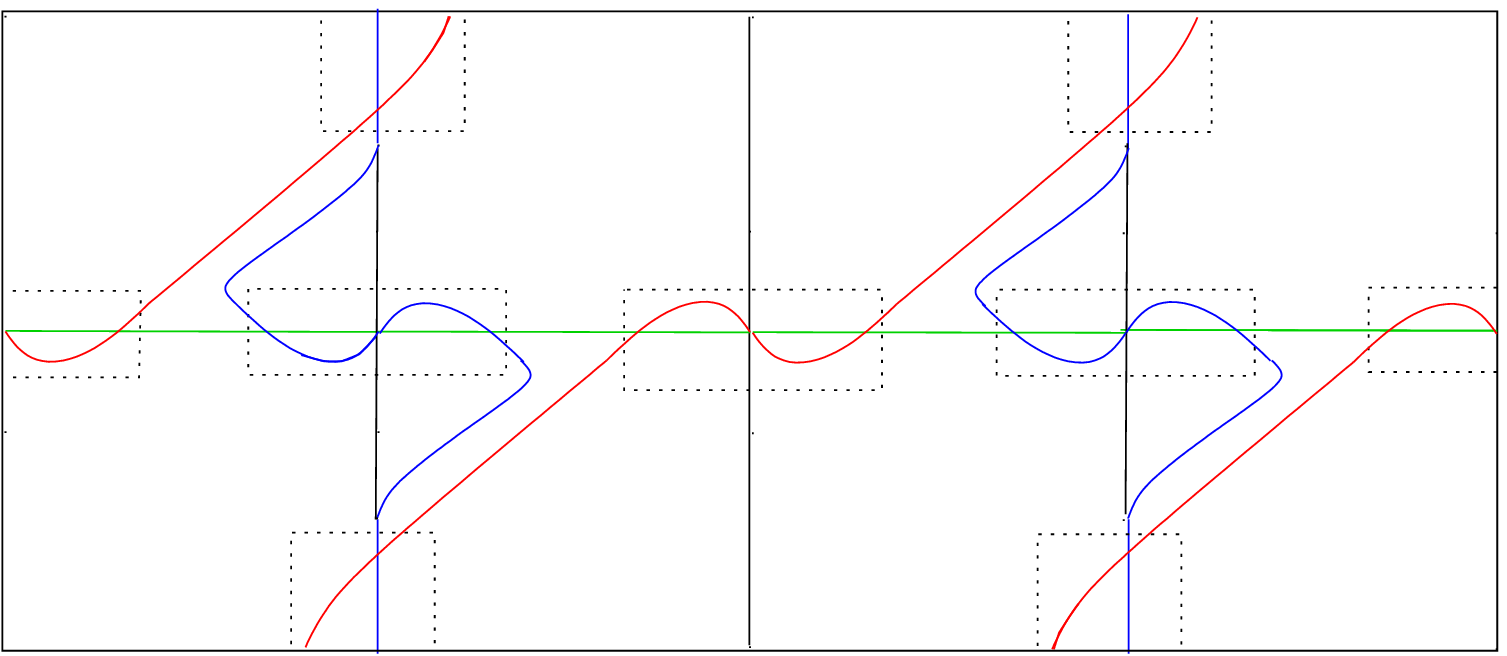}
%\put(-125,125){$\widehat{\psi}$}
%\caption{%Simple version of vanishing spheres
%$\Gamma_4$,  $\Gamma_2$, $\Gamma_0$ in $\bbR /2\pi\bbZ \times [-r,r]$.}
\caption{$\Gamma_4$ (red),  $\Gamma_2$ (green), $\Gamma_0$ (blue)
 in $\bbR /2\pi\bbZ \times [-r,r]$.}
\label{CorrectionGammaIIIbig}
\end{center}
\end{figure}
Then we define $V_4 \cap D(T^*L_2) $, $V_2 \cap D(T^*L_2)$, $V_0 \cap D(T^*L_2)$ 
to be the unique 3 dimensional submanifolds in $D(T^*L_2)$ whose 
intersection
 with every $D(T^*K_{ef})$ are given by  $\Gamma_4$, $\Gamma_4$, $\Gamma_4$.
Thus $V_4 \cap D(T^*L_2) $, $V_2 \cap D(T^*L_2)$, $V_0 \cap D(T^*L_2)$ are 
rotationally symmetric in the same sense as 
before. It is easy to see they are Lagrangian, and  exact isotopic to 
$L_4 \cap D(T^*L_2)$, $L_2 \cap D(T^*L_2)$, $L_0 \cap D(T^*L_2)$.
\\
\newline What this deformation accomplishes is the following. 
There are still four triangles in $D(T^*K_{ef})$ with boundary on
$V_4$, $V_2$, $V_0$, but now the vertices of these triangles do not lie
% on $(0,0), (\pi/2,0), (\pi,0),(3\pi/2,0)$ or $(2\pi,0)$, which are 
on any of the four points $\pm e, \pm f \in K_{ef} \cap K_2^{\pm}$, 
where $P: T^*L_2 \into \bbC$ is singular. 
This ensures that $P$ can be trivialized over $T$ in the
argument we discussed in \S \ref{firsttriangle}. 
(In addition $\Gamma_4$ and $\Gamma_0$ now intersect transversely, 
which will be useful when we compute Floer homology groups below.)
\\
\newline We complete the definition of  $V_0$ by making it coincide with 
$L_0$ outside of  $D(T^*L_2)$.
%(recall once again that $D(T^*K_-^2)$, 
%which corresponds to the two vertical  blue lines in figure \ref{IntrofigsimpleVCdim2}, is identified with a tubular neighborhood of $K_0$ in $L_0$, say $U_0 \subset L_0$). 
$V_4$ is extended outside of  $D(T^*L_2)$
in a similar way except we must take it to be the graph of an exact 1-form 
$df$ in $D(T^*L_0)$ defined over $L_0 \setminus U_0$ such that $df$  matches up with the  part of $V_4$ near the boundary of $D(T^*L_2)$; 
%(corresponding to the part of $\Gamma_4$ near the vertical boundary of $\bbR/2\pi \bbZ \times [-1,1]$ in figure \ref{CorrectionGammaIIIbig}) 
see  \S \ref{Floersketch} below for more details.

\subsection{Modification of  $V_4$, $V_2$, $V_0$  (version III): 
Computing the Floer homology groups} \label{Floersketch}

Usually to compute Lagrangian Floer homology one takes an exact isotopy of 
one of the Lagrangians which makes it transverse to the other. 
We cannot do that because it would disrupt the symmetry of 
  $V_4$, $V_2$, $V_0$  and ruin our argument in \S \ref{firsttriangle}. 
%above about classifying holomorphic triangles. 
However, there is a variant of Floer homology (isomorphic to the usual one)
which works for Lagrangians which only intersect cleanly (i.e. in Morse-Bott fashion); we call it Morse-Bott Floer homology. 
\\
\newline The old Lagrangians $L_4$, $L_2$, $L_0$ did not even intersect cleanly,  
because $L_4$ and $L_0$ intersected in a closed manifold with boundary, namely $L_0$ minus a small neighborhood of $K_0$. (This was partially visible in figure
\ref{IntrofigsimpleVCdim2}: $\Gamma_4^s$ and $\Gamma_0^s$
intersect in four closed intervals.) However, the new Lagrangians $V_4 \cap D(T^*L_2) $, $V_2 \cap D(T^*L_2)$, $V_0 \cap D(T^*L_2)$ do intersect cleanly, because the curves $\Gamma_4$, $\Gamma_2$, 
$\Gamma_0$ intersect transversely. 
\\
\newline It  is useful to understand how the various intersection points between
$\Gamma_4$, $\Gamma_2$, $\Gamma_0$ correspond to submanifolds in $D(T^*L_2)$.
First, consider  the two dotted rectangles surrounding the six intersection points 
of $\Gamma_0$ (blue) and $\Gamma_2$ (green). The two midpoints correspond to the two points $\pm f \in K_-^2$. Thus, as $e$ and $f$ vary, the two midpoints  together sweep out $K_-^2 \cong S^1$. The four intersection points on either side of the two  middle points in each rectangle  together sweep out
a torus, which is the boundary of a tubular neighborhood of $K_-^2$ in 
$V_2 = L_2$, and we denote this torus by $\Sigma_{20}$.
There is a similar story for the intersection points of $\Gamma_4$ (red) and $\Gamma_2$ (green), and we denote the torus there by $\Sigma_{42}$.
The four intersection points of $\Gamma_4$ (red) and $\Gamma_0$ (blue) together sweep out a torus, denoted $\Sigma_{40}$, 
which can be viewed either as the boundary of $D_s(\nu^*K_-^2)$ for a
 certain radius $s$, or as the boundary of a tubular neighborhood of 
$K_0$ in $L_0$ (recall $D(\nu^*K_-^2)$ is identified with 
$U_0 \subset L_0$, a tubular neighborhood of $K_0 \subset L_0$).  
%(recall $K_2 = K_-^2 \subset L_2$ is identified with 
%$K_0 \subset L_0$ when we plumb $D(T^*L_2)$ to $D(T^*L_0)$).
\\
\newline Before discussing the computation of Floer homology groups we 
have to set up some regions in $M$ where we will localize the holomorphic 
strips which are involved in the calculation.
%in which, for each pair of Lagrangians
%$(L_0, L_1)$, say $L_1$ can be regarded as the 
%graph of some  exact one form $df$ inside a Weinsetein neighborhood 
%$D(T^*\widetilde L_0)$ of some Lagrangian submanifold with boundary 
%$\widetilde L_0 \subset L_0$.  
First, consider  the two horizontal dotted rectangles in 
figure \ref{IntrofigsimpleVCdim2} which surround the six intersection points of
$\Gamma_4$ (red) and $\Gamma_2$ (green) (one of them is wrapped at the left and right edges of $\bbR/2\pi \bbZ \times [-1,1]$).
The intersection of these two rectangles with $\Gamma_2$ yields two closed intervals in $\Gamma_2$. As $e$ and $f$ vary, these intervals
sweep out a closed tubular neighborhood of $K_+^2$ in $V_2$, 
and we denote this by  $N_{42} \subset V_2$. 
Now let $D(T^*N_{42})\subset D(T^*L_2)$ denote a Weinstein neighborhood of $N_{42}$ which intersects each slice 
$D(T^*K_{ef})$ precisely in these two dotted rectangles. 
Note that inside the two rectangles 
we can view $\Gamma_4$ as the graph of two function over $\Gamma_2$.
Corresponding to these functions there is an exact 1-form $df_{42}$ defined on $N_{42}$ such that the graph of $df_{42}$ in  $D(T^*N_{42})$ is precisely $V_4 \cap D(T^*N_{42})$. From the shape of $\Gamma_4$ we see that $f_{42}$ is a 
Morse-Bott function on $N_{42}$ with two critical components: it has a maximum at $K_{2}^+$ and a minimum at the torus $\Sigma_{42}$.
In a similar way we  define $D(T^*N_{20})$ which corresponds to the two dotted rectangles surrounding the intersection points of $\Gamma_0$ (blue) and $\Gamma_2$ (green). We define $df_{20}$ as before; from the shape of $\Gamma_0$ we can see
$f_{20}$ is a Morse-Bott function on $N_{20}$ with two critical components: 
it has a minimum at $K_{2}^-$ and  a maximum at the torus $\Sigma_{20}$.
\\
\newline Now take a look at the four  partially formed vertical 
dotted rectangles surrounding the four intersection points of
$\Gamma_4$ (red) and $\Gamma_0$ (blue); these rectangles intersect 
$\Gamma_0$ in four closed intervals. Together, these intervals correspond to an 
 annular region in  $L_0$ which surrounds $K_0$. Denote this region by 
$U_0 \setminus \widetilde U_0$, where $\widetilde U_0 \subset U_0$ is a
 smaller tubular 
neighborhood of $K_0$. We extend this region into $L_0$ by setting
$N_{40}= L_0 \setminus \widetilde U_0$.
Then 
%Let $N_{40} \subset L_0$ denote a manifold with boundary of the form
%$L_0$ minus a small tubular neighborhood of $K_0$. We assume that 
$N_{40}$ intersects $D(T^*K_{ef}) \subset  D(T^*L_2)$ 
precisely in the four subintervals of $\Gamma_0$. 
Now take a Weinstein 
neighborhood $D(T^*N_{40}) \subset D(T^*L_0)$ which intersects   
$D(T^*K_{ef}) \subset  D(T^*L_2)$ precisely  in the four partially formed
 rectangles.  To define $V_4$ outside of $D(T^*L_2)$ more precisely, we take an
 Morse-Bott function  $f_{40}: N_{40}\into \bbR$ such that the 
graph of $df_{40}$ in $D(T^*N_{40})\cap D(T^*L_2) $ agrees with 
$V_4 \cap D(T^*N_{40})$. Thus, the part of the graph of $df_{40}$ in $D(T^*K_{ef})$ is represented by $\Gamma_4$ in the four partially formed rectangles. When we do the plumbing identification, these four 
rectangles in each slice $D(T^*K_{ef})$ are rotated by ninety 
degrees, via multiplication by $i$. (The complete sub-region of 
$D(T^*K_{ef})$ 
which is rotated is a neighborhood of $D(\nu^*K_-^2)$, which would correspond to  a rectangle around $\Gamma_0^s$, i.e.  the two horizontal blue curves  in figure \ref{IntrofigsimpleVCdim2}.) 
After this rotation $\Gamma_4$ becomes the graph of a function
 defined over $\Gamma_0$ 
(the effect of $i$ or $-i$ is the same). Thus, $f_{40}$ is critical
along $\Sigma_{40}$, and we assume that $f_{40}$ has isolated critical points 
aside from that.  Then, from the shape of $\Gamma_4$
one can see $f_{40}$ has a maximum at $\Sigma_{40}$.
\\
\newline We are now ready to  sketch how to compute the  
Floer homology groups
$HF(V_4,V_2)$, $HF(V_2,V_0)$, and $HF(V_4,V_0)$. To do this calculation we 
modify $V_4$, $V_2$, $V_0$, but we do not change their basic shape. Namely, we 
 replace  $f_{40}$, $f_{42}$,  $f_{20}$ by $\frac{1}{n}f_{40}$, $\frac{1}{n} f_{20}$, $\frac{1}{n} f_{42}$ for some large $n \geq 1$. Then we patch the graphs of $\frac{1}{n}df_{40}$,
$\frac{1}{n}df_{42}$, $\frac{1}{n}df_{20}$ back into $V_4$, $V_2$, $V_0$ using some smooth bump function, and  denote the result by  $V_4^n$, $V_2^n$, $V_0^n$ (although $V_2^n = V_2 = L_2$, as usual). Now let $J_n$ denote any sequence of
almost complex structures on $M$ converging to some fixed $J$. 
Using a simple energy argument one can show that, for $n$ sufficiently large,
the moduli space of $J_n$-holomorphic strips in $M$ with boundary on
$(V_4^n, V_2^n)$ is completely contained in $D(T^*N_{42})$, 
and similarly for $(V_4^n, V_0^n)$, or $(V_2^n, V_0^n)$.
Once this is known, we show there exist almost complex structures
 $J_{40}^n$, $J_{20}^n$, $J_{42}^n$ such that the above moduli spaces are in 1-1 correspondence with  the gradient flow lines of 
 $\frac{1}{n}f_{40}$, $\frac{1}{n} f_{20}$, $\frac{1}{n} f_{42}$ in 
$N_{40}$, $N_{20}$, $N_{42}$,  when $n$ is 
large (this relies on Floer's standard result of this type for
 closed manifolds).
(As in Floer's work, it follows from a linearized version of this correspondence that  $J_{40}^n$, $J_{20}^n$, $J_{42}^n$ are regular for  large $n$ as well.)  
Now, we fix an $n$ sufficiently large once for all and 
replace $V_4$, $V_2$, $V_0$ by $V_4^n$, $V_2^n$, $V_0^n$ and  replace
$f_{40}$, $f_{42}$,  $f_{20}$ by $\frac{1}{n}f_{40}$, $\frac{1}{n} f_{42}$, 
$\frac{1}{n} f_{20}$, 
but we keep the old notation in both cases. Also we drop the $n$ 
from  $J_{40}^n$, $J_{20}^n$, $J_{42}^n$ and denote them 
$J_{40}$, $J_{20}$, $J_{42}$. The above correspondence of moduli spaces
yields an identification of $HF(V_4,V_2, J_{42})$ with 
$(H_{MB})_*(N_{42},f_{42})$, where the latter is a version of Morse homology for Morse-Bott functions called Morse-Bott homology. Similarly, $HF(V_2,V_0, J_{20}) $ and  $HF(V_4,V_0, J_{40})$
are isomorphic to $(H_{MB})_*(N_{20},f_{20})$ and  $(H_{MB})_*(N_{40},f_{40})$.
With our explicit Morse-Bott functions it is fairly easy to 
show that the Morse-Bott homology groups are isomorphic to  $H_*(K_+^2)$, $H_*(K_-^2)$, and $H_*(L_0 \setminus U_0)$, respectively. Moreover, there are 
explicit generators and relations for the Morse-Bott homologies which match up 
 with the ones for $H_*(K_+^2)$, $H_*(K_-^2)$, and $H_*(L_0 \setminus U_0)$ that we used in our computation of the flow category of $(N,f,g)$ in  \S  
\ref{sectionFlow}.

\subsection{Construction of  $\widetilde V_4$, $\widetilde V_2$, $\widetilde V_0$ (version IV):  Classification of  holomorphic triangles in $M$, and a final correction 
to \S \ref{firsttriangle}} 
\label{secondtriangle}
At the moment our Lagrangians $V_4$, $V_2$, $V_0$ are 
set up for computing the Floer homology groups. In this section we will
deform them one last time, by exact isotopies, to new versions 
 $\widetilde V_4$, $\widetilde V_2$, $\widetilde V_0$ (but with the same rough shape); these will allow us to describe  the moduli space of 
$J$-holomorphic triangles in $M$ with boundary on $\widetilde V_4$,
$\widetilde V_2$, $\widetilde V_0$, for a certain almost complex structure $J$. (In \S \ref{contnmapsketch} below we will explain how to 
combine the results of this section and \S \ref{Floersketch}.)  
There is also one last mistake to fix in \S \ref{firsttriangle}: In order to invoke the Riemann
mapping theorem to construct the standard holomorphic triangles $w_{ef}$, it is necessary that the boundary
arcs of the four triangles in figure \ref{CorrectionGammaIIIbig} 
are real analytic. The new  versions $\widetilde V_4$,
$\widetilde V_2$, $\widetilde V_0$ will also remedy this problem.
(Note that $\Gamma_4$,  $\Gamma_2$, $\Gamma_0$ 
in figure \ref{CorrectionGammaIIIbig} could not have  triangles 
with real analytic edges  because the edge corresponding to $\Gamma_0$ 
contains an interval in the region corresponding to $D(T^*N_{40})$.)
\\
\newline Fix a small connected compact neighborhood $U$ in $D(T^*L_2)$ which contains
all of the four triangles in $D(T^*K_{ef})$ in figure \ref{IntrofigsimpleVCdim2} for every $e,f$. 
%[See figure \ref{} for a picture of $U$ intersected with $D(T^*K_{ef})$.]
Let $J_\zeta$, $\zeta \in V$ be any family of almost complex structures on $M$ 
such that $J_\zeta|_{U} = J_\bbC|_U$ for all $\zeta \in V$ (recall $V$ is the 
domain of our holomorphic triangles). First, we show that if the area of the
four triangles in  $D(T^*K_{ef})$ in figure \ref{CorrectionGammaIIIbig} 
are sufficiently small then
any $J$-holomorphic triangle must be contained in $U \subset D(T^*L_2)$ (this 
follows from a relative version of the monotonicity lemma).
Therefore, we deform  $\Gamma_4$, $\Gamma_2$, $\Gamma_0$ 
to new versions  $\widetilde \Gamma_4$, $\widetilde \Gamma_2$, $\widetilde \Gamma_0$ in such a
way that these triangles become sufficiently small, as in
 figure \ref{CorrectionGammaIIbig}.
\begin{figure}
\begin{center}
\includegraphics[width=5in]{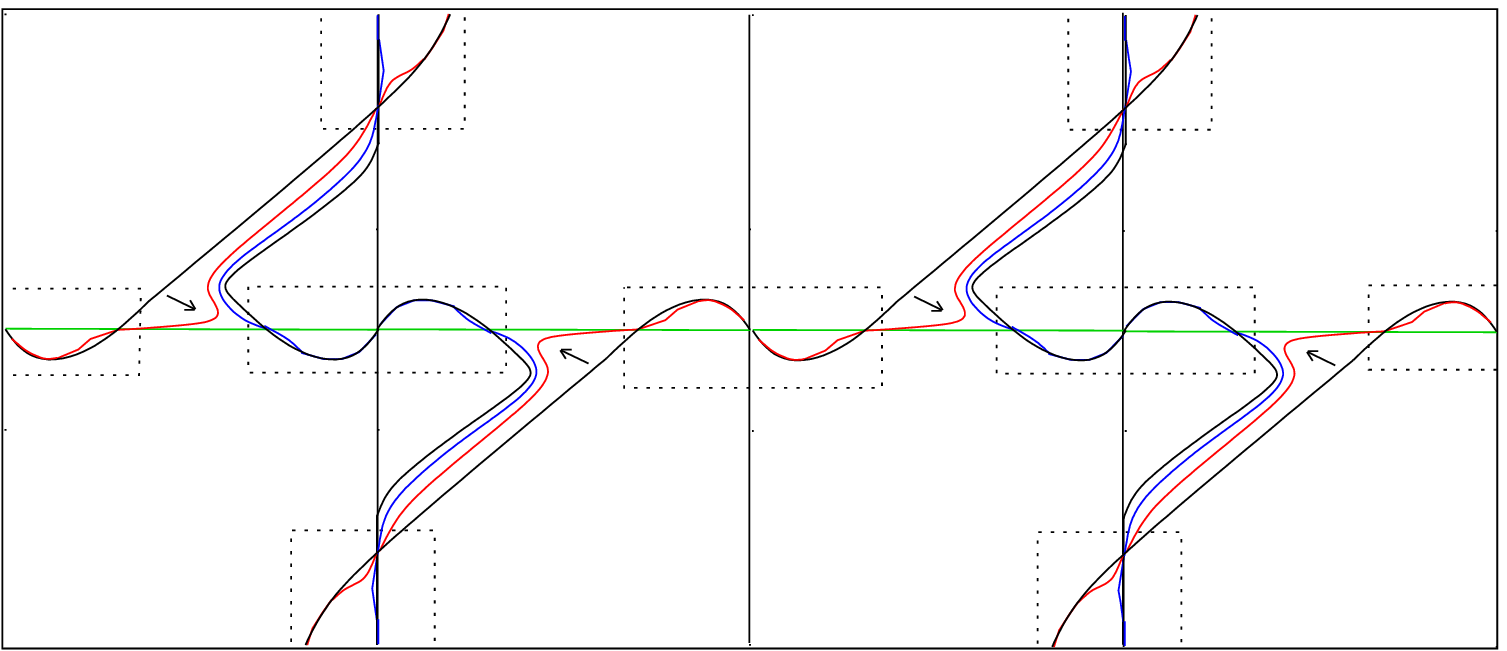}
%\put(-125,125){$\widehat{\psi}$}
%Simple version of vanishing spheres
%$\Gamma_4$,  $\Gamma_2$, $\Gamma_0$ in $\bbR /2\pi\bbZ \times [-r,r]$.}
\caption{$\widetilde \Gamma_4$ (red),  $\widetilde \Gamma_2$ (green), $\widetilde \Gamma_0$ (blue) in $\bbR /2\pi\bbZ \times [-r,r]$.} \label{CorrectionGammaIIbig}
\end{center}
\end{figure}
In addition, we arrange that the the boundaries of the new triangles are real 
analytic. 
(In order to be real analytic the new $\widetilde \Gamma$'s must interpolate back to the old $\Gamma$'s \emph{away} from the triangles; that is why there are small humps 
between  $\widetilde \Gamma_i$ and $\Gamma_i$ next to each vertex of each triangle.) We denote the corresponding new  $V_4$, $V_2$, $V_0$ by $\widetilde V_4$,
$\widetilde V_2$, $\widetilde V_0$. Since every $J-$holomorphic triangle $w$ 
in $M$ with boundary on  $\widetilde V_4$, $\widetilde V_2$, $\widetilde V_0$ lies in $U \subset D(T^*L_2)$, and 
$J|_U = J_\bbC$, our classification result from \S \ref{firsttriangle} applies
and we conclude that all such $w$ are equal to one of our standard ones 
$w_{ef}$. 
(For the monotonicity argument (see \S \ref{sectionmonotonicity})
we take a  set 
$\widetilde U \subset U$, slightly smaller than $U$ 
and set $B = U \setminus \widetilde U$ to be an annular region 
surrounding the four triangles in $D(T^*K_{ef})$ in figure \ref{IntrofigsimpleVCdim2} for every $e,f$. Then $B$ acts as a barrier to holomorphic 
triangles in $U$ with large energy. It is important that 
$V_i \cap B = \widetilde V_i \cap B$ for every $i$; this can be arranged by 
making the small humps between   $\widetilde \Gamma_i$ and $\Gamma_i$ in figure 
\ref{CorrectionGammaIIbig} at each vertex small enough.)
\\
\newline In addition, one can show that for any $J$ with $J|_U = J_\bbC|_U$ as 
above, $J$ is regular. 
This is proved in two steps. First one shows that the kernel of the linearized Cauchy-Riemann 
operator of $J$ has dimension 2; this proved by running through a version of the argument in \S \ref{firsttriangle} at the linearized level. Then one shows that the index of the operator is also equal to 2; this is proved using
a gluing formula in the Morse-Bott setting which reduces the calculation to 
computing the Maslov index of a certain loop of Lagrangian 
tangent planes; the latter is fairly easy to compute because of the rotational symmetry of  
$\widetilde V_4$, $\widetilde V_2$, $\widetilde V_0$ in $D(T^*L_2)$.

\subsection{Computing continuation maps; setting  up to compute the triangle map  $\mu_2$}
%combining the results of \S \ref{Floersketch} and \S \ref{secondtriangle}}
\label{contnmapsketch}

In this section we combine the two calculations from \S \ref{Floersketch} 
and \S \ref{secondtriangle} by inspecting the appropriate  continuation maps.  
There are two issues we have to deal with: First, in the Floer homology calculation we used
$(V_4,V_2,V_0)$, whereas in the holomorphic triangle
calculation we used $(\widetilde V_4,\widetilde V_2,\widetilde V_0)$. 
Second, $J$  
is not compatible with the almost complex structures 
 we used in the Floer homology calculation, $J_{40}$, $J_{20}$, $J_{42}$.
 Here, compatible means that $J_\zeta$ is supposed to agree with 
 $J_{40}$, $J_{20}$, or $J_{42}$,
%$\widetilde J_{40}$, $\widetilde J_{20}$, or $\widetilde J_{42}$, 
whenever $\zeta$ lies in a small neighborhood of one of the three corresponding boundary punctures of $V$ (but we have not arranged this to be so).
%(see below for the meaning of ``compatible'').
\\
\newline
To address the second issue we pick some regular almost
complex structures on $M$ for the Floer homology groups, say 
$\widetilde J_{40}$, $\widetilde J_{20}$, $\widetilde J_{42}$ 
such that  $\widetilde J_{40}|_U$, $\widetilde J_{20}|_U$, $\widetilde J_{42}|_U$
are all equal to $J_\bbC|_U$ (it is easy to see we can pick them to be regular under this constraint for abstract reasons); then we pick our almost complex structure for the triangle calculation
$J_\zeta$, $\zeta \in V$ also satisfying  $J_\zeta|_U = J_\bbC|_U$ for all $\zeta$ and  such that $J$ is compatible with $\widetilde J_{40}$, $\widetilde J_{20}$, $\widetilde J_{42}$. 
%Here, compatible means that $J_\zeta$ agrees with $\widetilde J_{40}$, $\widetilde J_{20}$, or $\widetilde J_{42}$, whenever $\zeta$ lies in a small neighborhood of one of the three corresponding boundary punctures of $V$.
The  condition $J|_U = J_\bbC|_U$ ensures that the whole discussion
from \S \ref{secondtriangle} applies to $J$; in particular, $J$ is regular.
 \\
\newline To address the first issue we take 
two functions $G,H: M \into \bbR$ such that the Hamiltonian flows $\phi_G^1, \phi_H^1: M \into M$
satisfy $\phi^1_G(V_4) = \widetilde V_4$ and $\phi_H^1(V_4)=\widetilde V_4$.
(Recall that $\widetilde V_2 = V_2$.)
 Then, we have canonical  isomorphisms
\begin{gather} \label{naturalisos}
HF(\widetilde V_4, \widetilde V_0, 0, \widetilde J_{40}) 
\cong HF(V_4, V_0, \{(G-H)\circ \phi_G^t\},
\widetilde J_{40}^{*}),\\ \notag
HF(\widetilde V_4, \widetilde V_2, 0, \widetilde J_{42}) 
\cong HF(V_4, V_2, H,
\widetilde J_{42}^{*}),\\ 
HF(\widetilde V_2, \widetilde V_0, 0, \widetilde J_{20}) 
\cong HF(V_2, V_0,\{-G \circ \phi^t_G \} ,
\widetilde J_{20}^{*}). \notag
\end{gather}
These isomorphisms arise from a straight-forward 
 equivalence between the underlying moduli spaces. For example, the 
middle one takes the form 
$u(s,t) \mapsto (\phi_H^t)^{-1} (\phi_H^1)^{-1}(u(s,t))$, where 
$u$ is a  $\widetilde J_{42}$-holomorphic strip with boundary on 
$(\widetilde V_4, \widetilde V_2)$ $=$ $(\widetilde V_4, V_2)$.
%$p \in  \widetilde V_4 \cap \widetilde V_2 =  \widetilde V_4 \cap  V_2$. 
% $t \in [0,1]$, 
%where $\gamma$ is a smooth path starting at $\widetilde V_4$ and ending at 
%$\widetilde V_2 = V_2$. 
Above, $\widetilde J_{40}^{*}$, $\widetilde J_{20}^{*}$, $\widetilde J_{42}^{*}$ 
stand for the corresponding pull backs and push forwards
of   $\widetilde J_{40}$, $\widetilde J_{20}$, $\widetilde J_{42}$, for example
$(\widetilde J_{40}^*)_t = (\phi^t_{\{(G-H)  \circ \phi^t_G   \}})_*(\phi^1_H)_*^{-1}(\widetilde J_{40})_t.$ 
(Each $J_{40}$, etc. and $\widetilde J_{40}$, etc. come in a 
family parameterized by $t\in [0,1]$, although we 
suppressed this earlier.) The precise formulas for the $\widetilde J^{*}$ are not relevant, however;  the only thing that matters is that 
they are also regular; this is  because of the natural equivalence between the 
moduli spaces (at the linearized level). Thus we are led to consider the continuation maps
\begin{gather*}% \label{contnmaps}
\phi_{40}: HF(V_4, V_0, 0, J_{40}) \into HF(V_4, V_0, \{(G-H)\circ \phi_G^t\},
\widetilde J_{40}^{*}), \\ \notag
\phi_{42}:HF(V_4, V_2, 0, J_{42}) \into HF(V_4, V_2, H, \widetilde J_{42}^{*}), 
\\
 \phi_{20}:  HF(V_2, V_0, 0, J_{20}) \into HF(V_2, V_0,\{-G \circ \phi^t_G \} ,
\widetilde J_{20}^{*}). \notag
\end{gather*}
Recall near the end of  \S \ref{Floersketch} we mentioned there is 
a nice explicit set of generators and relations for 
$HF(V_4,V_2)$, $HF(V_2,V_0)$, and $HF(V_4,V_0)$ when we use the particular almost complex structures $J_{40}$, $J_{20}$, $J_{42}$. The important thing is  they correspond exactly to the generators and relations  
we found for the flow category in \S \ref{sectionFlow}. Therefore we want to show that the 
above continuation maps \emph{fix} all the generators in this presentation (so that each continuation map in fact equals the identity on homology). 
%[, so that we can compare easily with the Flow category calculaton.]
We prove this by observing that each of the Hamiltonians, $(G-H)\circ \phi_G^t$, $H$, and $-G \circ \phi^t_G$,
% carefully inspecting  the Hamiltonians  
%$(H-G)\circ \phi_G^t$, $H$, and $-G \circ \phi^t_G$. For each Hamiltonian we show that it 
has either an absolute minimum  or an absolute maximum along the relevant torus $\Sigma_{42}$, $\Sigma_{20}$, or $\Sigma_{40}$ (where all the generators live).
Using this we can show that all $s-$dependent holomorphic 
strips $u$ (those which define the continuation map) which start and end at 
a point in $\Sigma_{42}$, $\Sigma_{20}$, or $\Sigma_{40}$  must be constant.
%Then we can take a homotopy between the given Hamiltonian $F_t$ and $0$
%which is \emph{monotone}; this means the homotopy, say $F^{01}_{s,t}$, $s,t \in [0,1]$ satisfies $\partial_s  F^{01}_{s,t}(p) \leq 0$ for all $p \in M$.
%(In case $F$ has a  maximum at $\Sigma$, one considers the continuation map going in the reverse direction.)
%Using this monotone condition we can show that all $s-$dependent holomorphic 
%strips $u$ (those which define the continuation map) which start at 
%a generator in $\Sigma$ must be constant, essentially becuase $\Sigma$ is a 
%minimum of the Hamiltonian.
\\
\newline The groups  
$HF(\widetilde V_4, \widetilde V_0, \widetilde J_{40})$, $HF(\widetilde V_2, \widetilde V_0, \widetilde J_{20})$,
$HF(\widetilde V_4, \widetilde V_2, \widetilde J_{42})$,
are suitable for computing the triangle product because we can use our explicit 
classification of $J-$ holomorphic triangles in $M$, where 
$J$ is compatible with $ \widetilde J_{40}$, $ \widetilde J_{20}$, 
$ \widetilde J_{42}$.
We also know from what we've just done that these groups
have generators and relations which correspond exactly to 
ones  we found for the flow category.
(One sees this by using the fact that the   
above continuation maps 
fix all the nice generators (and relations) on the left-hand side; 
then we use the 
natural isomorphisms from (\ref{naturalisos}), which also fix all 
the generators.)
Once we are in this situation it is fairly straight-forward to compute the triangle product (using our explicit knowledge of the holomorphic triangles) 
and compare that to the product in the flow category generator by generator; 
in fact the computations are completely parallel. 
We will not try to summarize that calculation in our present technical
set up 
(this is presented fairly clearly in the proof of Theorem \ref{fuk=flow}).
Instead, we will present in \S \ref{FukFlowsketch} below an intuitive 
argument which explains why we expect the triangle  product $\mu_2$, and the 
flow category product, $\mu^{Flow}$, to be the same. In many ways 
this argument parallels the actual argument, 
but  hopefully the main ideas are less obscured by technicalities.

%We leave the details to  the proof of Theorem \ref{fuk=flow}.

\subsection{Intuitive sketch of the isomorphism in Theorem $B$}
\label{FukFlowsketch}

%Instead of further discussing the details of the proof of Theorem  
%\ref{fuk=flow} (i.e Theorem $B$) in a way which is completely correct 
%and consistent with what we have done so far, we now present an intuitive argument 
%that explains why we expect the two categories to be isomorphic. In many ways 
%it parallels the actual argument, but the main ideas are hopefully less 
%obscured by techincalities. 

In this section  we will ignore all the technical difficulties 
that we have dealt with up until now and work with the simplest version
of the Lagrangians $L_0$, $L_2$, $L_4$. Our goal is to illustrate the essential idea of 
the calculation of the triangle product with the technical aspects 
stripped away. Along the way we will assume various things which are 
not quite true as stated, but it should be clear from the above 
that it is possible to tweak the Lagrangians so that the 
statements transform into new ones which are true, 
and are essentially  the same as the old ones. 
In this sense the argument is morally correct; 
in any case it should be a  useful
guide to the proof of Theorem \ref{fuk=flow}, which 
presents a technically correct 
argument.
%In any case, as we mentioned above, our purpose here is only to 
%provide an argume
\\
\newline Let us return to the more general case where $N$ has any number of 
2-handles,  so that we that can compare more directly with the flow 
category calculation in \S \ref{sectionFlow}.
%work in the generality of Theorem $B$,
%explain Theorem $B$ in full generality,  Thus, let 
%N$ now be any closed 4-manifold, 
%which admits a handle-deomposition of $N$ without 1- and 3-handles, 
%and let  $(f,g)$ be some Morse-Smale pair on $N$ which induces this 
%handle-deomposition. % of $N$ without 1- and 3-handles. 
Then, the handle-decomposition of $N$ gives rise to $k$ framed knots 
$K_1, \ldots, K_k \subset S^3$, one for each 2-handle. Let 
$L_0 = S^3$, and take $k$ parameterizations of tubular neighborhoods 
of $K_1, \ldots, K_k \subset L_0$, which are determined up to isotopy, say 
$\phi_j : S^1 \times D^2 \into L_0$, $j =1, \ldots, k$.
Now take $k$ additional copies of $S^3$, denoted $L_2^j$, 
$j =1, \ldots, k$. As before we have two natural unknots  $K_{\pm}^j \subset L_2^j$ for each $j$, 
and for each $e \in K_+^j$, $f \in K_-^j$,  we have the great circle $K_{ef}^j \subset L_2^j$. We construct $M$ by taking the disk bundle
$D(T^*L_0)$ and plumbing on each of the disk bundles $D(T^*L_2^j)$ along the knots $K_j$. 
Here $D(\nu^*K_-^j) \cong S^1 \times D^2$ will be identified with a tubular 
neighborhood of $K_j$ in $L_0$, say $U_0^j \subset L_0$ using the map $\phi_j$.
The last Lagrangian $L_4$ is defined by doing Lagrangian surgery 
of $L_0$ with each of the $L_2^j$ (in the same way $L_4$ was defined before).  
\\
\newline We make two main assumptions. First, we assume that inside each 
$D(T^*L_2^j)$, we have our standard holomorphic triangles $w_{ef}^j$, 
one for each
 $e \in K_+^j$, $f \in K_-^j$, and  we  assume that every holomorphic triangle 
in $M$ coincides with one of these $w_{ef}^j$. Second, we assume we have identifications
$$HF(L_4, L_2^j) \cong H_*(L_4 \cap L_2^j) \cong H_*(K_2^j),$$ 
$$HF(L_2^j, L_0) \cong  H_*(L_2^j\cap L_0) \cong H_*(K_j),$$
$$HF(L_4 , L_0) \cong H_*(L_4 \cap L_0)\cong H_*(L_0 \setminus (\cup_j U_0^j)).$$In this way the morphism spaces in 
$H(Fuk^{\rightarrow}(M; L_4, \{L_2^j\}, L_0))$ are identified with those in
 $Flow(N,f,g)$. Recall that $w_{ef}^j$ is the unique
% $w_{ef}^j: V \into D(T^*K_{ef}^j) \subset D(T^*L_2^j)  $ 
holomorphic triangle in $D(T^*K_{ef}^j)$ with two of its vertices at $e$, $f$.  The final vertex of $w_{ef}^j$ determines a third point, which we denote
$\psi_j(e,f) \in D(T^*K_{ef})$; it lies in the sphere conormal bundle 
$S(\nu^*K_-^j) \subset D(T^*L_2^j)$, which is a torus, since $\nu^*K_-^j \cong S^1 \times \bbR^2$.
 % \subset D(T^*L_2^j)$.
Now,  $S(\nu^*K_-^j)$ is identified by $\phi_j$ 
with the boundary of a smaller neighborhood $U_j \subset U_0^j$ 
of $K_j$ in $L_0$.
\\
\newline Let us choose some generators for  $H_*(K_+^j \times K_-^j)$, 
$H_*(S(\nu^*K_-^j))$, and $H_*(L_0 \setminus (\cup_j U_0^j))$. 
We choose notation similar to that used in \S \ref{sectionFlow}.
For  $H_*(K_+^j \times K_-^j)$, fix  $r_{\pm}^j \in K_+^j$ and take 
the generators
$$T_+^j = [K_+^j \times K_-^j], \,\,
 \lambda_+^j =  [r_+^j \times K_-^j],\,\,
\mu_+^j = [K_+^j \times r_-^j], \,\,
 s_+^j = [r_+^j \times r_-^j].$$ 
For $H_*(S(\nu^*K_-^j))$, we identify $S(\nu^*K_-^j) = K_-^j \times K_+^j$ (this is natural because a point in 
$K_+^j$ determines a direction in the normal bundle of $K_-^j$) and
 take the generators
$$T_-^j  =  [K_-^j \times K_+^j],\,\,
\lambda_-^j =  [K_-^j \times r_+^j],\,\,
\mu_-^j = [r_-^j \times K_+^j],\,\,
s_-^j = [r_-^j \times r_+^j]. $$
Now $H_*(L_0 \setminus (\cup_j U_0^j))$ is generated by cycles in 
$\partial U_j$ (where $U_j \subset U_0^j$ and $\phi_j(S(\nu^*K_-^j))$ $=$ $\partial U_j$). Let $q_j \in \partial U_j$ and let 
$\lambda_j, \mu_j \subset H_1(\partial U_j)$ be represented by 
circles with $\ell (\lambda_j, K_j) =0$ and $\ell k (\mu_j, K_j) = 1$.
Then  $H_*(L_0 \setminus (\cup_j U_0^j))$ is generated by 
$[\partial U_j],$ % \, \, 
$\lambda_j,$ % \, 
$\mu_j,$ %\, 
and $[q_j]$, % $$
with the same relations (\ref{homologyrelations}) 
from \S \ref{sectionFlow}.
%\begin{gather} 
% \Sigma_j [\partial U_j] =0, \, [q_1] = \ldots =[q_k],\\
% \lambda_j = \Sigma_{i\neq j} \ell k(K_-^j,K_i^-)\mu_i, \,  j =1,\ldots, k, 
%\notag
%\end{gather}
 \\
\newline Here is an intuitive way to  think about the  triangle product
$$\mu_2: HF(L_4, L_2^j) \otimes  HF(L_2^j, L_0) \into HF(L_4, L_0)$$
in terms of the geometric cycles above. Take a pair of cycles 
 on the right hand side, which we represent by submanifolds 
 $C_+^j$, $C_-^j$ in $K_+^j$, $K_-^j$ (so  $C_{\pm}^j = K_j^{\pm}$ or $C_{\pm}^j = r_j^{\pm}$). Now, as a first step, restrict attention to $D(T^*L_2^j)$
and consider  the cycle in $S(\nu^*K_-^j)$ 
that gets swept out by the points $\psi_j(e,f)$, as $(e,f)$ range over 
$C_+^j \times C_-^j$, and denote this by $\mu_2|_{D(T^*L_2^j)}(C_+^j,C_-^j)$.
Next, $S(\nu^*K_-^j)$ is identified by the framing $\phi_j$ with 
$\partial U_j$, and  in this way we get a cycle in $\partial U_j$
%$$H_*(L_0 \setminus \cup_j U_0^j) \cong HF(L_4, L_0)$$
which is $\mu_2(C_+^j, C_-^j) = \phi_j(\mu_2|_{D(T^*L_2^j)}(C_+^j,C_-^j))$.
\\
\newline In the first step above, it is easy to see that
 the restricted triangle product 
$$\widetilde \mu_2 = \mu_2|_{D(T^*L_2^j)}: H_*(K_+^j) \otimes H_*(K_-^j) \into H_*(S(\nu^*K_-^j))$$
satisfies $\widetilde \mu_2(T_+^j) = T_-^j$, 
$\widetilde \mu_2(\lambda_+^j) = \mu_-^j$,
$\widetilde \mu_2(\mu_+^j) = \lambda_-^j$, 
$\widetilde \mu_2(s_+^j) = s_-^j$.
%$\widetilde \mu_2(T_+^j) = T_-^j, \, \, \, \, 
%\widetilde \mu_2(\lambda_+^j) = \mu_-^j, \, \,  \, \,
%\widetilde \mu_2(\mu_+^j) = \lambda_-^j,  \,\, \, \,
%\widetilde \mu_2(s_+^j) = s_-^j.$
These are analogous to the relations (\ref{PhiRelations}) satisfied by the flow map
$\Phi$ in the fixed handle $H_j$ in \S \ref{sectionFlow}. In the second step, one composes $\widetilde \mu_2$  with the 
map on homology induced by $\phi_j: S(\nu^*K_-^j) \into \partial U_j$ and
 the inclusion  $ \partial U_j \subset L_0 \setminus (\cup_j U_0^j)$, 
$$(\phi_j)_*: H_*(S(\nu^*K_-^j)) \into H_*(L_0 \setminus (\cup_j U_0^j)).$$
As in (\ref{phi_jformula}), we have 
$$ (\phi_j)_*(\lambda_-^j)= \lambda_j + m_j\mu_j,  \, (\phi_j)_*(\mu_-^j) = 
\mu_j.$$
Combining these we have, for example, 
$$\mu_2([K_+^j]\otimes [r_-^j]) = (\phi_j)_*(\widetilde \mu_2 (\mu_+^j))
= (\phi_j)_*(\lambda_-^j) =  \lambda_j + m_j\mu_j,\text{ while }$$
$$
\mu^{Flow}([K_+^j] \otimes [p_-^j]) = (\phi_j)_*(\Phi(\mu_+)) = (\phi_j)_*
(\lambda_-) =  \lambda_j + m_j\mu_j.$$
%which is the same as the formula for $\mu^{Flow}([K_+^j] \otimes [p_-^j])$.
Notice that the answers agree, but also the calculations are parallel.

\subsection{The case general case $\dim N=2n$,  when
$f$ has critical values $0,n,2n$}\label{0,n,2n}
%, and some remarks on 3-manifolds.

In this section we briefly summarize how things work in an arbitrary dimension
$\dim N=2n$.
The handle decomposition of $N$ corresponding to 
$(f,g)$ is determined by $k$ framed $(n-1)$-spheres in a $(2n-1)$-sphere,
$K_j \subset S^{2n-1}$, $j =1, \ldots, k$. The flow category calculation 
is much the same, and we omit that discussion.
To construct $M$ we take 
$L_0 = S^{2n-1}$ and  $L_n^j = S^{2n-1}$, $j =1, \ldots, k$.
In $L_0$ we have the framed  spheres $K_j$ and inside
$L_n^j$ we have the two $(n-1)$-spheres with obvious framings 
$$K_+^j = \{(u_1, \ldots, u_{n}, 0, \ldots, 0) : \Sigma u_j^2 =1\}, 
 K_-^j = \{(0, \ldots, 0, 
u_{n+1}, \ldots, u_{2n}): \Sigma u_j^2 =1 \}.$$
We plumb as before; and $L_4$ is defined in the same way, by 
applying the time $\pi/2$ Hamiltonian flow to $D(\nu^*K_-^j)$.
The classification of holomorphic triangles in
$M$ works as before: For each  $e \in K_+^j$, $f \in K_-^j$ we have the great circle $K_{ef}^j \subset L_2^j$ and 
there is a unique holomorphic triangle $w_{ef}^j$ in $D(T^*K_{ef}^j) \subset D(T^*L_2^j)$
with each of its vertices determined by $(e,f)$; and these are the only holomorphic triangles in $M$. Then the identification of the two categories is done in the 
same way as the sketch above. 
%(In general, it is worth noting that we do not need to know teh relations 
(One thing to note is that we do not need to explicitly know the relations
(\ref{homologyrelations}). Indeed, whatever these relations are, they 
show up abstractly in the calculation of the flow category and the 
directed Fukaya category in exactly the same way.)

\section{Constructing  $M$ (the fiber)} \label{fiber}

In \S \ref{plumbing} we give a precise treatment of the
 plumbing construction sketched in \S \ref{fibersketch}. 
It turns out that the plumbing, which we denote $M_0$, 
 has boundary which is not smooth. 
Therefore, in \S \ref{handles} we construct a slightly bigger space
$M$ which is 
obtained from $M_0$ by smoothing the boundary, so that $M$ has convex, 
contact type boundary. % (see  \S \ref{handles} for more details).
%Ultimately the results of this paper will apply to $M$, but most of the time 
%we will work with $M_0 \subset M$, since it is easier to describe 
%Lagrangian submanifolds there.

\subsection{Plumbing} \label{plumbing} Let $(N,f,g)$ be as in \S \ref{sectionFlow}.
Denote the critical points of $f$ by $x_0, x_2^j, x_4,$ \label{x_0, x_2^j, x_4}
$j=1,\ldots,k$.  Here the index of the critical point is given by the subscript. Let $$L_0 = f^{-1}(1) \cong S^3, \text{ and }L_2^j = S^3, \label{L_0,L_2^j} \, j = 1\ldots k.$$ 
Now define
$$K_j = U(x_2^j) \cap f^{-1}(1) \subset L_0,\label{K_j} $$  
and let
$$\phi_j: S^1 \times \bbR^2 \into L_0 \label{phi_j}$$ 
denote a parameterization of a tubular neighborhood of $K_j$ determined by  $(f,g)$ up to isotopy.  
Let $K_-^j, K_+^j \subset L_2^j$, \label{K_-^j} denote the two unknots  
$$K_-^j= \{(0,0, x_3, x_4) | x_3^2 + x_4^2 =1\} \text{ and } 
K_+^j = \{ (x_1, x_2,0,0) | x_1^2 + x_2^2 =1\}.\label{K_+^j}$$
We specify a  parameterization of a tubular neighborhood of $K_-^j$ in $L_2^j$
as follows. First take the obvious identification 
of $S^1 \times \bbR^2$ with the normal bundle of $K_-^j$ in $TS^3$. %with respect to the round metric $g_{S^3}$. %\cong \nu^{\perp}_{g_{S^3}}K_-^j $. %, where $\perp = \perp_{g_{S^3}}$. 
Then, using this identification, we let 
$$\phi_-^j :  S^1 \times D^2_{\pi/3} \into L_0,\label{phi_-^j}$$
be the  map given by geodesic coordinates with respect to the round metric on $S^3$. (Here we take $\pi/3 <\pi/2$ so that $\phi_-^j$ is an embedding.)
For use later on we define $\phi_+^j: S^1 \times D^2_{\pi/3} \into L_2^j$ \label{phi_+^j} for $K_+^j$ in the same way. 
%[?? Do I need $\phi^+_j$?? Yes for $f_{42}]
%[?? Why do I need $r < \pi/4$ ?]
\\
\newline We start with the disjoint union of the 
disk bundles $D(T^*L_0)$ and $D(T^*L_2^j)$, $j=1, \ldots, k$. 
Then for each $j$ we glue a neighborhood of $K_j \subset D(T^*L_0)$  to a neighborhood of $K_-^j \subset D(T^*L_2^j)$. Each neighborhood can be symplectically identified with a neighborhood of 
$S^1\times\{0\}$ in  $T^*S^1\times  T^*\bbR^2 \cong T^*S^1 \times \bbC^2$ and 
each gluing map  is of the form $id_{T^*S^1} \times m_i$, where $m_i$ is multiplication by $i$ on $\bbC^2$.
\\
\newline To be more precise, we choose some convenient metrics with which to form the disk 
bundles $D(T^*L_0)$ and $D(T^*L_2^j)$, and we fix the radii. Let 
$g_2$ \label{g_2} be a Riemannian metric on $\cup_j L_2^j$
such that, 
$$g_2 = (\phi_-^j)_*(g_{S^1} \times g_{\bbR^2}) \text{ on }\phi_-^j(S^1 \times D^2_{\pi/3}),$$ % 
$$g_2=  (\phi_+^j)_*(g_{S^1} \times g_{\bbR^2}) \text{ on }\phi_+^j(S^1 \times D^2_{\pi/3}),$$
and outside of a small neighborhood of $\phi_-^j(S^1 \times D^2_{\pi/3}) \cup \phi_+^j(S^1 \times D^2_{\pi/3})$, $g_2$ is the round metric $g_{S^3}$. On the intermediate region we linearly interpolate between $(\phi_j^{\pm})_*(g_{S^1} \times g_{\bbR^2})$ and $g_{S^3}$ using some cut-off function.
Let $g_0$ \label{g_0} be a Riemannian metric on $L_0$, such that 
 $$g_0 = (\phi_j)_*(g_{S^1} \times 
g_{\bbR^2}) \text{ on }\phi_j(S^1 \times \bbR^2),$$ %S^1 \times D^2_r ?? 
and outside of $\phi_j(S^1 \times \bbR^2)$, $g_0$ is arbitrary.
Now, fix $r,R>0$ \label{r, R} such that
$$0<r < R< \pi/3.$$ %such that $\pi/8 < \pi/2-R$. (r < \pi/4 < R why ??)
$R$ will be the radius of $D(T^*L_0) = D_r^{g_0}(T^*L_0)$ and $r$ will be the radius of $D(T^*L_2^j)= D_r^{g_2}(T^*L_2^j)$.
\\
\newline %We are now ready to describe the plumbing construction in detail.
Now $\phi_-^j, \phi_j$ respectively give rise to exact symplectic  
identifications
$$\rho_-^j : T^*(S^1) \times T^*(D^2_r) \into T^*(\phi_-^j(S^1 \times D^2_r)),$$
$$\rho_j:  T^*(S^1) \times T^*(\bbR^2) \into T^*(\phi_j(S^1 \times \bbR^2)).$$
We take two sets of coordinates on $T^*(S^1) \times T^*(\bbR^2)$
$$q \in S^1, p \in T^*_q S^1, x \in \bbR^2, y \in T^*_x\bbR^2,$$ 
$$ \widetilde q \in S^1, \widetilde p \in T^*_{\widetilde q} S^1, \widetilde x \in \bbR^2, 
\widetilde y \in T^*_{\widetilde x} \bbR^2.$$ 
To do the plumbing we will identify the following two regions in 
$D_R^{g_0}(T^*L_0)$ and
$D_r^{g_2}(T^*L_2^j)$ respectively. Let
\begin{gather*}
U_0^j = \rho_j( \{((q,p),(x,y)) : 
  |p|_{S^1}^2+ |y|^2_{\bbR^2} \leq R^2, 
|x|_{\bbR^2}^2 + |p|_{S^1}^2 \leq r^2 \}),\\
U_2^j= \rho_-^j(\{((\widetilde q, \widetilde p),(\widetilde x,\widetilde y)) : 
 |\widetilde p|_{S^1}^2+ |\widetilde y|^2_{\bbR^2} \leq r^2, 
|\widetilde p|_{S^1}^2 + |\widetilde x|^2_{\bbR^2}  \leq R^2\}).\label{U_2^j,U_0^j}
\end{gather*}
Note that, since $|x|^2_{\bbR^2} \leq |p|_{S^1}^2 + |x|^2_{\bbR^2} \leq r^2 < (\pi/3)^2$, $g_0$ agrees 
with the metric used in $U_0^j$, and so indeed have
$$U_0^j \subset D_{R}^{g_0}(T^*L_0)\text{ and, similarly, } U_2^j \subset D_{R}^{g_2}(T^*L_2^j).$$ 
%Set $U_2 = \cup_j  U_2^j$ and $U_0= \cup_j U_0^j$. 
Define a symplectomorphism
$$\tau^j : U_2^j \into U_0^j \label{tau} \text{ by }  (\widetilde q, \widetilde p, \widetilde x, \widetilde y) \mapsto (q,p,-y,x).$$ 
Define $M_0$ by gluing  $D_{R}^{g_0}(T^*L_0)$
and $D_{r}^{g_2}(T^*L_2^j)$ along the subsets $U_0^j$, 
$U_2^j$  for each $j=1,\ldots,k$ using $\tau^j$.
We call this a plumbing (along $K_j, K_-^j$) and denote it 
$$M_0 =  D_{R}^{g_0}(T^*L_0) \boxplus_{\,j=1}^{\,j=k} D_{r}^{g_2}(T^*L_2^j).$$

\begin{lemma} The interior of $M_0$ is smooth and has an exact symplectic structure which agrees with the standard symplectic structures on $D_R^{g_0}(T^*L_0)$ and $D_r^{g_2}(T^*L_j^2)$, $j=1,\ldots,k$.
\end{lemma}

\begin{proof} The main issue is topological: We must prove that $U_0^j$ and $U_2^j$ reach the boundary of  $D_R^{g_0}(T^*L_0)$ and  $D_r^{g_2}(T^*L_j^2)$ respectively. (To see what could go wrong imagine gluing $\bbR \times [-1,1]$ to $\bbR \times [-1,1]$ along $[-1/2,1/2]\times [-1/2,1/2]$,  using the map $(x,y) \mapsto (-y,x)$.)
% the resulting space is not locally homeomorphic to $\bbR^2$.) 
The boundary of $U_0^j$ has two parts, which we will call horizontal and vertical. The horizontal part is
$$\partial_h U_0^j = \{((q,p),(x,y)) : 
  |p|_{S^1}^2+ |y|^2_{\bbR^2} = R^2, 
|x|_{\bbR^2}^2 + |p|\\
\newline_{S^1}^2 \leq  r^2 \},$$
and the vertical part is 
$$\partial_v U_0^j  = \{((q,p),(x,y)) : 
  |p|_{S^1}^2+ |y|^2_{\bbR^2} \leq R^2, 
|x|_{\bbR^2}^2 + |p|_{S^1}^2  = r^2 \},$$
Similarly we have $\partial_h U_2^j$ and $\partial_v U_2^j$. Note that $\tau^j$ maps $\partial_h U_2^j$ to $\partial_v U_0^j$ and $\partial_v U_2^j$ to  $\partial_h U_0^j$. 
Now, since $|x|^2_{\bbR^2} \leq |p|_{S^1}^2 + |x|^2_{\bbR^2} \leq r^2 < (\pi/3)^2$, $g_0$ agrees 
with the metric used in $U_0^j$, and so we have 
$$\partial_h U_0^j \subset S_{R}^{g_0}(T^*L_0).$$
Similarly 
$\partial_h U_2^j \subset  S_{r}^{g_2}(T^*L_2^j).$ 
This shows that  $U_0^j$ and $U_2^j$ each reach the boundary of $D_R^{g_0}(T^*L_0)$ and  $D_r^{g_2}(T^*L_j^2)$ respectively, and so $\Int(M_0)$ is locally Euclidean. It is smooth and symplectic since $\tau^j$ is, and
the Mayer-Vietoris sequence in de Rham cohomology shows the symplectic form is exact. \end{proof} 
Let us denote the symplectic structure on $M_0$ by $\omega_0$. We fix a primitive $\theta_0$, with $\omega_0= d\theta_0$. \label{omega_0,theta_0}

\subsection{Handle attachments} \label{handles} The simplest example of a plumbing is obtained by gluing $\bbR \times [-1,1]$
to $\bbR \times [-1,1]$  along $[-1,1]\times [-1,1]$ using $(x,y) \mapsto (-y,x)$. This shows $M_0$  is not an exact symplectic manifold with corners in the sense of \cite[\S 7a]{S08}: it has obtuse corners, and the combined convexity of $S_R^{g_0}(T^*L_0)$ and $S_r^{g_2}(T^*L_2^j)$ are not enough to prevent holomorphic curves from hitting a corner.
%and it is not obvious how to choose $\theta_0$ so that convexity condition is satisfied. 
To remedy this we have an alternative construction from  \cite{JC} (see also 
\cite{JA}). There
 we construct $M$ by attaching $k$ Morse-Bott type handles to $D(T^*L_0)$
in the style of \cite{W}, one for each $j=1, \ldots, k$. 
Each handle $H_j$ is diffeomorphic to 
$D(T^*(S^1 \times D^2))$, where we think of $S^1$ as a critical manifold (which corresponds to $K_+^j$)
and $S^1 \times D^2$ as its unstable manifold. To attach $H_j$ we identify $D(T^*(S^1 \times D^2))|(S^1 \times\partial D^2)$
with a neighborhood of $S(\nu^*K_j)$ in $S(T^*L_0)$ so that $S^1 \times D^2$ and
$D(\nu^*K_j)$ together form a Lagrangian 3-sphere $\widetilde L_2^j$. 
The following lemma summarizes the relation between $M_0$ and $M$.

\begin{lemma}[\cite{JC}] \label{embedding} There is an exact symplectic manifold 
$(M,\omega,\theta)$ \label{(M,omega,theta)} with smooth, convex, contact type boundary, 
together with an exact symplectic embedding 
$$j:(M_0,\omega_0,\theta_0) \into (M,\omega,\theta).$$
such that $j(L_2^j) = \widetilde L_2^j$ and $j|L_0$ is the obvious embedding.
%There is a homeomorphism  $M \into M_0$ which retracts $M$ onto $j(M_0)$.
%and $dist(\partial M_0, \partial M) >0$. 
See figure \ref{figureM_0}. $\square$
\end{lemma}
 We mention also that $M$ and $M_0$ are homeomorphic, via a retraction  
$M \into j(M_0)$, but we will not need this.
\begin{figure}
\begin{center}
\includegraphics[width=3in]{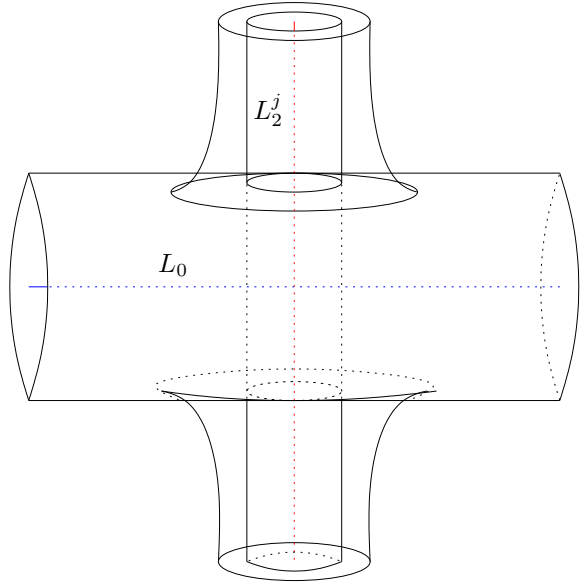} 
\put(-124,175){$L_2^j$}
\put(-160,117){$L_0$}
\caption{Schematic of $M_0$ (inner) and $M$ (outer) near  the plumbing region. }
\label{figureM_0}
\end{center}
\end{figure}

\section{Constructing $V_4, V_2^j, V_0$ (the vanishing spheres) }\label{bigsectionvanishingcycles}

In this section we construct the vanishing spheres $V_4, V_2^j, V_0$ in $M_0$;
these correspond to version II of the vanishing spheres in 
\S \ref{VCsketch}. We will only have a few 
things to add beyond the description  there.
\\
\newline  We regard $T^*L_2^j=T^*S^3$ as 
$\{(u,v) \in \bbR^4 \times \bbR^4 : |u| =1, u \cdot v = 0\}$, \label{(u,v)}
with the restriction of the symplectic structure $\Sigma_j dv_j \wedge du_j$ 
from $\bbR^8$. Recall from \S \ref{fiber} we have the two standard unknots
$K_+^j$, $K_-^j$ in $L_2^j$. For each $e \in K_+^j$, $f \in K_-^j$, we
 have the great circle through $e,f$ in $L_2^j$, denoted 
$K_{ef}^j = \{(\cos(\theta)e, \sin(\theta)f) : \theta \in  \bbR/ 2\pi\bbZ \}\label{K_{ef}}$. (Note that $K_{ef}^j = K_{\pm e, \pm f}^j$.) 
We also embed $T^*K_{ef}^j$ 
in $T^*L_2^j$ as  %using the coordinates $(u,v)$:
$$T^*K_{ef}^j = \{( (\cos(\theta)e, \sin(\theta)f), \lambda (-\sin(\theta)e, \cos(\theta)f): \theta \in  \bbR/ 2\pi\bbZ , \lambda \in \bbR \}.$$
For each $e \in K_+^j$, $f \in K_-^j$ we fix an identification
$$\sigma_{ef}^j :  (\bbR /2\pi\bbZ) \times \bbR \into T^*K_{ef},\phantom{2}
\sigma_{ef}^j(\theta, \lambda) = (\cos(\theta)e, \sin(\theta)f), 
\lambda (-\sin(\theta)e, \cos(\theta)f). \label{(theta,lambda)}$$
Note that  $\sigma_{ef}^j$ maps $(\bbR /2\pi\bbZ) \times [-r,r]$ 
onto $D_r^{g_{S^3}}(T^*K_{ef}^j)$, 
where $g_{S^3}$ is the round metric. But in fact 
 $D_r^{g_{S^3}}(T^*K_{ef}^j) = D_r^{g_2}(T^*K_{ef}^j)$, because
$g_{S^3}(v,v) = g_2(v,v)$  for any $v \in T^*K_{ef}^j$.
%\begin{equation}
%g_{S^3}(v,v) = g_2(v,v) \text{ for any } v \in T^*K_{ef}.\label{roundvsg_2}
%\end{equation}
(Indeed, it is easy to check that 
$g_{S^3}$ and $(\phi_-^j)_*(g_{S^1} \times g_{\bbR^2})$ agree on $T^*K_{ef}^j$ 
in a neighborhood of $K_-^j$, and then it follows that the 
equality holds on all of $S^3$ because
$g_2$ is defined by linearly interpolating between $g_{S^3}$ and $(\phi_-^j)_*(g_{S^1} \times g_{\bbR^2})$.) From now on we will omit the metric and radius from 
the notation in $D(T^*L_2^j) = D_r^{g_2}(T^*L_2^j)$ and $D(T^*L_0) = D_R^{g_0}(T^*L_0)$.
\\
\newline Let $\Gamma_0$,$\Gamma_2$,$\Gamma_4$ be the three curves as in 
figure \ref{CorrectionGammaIIIbig}.
\begin{figure}
\begin{center}
\includegraphics[width=5in]{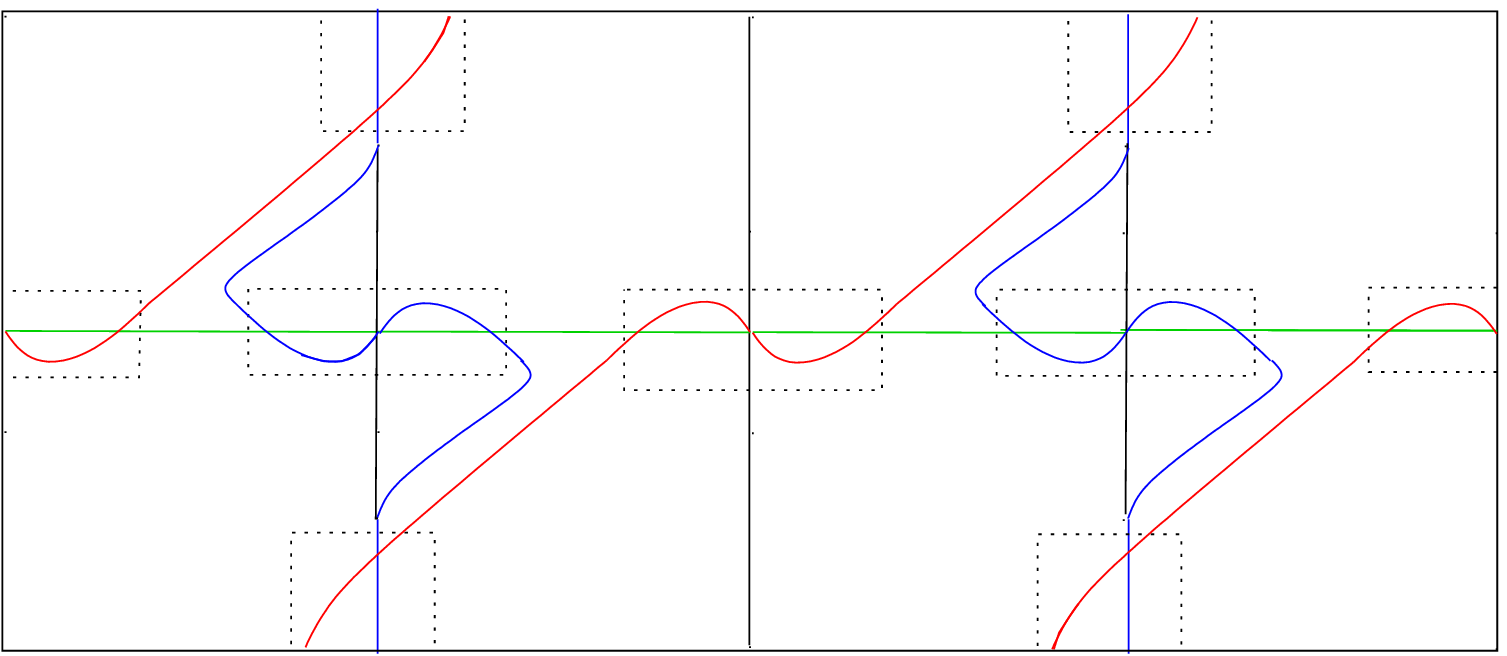}
%\put(-125,125){$\widehat{\psi}$}
%\caption{%Simple version of vanishing spheres
%$\Gamma_4$,  $\Gamma_2$, $\Gamma_0$ in $\bbR /2\pi\bbZ \times [-r,r]$.}
\caption{$\Gamma_4$ (red),  $\Gamma_2$ (green), $\Gamma_0$ (blue) in 
$\bbR /2\pi\bbZ \times [-r,r]$.}
\label{CorrectionGammaIIIbig(Floer)}
\end{center}
\end{figure}
We define some 3 dimensional submanifolds (possibly with boundary) 
$$W_0^j, W_2^j, W_4^j \subset D(T^*L_2^j), \text{ by } W_i^j= \cup_{e,f} \, \sigma_{ef}^j(\Gamma_i), \, i=0,2,4. \label{W_0^j,W_4^j} $$
This means that $W_0^j, W_2^j, W_4^j$ are rotationally symmetric 
in the sense that  the intersection of  $W_0^j$, $W_2^j$, $W_4^j$
with any slice $D(T^*K_{ef}^j)$,
%$W_i^j \cap T^*K_{ef}^j,$$ $i=0,2,4$, 
is given by figure \ref{CorrectionGammaIIIbig}.
We define the parts of $V_0, V_2^j, V_4 \subset M_0$ 
which lie in $D(T^*L_2^j)$ by 
$$V_0 \cap D(T^*L_2^j) = W_0^j, \, V_2^j \cap D(T^*L_2^j) = W_2^j, \, 
V_4 \cap D(T^*L_2^j) = W_4^j.$$
Note that $V_2^j= L_2^j = W_2^j$. It is easy to see that  
 $W_0^j$, $W_2^j$, $W_4^j$ are Lagrangian. 
Note that $\Gamma_0$,  $\Gamma_2$, $\Gamma_4$ are exact isotopic to 
the three corresponding curves $\Gamma_0^s$,  $\Gamma_2^s$, $\Gamma_4^s$ in 
figure \ref{IntrofigsimpleVCdim2} in \S \ref{VCsimplesketch}, since the signed area between $\Gamma_i$ and $\Gamma_i^s$ is zero.
From this it is not it difficult to show that  
$W_0^j$ and $W_4^j$ are exact isotopic to 
$D(\nu^*K_-^j)$ and $D(\nu^*K_+^j)$, respectively.
\\
\newline We define $V_0 \subset M_0$ %\label{V_0,V_2^j,V_4(c)} 
as the union of 
$L_0 \setminus \big(\cup_{j=1}^{j=k}  \phi_j(S^1 \times \Int(D^2_r))\big)$ and $\cup_{j=1}^{j=k} \, W_0^j$; here $D^2_r$ is the closed disk of radius $r$. 
(Notice that these two pieces match up, 
since  $W_0^j$ coincides with $D_r(\nu^*K_-^j) = \phi_j(S^1 \times D^2_r)$ 
near the boundary of $D(T^*L_2^j)$.) Then 
$V_0$ is exact isotopic to $L_0$, since 
$W_0^j$ is  exact isotopic to $D(\nu^*K_-^j)$.
\\
\newline  We give a rough definition of $V_4$ and leave more precise
 details to 
\S \ref{graphs} below. First take the graph of an exact one-form 
$df$ in $D(T^*L_0)$
defined on $L_0 \setminus \big(\cup_{j=1}^{j=k}  \phi_j(S^1 \times \Int(D^2_{r}))\big)$ 
such that the graph of $df$ matches up with $\cup_{j=1}^{j=k} \, W_4^j$.
(In \S \ref{graphs} below we will specify $df$ more precisely and denote it 
$df_{40}$.)
Now define $V_4 \subset M_0$ %\label{V_0,V_2^j,V_4(c)} 
as the union of 
the graph of $df$ and $\cup_{j=1}^{j=k} \, W_4^j$.
Then, $V_4$ is an exact Lagrangian  sphere because $df$ is exact isotopic to  $L_0 \setminus \big(\cup_{j=1}^{j=k}  \phi_j(S^1 \times \Int( D^2_{r})) \big)$ and 
$W_4^j$ is  exact isotopic to $D_{r}(\nu^*K_+^j) \cong S^1 \times D^2_r$.
This also shows $V_4$ is diffeomorphic to the result of doing surgery on 
$L_0$ along the framed link $\cup_jK_j$. Related to this, 
$V_4$ is exact isotopic to $L_4$ from \S \ref{FukFlowsketch}
(see also \S \ref{VCsimplesketch}); recall $L_4$ is the 
 (Morse-Bott) Lagrangian surgery of $L_0$ with each of the $L_2^j$.
% [Maybe put in 2D diagram to illustrate ''matching up'' like figure 
%\ref{2Dplumbing} ]
%\begin{figure}
%\begin{center}
%\includegraphics[width=3in]{figure2Dplumbing} 
%\put(-150,207){A}
%\put(-150,-3){B}
%\caption{Plumbing of $D_r(T^*S^1)$ (rectangular) and $D_R(T^*S^1)$ (round) along an $S^0$. (The two horizontal edges at $A$ and $B$ are identified.) 
%The vanshing cycles are $V_4$ (red), $V_2$ (green), $V_0$ (blue). }
%\label{2Dplumbing}
%\end{center}
%\end{figure}
\section{Computing the  Floer homology groups}\label{sectionFloer}

In this section we compute the Floer homology groups 
$HF(V_4, V_2^j)$, $HF(V_2^j,V_0)$, $HF(V_4, V_0)$, $j=1, \ldots,k$
(compare with the summary in \S \ref{Floersketch}).
Using results from \S \ref{localstrips} it follows immediately that 
these groups are identified with the 
Morse-Bott homology groups of certain functions (see \S \ref{Floergroups}).
The focus of this section is to find explicit generators and relations for 
these which correspond perfectly to the ones we found for the flow category 
in \S \ref{sectionFlow} (see \S \ref{sectionHF}) .

\subsection{Expressing certain parts of $V_4$,   $V_2^j$,  $V_0$ as graphs}
\label{graphs}
% in certain subregions of $M$}
We first set up some regions in $M$ where we 
will localize the holomorphic strips which are involved in the calculation.
This is basically the same as our discussion in  \S \ref{Floersketch}, but we run through it again
to set up the notation, and make things slightly more precise.
\\
\newline To begin, it is useful to understand how the various intersection points between
$\Gamma_4$, $\Gamma_2$, $\Gamma_0$   in figure 
\ref{CorrectionGammaIIIbig(Floer)} correspond to submanifolds in $D(T^*L_2^j)$ (where we look at 
their image under $\sigma_{ef}^j$, and let $e$ and $f$ vary). First, consider  
the two dotted rectangles surrounding the six intersection points 
of $\Gamma_0$ (blue) and $\Gamma_2$ (green). The two midpoints correspond to the two points $\pm f \in K_-^j$. Thus, as $f$ varies, the two midpoints  together sweep out $K_-^j \cong S^1$.  As 
$e$ and $f$ vary, the four intersection points on either side of the two  middle points in each rectangle  together sweep out
a torus, which is the boundary of a tubular neighborhood of $K_-^j$ in 
$V_2^j = L_2^j$, and we denote this torus by $\Sigma_{20}^j$.
There is a similar story for the intersection points of $\Gamma_4$ (red) and $\Gamma_2$ (green):
the two midpoints correspond to $K_+^j$ and the other four points correspond to a torus
 $\Sigma_{42}^j$. The four intersection points of $\Gamma_4$ (red) and $\Gamma_0$ (blue) together sweep out a torus, denoted $\Sigma_{40}^j$, which can be viewed either as the boundary of $D_t^{g_2}(\nu^*K_-^j)$ for a certain radius $0< t< r$, or as the boundary of  $\phi_j(S^1 \times D^2_t) \subset L_0$.
\\
\newline %Consider again the two horizontal dotted rectangles in
%figure \ref{IntrofigsimpleVCdim2} which surround the six intersection 
%points of $\Gamma_4$ (red) and $\Gamma_2$ (green) in figure 
%\ref{CorrectionGammaIIIbig}. 
The two rectangles which surround $\Gamma_0 \cap \Gamma_2$ intersect $\Gamma_2$ in two intervals; 
as $e$ and $f$ vary these intervals  
sweep out a closed tubular neighborhood of $K_-^j$ in 
$L_2^j$ which we denote  $N_{20}^j$.
%For each $j$ let $N_{42}^j$ be the closed tubular neighborhood of $K_-^j$ in 
%$L_2^j$ which corresponds to the intersection of these two rectangles with
%$\Gamma_2$. 
%the tubular neighborhood of $K_j$ in $L_0$ $\phi_j(S^1 \times D^2_t)$
Now for each $j$, take a  Weinstein 
neighborhood $D(T^*N_{20}^j) \subset D(T^*L_2^j)$ 
%be such that 
%$D(T^*N_{42}^j)\cap D(T^*K_{ef}^j)$ corresponds (via $\sigma_{ef}^j$) precisely to the two horizontal dotted rectangles in
%figure \ref{IntrofigsimpleVCdim2} which surround the six intersection 
%points of $\Gamma_4$ (red) and $\Gamma_2$ (green) in figure 
%\ref{CorrectionGammaIIIbig}. Here, $N_{42}^j$ is a closed tubular neighborhood of $K_-^j$ in 
%$L_2^j$.
such that $D(T^*N_{20}^j)\cap D(T^*K_{ef}^j)$ 
corresponds precisely %(via $\sigma_{ef}^j$) 
to these the two rectangles. 
Note that inside the two rectangles  we can view $\Gamma_0$ as the graph of two function 
over $\Gamma_2$.
Corresponding to these functions there is an exact 1-form $df_{20}^j$ defined on $N_{20}^j$ such that the graph of $df_{20}^j$ in  $D(T^*N_{20})$ is precisely 
$V_0 \cap D(T^*N_{20}^j)$. From the shape of $\Gamma_0$ we see that $f_{20}^j$ is a 
Morse-Bott function on $N_{20}^j$ with two critical components: it has a minimum at $K_+^j$ and a
 maximum at the torus $\Sigma_{42}^j$.
In a similar way we  define $D(T^*N_{42}^j)$; it corresponds to the two dotted rectangles surrounding  $\Gamma_4 \cap \Gamma_2$.
%the intersection points of $\Gamma_0$ (blue) and $\Gamma_2$ (green). 
We define $df_{42}^j$ as before; from the shape of $\Gamma_4$ we can see
$f_{42}^j$ is a Morse-Bott function on $N_{42}^j$ with two critical components: 
it has a maximum at $K_+^j$ and  a minimum at the torus $\Sigma_{42}^j$. 
\\
\newline Now take a look at the four  partially formed vertical 
dotted rectangles surrounding $\Gamma_4 \cap \Gamma_0$.
%the four intersection points of $\Gamma_4$ (red) and $\Gamma_0$ (blue); 
These rectangles intersect 
$\Gamma_0$ in four vertical closed intervals; as $e$, $f$ vary, these intervals 
 sweep out an annular region, which can be viewed 
as subset of $L_0$ of the form $\phi_j(S^1 \times D^2_{[s,r]})$, 
for some $0< s< r$, where $D^2_{[s,t]} = \{ x \in \bbR^2: |x| \in [s,t]\}$.
%subset $A_j \subset \nu^*K_-^j$; then $\phi_j$  identifies
%$A_j$ with an 
% annular region in  $L_0$ which surrounds $K_j$. Let's 
% say $\phi_j(A_j) = \phi_j(S^1 \times D^2_{[s,r]})$, for some $0< s< r$,
%where $D^2_{[s,t]}) = \{ x \in \bbR^2: |x| \in [s,t]\}$.
We extend this region into $L_0$ by setting
$$N_{40} = L_0 \setminus \big( \cup_j \phi_j(S^1 \times \Int D^2_s)  \big)$$
Then, the 
intersection of $N_{40}$ with $D(T^*K_{ef}^j) \subset  D(T^*L_2^j)$ corresponds precisely to the four subintervals of $\Gamma_0$ above.
Now take a Weinstein 
neighborhood $D(T^*N_{40}) \subset D(T^*L_0)$ which intersects   
$D(T^*K_{ef}^j) \subset  D(T^*L_2^j)$ precisely in the region corresponding  to  the four partially formed
 rectangles.  To define $V_4$ outside of $D(T^*L_2)$ more precisely 
(compare with the end of \S \ref{bigsectionvanishingcycles}), we take a Morse-Bott function  
$f_{40}: N_{40}\into \bbR$ such that the part of the graph of $df_{40}$ 
in $D(T^*L_2^j)$ agrees with $W_4^j \cap D(T^*N_{40})$. 
Thus, the part of the graph of $df_{40}$ in $D(T^*K_{ef}^j)$ is represented by 
$\Gamma_4$ in the four partially formed rectangles.
Recall from \S \ref{plumbing} that the plumbing map $\tau^j$ is basically defined to be 
$id_{T^*S^1} \times m_i$  on a certain subregion of $T^*S^1 \times T^*\bbR^2 \cong T^*S^1 \times \bbC^2$ 
(where $m_i$ is multiplication by $i$).
This implies that when we do the plumbing identification, the four 
rectangles in each slice $D(T^*K_{ef}^j)$
 are rotated by ninety 
degrees, via multiplication by $i$. (The complete sub-region of 
$D(T^*K_{ef}^j)$ 
which is rotated is a neighborhood of $D(\nu^*K_-^j) \cap D(T^*K_{ef}^j) $.) 
%which would correspond to a rectangle surounding 
%the two horizontal blue curves  in figure \ref{IntrofigsimpleVCdim2}.) 
After this rotation $\Gamma_4$ becomes the graph of a function
defined over $\Gamma_0$ (note the effect of $i$ or $-i$ is the same). Thus, $f_{40}$ is critical
along each torus $\Sigma_{40}^j$; from the shape of $\Gamma_4$
one can see $f_{40}$ has a maximum at $\Sigma_{40}^j$ and we assume that $f_{40}$ has isolated critical points aside from that. 

\subsection{Generators and relations for the Morse-Bott homology groups of  $(N_{40},f_{40})$ 
$(N_{42}^j,f_{42}^j)$, and  $(N_{20}^j,f_{20}^j)$} \label{sectionHF}
 We will find explicit generators and relations 
for the Morse-Bott homology groups of $(N_{40},f_{40})$, 
$(N_{42}^j,f_{42}^j)$, and  $(N_{20}^j,f_{20}^j)$ 
which correspond perfectly to the ones we found for the flow category in \S \ref{sectionFlow}.
Later we will identify these groups with the Floer homology groups.
\\
\newline We first specify 
Morse-Smale data $(h_{42}^j,g_{42}^j)$, $(h_{20}^j,g_{20}^j)$ 
\label{(h_{42},g_{42}),(h_{20},g_{20}),(h_{40},g_{40})} 
\label{(widetilde h_{40},widetilde g_{40})} 
on  $\Sigma_{42}^j \cup K_+^j$ and  $\Sigma_{20}^j \cup K_-^j$ respectively. Then we specify 
two sets of Morse-Smale data on $\Sigma_{40}^j$, denoted
$(h_{40}^j,g_{40}^j)$ and $(\widetilde h_{40}^j,\widetilde g_{40}^j)$; the first will be related  to  the identification $\Sigma_{40}^j = S_t(\nu^*K_-^j) \cong K_-^j \times K_+^j$; the second  will be related to the identification $\Sigma_{40}^j = \phi_j(S^1 \times \partial D^2_t)$.
(We will not use $(h_{40}^j,g_{40}^j)$ until \S \ref{sectiontriangleproduct}.)
\\
\newline  There are canonical parameterizations 
\begin{gather*}
\varphi^j_{40}: K_+^j \times K_-^j \into \Sigma_{40}^j,\,\,\,
\varphi^j_{42}: K_+^j \times K_-^j \into \Sigma_{42}^j,\,\,\,
\varphi^j_{20}: K_+^j \times K_-^j \into \Sigma_{20}^j.
\end{gather*}
(These arise from the fact that each torus is swept out by some points in $D(T^*K_{ef}^j)$
as $e \in K_+^j$ and $f \in K_-^j$ vary.)
Fix the obvious identifications of  $ K_+^j$, $K_-^j$ with 
$\bbR/\bbZ$, using $\sin$ and $\cos$. To simplify notation in what follows we will 
express points in $K_+^j$, $K_-^j$, $\Sigma_{42}^j$, $\Sigma_{20}^j$, $\Sigma_{40}^j$ 
in terms of the coordinates $x \in \bbR/\bbZ$, or $(x,y) \in (\bbR/\bbZ)^2$.
Choose the metrics  $g_{42}^j$, $g_{20}^j$, $g_{40}^j$ 
so that 
they correspond to the flat metric on $K_j^{\pm} \cong \bbR/\bbZ$ and
$\Sigma_{42}^j, \Sigma_{20}^j, \Sigma_{40}^j \cong (\bbR/\bbZ)^2.$ We will define  $h_{42}^j, h_{20}^j$, $h_{40}^j$ so that their flow lines are as in figure \ref{figureMorseSmaledata}. 
To define them  precisely,
 \label{(h_{42},g_{42}),(h_{20},g_{20}),(h_{40},g_{40})(b)} 
and give notation for their critical points, we take a Morse function 
$\varphi: \bbR/\bbZ \into \bbR$
with two critical points at $0=1$ and $1/2$, which are respectively a minimum and a maximum.
We define, for $x \in K_{\pm}^j \cong \bbR/\bbZ$, 
$$(h_{42}^j|K_+^j)(x) =\varphi(x-1/4), \phantom{bbb}(h_{20}^j|K_-^j)(x) =\varphi(x-1/4).$$ 
Then, $h_{42}^j|K_+^j$, $h_{20}^j|K_-^j$ %, $h_{40}^j|K_j^{\pm}$ 
have critical points, respectively, 
$$a_0 = 1/4, a_1 = 3/4,  \phantom{bbb} b_0 = 1/4, b_1 = 3/4.$$ 
%$$c_0 = 1/4, c_1 = 3/4,$$
 where $a_0$, $b_0$ %, $c_0$ 
are  minima and $a_1, b_1$ %, c_1$ 
are maxima. For $(x,y) \in \Sigma_{42}^j \cong K_+^j \times K_-^j \cong (\bbR/\bbZ)^2$, we define
$$(h_{42}^j|\Sigma_{42}^j)(x,y) = \varphi(x) + \varphi(y).$$
Then, $h_{42}^j|\Sigma_{42}^j$ has critical points
$$x_2^j= (1/2,1/2), x_1^j = (0, 1/2)= (1,1/2), (x_1^j)' = (1/2,0) = (1/2,1)$$
$$\text{and } x_0^j = (0,0) = (1,0) = (0,1) = (1,1),$$
where the subscript denotes the index. 
Let $$h^j_{20}(x,y) = h_{42}^j(x-1/8,y-1/8),$$  
%$$h^j_{4,0}(x,y) =  h_{42}^j(x+1/8,y+1/8).$$ 
Then, $h^j_{20}|\Sigma_{20}^j$
has critical points 
$$y_2^j, y_1^j, (y_1^j)', y_0^j$$
with $y_2^j = x_2^j+(1/8, 1/8)$, $y_1^j = x_1^j +(1/8, 1/8)$, etc.
%For $(x,y) \in (\bbR/\bbZ)^2 \cong \Sigma_{40}^j$ 
%\cong K_+^j \times K_-^j \cong (\bbR/\bbZ)^2$,
Let $$h^j_{40}(x,y) = h_{42}^j(x+1/8,y+1/8),$$  
%$$h^j_{4,0}(x,y) =  h_{42}^j(x+1/8,y+1/8).$$ 
then $h_{40}^j = h_{40}|\Sigma_{40}^j$ has critical points 
$$z_2^j, z_1^j, (z_1^j)', z_0^j$$
where $z_2^j = x_2^j-(1/8, 1/8)$, $z_1^j = x_1^j -(1/8, 1/8),$ etc.
\begin{figure}
\begin{center}
\includegraphics[width=4in]{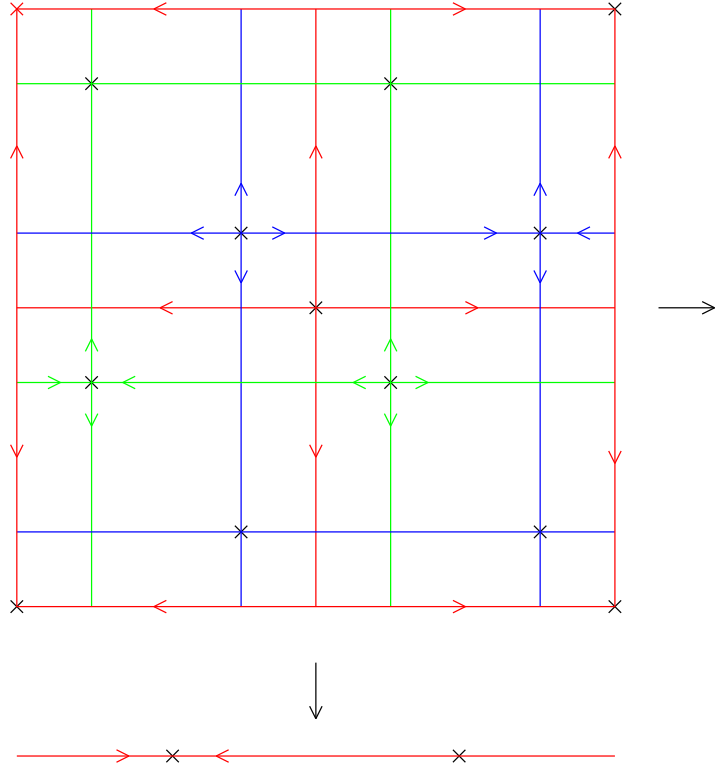} 
%\put(-125,125){$\widehat{\psi}$}
%\put(-200,100){$(0,0)$}
\caption{The (negative) flow lines of $(h_{42},g_{42})$, $(h_{20},g_{20})$  on $(\bbR/\bbZ \times \bbR/\bbZ) \cup \bbR/\bbZ$
(in red, blue respectively), and  $(h_{40}, g_{40})$ on $\bbR/\bbZ \times \bbR/\bbZ$ (in green). The two black arrows indicate projection maps.}
\label{figureMorseSmaledata}
\end{center}
\end{figure}
\\
\newline Let $\mu_{42}^j$ and $\mu_{20}^j$ be the Riemannian metrics on $N_{42}^j$ and $N_{20}^j$ 
which are of the form $(\phi_+^j)_*(g_{S^1} \times g_{\bbR^2})$ and
 $(\phi_+^j)_*(g_{S^1} \times g_{\bbR^2})$ respectively. Here 
$\phi_{\pm}^j$ are from  \S \ref{plumbing} (and we may assume 
 $N_{42}^j$ and $N_{20}^j$ are contained in $\phi_+^j(S^1 \times D^2_{\pi/3})$ and
$\phi_+^j(S^1 \times D^2_{\pi/3})$).

\begin{lemma} \label{MorsehomologyI} 
$((-f_{42}^j, \mu_{42}^j)    ; (h_{42}^j, g_{42}^j)) \text{ and }  
( (f_{20}^j, \mu_{20})    ; (h_{20}^j, g_{42}^j))  $
are regular Morse-Bott homology data in the sense of  \S \ref{MorseBotthomology}.
The corresponding ungraded Morse-Bott homology groups with $\bbZ/2$ coefficients
have the following generators and relations.
%and $$HF^*(V_4, V_2^j, J_{42}^j;(h_{42}^j,g_{42}^j))=
\begin{gather*}
 (H_{MB})_*( (N_{42}^j, -f_{42}^j, \mu_{42}^j),  (h_{42}^j,g_{42}^j), \partial)
%\text{ is freely generated by } x_0^j, (x_1^j)' \in \Sigma_{42}^j $$
\text{ is freely generated by } x_2^j, x_1^j \in \Sigma_{42}^j;\\
%$$HF^*(V_2^j,V_0 J_{20}^j;(h_{20}^j,g_{02}^j)) = 
(H_{MB})_*((N_{20}^j, f_{20}^j, \mu_{42}^j), (h_{20}^j,g_{20}^j), \partial)
%\text{ is freely  generated by } y_0^j, y_1^j \in \Sigma_{20}^j.$$
\text{ is freely  generated by } y_2^j, (y_1^j)' \in \Sigma_{20}^j.
\end{gather*}
\end{lemma}

\begin{proof}%[Proof of lemma \ref{MorsehomologyI}] 
We focus on $(-f_{42}^j, \mu_{42}^j)$, first addressing regularity. The unstable and stable manifolds of $-f_{42}^j$ 
are both codimension 0 submanifolds of $N_{42}^j$, hence they intersect transversely.
This shows that  $(-f_{42}^j, \mu_{42}^j)$ is regular. 
Next we show that $((-f_{42}^j, \mu_{42}^j) ; (h_{42}^j,g_{42}^j))$ is regular. 
Let $\mathcal{M}$ denote 
$$\{ \gamma \in C^{\infty}(\bbR, N_{42}^j) : \gamma'(s) = 
- \nabla_{\mu_{42}^j}(-f_{42}^j)(\gamma(s)),
\gamma(-\infty) \in \Sigma_{42}^j, \gamma(\infty) \in K_+^j \},$$
and let $\mathcal{M}^* = \mathcal{M} /\bbR$,  $ev_{+}(\gamma) =\gamma(+\infty)$,  $ev_{-}(\gamma)= \gamma(-\infty)$.
%For that, recall that
%$\varphi_{42}^j: K_+^j \times K_-^j \into \Sigma_{42}^j$ 
%denotes the standard identificationmentioned at the beginning of this section. .
Recall $\mu_{42}^j =  (\phi_+^j)_*(g_{S^1} \times g_{\bbR^2})$ on $N_{42}^j$,  
and note that $f_{42}^j$ is rotationally symmetric as well. 
Therefore,
for each $(e,f) \in K_+^j \times K_-^j$, there is exactly one $\gamma \in \mathcal{M}^*$
with $ev_{+}(\gamma) = e$ and $ev_{-}(\gamma)= \varphi_{42}^j(e,f) \in \Sigma_{42}^j$. This means we can identify $\mathcal{M}^*$ with $\bbR/\bbZ \times \bbR/\bbZ$,
%, where $K_j^{\pm} \cong \bbR/\bbZ$, and 
where the evaluation maps 
\begin{gather*}
 ev_{+}: \mathcal{M}^* \into K_+^j, \, ev_{-}: \mathcal{M}^* \into \Sigma_{42}^j 
\text{ become, respectively,}\\
\pi_{+}: \bbR/\bbZ \times \bbR/\bbZ \into \bbR/\bbZ, \, \pi_{-}: \bbR/\bbZ \times \bbR/\bbZ\into \bbR/\bbZ \times \bbR/\bbZ,
\end{gather*}
where $\pi_+$ is projection to the first factor and $\pi_-$ is the identity map.
Then $((-f_{42}^j, \mu_{42}^j) ; (h_{42}^j,g_{42}^j))$ is regular because 
$\pi_-^{-1}(U(x)) \cap \pi_+^{-1}(S(a))$ is a transverse intersection  for all 
$x,a \in Crit(h_{42}^j)$, where
$x \in \bbR/\bbZ \times \bbR/\bbZ \cong \Sigma_{42}^j$
and $ a \in\bbR/\bbZ \cong K_+^j$. (See figure \ref{figureMorseSmaledata}.)
\\
\newline Because $-f_{42}$ has a maximum along $\Sigma_{42}^j$, 
%\cong \partial N_{42}^j$, 
the Morse Bott homology will be isomorphic to $H_*(N_{42}, \partial N_{42})$.
(In general, for a manifold $N$ with boundary, if the Morse-Bott function 
$f: N\into \bbR$ takes a minimum (resp. maximum) along the boundary, then the Morse-Bott homology of $f$ is isomorphic to $H_*(N)$ 
(resp. $H_*(N, \partial N)$).)
Since we have already analyzed $\mathcal{M}^*$ we may as well verify this explicitly.
For $x\in Crit(h_{42}^j|\Sigma_{42}^j)$, %, the differential $\partial$ is given by 
$$\partial x = \Sigma_a \# (\pi_-^{-1}(U(x)) \cap \pi_+^{-1}(S(a))) a,$$
where $x$, $a$ are as above and $\#(\pi_-^{-1}(U(x)) \cap \pi_+^{-1}(S(a)))$ is the number of zero dimensional components mod 2. Using this it is easy to see that
$$\partial x_2^j =0, \, \partial x_1^j = 0, \,\partial (x_1^j)' = a_1^j,  \,\partial x_0^j = a_0^j.$$
%In the last two  cases there is exactly one element of $(\pi_-^{-1}(U(x)) \cap \pi_+^{-1}(S(a)))$. 
%For $a \in  Crit(h_{42}^j|K_-^j)$, $\partial a =0$, since $a_0, a_1$ generate 
%the usual Morse complex of $S^1$.
For $(H_{MB})_*(N_{20}^j, (f_{20}^j, \mu_{42}^j) ; (h_{20}^j,g_{20}^j), \partial)$
everything is the same except $ev_{+}$ is projection to the second factor. In that case we get 
$$\partial y_2^j =0, \, \partial (y_1^j)' = 0, \,\partial (y_1^j) = b_1^j,  \,\partial y_0^j = b_0^j.$$  \end{proof}

%We now define the second Morse-Smale pair $(\widetilde h_{40}^j, \widetilde g_{40}^j)$  
Now let $(\widetilde h_{40}^j, \widetilde g_{40}^j)$  \label{(widetilde h_{40},widetilde g_{40})(b)} be any Morse-Smale pair on $\Sigma_{40}^j$ such that $\widetilde h_{40}^j$ has critical points
$\tilde z_0^{\,j}, \tilde z_1^{\,j},  (\tilde z_1^{\,j})', \tilde z_2^{\,j}$ (where the index is given by the subscript) such that the closure of the unstable manifolds 
$$\widetilde \mu_j  = \overline{U((\tilde z_1^{\,j})')} \cong S^1 \text{ and } \widetilde \lambda_j = \overline{U(\tilde z_1^{\,j})}\cong S^1$$ 
have linking numbers in $L_0 \cong S^3$ respectively%\footnote{The sign depends on various orientations, but this is not relevant over $\bbZ/2$.} with $K_j$  in $L_0 \cong S^3$ respectively
$$\ell k(\widetilde \lambda_j, K_j) = 0 \text{ and } \ell k(\widetilde \mu_j, K_j) = 1.$$
To be concrete, we can take  $\widetilde \lambda_j = \phi_j(S^1 \times \{p\})$ for some $p \in \partial D^2$ and $\widetilde \mu_j = \phi_j(\{q\} \times \partial D^2)$ for some $q \in  S^1$.
Recall the singular homology $H_*(N_{40})$ is generated by the cycles  
$\Sigma_{40}^j$, $\widetilde \lambda_j$,  $\widetilde \mu_j$,  $\widetilde p_j$, $j=1, \ldots,k$,
where $ \widetilde p_j \in \Sigma^j_{42}$ is any point. The relations are as in \S\ref{sectionFlow}:
\begin{gather} [ \widetilde \lambda_j] = \Sigma_{i \neq j} \ell k(K_j,K_i) [ \widetilde\mu_i]; 
\, \Sigma_j [ \widetilde \Sigma_{40}^j] = 0; \,  [\widetilde p_1] = \ldots =  [ \widetilde p_k].
\label{H_*(N_{40}relations}
\end{gather}

\begin{lemma}\label{MorsehomologyII} 

The Morse-Bott function $f_{40} : N_{40} \into \bbR$ can be chosen  
(maintaining  same form as in \S \ref{graphs}) so that, for some metric $\mu_{40}$ on $N_{40}$, the Morse-Bott homology data   
$((-f_{40},  \mu_{40}); (\widetilde h_{40}, \widetilde g_{40}))$
is regular in the sense of  \S \ref{MorseBotthomology}. Moreover
the Morse-Bott homology
$$(H_{MB})_*((N_{40},-f_{40},  \mu_{40}) , (\widetilde h_{40}, \widetilde g_{40}), \partial )$$
is generated by $\tilde z_0^{\, j}, \tilde z_1^{\, j},  ( \tilde z_1^{\, j})',  \tilde z_2^{\, j}$,  $j=1, \ldots,k$ and the relations are 
$$ [ \tilde z_1^{\, j}] = \Sigma_{i \neq j} \ell k(K_j,K_i) 
[(\tilde z_1^{\, i})'];\,\, \Sigma_j [ \tilde z_2^{\, j}] = 0; \, \,
  [ \tilde z_0^{\, 1}] = \ldots =  [ \tilde z_0^{\, k}].$$
\end{lemma}
\begin{proof}%[Proof of lemma \ref{MorsehomologyII}]
We will describe a suitable handle decomposition of $N_{40}$ and then take  $(-f_{40}, \mu_{40})$ so that it realizes this handle decomposition. Regularity
will easily follow by suitably adjusting the attaching maps by some isotopy, 
which corresponds to adjusting  $\mu_{40}$ by a small isotopy.
%so that certain transveralities occur. 
\\
\newline Assume for a moment that this has been done. 
%the Morse-Bott homology has the expected generators and relations because,
Then, since $-f_{40}$ takes a 
minimum along $\cup_j \Sigma_{40}^j \cong \partial N_{40}$, the Morse-Bott homology is isomorphic to the singular homology $H_*(N_{40})$ 
(as in the proof of lemma \ref{MorsehomologyI}), where the generators $\widetilde \Sigma_{40}^j$, $\widetilde \lambda_j$, $\widetilde \mu_j$, $\tilde p_j$ correspond respectively to $\tilde z_2^j$, $\tilde z_1^j$,  $( \tilde z_1^j)',$  $\tilde z_0^j \in Crit(\tilde h_{40}| \widetilde \Sigma_{40}^j),$  $j=1, \ldots,k$. Thus we have the expected generators and relations from (\ref{H_*(N_{40}relations}). (One can also verify this explicitly by inspecting the handle-decomposition 
we describe below; compare with the end of the proof of lemma \ref{fixgenerators}).
\\
\newline We now explain the handle-decomposition of $N_{40}$. Since $-f_{40}$ is Morse-Bott near its minimum value our handle decomposition  
starts with $\Sigma \times [0,1]$, where $\Sigma = \cup_j \widetilde \Sigma_{40}^j$. Here, the corresponding Morse-Bott function would depend only on the $[0,1]$ factor, and have a single critical point at 1/2, which is a minimum. From now on we attach standard handles in the usual way to $\Sigma \times \{1\}$ because $-f_{40}$ has only isolated critical points away from $\Sigma$.
\\
\newline We do a variation on  the standard Heegaard diagram representation of a link complement.  Consider the link diagram of $\cup_j K_j \subset L_0$ with over-crossings and under-crossings. We can visualize $\Sigma \times \{1\}$ as the boundary of a tubular neighborhood of
the link diagram of $\cup_j K_j \subset L_0 = S^3$. Each crossing gives rise to 
two points, one on the bottom of the over-crossing and one on the top of the under-crossing; these two points form an $S^0$ and we attach a 1-handle to this 
$S^0$ at each crossing. We denote the result by $X_1$.
\\
\newline Now consider the singularized link diagram $D\subset \bbR^2$, where each crossing no longer has over/under information recorded. For each bounded component $U\cong D^2$ of $\bbR^2 \setminus D$ we attach a 
2-handle to $X_1$ whose core corresponds to the disk $U$; the attaching circle will run once along a 1-handle each time it encounters a crossing. Let $X_2$ denote result of attaching all these 2-handles; it has one $S^2$ boundary component for each component of $D$, and the other boundary components are tori, one for each link component $K_j$.
\\
\newline To finish, if $X_2$ 
has several components $C_1, \ldots, C_l$ then  we attach a 1-handle to each consecutive pair $(C_1,C_2),(C_2,C_3), \ldots, (C_{l-1}, C_l)$, where the attaching regions lie in 
$\Sigma\times \{1\}$. The result is a new space $X_2'$ which is connected. 
The boundary components of $X_2'$ consist of the the same tori as before, plus one $S^2$ component. We attach a 3-handle to this $S^2$ and the result is diffeomorphic to $N_{40}$.
\\
\newline Now $(-f_{40}, \mu_{40})$ is regular if and only if 
$U(C) \cap S(C\,')$ is transverse for every pair of critical components $C,C\,'$.
This is in turn means that the attaching sphere 
of each handle should be transverse to the descending spheres of all handles up to that point. (This does not constrain how the attaching spheres meet the tori.)
Also, $((-f_{40}, \mu_{40}),(\widetilde h_{40}, \widetilde g_{40}))$ is regular if and only if
the attaching sphere of each handle meets $U(\tilde z) \subset \Sigma \times \{1\}$
transversely in $\Sigma \times \{1\}$ for every $\tilde z \in Crit(\widetilde h_{40})$. 
Both of these conditions can be arranged by suitably isotoping the attaching spheres of the handles, and this corresponds to adjusting 
$\mu_{40}$ by a small isotopy.
\end{proof}

\subsection{Identifying the Floer homology  groups with the Morse-Bott homology groups} \label{Floergroups}
 
In this section we identify  the Floer homology  groups we are interested in with the Morse-Bott homology groups from the last section (in fact the chain complexes are isomorphic). 
We use the results of \S \ref{localstrips}; there we explain
 how to replace   
$V_4$ and $V_0$ by new exact isotopic versions obtained, roughly speaking, by 
 replacing $d f_{40}$, $d f_{42}^j$, $df_{20}^j$ by 
by $\frac{1}{n} d f_{40}$, $\frac{1}{n} df_{42}^j$, $\frac{1}{n}df_{20}^j$, for some fixed $n\geq 1$ sufficiently large, and patching those back into $V_4$ and 
$V_0$ using a cut-off function. Let us assume that this replacement has been made, but we keep the same notation. (Note that the re-scaling $f \leadsto \frac{1}{n}f$ does not affect the arguments in \S\ref{sectionHF}.)
Then the main result of \S \ref{localstrips} is given by the following lemma; it corresponds to Proposition \ref{correspondence}.% there. 
%First, let $N_{20} = \cup_j N_{20}^j$ and $N_{42} = \cup_j N_{42}^j$; let 
%$f_{20}$, $f_{42}$, $\mu_{42}$, $\mu_{20}$, $(h_{42},g_{42})$,
%$(h_{20}, g_{20})$ be the data such that $f_{20}|N_{20}^j = f_{20}^j$, etc.

\begin{lemma}\label{strips=gradientlines} 

There exist almost 
complex structures $J_{40}$, $J_{42}^j$, $J_{20}^j$ $\in$ $\mathcal{J}([0,1],M,I)$
such that $(J_{42}^j; h_{42}^j,g_{42}^j)$,  $(J_{20}^j; h_{20}^j,g_{20}^j)$,   
$(J_{40}; (\widetilde h_{40}, \widetilde g_{40}))$ 
are each regular as Morse-Bott Floer data (see \S \ref{specialcaseI}).
%(Note we can pick $J_{20}^j$ (resp. $J_{42}^j$) to be all equal to 
%$J_{20}$ (resp. $J_{42}$).)
Moreover, %the following isomorphisms  of moduli spaces hold:
for every $p$, $q$ $\in$ $Crit(f_{40})$ $=$ $V_4 \cap V_0$, 
%the elements of
$\mathcal{M}(V_4,V_0, J_{40}; p,q)$ is in one-one correspondence with
%are in one-one correspondence with the elements of 
$$ \{ \gamma\in C^{\infty}(\bbR, N_{40}) :\gamma'(s) = 
-\nabla_{\mu_{40}} (-f_{40})(\gamma(s)), \gamma(-\infty) = p, \gamma(\infty) = q \}$$
%Namely the correspondence sends $\gamma$ to $u$, where
%$$u(s,t) = \phi_t (\gamma((s));$$
%%here $n>0$ is some fixed number, and 
%here, $\phi_t$ is a certain  exact isotopy such that 
%$\phi_t(p) = p$ for all $p \in  V_0 \cap V_4$.
Similarly,
for every $p, q \in Crit(f_{42}^j) = V_4 \cap V_2^j$,
$\mathcal{M}(V_4, V_2^j, J_{42}^j; p,q)$ is in one-one correspondence with
$$\{ \gamma\in C^{\infty}(\bbR, N_{42}^j) :\gamma'(s) = -\nabla_{\mu_{42}^j}(- f_{42}^j)(\gamma(s)), \gamma(-\infty) = p, \gamma(\infty) = q \},\text{ and }$$
for every $p, q \in Crit(f_{20}^j) = V_2^j\cap V_0$, 
$\mathcal{M}(V_2^j, V_0, J_{20}^j; p,q)$ is in one-one correspondence with
$$\{ \gamma\in C^{\infty}(\bbR, N_{20}^j) :\gamma'(s) = -(\nabla_{\mu_{20}^j} f_{20}^j)(\gamma(s)), \gamma(-\infty) = p, \gamma(\infty) = q \}.$$
%The isomorphism in both cases takes the form $\gamma \mapsto u$, where
%$u(s,t) = \phi_t(\gamma(s))$ and $\phi_t: M \into M$ is some exact  isotopy  which fixes every $p \in Crit(f_{40})$ (resp. $p \in Crit(f_{20})$). 
\end{lemma}

As an immediate consequence we have the following generators and relations for
 the Floer homology groups.

\begin{prop}\label{generators} 
%First, the following Morse-Bott Floer data are regular:
%$(J_{42}; h_{42},g_{42})$,  $(J_{20}; h_{20},g_{20})$,   
%$(J_{40}; (\widetilde h_{40}, \widetilde g_{40})).$
%Second,
$$HF_*((V_4, V_2^j, J_{42});(h_{42}^j,g_{42}^j))
\text{ is freely generated by } x_2^j, x_1^j \in \Sigma_{42}^j;$$
$$HF_*((V_2^j,V_0,  J_{20});(h_{20}^j,g_{20}^j))
\text{ is freely  generated by } y_2^j, (y_1^j)' \in \Sigma_{20}^j;$$
and $HF_*((V_4, V_0, J_{40}); (\widetilde h_{40}, \widetilde g_{40}))$
is  generated by 
$\tilde z_0^{\,j},$ $\tilde z_1^{\,j},$  $( \tilde z_1^{\,j})',$  
$\tilde z_2^{\,j} \in \Sigma_{40}^j$, $j=1, \ldots,k$, with the relations 
$$ [\tilde z_1^{\,j}] = \Sigma_{i \neq j} \ell k(K_j,K_i) [(\tilde z_1^{\,i})'], \phantom{bb} \Sigma_j [\tilde z_2^{\,j}] = 0, \, \,   [\tilde z_0^1] = \ldots =  [ \tilde z_0^k].$$
\end{prop}

\begin{proof} As ungraded complexes with $\bbZ/2$ coefficients, we have isomorphisms 
$$(CF((V_4, V_2^j, J_{42}),(h_{42}^j,g_{42}^j),\partial) \cong 
(C_{MB})((N_{42}^j,-f_{42}^j, \mu_{42}^j), (h_{42}^j,g_{42}^j),\partial ).$$
$$(CF((V_2^j,V_0,  J_{20}),(h_{20}^j,g_{20}^j),\partial) \cong (C_{MB})((N_{20}^j, f_{20}^j, \mu_{20}^j), (h_{20}^j,g_{20}^j), \partial)$$
$$(CF((V_4, V_0, J_{40}), (\widetilde h_{40}, \widetilde g_{40}),\partial) \cong (C_{MB})((N_{40},-f_{40}, \mu_{40}), (\widetilde h_{40},\widetilde g_{40}),\partial).$$
All three isomorphisms are given by the identity map (for example from  $V_4 \cap V_0$ to $Crit(f_{40})$). The identification of these complexes is immediate from the last lemma, and the definition of these complexes (see  \S \ref{MorseBott}).  Propositions \ref{MorsehomologyI} and \ref{MorsehomologyII} then imply the stated generators and relations.\end{proof}
Note that the generators and relations in the last proposition match up 
perfectly with the ones for the flow category (\ref{homologyrelations})
in \S \ref{sectionFlow}.

\begin{comment}
Because the moduli spaces underlying these complexes are isomorphic, regularity for the Morse-Bott Floer data follows from the regularity of 
the Morse-Bott homology data. Strictly speaking, to 
deduce that the regularity of $(f_{40},\widetilde g_0)$ implies that of $J_{40}$
(and similarly for $J_{42}^j, J_{20}^j$),
we need to know the  isomorphism of moduli spaces in lemma \ref{strips=gradientlines}
works at the linearized level.
%conclude that regularity for the Morse moduli space implies regularity for the Floer moduli space. 
But the proof of that result (see Proposition \ref{strips}) boils down to Floer's basic correspondence, Theorem \ref{Floer}, and that is well known to  work at the linearized level;  see  \cite{Poz} p. 86.)
\end{comment}

%\section{Computing the triangle product}
%\label{Floergroupsandtriangleproduct}

\section{The monotonicity lemma; constructing  $\widetilde V_4$, $\widetilde V_2^j$, $\widetilde V_0$} 
%constraining holomorphic triangles  using 
\label{sectionmonotonicity}

In this section we carry out the details of the 
arguments involving the monotonicity lemma (sketched in 
\S \ref{secondtriangle}). 
%, the continuation map (sketched in  \ref{contnmapsketch}), and the proof of Theorem $B$ (sketched in \S \ref{FukFlowsketch}). 
%In \S \ref{monotonicity} 
In the process we replace $V_0$, $V_2^j$, $V_4$ by $\widetilde V_0$,
 $\widetilde V_2^j$, $\widetilde V_4$, which correspond to version IV of the vanishing spheres  in \S \ref{secondtriangle}.
%\subsection{Constraining holomorphic triangles 
%using the Monotonicity lemma} \label{monotonicity}
\\
\newline Let $\Delta$ denote any one of the four triangles bounded by 
$\Gamma_4$, $\Gamma_2$ and  $\Gamma_0$ in figure 
\ref{CorrectionGammaIIIbig(Floer)}. Let $\widetilde U$ be a small connected
compact neighborhood of $\Delta$ in $\bbR/2\pi\bbZ \times [-r,r]$, where $\widetilde U$ is so small that it 
does not contain the points corresponding to $K_{\pm}^j$ 
(in figure \ref{figureB}, $\widetilde U$ is bounded by the 
largest triangle). Let $\widetilde B \subset \widetilde U$ be a compact  annular subregion which surounds $\Delta$, but does not touch it, and let  $\widetilde B_0 \subset \Int \widetilde B$ be another smaller compact annular region surounding $\Delta$ (see figure   \ref{figureB}, where  $\widetilde B$ is bounded by the 
two black triangles, and  $\widetilde B_0$  is bounded by the 
two brown triangles inside $B$).
%\begin{comment}
\begin{figure}
\begin{center}
\includegraphics[width=4in]{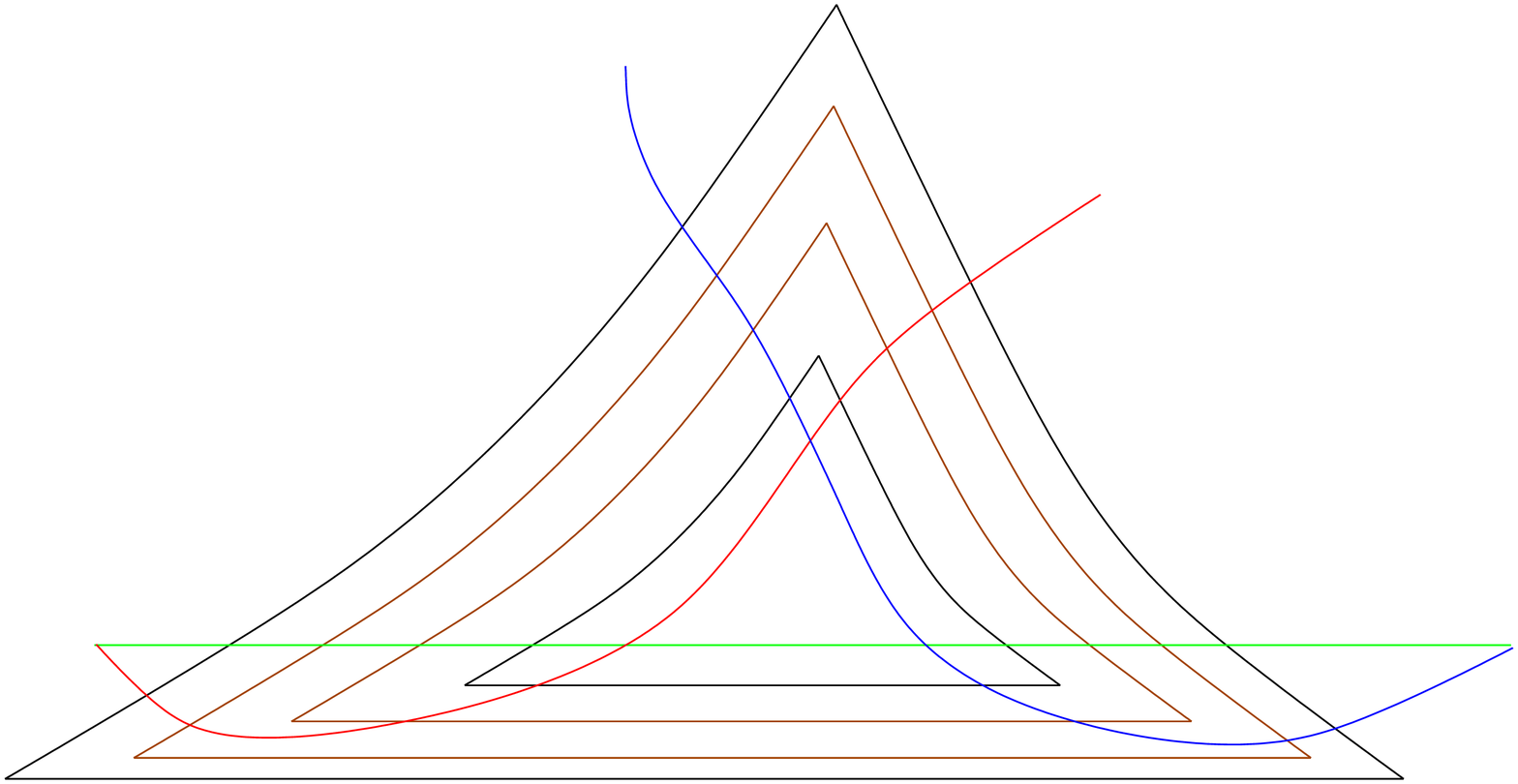} 
%\put(-125,125){$\widehat{\psi}$}
%\put(-200,100){$(0,0)$}
\caption{The regions $\widetilde U, \widetilde B, \widetilde B_0$  surounding $\Delta$ in  $\bbR/2\pi\bbZ \times [-r,r]$. }
%\caption{The regions $\widetilde U, \widetilde B, \widetilde B_0$ in $[0,\pi/2] \times [-r/4,r]$. $\widetilde B_0$ is bounded by the two brown (inner) triangles; $\widetilde B$ is bounded by the two black (outer) triangles; and $U$ is bounded by the largest black triangle.}
\label{figureB}
\end{center}
\end{figure}
%\end{comment}
Now let $T^j = \cup_{e,f} \sigma_{ef}^j(\Delta)$ 
(this is the union of all our expected holomorphic triangles)
and let $U^j \subset D(T^*L_2^j)$ be a  connected compact neighborhood of $T^j$ 
such that $U^j \cap D(T^*K_{ef}^j) = \sigma_{ef}^j(\widetilde U)$ 
for every $e$, $f$.  Similarly, let $B_0^j \subset B^j \subset U^j$
be annular regions surounding $T^j$ 
corresponding to $\widetilde B_0 \subset\widetilde  B \subset\widetilde  U$
in the sense that
 $B^j \cap D(T^*K_{ef}^j) = \sigma_{ef}^j(\widetilde B)$
and $B^j_0 \cap D(T^*K_{ef}^j) = \sigma_{ef}^j(\widetilde B_0)$ 
for every $e$, $f$. 
%(Here, by ``annular region surounding $T^j$'' we mean that  $B^j$ 
%is the closure of a connected  open set, and 
%$T^*L_2^j \setminus B^j$ has exactly two components, one of 
%which contains $T^j$, and similarly for $B_0^j$.)
%Set  $T= \cup_j T^j$, $U = \cup_j U^j$, $B = \cup_j B^j$, $B_0 = \cup_j B^j_0$.%\begin{comment}
More precisely, we may suppose that $B_0^j \subset B^j \subset U^j$ are  of the form 
$$U^j = \{x \in T^*L_2^j : dist(x,T^j) \leq \eta\},$$
$$B^j = \{x \in T^*L_2^j : \lambda \leq dist(x,T^j) \leq  \eta\}, \, \, B^j_0 = \{x \in T^*L_2^j : \lambda_0 \leq dist(x,T^j) \leq  \eta_0\},$$
where $0<\lambda< \lambda_0 < \eta_0 < \eta$. 
%and $dist$ is the metric on $T^*S^3 \subset \bbR^4 \times \bbR^4$.
Then $T^*L_2^j \setminus B^j$ has exactly two components, one of 
which contains $T^j$, and similarly for $B_0^j$.
%, where $\eta >0$ and $dist$ is the 
%standard distance function on $T^*L_2^j = T^*S^3 \cong \{z \in \bbC^4 : 
%\Sigma_i z_i^2 =1\}$. Then $B$ $\widetilde B$ ill be of the form  
%$\{x \in U_j: dist_g(p,\partial U_j)\leq \lambda \}$ 
%\end{comment}
\\
\newline 
Let $g$ denote  the metric on $\cup_j \, D(T^*L_2^j)$ determined by 
$\omega|(\cup_j \,D(T^*L_2^j))$ %= \omega_{T^*S^3}|D(T^*S^3)$ 
and $J_\bbC$.
Fix $\mu>0$ such that for  every $y \in B_0^j$ we have
\begin{gather} \label{B}
D_\mu^{\,g}(y) = \{x \in D(T^*L_2^j) : dist_g(x,y) \leq \mu \} \subset B^j.
\end{gather}
Also assume that $\mu$ is small enough that 
$$exp_y^{\,g} : D_\mu(T_y(T^*L_2^j), g) \into  D_\mu^{g}(y)$$
is a diffeomorphism for any $y\in B_0^j$.
We have the following relative form of the monotonicity lemma (see \cite{FSS} lemma 12). (The essential ingredient is an isoperimetric inequality, with respect to the metric $g$, applied inside small balls $D_\mu^{g}(y)$.)

\begin{lemma}\label{monotonicity} 

Fix $j=1,\ldots, k$.
Suppose $L$ is a (not necessarily connected) Lagrangian submanifold of $B^j$ with 
$\partial L \subset \partial B^j$. Set 
\begin{gather*}
d = (1/2)min\{dist(C,C'): C,C'\text{ are distinct components of L}\}.
\end{gather*}
There are constants $0< \gamma, 0< \rho < min(\mu,d)$, 
 depending on $L$, $J_{\bbC}|(\cup_j B^j)$, and $\omega|(\cup_j B^j)$  
such that the following holds. 
%Let $B_r(y)$ be the ball of radius $0<r \leq \rho$ around a point $y \in B_0$. 
%(So in particular $B_r(y)$ is contained in $\Int B$, and 
%$\partial B_r(y) \cap \partial L =\emptyset$.) 
Let $\Sigma$ be a compact 
Riemann surface with (nonempty) boundary with corners,
and let $u: \Sigma \into B^j$ be a nonconstant $J_{\bbC}-$holomorphic curve such that there is some $y \in B_0^j \cap u(\Sigma)$, and $u(\partial \Sigma) \subset \partial D_{s}^g(y) \cup L$ for some $0< s \leq \rho$. Then
$$\int_\Sigma u^*\omega \geq \gamma s^2. \, \square$$
\end{lemma}

\begin{remark} It doesn't matter that $B^j$ has boundary because the proof of this type of lemma takes place on small balls of radius less than $\mu$ centered at points $y \in B_0$, which are contained in $\Int(B)$. 
%which are the images of balls under the exponential map. Thus it makes sense to formulate the lemma as stated, referring only to $B$ and $B_0$. 
Also, the version of this lemma in \cite{FSS} assumes that $L$ is connected.
Because $\rho$ is less than half the distance between any pair of components of $L$, we can apply that version to each component successively, shrinking $\rho$ each time, to deduce the version above.
\end{remark}
We will apply the above lemma to $L=L^j$, where
$$L^j = (V_4 \cap B^j) \cup (V_0\cap B^j ) \cup (V_2^j \cap B^j)$$
% \, L = \cup_j L^j.$$
From figure \ref{figureB} we see that $L^j$ is a  
smooth Lagrangian (not connected) submanifold of $B^j$ with 
$\partial L_j \subset \partial B_j$.
%because $B$ does not contain point common to two of $V_4, V_2^j, V_0$, and $V_4,V_2^j, V_0$ meet $\partial B$ transversely and away from the corners of $B$. 
\\
\newline Fix an $I \in \mathcal{J}(M)$ which makes $\partial M$ $I-$convex. 
 Equip $T^*L_2^j = T^*S^3$  the complex structure $J_\bbC$ \label{J_{bbC}} inherited from the symplectic identification 
$T^*S^3 \cong \{z \in \bbC^4 : \Sigma_i z_i^2 =1\}$ (see page \pageref{mu}). Fix any $J \in \mathcal{J}(V,M,I)$ which satisfies 
\begin{gather}\label{J}
J_\zeta|U^j = J_{\bbC}|U^j \text{ for all }j \text{ and for all } \zeta \in V.
%J_\zeta = I \text{ near } \partial M \text{ for all } \zeta \in V \notag
\end{gather}

\begin{lemma} \label{containment} There is a constant $\epsilon >0$ depending on
$J_{\bbC}|(\cup_j B^j)$, $\cup_j L^j \subset \cup_j B^j$, 
and $\omega|(\cup_j B^j)$ such that the following holds.
Let $J \in  \mathcal{J}(V,M,I)$ satisfy (\ref{J}); then for any 
$w \in \mathcal{M}(M,V_4,V_2^j,V_0,J)$, if $w(p) \in B_0^j$ for some $p$ then 
$\int_V w^*\omega \geq \epsilon.$ 
\end{lemma}

\begin{proof} %For topological reasons we may assume $w(p) \in B_0^j$ for some $j$. 
Set $y = w(p) \in B_0^j$. %$w$ is nonconstant because of the boundary conditions.
We invoke lemma \ref{monotonicity}  for $L^j$ as defined above.
Let $\rho>0$ and $\gamma>0$ be the resulting constants.
%Shrink $\rho$ if neccesary so that it less than half the distace between
%between any pair of components of $L$.
Now take $s>0$ such that $\rho/2 < s < \rho$ and such that $w$ is transverse 
to $\partial D_s^{\,g}(y)$. Set $\Sigma = w^{-1}(D_s^{\,g}(y))$. 
Then $w(\partial \Sigma) \subset \partial D_s^{\,g}(y)\cup L^j$. We have 
$\partial \Sigma \neq \emptyset$ and $w|\Sigma$  nonconstant because of the boundary conditions for $w$ and because $D_s^{\,g}(y) \subset \Int(B^j)$.
Lemma \ref{monotonicity} applied to $u = w|\Sigma$ says that
$$\int_V w^*\omega \geq  \int_{\Sigma} u^*\omega \geq \gamma s^2 \geq \gamma \rho^2/4,$$
so we can take $\epsilon =  \gamma \rho^2/4.$
\end{proof}

We now consider new  Lagrangians $\widetilde V_4$, $\widetilde V_2^j$, 
$\widetilde V_0$, obtained by applying exact isotopies to
$V_4$, $V_2^j$, $V_0$. These isotopies arise by adjusting
$\Gamma_4$, $\Gamma_2$, $\Gamma_0$ in $\bbR/2\pi\bbZ \times [-r, r]$ 
by exact isotopies as in figure \ref{CorrectionGammaIIbig(mu2)}.
%\begin{comment}
\begin{figure}
\begin{center}
\includegraphics[width=5in]{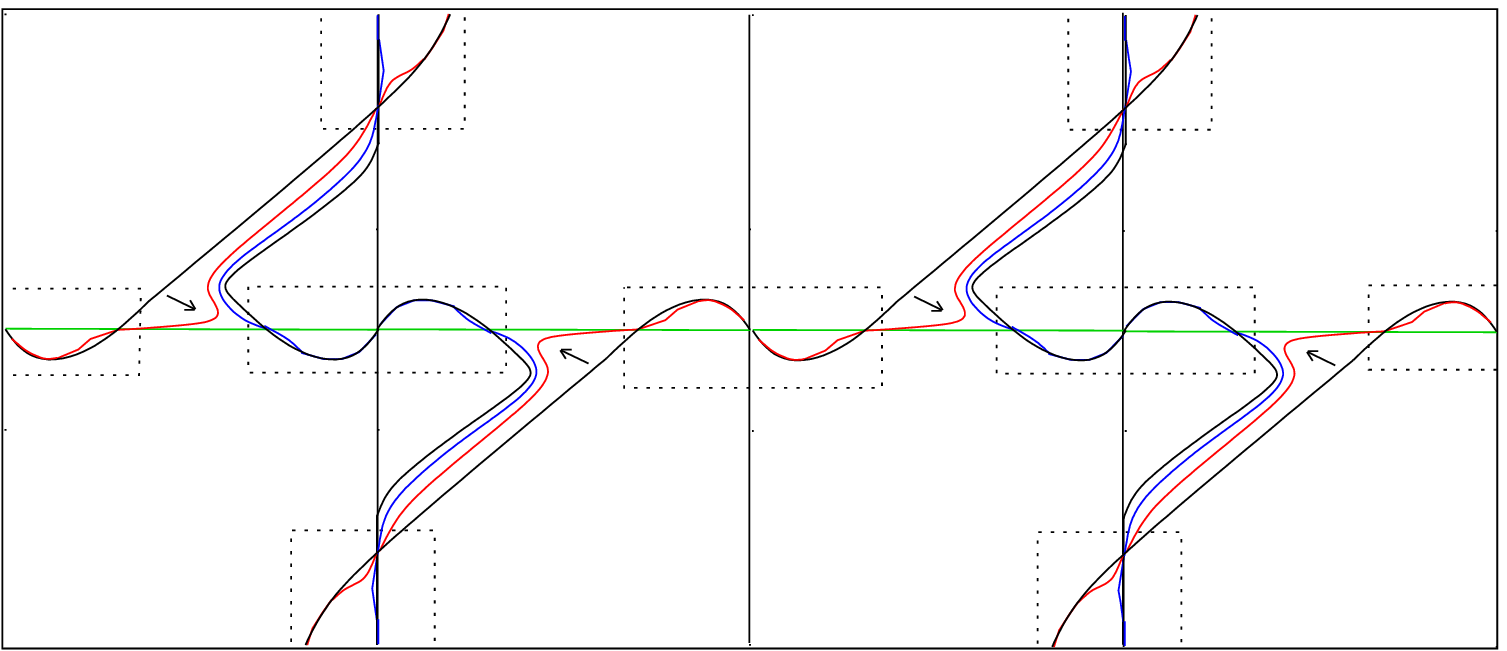}
%\put(-125,125){$\widehat{\psi}$}
%Simple version of vanishing spheres
%$\Gamma_4$,  $\Gamma_2$, $\Gamma_0$ in $\bbR /2\pi\bbZ \times [-r,r]$.}
\caption{ $\widetilde \Gamma_4$ (red), $\widetilde \Gamma_2$ (green), $\widetilde \Gamma_0$ (blue) in $\bbR /2\pi\bbZ \times [-r,r]$.} \label{CorrectionGammaIIbig(mu2)}
\end{center}
\end{figure}
%\end{comment}
The new versions $\widetilde \Gamma_4$, $\widetilde \Gamma_2$, $\widetilde \Gamma_0$ are chosen so that the 
four triangles bounded by $\widetilde \Gamma_4$, $\widetilde \Gamma_2$, $\widetilde \Gamma_0$ each have
area less that $\epsilon$ (from lemma \ref{containment}). 
In addition, the 
boundary arcs of each triangle are chosen to be real analytic (we will need this to construct $w_{ef}^j$ as in Proposition \ref{classificationII}). In order for this to happen, the new 
$\widetilde \Gamma$'s must interpolate back to the old $\Gamma$'s \emph{away} from the triangles; that is why there are small humps between  $\Gamma_i$
and $\widetilde \Gamma_i$ at each vertex. However, we also insist that 
$\widetilde \Gamma_i \cap \widetilde B = \Gamma_i \cap \widetilde B$
 (this is ensured if the small humps at each vertex are small enough).
The last condition implies  that 
$\widetilde V_4 \cap B = V_4 \cap B $, 
$\widetilde V_2^j \cap B = V_2^j \cap B $, and
$\widetilde V_0 \cap B = V_0 \cap B $. 
 Thus, for any $w \in \mathcal{M}(M,\widetilde V_4,\widetilde V_2^j,
\widetilde V_0,J)$, 
 \begin{gather} \label{epsilon}
 %\text{For any }
 %w \in \mathcal{M}(M,\widetilde V_4,\widetilde V_2^j,\widetilde V_0,J),
\text{ if } w(p) \in B_0^j \text{ for some } p, \text{ then } 
\int_V w^*\omega \geq \epsilon, 
\end{gather}
where $\epsilon$ is the same one from  lemma \ref{containment}.
(Note the proof of lemma \ref{containment} takes place entirely in $B$, so it will not feel the difference between  
$\widetilde V_4$, $\widetilde V_2^j$, 
$\widetilde V_0$, and
$V_4$, $V_2^j$, $V_0$.) 
%lemma \ref{containment}applies to any $w \in \mathcal{M}(M,\widetilde V_4,\widetilde V_2^j, \widetilde V_0,J)$ \emph{with the same  $\epsilon$ as before}.
\\
\newline We keep the same notation 
$\Delta \subset  \bbR/2\pi\bbZ \times [-r, r]$ 
for any one of the  four triangles bounded by $\widetilde \Gamma_4$, $\widetilde \Gamma_2$, $\widetilde \Gamma_0$. Note that the new functions $f_{40}$,  
$f_{42}^j$,  $f_{20}^j$, 
corresponding to $\widetilde V_4$, $\widetilde V_2^j$, 
$\widetilde V_0$ have the same basic form, so we keep the same notation there as well.

\begin{lemma} \label{area}
Let $J \in  \mathcal{J}(V,M,I)$ satisfy (\ref{J}), and 
 let $w \in \mathcal{M}(\widetilde V_4, \widetilde V_2^j, \widetilde V_0 , J;  M)$
for some $j =1, \ldots, k$. 
%Assume that $$lim_{\zeta \rightarrow \zeta_0}w(\zeta) \in \Sigma_{0,4}^j.$$
% $$lim_{\zeta \rightarrow \zeta_1}w(\zeta) \in \Sigma_{2,0}^j.$$
% $$lim_{\zeta \rightarrow \zeta_2}w(\zeta) \in \Sigma_{4,2}^j.$$
%and assume that the  puncture points converge to some 
%$p_0 \in \Sigma_{0,4}^j$, $p_1 \in \Sigma_{4,2}^j$,  $p_2 \in \Sigma_{2,0}^j$. 
Then the symplectic area of $w$ is less than or equal to the Euclidean area of the triangle $\Delta \subset \bbR /2\pi \bbZ\times \bbR$, hence
it is less than $\epsilon$ (from lemma \ref{containment}). 
\end{lemma}

\begin{proof}  
Let $w: V \into M$ be any smooth map, which extends continuously to $D^2$,
and which satisfies $w(I_0) \subset \widetilde V_0$, $w(I_1) \subset \widetilde V_2^j$, 
$w(I_2) \subset \widetilde V_4$ (for example, any $w \in  
\mathcal{M}(\widetilde V_4,  \widetilde V_2^j, \widetilde V_0 , J;  M)$).
 Then 
$$p_{40} =\lim_{\zeta \rightarrow \zeta_0} w(\zeta) \in  \Sigma_{40}^j \cup \{\text{isolated critical points of }f_{40}\},$$ 
$$ p_{42} = \lim_{\zeta \rightarrow \zeta_1}w(\zeta) \in K_+^j \cup \Sigma_{42}^j, \text{ and }
p_{20} = \lim_{\zeta \rightarrow \zeta_2}w(\zeta) \in K_+^j \cup \Sigma_{20}^j.$$
Take functions $k_0, k_2, k_4$ defined on $\widetilde V_0, (\cup_j \widetilde V_2^j), \widetilde V_4$ 
respectively, such that
$\theta|\widetilde V_2^j = dk_2$, $\theta|\widetilde V_0 = dk_0$, $\theta|\widetilde V_4 = dk_4$. 
Up to adding constants we have, 
\begin{gather*}
k_4|(\widetilde V_4 \cap D(T^*N_{42})) = f_{42}, \, 
k_4|(\widetilde V_4 \cap D(T^*N_{40})) = f_{40}, \,
k_0|(\widetilde V_0 \cap D(T^*N_{20})) = f_{20}, 
\end{gather*}
%as well as  %$k_0|(\widetilde V_0 \cap D(T^*N_{40})) =0$, $k_2|\widetilde V_2 = 0$.
%$k_4|(V_4 \cap D(T^*N_{42})) = f_{42}$, 
%$k_4|(V_4 \cap D(T^*N_{40})) = f_{40}$,
%$k_0|(V_0 \cap D(T^*N_{20})) = f_{20}$,
%$k_0|(V_0 \cap D(T^*N_{40})) =0$.
Then, by Stoke's theorem, %(applied twice) we have 
\begin{gather} \label{stokes}
\int_V w^*\omega = k_4(p_{40})-k_4(p_{42})
+ k_2(p_{42})-k_2(p_{20}) 
+ k_0(p_{20})-k_0(p_{40}).
\end{gather}
\\
Consider the case when all three vertices lie on the tori,
$p_{kl} \in \Sigma_{kl}^j$. Then 
$\int_V w^*\omega = area(\Delta)$ because the above formula shows the 
symplectic area  will not change if 
we  assume  $w$ to be of the form $\sigma_{ef}^j \circ \varphi$, 
where $\varphi: V \into \bbR/2\pi\bbZ \times [-r, r]$ is a smooth 
map parametrizing $\Delta$, and then, using the coordinates $(u,v)$ 
on $T^*L_2^j$, 
\begin{gather*}
\int_V(\sigma_{ef}^j \circ \varphi)^*\omega  
= \ \int_V (\sigma_{ef}^j \circ \varphi)^*(\Sigma_k \, d v_k\wedge du_k)
= \int_V \varphi^*(d\lambda \wedge d\theta).
\end{gather*}
(Note: $\varphi$ should be orientation reversing; compare with 
the proof of Proposition \ref{existence}.)
Now consider the case 
$p_{40} \in 
\{\text{isolated critical points of }f_{40}\}$, $p_{42} \in \Sigma_{42}^j$,
and $p_{20} \in \Sigma_{20}^j$. Let  $\widetilde w : V \into M$ 
be second map satisfying the same conditions, 
except this time with
$q = \lim_{\zeta \rightarrow \zeta_0} \widetilde w(\zeta) \in  \Sigma_{40}^j$. 
Then $\int_V \widetilde w^* \omega = area(\Delta)$, but also 
 $\int_V \widetilde w^* \omega$ is given by the same formula
(\ref{stokes}), but with $p_{40}$ replaced by $q$.
Therefore,
$$\int_V w^*\omega = area(\Delta)  - k_4(q)+ k_4(p_{40}) - k_0(q) + k_0(p_{40}). 
$$ 
Since $k_4 = f_{40}$ in $D(T^*N_{40})$, and $f_{40}$ takes a maximum at $\Sigma_{40}^j$
(and we may assume the value of $f_{40}$ at all isolated critical points is much smaller)
it follows that $- k_4(q)+ k_4(p_{40})< 0$. Similarly, by inspecting the shape of $\widetilde \Gamma_0$ near the vertex of $\Delta$ corresponding to $\Gamma_4 \cap \Gamma_0$ (see figure \ref{CorrectionGammaIIbig(mu2)}), we see that 
$k_0$ has a maximum
along $\Sigma_{40}^j$, and decreases to some constant value in the region where 
$p_{40}$ lies. Therefore $-k_0(q) + k_0(p_{40}) < 0$ as well. 
We conclude  $\int_V w^*\omega < area(\Delta)$ (if it is negative, it means that no such 
holomorphic $w$ could exist).
The other cases, $p_{42} \in K_+^j$,
and/or $p_{20} \in K_-^j$, are handled similarly; one uses the fact that
$f_{42}$ has a maximum at $\Sigma_{40}^j$, and $f_{20}$ has a minimum at $\Sigma_{20}^j$. \end{proof}

%Lemmas \ref{containment} 
Now (\ref{epsilon}) and lemma \ref{area} immediately imply: 

\begin{prop} \label{containmentII} Any  $w \in \mathcal M( \widetilde V_4, \widetilde V_2^j, \widetilde V_0, J; M)$
must be contained in the interior of $U^j \subset T^*L_2^j$. $\square$
\end{prop}
%\emph{.}
Recall that $J|U^j = J_{\bbC}|U^j$. In \S  \ref{localtriangles}
we classify all $J_\bbC-$holomorphic triangles in  $T^*L_2^j$, as summarized 
in Proposition \ref{classificationII} below.
Let $\delta_{42}, \delta_{20}, \delta_{40}$ denote the vertices of $\Delta$, where $\delta_{ij} \in \Gamma_i \cap \Gamma_j$, $i,j =0,2,4$, and set $\Delta^* = \Delta \setminus \{ \delta_{42}, \delta_{20}, \delta_{40} \}.$
 Note that $\sigma_{ef}^j(\delta_{40}) \in \Sigma_{40}^j$,
 $\sigma_{ef}^j(\delta_{42}) \in \Sigma_{42}^j$, and $\sigma_{ef}^j(\delta_{20}) \in \Sigma_{20}^j$.
Below $W_4^j =\widetilde V_4 \cap D(T^*L_2^j)$, 
$W_2^j = \widetilde V_2^j \cap D(T^*L_2^j)$, 
$W_0^j =  \widetilde V_0 \cap D(T^*L_2^j)$.
%First we summarize, in lemma \ref{classificationII} below, the main results from section \ref{localtriangles}, see Propositions \ref{existence}, \ref{uniqueness}, and \ref{regular}. 

\begin{prop}\label{classificationII}
%\begin{enumerate}
%\item 
(1) (Existence) For each $e \in K_+^j$, $f \in  K_-^j$ there exists 
$w_{ef}^j \in \mathcal{M}(W_4^j, W_2^j, W_0^j, J_\bbC; D(T^*L_2^j))$ which satisfies $w_{ef}^j(V) = \sigma_{ef}^j(\Delta^*) \subset D(T^*K_{ef}^j),$ and 
$$\lim_{\zeta \rightarrow \zeta_0} w(\zeta) = \sigma_{ef}^j(\delta_{40}), \, \, 
\lim_{\zeta \rightarrow \zeta_1} w(\zeta) = \sigma_{ef}^j(\delta_{20}),
\, \, 
\lim_{\zeta \rightarrow \zeta_2} w(\zeta) = \sigma_{ef}^j(\delta_{42}).$$
%\\$\bullet$ 
%\item 
(2) (Uniqueness) Every $w \in \mathcal{M}(W_4^j, W_2^j, W_0^j, J_\bbC; D(T^*S^3))$
is equal to $w_{ef}^j$ for some $e,f$.
%\newline $\bullet$ 
%\\
%\item 
(3) (Regularity) $J_\bbC$ is regular.
%\end{enumerate}

\end{prop}

\begin{remark} \label{boundaryworries}The above result implies in particular 
that no 
$$w \in \mathcal{M}(W_4^j, W_2^j, W_0^j , J_\bbC;  D(T^*L_2^j))$$ 
touches the boundary of $D(T^*L_2^j)$, or the boundary of $W_4^j$, $W_0^j$. 
This is proved by composing a given $w$ with a certain holomorphic map 
$P: T^*L_2^j \into \bbC$ and applying the  maximum principle.
\end{remark}

From Propositions \ref{classificationII} and \ref{containmentII} we conclude:
\begin{prop} \label{moduli} We have an equality of moduli spaces
\begin{gather*}
\mathcal{M}(J, \widetilde V_4, \widetilde V_2^j, \widetilde V_0; M)
%=\mathcal M (J_{\bbC}, \widetilde V_4 \cap U^j, \widetilde V_2^j \cap U^j, \widetilde V_0 \cap U^j; U^j) \\
= \mathcal{M} (W_4^j, W_2^j,W_0^j, J_{\bbC};D(T^*L_2^j))
\end{gather*}
%in fact, both moduli spaces are equal to $$\mathcal M (J_{\bbC},W_4^j, W_2^j,W_0^j;D_r^{g_2}(T^*L_2^j)),$$ from lemma \ref{classificationII}.
Moreover, $J$ is a regular almost complex structure.% for $\mathcal{M}(J, \widetilde V_4, \widetilde V_2^j, \widetilde V_0;M)$
\end{prop}

\section{Computing the continuation map} \label{sectioncontnmap}

In this section we give a precise treatment of the argument summarized in
\S \ref{contnmapsketch}. 
%consider the Floer homology groups
%$HF(\widetilde V_4,\widetilde V_0)$, $HF(\widetilde V_4,\widetilde V_2^j)$, 
%$HF(\widetilde V_2^j,\widetilde V_0)$, and compute appropriate continuation
%maps $\phi: HF(V_4,V_0) \into  HF(\widetilde V_4,\widetilde V_0)$, etc.
%We show the continuation maps fix all the generators (and relations) from 
%Proposition \ref{generators}, so that , 
The main point is to show  $HF(\widetilde V_4,\widetilde V_0)$,
$HF(\widetilde V_4,\widetilde V_2^j)$, and
$HF(\widetilde V_2^j,\widetilde V_0)$,
inherit the same generators and relations from Proposition \ref{generators}; 
recall these match up perfectly with the ones for the flow category (\ref{homologyrelations}) in 
\S \ref{sectionFlow}.
\\
\newline For $HF(\widetilde V_4,\widetilde V_0)$, $HF(\widetilde V_4,\widetilde V_2^j)$, $HF(\widetilde V_2^j,\widetilde V_0)$, we take the same
Morse-Smale data as in \S \ref{sectionHF}: $(h_{42}^j,g_{42}^j)$, $(h_{20}^j,g_{20}^j)$,  and $(\widetilde h_{40}^j,\widetilde g_{40}^j)$
on  $\Sigma_{42}^j \cup K_+^j$,  $\Sigma_{20}^j \cup K_-^j$, and $\Sigma_{40}^j$, respectively. For the almost complex structures let
$$\widetilde J_{42}^j, \widetilde J_{20}^j, \widetilde J_{40} \in \mathcal{J}([0,1],M, I), j=1,\ldots, k.
\label{widetilde J_{40} ,widetilde  J_{42},widetilde  J_{20}(b)}$$
%with $\widetilde (J_{42}^j)_t, (\widetilde J_{20}^j)_t, (\widetilde J_{40})_t$  \mathcal{J}$\widetilde J_$ is equal to $I$ near $\partial M$. 
be such that
%the following properties. All complex structures satisfy
\begin{gather}\label{J_t|U}
(\widetilde J_{42}^j)_t|U, (\widetilde J_{20}^j)_t|U, (\widetilde J_{40})_t|U= J_\bbC, t \in [0,1],
\end{gather}
and assume that the following Morse-Bott Floer data are regular: 
$$(\widetilde J_{42}^j,(h_{42}^j,g_{42})), (\widetilde J_{20}^j,(h_{20}^j,g_{20})) (\widetilde J_{40},(h_{40}^j,g_{40})). $$
To see  that regularity can be achieved under the above constraints, we follow
a basic principle outlined in \cite[p. 35]{MS}: If every nonconstant 
$J-$holomorphic curve meets some neighborhood $\Omega$, then to make $J$ 
regular it suffices to 
perturb $J$ on $\Omega$ only. In our situation, we 
fix a neighborhood $\Omega \subset (M \setminus \cup_j U^j)$ containing  
$\cup_j (K_+^j \cup  K_-^j)$ as well as all isolated critical point of 
$f_{40}$ (this is possible by construction of the $U^j$: 
see figure \ref{figureB}).
Now note that every nonconstant finite area holomorphic strip will meet
$\Omega$, and,  since $\Omega \subset  (M \setminus \cup_j U^j)$, we can  pick $\widetilde J_{40}$
$\widetilde J_{42}^j$, and $\widetilde J_{40}$ with complete freedom on 
$\Omega$.
\\
\newline In fact, we can obviously pick all $\widetilde J_{20}^j$ to be \emph{equal}
and similarly for $\widetilde J_{42}^j$; so, for convenience of notation 
can drop the $j$'s and denote these by 
$\widetilde J_{20}$, $\widetilde J_{42}$.
Now pick $J\in \mathcal{J}(V,M,I)$ compatible with  $\widetilde J_{40}$
$\widetilde J_{42}$, and $\widetilde J_{20}$ in the sense of 
(\ref{Jcompatible}), that is, for all $(s,t) \in [0,\infty) \times [0,1]$, %with $J_\zeta =I$ near $\partial M$, for all $\zeta \in V$, such that 
\begin{gather*}
J_{\epsilon_0(s,t)}  = (\widetilde J_{40})_t, \, 
J_{\epsilon_1(-s,t)} = (\widetilde J_{42})_t, \,
J_{\epsilon_2(-s,t)} = (\widetilde J_{20})_t, \text{ and } 
J_\zeta|U = J_\bbC, \, \zeta \in V. %J_\zeta =I \text{ near }\partial M, \, 
\end{gather*}
The  first three conditions are compatible with last one, since 
each $\widetilde J_t$ satisfies (\ref{J_t|U}). 
Since $J$ satisfies (\ref{J}), the whole discussion in 
\S \ref{sectionmonotonicity} applies to $J$; in particular we have 
Proposition \ref{moduli}, and $J$ is regular.
\\
\newline We now pick some smooth functions
$G,H: M \into \bbR$  such that their Hamiltonian flows satisfy
$\phi_1^H(V_4) = \widetilde V_4$ and $\phi_1^G(V_0) = \widetilde V_0$. 
These will be determined by 
$H_0, G_0: \bbR/2\pi \bbZ \times [-r,r] \into \bbR$ which satisfy  
$\phi_1^{H_0}(\Gamma_4) = \widetilde \Gamma_4$ and $\phi_1^{G_0}(\Gamma_0) = \widetilde \Gamma_0$. %Pramaterize $\Gamma_0$ by an interval $[0,1]$
Consider  a Weinstein tubular neighborhood of $\Gamma_0$, say
$D(T^*\Gamma_0) \subset  \bbR/2\pi \bbZ \times [-r,r]$, 
which contains $\widetilde \Gamma_0$ in its interior.
Inside  $D(T^*\Gamma_0)$, we view $\widetilde \Gamma_0$ as 
the graph of an exact 1-form $g'(x)dx$, where $x \in [0,1] \cong \Gamma_0$
is some coordinate. Similarly, we view $\widetilde \Gamma_4$ 
as the graph of $h'(x)dx$ over $\widetilde \Gamma_4$. Then 
the graphs of $g'$, $h'$, $g$, and $h$ are as shown in 
figure \ref{CorrectionGH}. 
(The graphs of  $g'$ and $h'$ are found by inspecting  figure 
\ref{CorrectionGammaIIbig(mu2)} and 
copying out  $\widetilde \Gamma_0$ and $\widetilde \Gamma_4$, viewed as graphs 
of functions over $\Gamma_0$ and  $\Gamma_4$.)
\begin{figure}
\begin{center}
\includegraphics[width=3in]{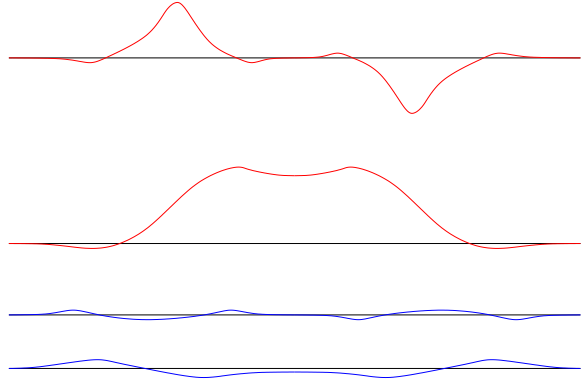} 
%\put(-200,100){$(0,0)$}
\caption{From top to bottom: the graphs of $h'$ and $h$ (red), and the graphs of $g'$ and $g$ (blue).}
\label{CorrectionGH}
\end{center}
\end{figure}
Note that in figure \ref{CorrectionGH} (compare with 
figure \ref{CorrectionGammaIIbig(mu2)})
the two minima of $h$ correspond to 
points in $\Sigma_{40}^j$, and the two maxima correspond to points in 
$\Sigma_{42}^j$. Similarly, the two maxima of $g$ correspond to 
points in $\Sigma_{40}^j$, and the two minima correspond to points in 
$\Sigma_{20}^j$.
\\
\newline Let $(x,y)$ denote the usual coordinates on
$D(T^*[0,1])$, then we define 
$G_0$, $H_0$ %: \bbR/2\pi \bbZ \times [-r,r] \into \bbR$ 
to be zero outside of $D(T^*\Gamma_0)\cong D(T^*[0,1])$ and
 $D(T^*\Gamma_4)\cong D(T^*[0,1])$, and on those regions we set
%$G_0,H_0: \bbR/2\pi \bbZ \times [-r,r] \into \bbR$ by 
$H_0(x,y) = - \psi(|y|) h(x)$, $ G_0(x,y) = - \psi(|y|) g(x),$
where $\psi$ is some suitable cut-off function which makes $G_0$, $H_0$ 
equal to zero near the horizontal boundary of $D(T^*[0,1])$.
% which is identically 1
%on a neighborhood of the zero section containing the graphs of $g$ and $h$, 
%and which damps down to 0 near the horizontal boundary of $D(T^*[0,1])$.
(Note that we have taken $g$ and $h$ to be 
zero near the boundary of $[0,1]$ also.)
We need the $-$ signs in order to have $\phi_1^{H_0}(\Gamma_4) = \widetilde \Gamma_4$ and $\phi_1^{G_0}(\Gamma_0) = \widetilde \Gamma_0$.
%$$\phi_{H_0}^1([0,1]\times \{0\}) = \Gamma(dh), \,\, \phi_{G_0}^1([0,1]\times \{0\}) = \Gamma(dg).$$
This is because of our conventions (see  \S \ref{conventions})
and because the symplectic form on 
$\bbR/2\pi \bbZ \times [-r,r]$ is $d\lambda \wedge d\theta$, where
$(\theta,\lambda) \in  \bbR/2\pi \bbZ \times [-r,r]$ (see Proposition \ref{existence} to see why). Now let us take coordinates on  
$D(T^*L_2^j) = D_{r}^{g_2}(T^*S^3)$
given by $(u,v) \in \bbR^4 \times \bbR^4$, $|u|=1$, $u \cdot v=0$,  where 
$$u = (\cos\theta e, \sin \theta f), \phantom{bbb} v=\lambda_1(-\sin\theta e, \cos \theta f)
+ \lambda_2(J_0 e)  + \lambda_3 (J_0 f)$$ 
%(u,v)  = ((\cos\theta e, \sin \theta f), 
%(\lambda_1(-\sin\theta e, \cos \theta f)
%+ \lambda_2(Je)  + \lambda_3 (Jf) ),
where $\lambda_1, \lambda_2, \lambda_3 \in \bbR$, 
$e\in K_+^j$, $f\in K_-^j$, $\theta \in \bbR/2\pi \bbZ$, 
$J_0e$ $=$ $J_0(e_1, e_2,0,0)$ $ =$ $ (e_2, -e_1,0,0)$, similarly for  $J_0 f$,
% = J(e_1, e_2,0,0) = (e_2, -e_1,0,0)$
 and $(g_2)_u(v,v) \leq r^2$. Then define
$$G,H: D(T^*L_2^j) \into \bbR, \phantom{bbb} H(u,v)  = H_0 (\theta, \lambda_1), \phantom{bbb} G(u,v)  =  G_0 (\theta, \lambda_1).$$
%Define $$G:  D_{r}^{g_{S^3}}(T^*S^3) \into \bbR $$ similarly. 
Since $G$ and $H$ are both zero near $\partial D(T^*L_2^j)$, we can extend them by zero to the rest of $M$, and we denote these also by $G$, $H$. 
Then we have 
$\phi_1^H(V_4) = \widetilde V_4$ and  $\phi_1^G(V_0) = \widetilde V_0.$
Because we have defined $G$ and $H$ in a symmetrical way 
explicitly in terms of $g$ and $h$, we can see from figure \ref{CorrectionGH}
that $H$ has an absolute minimum at $\Sigma_{40}^j$ and an absolute maximum at 
$\Sigma_{42}^j$.  Similarly, $G$ has an absolute maximum at $\Sigma_{40}^j$ and an absolute minimum at $\Sigma_{20}^j$. 
\\
\newline Recall there is an elementary isomorphism of moduli spaces
$$\mathcal{M}(L_0,L_1,\{H_t\},\{J_t\}) \into
\mathcal{M}(L_0,(\phi_1^H)^{-1}(L_1), 0, \{(\phi_t^H)^*J_t\})$$
given by $u \mapsto \widetilde u$, where
$\widetilde u(s,t)= (\phi^H_t)^{-1}(u(s,t))$.
This isomorphism also has a version at the 
linearized level which carries regular data into regular data.
From this we obtain
$HF(L_0, (\phi_1^H)^{-1}(L_1), 0, \{J_t\}) \cong HF(L_0,L_1,\{H_t\},\{(\phi_t^H)_*J_t\})$.
%There is a similar even simpler isomorphism where one pushes everything forward by a symplectomorphism.
We apply this (as well as a simpler isomorphism where one pushes everything forward by a 
symplectomorphism) and  obtain
$$HF(\widetilde V_4, \widetilde V_0, 0, \{(\widetilde J_{40})_t\}) 
= HF(\phi_1^{H}(V_4), \phi_1^{G}(V_0), 0, \{(\widetilde J_{40})_t\})$$
$$\cong HF(V_4, ((\phi_1^{H})^{-1}\circ \phi_1^{G})(V_0), 0, 
\{ (\phi_1^{H})^{-1}_*(\widetilde J_{40})_t\})$$
  $$ =  HF(V_4, (\phi_1^{(G-H) \circ \phi_t^G})^{-1}(V_0), 0, 
\{ (\phi_1^{H})^{-1}_*(\widetilde J_{40})_t\})$$                    
 $$  \cong   HF(V_4,V_0, \{(G-H) \circ \phi_t^G\},
\{ (\phi_1^{(G-H) \circ \phi_t^G})_*(\phi_1^{H})^{-1}_*(\widetilde J_{40})_t\}).$$
Above we have used $(\phi_t^F)^{-1} = \phi_t^{\{-F \circ \phi_t^F\}}$ and
$\phi_t^{\{G_t\}} \circ \phi_t^{\{H_t\}} = \phi_t^{\{G_t + H_t \circ (\phi^G_t)^{-1}\}}$.
Similarly, we have 
$$HF(\widetilde V_4, \widetilde V_2^j, 0, \{(\widetilde J_{42})_t\}) 
= HF(\phi_1^{H}(V_4), V_2^j, 0, \{(\widetilde J_{42})_t\})$$
$$\cong    HF(V_4,V_2^j, H , \{(\phi_t^{H})_*(\phi_1^{H})^{-1}_*(\widetilde J_{42})_t\}), 
\text{ and }$$
$$HF(\widetilde V_2^j, \widetilde V_0, 0, \{(\widetilde J_{20})_t\}) 
= HF( V_2^j, (\phi_1^{-G \circ \phi_t^G })^{-1}(V_0), 0, \{(\widetilde J_{20})_t\})$$
$$\cong    HF(V_2^j,V_0,\{ -G \circ \phi_t^G\} , \{(\phi_t^{-G \circ \phi_t^G} )_*(\widetilde J_{20})_t\}).$$ 
As we mentioned above, all these new Floer data $(H,J)$ for  $V_4$, $V_2^j$, $V_0$ are regular.
%, because they arise by natural isomorphisms of moduli spaces. 
Now set $(H_{40}^\alpha)_t = 0 $,
$(H_{42}^\alpha)_t =   H$, $(H_{20}^\alpha)_t =   -G \circ \phi_t^G$,
$H_{40}^\beta = (G-H) \circ \phi_t^G$,  $H_{42}^\beta = 0$, $H_{20}^\beta = 0$, for $t \in [0,1]$.
Note that  $(H_{40}^\alpha)_t$ has an absolute maximum at $\Sigma_{40}^j$ for all $t \in [0,1]$, and $(H_{42}^\alpha)_t$, $(H_{20}^\alpha)_t$
 each have an  absolute minimum respectively at 
$\Sigma_{42}^j$, $\Sigma_{20}^j$.
(For example, to see this for  $(H_{40}^\alpha)_t$, note that
$G$ and $-H$ both have a maximum at 
$\Sigma_{40}^j$, and $\phi_t^G$
is a diffeomorphism which fixes $\Sigma_{40}^j$ point-wise for all $t$.)
Add constants to each of  $(H_{40}^\alpha)_t$, $(H_{42}^\alpha)_t$ and $(H_{20}^\alpha)_t$ so that these maxima and minima 
%along $\Sigma_{40}^j$,  $\Sigma_{42}^j$, and  $\Sigma_{20}^j$ 
are all equal to \emph{zero}.
We take  homotopies  of time dependent Hamiltonians,
$$(H_{40}^{\alpha \beta})_{s,t}, (H_{42}^{\alpha \beta})_{s,t}, (H_{20}^{\alpha \beta})_{s,t}: M \into \bbR, s \in \bbR, t \in [0,1], \text{ with} $$
 $$(H_{40}^{\alpha \beta})_{s,t} = \varphi(s)(H_{40}^\alpha)_t, \, 
(H_{42}^{\alpha \beta})_{s,t} = \varphi(s)(H_{42}^\alpha)_t, \,
(H_{20}^{\alpha \beta})_{s,t} = \varphi(s)(H_{20}^\alpha)_t,$$
where $\varphi: \bbR \into \bbR$ is some monotone cut-off function with
$\varphi(s) = 1$ for $s \leq -1$ and $\varphi(s) = 0$ for $s \geq 1$.
Because $\partial_s (H_{40}^{\alpha \beta})_{s,t}(p) \leq 0$ for all
$p\in M$, $s\in \bbR$, $t \in [0,1]$, we say that 
$(H_{40}^{\alpha \beta})_{s,t}$ is a \emph{monotone} homotopy (see 
\cite{CFH, FH, CR}), and similarly for $(H_{42}^{\alpha \beta})_{s,t}$, 
$(H_{20}^{\alpha \beta})_{s,t}$. This condition will be important in what follows.
Set 
$$(J_{40})_t^\alpha= (J_{40})_t,
\phantom{bb}   (J_{40}^\beta)_t 
= (\phi_1^{(G-H) \circ \phi_t^G})_*(\phi_1^{H})^{-1}_*(\widetilde  J_{40})_t,$$
$$(J_{42}^\alpha)_t =(\phi_t^{H})_*(\phi_1^{H})^{-1}_*(\widetilde J_{42})_t,
\phantom{bb}  
(J_{42})_t^\beta=(  J_{42})_t,\text{ and } $$
$$(J_{20}^\alpha)_t = (\phi_t^{-G \circ \phi_t^G} )_*(\widetilde J_{20})_t, 
\phantom{bb}   (  J_{20})_t^\beta= (  J_{20})_t.$$
Pick homotopies of almost complex structures
$(J_{40}^{\alpha \beta})_{s,t}$, $(J_{42}^{\alpha \beta})_{s,t}$, $(J_{20}^{\alpha \beta})_{s,t}$ respectively equal to  
$(  J_{40}^\alpha)_t$, $(  J_{42}^\alpha)_t$, $(  J_{20}^\alpha)_t$ for $s \leq -1$ and equal to 
$(  J_{40}^\beta)_t$, $(  J_{42}^\beta)_t$, 
$(  J_{20})_t^\beta$ for  $s \geq 1$.
We will consider the continuation map
$$\phi_{40}: HF(V_4, V_0, H^\alpha_{40}, J^\alpha_{40}) \into  
HF(V_4, V_0, H^\beta_{40}, J^\beta_{40}),$$
and similarly for $(V_4,V_2)$ and $(V_2,V_0)$.
The main ingredient for this 
is the moduli space of $s$-dependent 
$(H^{\alpha \beta}_{40}, J^{\alpha \beta}_{40})$-holomorphic curves:
\begin{gather}\label{phimoduli}
\mathcal{M}(V_4,V_0,  H^{\alpha \beta}_{40}, J^{\alpha \beta}_{40}) 
=\{ u  \in C^\infty(\bbR \times [0,1], M) :\phantom{b}
 u(\{0\} \times \bbR ) \subset V_4,  \phantom{b} \\\notag
 u(\{1\} \times \bbR ) \subset V_0, \phantom{bb}
\partial_s u + (J^{\alpha \beta}_{40})_{s,t} 
(\partial_t u - X_{s,t}^{ H^{\alpha \beta}_{40}} ) = 0, \phantom{b} \int u^*\omega < \infty\}. 
\end{gather}
(See \S \ref{generalMorseBott}
 for more details on the continuation map in
 the Morse-Bott set up.) 
We choose the $J^{\alpha \beta}_{40}$ 
so that each pair $(H^{\alpha \beta}_{40},J^{\alpha \beta}_{40} )$ is regular in 
the sense that the linearized Cauchy-Riemann operator associated to the above moduli space is regular (in addition, we pick 
 $J^{\alpha \beta}_{40}$
so that the Morse-Bott data
 $((h_{40}^j, g_{40}^j), H^{\alpha \beta}_{40},J^{\alpha \beta}_{40} )$
is regular for each $j$). We impose similar conditions regarding
$(V_4,V_2^j)$ and $(V_2^j,V_0)$.

\begin{lemma} \label{uconst} If $u \in \mathcal{M}(V_4,V_0,  H^{\alpha \beta}_{40}, J^{\alpha \beta}_{40})$ satisfies $\lim_{s\rightarrow \pm \infty} u(s,t) \in \Sigma_{40}^j$, then 
$u$ must be constant. The corresponding statements hold for 
$(V_4,V_2^j)$ and $(V_2^j,V_0)$ as well.
\end{lemma}

\begin{proof} We prove this for $(V_4,V_0)$; the other cases are exactly the
same. Consider the $s-$dependent action functional:
%whose formal $s-$dependent negative gradient lines 
%correspond to  $s-$dependent $(H^{\alpha \beta}_{40}, J^{\alpha \beta}_{40})$-holomorphic curves:
For $y \in C^\infty([0,1], M)$ with $y(0) \in V_4$, $y(1) \in V_0$, 
$$A^{H^{\alpha \beta}_{40}}_s(y) =
   - \int_0^1 y^*\theta \,dt
+ h_{V_4}(y(1)) +h_{V_0}(y(0)) + \int_0^1 (H^{\alpha \beta}_{40})_{s,t}(y(t))\, dt,$$
where $\theta|V_i = dh_{V_i}, i=4,0$. Now consider $a(s) = A^{H^{\alpha \beta}_{40}}_s(u(s,\cdot))$
%where $u$ is in (\ref{phimoduli}).
First, we have 
\begin{gather}\label{da}
\partial_s a(s)  = d(A^{H^{\alpha \beta}_{40}}_s) \cdot \partial_s u + 
\int_0^1 [ \partial_s (H^{\alpha \beta}_{40})_{s,t}] (u(s,t)) \,dt \\
=-\int_0^1 (g^{\alpha \beta}_{40})_{s,t}(\partial_s u,\partial_s u)\,dt \notag
+  \int_0^1 [ \partial_s (H^{\alpha \beta}_{40})_{s,t}] (u(s,t)) \,dt  \leq 0.
\end{gather}
Here the first term comes from 
the standard calculation that the action functional 
decreases along holomorphic strips, and $(g^{\alpha \beta}_{40})_{s,t}$ 
is the metric determined by  $(J^{\alpha \beta}_{40})_{s,t}$ and $\omega$, 
so the first term is $\leq 0$. The second term is $\leq 0$ because we assumed
$\partial_s (H^{\alpha \beta}_{40})_{s,t}(p) \leq 0$ for all $p \in M$.
\\
\newline Next, set  $p_{\pm} =\lim_{s\rightarrow \pm \infty} u(s,t) \in \Sigma_{40}^j$. Since $(H^{\alpha}_{40})_t$ and  $(H^{\beta}_{40})_t$
are both equal to zero at $\Sigma_{40}^j$, we have
$$\lim_{s \rightarrow  - \infty} a(s) = 
 \int_0^1(H^{\alpha}_{40})_t(p_-)\,dt =0 = \int_0^1(H^{\beta}_{40})_t(p_+)\,dt= \lim_{s \rightarrow  + \infty} a(s).$$
Since $\partial_s a \leq 0$, this shows $a$ is constant, so $\partial_s a =0$.
Now, returning to (\ref{da}), and recalling that each of the
two terms is $\leq 0$, we conclude that each term must be zero; in particular
$-\int_0^1 (g^{\alpha \beta}_{40})_{s,t}(\partial_s u,\partial_s u) =0$.
Hence $\partial_s u =0$ and therefore $u(s,t) = x(t)$ depends only on $t$ and
$(\partial_t x(t) - X_{s,t}^{ H^{\alpha \beta}_{40}}(x(t)) = 0$ for all $s$.
But for $s \geq 1$, we have  $X_{s,t}^{ H^{\alpha \beta}_{40}} =0$, so 
$\partial_t x(t) =0$ and we conclude $u$ is constant.
\end{proof}

Recall that in Proposition \ref{generators} we found specific generators and relations for 
$$HF(V_4, V_0, J_{40}^\alpha), 
HF(V_4, V_2, J_{42}^\beta), \text{ and } HF(V_2, V_0, J_{20}^\beta),$$ 

in terms of the critical points 
of $(h_{42}^j,g_{42}^j)$, $(h_{20}^j,g_{20}^j)$, $(\widetilde h_{40}^j,\widetilde g_{40}^j)$
on  $\Sigma_{42}^j $,  $\Sigma_{20}^j$, $\Sigma_{40}^j$.

\begin{lemma} \label{fixgenerators} The continuation maps
$$\phi_{40}: HF(V_4, V_0, H^\alpha_{40}, J^\alpha_{40}) \into  
HF(V_4, V_0, H^\beta_{40}, J^\beta_{40}),$$
$$\phi_{42}: HF(V_4, V_2, H^\alpha_{42}, J^\alpha_{42}) \into  
HF(V_4, V_2, H^\beta_{42}, J^\beta_{42}),$$
$$\phi_{20}: HF(V_2, V_0, H^\alpha_{20}, J^\alpha_{20}) \into  
HF(V_2, V_0, H^\beta_{20}, J^\beta_{20})$$
are such that  $\phi_{40}$,  $\phi_{42}^{-1}$
 $\phi_{20}^{-1}$ fix all  the generators given in  
Proposition \ref{generators}.
\end{lemma}

\begin{proof} 
%We will show   $\phi_{40}$,  $\phi_{42}$ and $\phi_{20}$respectively fix the critical points of 
%$(h_{42}^j,g_{42}^j)$, $(h_{20}^j,g_{20}^j)$, and  $(\widetilde h_{40}^j,\widetilde g_{40}^j)$
% in   $\Sigma_{42}^j $,  $\Sigma_{20}^j$, and $\Sigma_{40}^j$.
 First we deal with $HF(V_4, V_2^j)$ and $HF(V_2^j, V_0)$. 
Recall we have  generators of $CF(V_4,V_2^j)$ given by
$$x_2^j, x_1^j, (x_1^j)', x_0^j \in \Sigma_{42}^j, \, a_0^j,a_1^j \in K_+^j$$ 
and generators of $CF(V_2^j, V_0)$ given by
$$y_2^j,y_1^j, (y_1^j)', y_0^j \in \Sigma_{20}^j, \, b_0^j, b_1^j \in K_-^j.$$
In the proof of lemma \ref{MorsehomologyI}
we saw that $x_2^j, x_1^j$ and $y_2^j,(y_1^j)'$ freely generate their respective  homology groups, while
$$\partial(x_1^j)' = a_1^j, \partial x_0^j = a_0^j; \, \partial (y_1^j) = b_1^j,  \partial y_0^j= b_0^j.$$ 
It follows from the definition of the continuation map 
(see \ref{generalMorseBott}) and  lemma \ref{uconst}
that
$$\phi_{42}^j(x_2^j) = x_2^j + A^ja_0^j +B^ja_1^j,$$
for some coefficients $A^j,B^j$. This shows $\phi_{42}^j(x_2^j)$ is homologous to $x_2^j$, 
so $\phi_{42}^j([x_2^j]) =[x_2^j]$ in homology. Similarly for $\phi_{20}^j$.
\\
\newline For  $HF(V_4, V_0)$ we argue similarly: Let
$z \in Crit(h_{40}^j)$ be one of the critical points of 
$h_{40}^j: \Sigma_{40}^j \into \bbR$. Let $p_1, \ldots, p_r \in N_{40}$ denote the isolated critical points of $f_{40}$. Then again we have
$$\phi_{40}([z]) = [z] + \Sigma_j A_j [p_j]$$ 
for some coefficients $A_j$. 
We will now argue that all $p_j$ are such that $\partial p_j \neq 0$, so that 
$[p_j] =0$ in homology. 
To see this we recall the details of the chain complex 
$(C_{MB})_*((N_{40},f_{40},\mu_{40}), (h_{40},g_{40}))$, which 
%which were described in the proof of  lemma \ref{MorsehomologyII}.
is determined by a certain (Morse-Bott) handle decomposition of $N_{40}$ in 
the proof of  lemma \ref{MorsehomologyII}.
In that handle decomposition, we start with the union of tori $\cup_j \Sigma_{40}^j$ which 
are minima.
Then we attach 1-handles to each pair of adjacent tori. This gives the relations
$[z_0^1]= \ldots = [z_0^n]$. 
Thus the corresponding critical points $p$ of $f_{40}$ of index 1 satisfy
$\partial p = z_0^j + z_0^k$, so $\partial p \neq 0$.
Next, we attach certain 2-handles which always run along exactly two 1-handles. Thus
each critical point $p'$ of $f_{40}$ of index 2 satisfies
$\partial p' = p_1^p+ p_1^q + \Sigma_j C_j z_j$, $p \neq q$, where $\Sigma_j C_j z_j$ represents some 
terms (of index 1) in $Crit(h_{40})$. Thus, $\partial p' \neq 0$.
Finally there is one 3-handle, and the corresponding critical point $p''$ satisfies
$\partial p'' = \Sigma_{j=1}^{j=k} z_2^j$;  it gives rise to the relation $\Sigma_j [z_2^j] =0$. Thus $\partial p'' \neq 0$.
\end{proof}

\section{The proof of Theorem  $B$: Computing the triangle product}
\label{sectiontriangleproduct}
%Combining lemma \ref{newgenerators} and the isomorphisms we discussed at the beginning of \S \ref{sectioncontnmmap} we obtain the following
%\begin{prop} $HF(\widetilde V_4,\widetilde V_0)$, $HF(\widetilde V_4,\widetilde V_2^j)$, $HF(\widetilde V_2^j,\widetilde V_0)$ have the same generators and relations as in Proposition \ref{generators}. $\square$
%\end{prop}

\begin{thm}\label{fuk=flow}  The directed Donaldson-Fukaya category
of $(M,,\widetilde V_4, \{\widetilde V_2^j\},\widetilde V_0 )$ is described 
as follows. First, $HF(\widetilde V_4,\widetilde V_0)$, 
$HF(\widetilde V_4,\widetilde V_2^j)$, 
and  $HF(\widetilde V_2^j,\widetilde V_0)$, computed with the data
$(\widetilde J_{40}, (\widetilde h_{40}, \widetilde g_{40}))$,
$(\widetilde J_{42}, ( h_{42},g_{42}))$,
$(\widetilde J_{20}, ( h_{20},g_{20}))$,
are generated by 
$$x_2^{\, j}, x_1^{\, j} \in \Sigma_{42}^{\, j},\phantom{bb} y_2^{\, j}, (y_1^{\, j})' \in \Sigma_{20}^{\, j},
 \phantom{bbb} \tilde z_2^{\, j}, \widetilde {z}_1^{\, j}, 
(\widetilde{z}_1^{\, j})', \widetilde {z}_0^{\, j} \in {\Sigma}_{40}^{\, j},\phantom{bb} j=1, \ldots, k.$$ 
with the relations 
$$[ \tilde z_1^{\, j}] = \Sigma_{i \neq j} \ell k(K_j,K_i) 
[(\tilde z_1^{\, i})'];\,\, \Sigma_j [ \tilde z_2^{\, j}] = 0; \, \,
  [ \tilde z_0^{\, 1}] = \ldots =  [ \tilde z_0^{\, k}].$$
Here $m_j\in \bbZ$ is the framing coefficient for the 
$j$th 2-handle whose attaching sphere is $K_j$, 
and $\ell k(K_j,K_i)$ is the linking number of the knots.
%of $(N,f,g)$ that attaches at $K_j\subset L_0 = f^{-1}(1)$, 
%$j=1,\ldots,k$. 
Second,  $\mu_2$, computed with respect to $J$, 
is given by the formulas
$$\mu_2([x_2^{\, j}], [y_2^{\, j}]) = [\tilde z_2^{\, j}], \phantom{bbb}
\mu_2([x_1^{\, j}], [(y_1^{\, j})']) =[\tilde z_0^{\, j}],$$
$$\mu_2([x_2^{\, j}], [(y_1^{\, j})']) = [(\tilde z_1^{\, j})'],  \phantom{bbb}
\mu_2([x_1^{\, j}], [y_2^{\, j}]) = [\tilde z_1^{\, j}] + m_j [(\tilde z_1^{\, j})'].$$
%Moreover, we have the relations
%$$\Sigma_{j=1}^{j=k} \mu_2([x_2^{\, j}], [y_2^{\, j}])=0,  \phantom{bbb}
%\mu_2([x_1^1], [(y_1^1)']) = \mu_2([x_1^2], [(y_1^2)'])
%= \ldots = \mu_2([x_1^{\, k}], [(y_1^{\, k})']),$$ 
%$$\mu_2([x_1^{\, j}], [y_2^{\, j}]) = m_j \mu_2([x_2^{\, j}], [(y_1^{\, j})']) + \Sigma_{i \neq j} \ell k(K_j,K_i) \mu_2([x_2^i], [(y_i^{\, j})']).$$
In particular, these generators and relations are identical with those of the flow category stated in Proposition \ref{propFlow}, so the categories are isomorphic. \end{thm}

\begin{proof} %We compute $\mu_2$ in two steps. 
The generators and relations for the Floer groups are obtained by
combining lemma \ref{fixgenerators} and the isomorphisms we discussed at the beginning of \S \ref{sectioncontnmap}, and then using Proposition 
\ref{generators}. 
\\
\newline We first compute the  triangle product, restricting attention to $D(T^*L_2^j)$:
$$\mu_2^j: HF(W_4^j,W_2^j) \otimes HF(W_2^j, W_0^j) \into HF(W_4^j, W_0^j),$$  
for any fixed $j$, where we recall $W_i^j = \widetilde V_i \cap D(T^*L_2^j)$, 
$i=0,2,4$. We use $J_\bbC|D(T^*L_2^j)$ and the Morse-Smale pairs
$(h_{42}^j, g_{42}^j)$, $(h_{20}^j, g_{20}^j)$, $(h_{40}^j, g_{40}^j)$
on $\Sigma_{40}^j, \Sigma_{40}^j, \Sigma_{40}^j$. Each of these Floer groups is simply the corresponding Morse homology. We identify each of 
$\Sigma_{42}^j, \Sigma_{20}^j, \Sigma_{40}^j$ with $K_+^j \times K_-^j$.
%are contained in the interior of $U^j \subset \Int D_r^{g_2}(T^*L_2^j)$ and have their vertices on $\Sigma_{42}^j, \Sigma_{20}^j, \Sigma_{40}^j$. 
%Since $J|U^j = J_\bbC|U^j$, we can appeal directly to the classification lemma \ref{classificationII}. 
Then, by Proposition %\ref{moduli} and 
\ref{classificationII}, we have the following: For any given triple of points 
$(e,f),(e',f'), (e'',f'') \in K_+^j \times K_-^j$,
if $(e,f)=(e',f')=(e'',f'')$, there is exactly one 
$w \in \mathcal{M}(W_4^j,W_2^j, W_0^j, J_\bbC; D(T^*L_2^j))$ with those points at its vertices;
otherwise there is no such $w$ with those points at the vertices. This means that for any $x \in Crit(h_{42}^j), y \in Crit(h_{20}^j),
z \in Crit(h_{40}^j)$, we have 
$$\mu_2^j(x,y) = \Sigma_z \# [U(x) \cap U(y) \cap S(z)] z$$
where $\# [U(x) \cap U(y) \cap S(z)]$ is the number of zero-dimensional components of 
$U(x) \cap U(y) \cap S(z)$ counted modulo 2. 
%Note that in order for the expected dimension of the moduli space to be zero, we must have $2-\index(x) + 2-\index(y) = 2-\index(z)$. 
Referring to figure \ref{figureMorseSmaledata} we see that
\begin{gather}\label{mu_2^j}
\mu_2^j([x_2^j], [y_2^j]) = [z_2^j],\phantom{bbb} \mu_2^j([x_1^j], [y_2^j]) = [z_1^j], 
\notag \\ 
 \mu_2^j([x_2^j], [(y_1^j)']) = [(z_1^j)'],\phantom{bbb} \mu_2^j([x_1^j], [(y_1^j)']) =[z_0^j] 
.\notag
\end{gather}
In each case there is a unique point of intersection in 
$U(x) \cap U(y) \cap S(z)$ or the intersection is empty.
\\
\newline We now turn to computing $\mu_2$ in $M$, using $J$. Recall that when we constructed $M_0 \subset M$  
the plumbing map $\tau^j: U_2^j \into U_0^j$ maps $D(\nu^*K_-^j)$ onto a 
neighborhood of $K_j$ in $L_0$. In particular,
$\Sigma_{40}^j$ can be viewed either as 
$\partial D_s(\nu^*K_-^j)$ or as $\partial \phi_j(S^1\times D^2_s)$ for some
 $0 < s < r$. We have two sets of Morse-Smale data on 
$\Sigma_{40}^j$: $(h_{40}^j, g_{40}^j)$ and 
$(\widetilde h_{40}^j,\widetilde  g_{40}^j)$; the first is compatible with 
$\Sigma_{40}^j = \partial D_s(\nu^*K_-^j)$, and the second with
$\Sigma_{40}^j = \partial \phi_j(S^1\times D^2_s)$.
We will use the notation $\widetilde \Sigma_{40}^j$ when we are thinking of
$\partial D_s(\nu^*K_-^j)$ equipped  $(h_{40}^j, g_{40}^j)$, and
we will use $\Sigma_{40}^j$ when we are thinking of 
 $\partial \phi_j(S^1\times D^2_s)$ equipped with $(\widetilde h_{40}^j,\widetilde  g_{40}^j)$.
\\
\newline We express $\mu_2$  
as the composite of the following maps, suppressing Morse-Smale data except where necessary. The reason $\mu_2$ can be expressed in terms of $\mu_2^j$ is of course
Proposition \ref{moduli}, which says that the underlying moduli spaces of 
holomorphic triangles  are exactly the same.
\begin{gather}
\begin{CD}
HF(\widetilde V_4, \widetilde V_2^j) \otimes HF(\widetilde V_2^j, \widetilde V_0) \label{mu_2}\\ % \text{ in $M$}\\
@V{\cong}VV \\
HF(W_4^j,W_2^j) \otimes HF(W_2^j, W_0^j) \\ % \text{ in $D_r^{g_2}(T^*L_2^j)$}\\
@VV{\mu_2^j}V \\
HF(W_4^j, W_0^j)  \\%\text{ in $D_r^{g_2}(T^*L_2^j)$}\\
@V{\cong}VV  \\
HF(W_4^j \cap U_2^j, W_0^j \cap U_2^j, (h_{40}^j, g_{40}^j))  \text{ in $U_2^j$} \\
@V{\cong}V{\tau^j_*}V \\ 
HF(\tau^j(W_4^j \cap U_2^j), \tau^j(W_0^j \cap U_2^j), \tau^j_*(h_{40}^j, g_{40}^j))\text{ in $U_0^j$} \\
@V{\cong}V{\phi}V \\ 
HF(\tau^j(W_4^j \cap U_2^j), \tau^j(W_0^j \cap U_2^j), (\widetilde h_{40}^j, 
\widetilde g_{40}^j)) \text{ in $U_0^j$}\\
@V{\cong}VV \\
HF_{loc}(\widetilde V_4,\widetilde V_0; \widetilde   \Sigma_{40}^j) \\
@VV{\iota}V\\
HF(\widetilde V_4, \widetilde V_0) % \text{ in $M$}
\end{CD}
\end{gather}
%$\mu_2$ is equal to this composite essentially because  $\mu_2^j$ is defined by exactly the same moduli space (see Proposition \ref{moduli}). 
Here: 
%$\widetilde W_4^j = V_4 \cap U_2^j$, $\widetilde W_0^j = V_0 \cap U_2^j$; 
$\tau^j_*$ denotes the tautological isomorphism obtained by pushing 
forward all data 
from $U_2^j$ to $U_0^j$; $\phi$ is the continuation map in Morse homology
 induced by choosing any homotopy of Morse-Smale data on $\widetilde \Sigma_{40}^j$ from
$\tau^j_*(h_{40}^j,g_{40}^j)$ \text{ to } $(\widetilde h_{40}^j,\widetilde g_{40}^j);$ 
$HF_{loc}(\widetilde V_4, \widetilde V_0; \widetilde \Sigma_{40}^j)$ denotes the local Floer homology in $M$
defined by restricting attention to a small neighborhood of the component 
$\widetilde \Sigma_{40}^j$; % = \tau^j(\Sigma_{40}^j)$;
$\iota$ is the map induced by the inclusion of the subcomplex
$CF_{loc}(\widetilde V_4, \widetilde V_0;\widetilde  \Sigma_{40}^j) \into CF(\widetilde V_4, \widetilde V_0).$
\\
\newline In our case,  $HF_{loc}(\widetilde V_4,\widetilde  V_0;\widetilde  \Sigma_{40}^j)$ is simply 
the Morse homology $H_*(\widetilde \Sigma_{40}^j, (\widetilde h_{40}^j,\widetilde g_{40}^j))$ and 
$\iota([\tilde z])  = [\tilde z]$
for $\tilde z =  \tilde z_0^j,\tilde z_1^j, (\tilde z_1^j)', \tilde z_2^j$.
The Floer homology groups involved in $\phi$ and $\tau^j_*$ are by definition equal to the corresponding Morse homologies. After making these identifications, that part of the diagram becomes
\[
\begin{CD}
H_*(\Sigma_{40}^j, (h_{40}^j, g_{40}^j)) \\
@VV{\tau^j_*}V \\
H_*(\widetilde \Sigma_{40}^j, \tau^j_*(h_{40}^j, g_{40}^j)) \\
@VV{\phi}V \\
H_*(\widetilde \Sigma_{40}^j, (\widetilde h_{40}^j,\widetilde g_{40}^j)).
\end{CD}
\]
In terms of singular homology, the composite $\phi \circ \tau^j_*$ corresponds to the map
$\tau^j_*: H_*(\Sigma_{40}^j) \into H_*(\widetilde \Sigma_{40}^j)$.
Recall
$\Sigma_{40}^j = S(\nu^*K_-^j) \cong S^1 \times \partial D^2$ and
$\widetilde \Sigma_{40}^j = \phi_j(S^1 \times \partial D^2)$; and $\tau^j|S(\nu^*K_-^j)$ by definition is 
just $\phi_j : S(\nu^*K_-^j) \into L_0$. Thus, intuitively, $\tau^j_*$ is just the map $(\phi_j)_*$ 
as in \S \ref{FukFlowsketch} (satisfying the formulas (\ref{phi_jformula} in  \S  \ref{sectionFlow}).
Thus we expect the map  $\phi \circ \tau^j_*$ on Morse homology to be given by the corresponding 
formulas:
\begin{gather}\label{phicirctau}
(\phi \circ \tau^j_*)([z_1^j]) = [\tilde z_1^j] + m_j [(\tilde z_1^j)'],\\ \notag
(\phi \circ \tau^j_*)([(z_1^j)']) = [(\tilde z_1^j)'],\\  \notag
(\phi \circ \tau^j_*)[z_2^j] = [\tilde z_2^j],\, (\phi \circ \tau^j_*)[z_0^j] = [\tilde z_0^j]. \notag
\end{gather}
To see this, think of $ \Sigma_{40}^j$ and $\widetilde \Sigma_{40}^j$ as identified with
 $\bbR/\bbZ \times \bbR /\bbZ$, in such a way that $(h_{40}^j, g_{40}^j)$ and  $(\widetilde h_{40}^j,\widetilde g_{40}^j)$ both have unstable manifolds which are either horizontal or vertical. 
We take $\tau^j$ to be the map $(x,y) \mapsto (x+ m_j y, y)$ (which is correct up to isotopy), so that 
$\tau^j_*(h_{40}^j, g_{40}^j)$ has unstable manifolds which are either horizontal or 
have slope $m_j$. We take the homotopy of Morse-Smale data induced by 
$h_s(x,y) = (x+ sm_j y, y)$, $s \in [0,1]$. The continuation map $\phi$ counts 
(possibly constant) $s$-dependent flow lines which interpolate between the horizontal/vertical unstable manifolds and the horizontal/slope $m_j$ unstable manifolds; the result is that $\phi\circ \tau^j_*$ 
does indeed have the above form (\ref{phicirctau}).
\\
\newline To see that $\mu_2$ satisfies the stated formulas,
 we combine (\ref{phicirctau})  with our formulas 
for $\mu_2^j$, and use the fact that $\iota$ fixes all generators 
$[\tilde z]$. 
%To see that $\mu_2$ satisfies the stated relations, 
%we use the formulas for $\mu_2$ plus the relations in Proposition \ref{generators} (using Proposition \ref{newgenerators}).
\end{proof}

\section{Correspondence  holomorphic strips and gradient lines}\label{localstrips}

Assume $V_4$, $V_2^j$, $V_0$ are given as in \S \ref{bigsectionvanishingcycles}
and we have the corresponding functions  $f_{40}$, $f_{42}^j$, $f_{20}^j$ as in 
\S \ref{graphs}.  In this section we 
explain how to make  these functions small enough that we have a 1-1 correspondence between certain holomorphic strips in $M$ 
and the gradient lines of $f_{40}$, $f_{42}^j$, $f_{20}^j$  in $N_{40}$, 
$N_{42}^j$, and $N_{20}^j$, as stated in lemma \ref{strips=gradientlines}; this corresponds to Proposition \ref{correspondence} below in the case $(V_4, V_0)$; the other cases $(V_2^j,V_0)$ and $(V_4,V_2^j)$ have corresponding versions of  Proposition \ref{correspondence}  which 
are proved in exactly the same way. 
\\
\newline The main issue we have to worry about is that $N_{40}$, 
$N_{42}^j$, and $N_{20}^j$ have boundary. First we recall the standard 
correspondence for closed manifolds:
 See \cite[pages~84--87] {Poz} or \cite{FloerMorse}. Given a Riemannian metric $g$ on $L$ there is a natural corresponding $J_g \in \mathcal{J}(T^*L)$, 
which  exchanges the Horizontal and vertical subspaces of 
the Levi-Civita connection of $g$.
%Our sign convention is: if we identify $T^*\bbR^n \cong \bbC^n$ by 
%$ydx \in T^*\bbR^n$ $\mapsto$ $x+iy\in \bbC^n$, then  $J_{g_{\bbR^n}} \cong m_i$.

\begin{thm}[Floer]\label{Floer} Let $L$ be a smooth closed manifold %and consider $T^*L$ with its standard symplectic structure. 
and let $\pi: T^*L \into L$ be the projection. 
Let $g$ be a Riemannian metric on $L$. Then, there is a constant 
$A= A(L,g) >0$ such that,
for any  smooth function $f: L \into \bbR$,
if $||\nabla_g f||_\infty = max\{|\nabla_g f |_g  \}< A$ then 
 %$\mathcal M(L,L,-(f\circ \pi), -J_g) \cong 
$$\mathcal M(L, L, f\circ \pi, J_g)$$ 
is identical with 
$$\{\gamma \in C^{\infty}(\bbR, L) : \gamma'(s) = -(\nabla_g f)(\gamma(s))\}.$$
Namely, if $\gamma$ lies in second space then
$u(s,t) = \gamma(s)$
lies in  the first space and every $u$ in that space is of this form.
\end{thm}

The first thing we do is replace $f_{40}$, $f_{42}^j$, $f_{20}^j$ by 
 $\epsilon f_{40}$, $\epsilon f_{42}^j$, $\epsilon f_{20}^j$ for some small
 $\epsilon >0$ (but we keep the same notation)
so  that  
$$||\nabla_{\mu_{40}} f_{40}||_\infty < A(L_0,\mu_{40}), 
||\nabla_{\mu_{42}^j} f_{42}^j||_\infty < A(L_2^j,\mu_{42}^j), %\text{ and }
||\nabla_{\mu_{20}^j} f_{20}^j||_\infty < A(L_2^j,\mu_{20}^j).$$
On the left hand-sides, the functions are defined on  
$N_{40}$  $N_{42}^j$  $N_{20}^j$, but on the right hand side $\mu_{40}$
$ \mu_{42}^j$, $ \mu_{42}^j$ are the metrics from lemmas \ref{MorsehomologyI} 
and \ref{MorsehomologyII}, extended to the rest of $L_2^j$ and $L_0$ arbitrarily.
We patch the graphs $\Gamma(\epsilon df_{40})$, $\Gamma(\epsilon df_{42}^j)$,
$\Gamma(\epsilon df_{20}^j)$, into $V_4$, $V_2^j$, $V_0$ using some cut-off 
function; this does not alter their basic shape and we keep the same notation.
\\
\newline We will now focus our attention on $(V_4,V_0)$. The other cases 
 $(V_2^j,V_0)$, and $(V_4, V_2^j)$ can be treated in the same way.
First, fix some smooth extension $f_{40}^*: L_0 \into \bbR$ still satisfying  
$||\nabla_{\mu_{40}} f_{40}^*||_\infty < A(L_0,\mu_{40})$.
Denote the components of  $\partial N_{40}$ by $C_{40}^j$, $j=1,\ldots, k$, 
where $C_{40}^j$ is a torus adjacent to  $\Sigma_{40}^j$.
Let  $$\psi_j: T^2 \times [0,1] \into N_{40},$$  
$j=1,\ldots, k$, be a parameterization of a tubular neighborhood of $C_{40}^j$,
with   $\psi_j(T^2 \times \{1\}) = C_{40}^j$. Set 
$$\widehat N_{40} = N_{40} \setminus (\cup_j \psi_j(T^2 \times (0,1]))$$ 
and assume $\psi_j(T^2 \times [0,1]))$ is small enough that
$\Sigma_{40}^j \subset \widehat N_{40}$  for every $j$.
We now define some functions $f_{40}^n, \widehat f_{40}: L_0 \into \bbR$, 
$n \geq 1$, with $f_{40}^n \into \widehat f_{40}$ in $C^\infty$,
see figure \ref{figuretildef_40} for a schemetic near  
$\psi_j(T^2 \times [0,1]))$.
Define $\widehat f_{40}$ so that 
$$\widehat f_{40}|\widehat N_{40} =0, \phantom{bbb} \widehat f_{40}|
(L_0 \setminus N_{40}) =   f_{40}^*|(L_0 \setminus N_{40}),  $$
%to be equal to zero on  $\widehat N_{40}$, equal to $f_{40}^*$ on $L_0 \setminus N_{40}$, 
and on each intermediate region
$\psi_j(T^2 \times [0,1])$, 
set $$\widehat f_{40}(\psi_j(p,t)) = \varphi(t) f_{40}^*(\psi_j(p,t)),$$ where 
$\varphi$ is  smooth and monotone, with $\varphi(0) = 0$, 
$\varphi|{[1/2,1]} = 1$, so that   $\widehat f_{40} = f_{40}^*$ on 
$\psi_j(T^2 \times [1/2,1])$. Now define
$f_{40}^n : L_0 \into \bbR$ so that 
$$f_{40}^n|\widehat N_{40} = \frac{1}{n}f_{40}|\widehat N_{40},  \phantom{bbb}
f_{40}^n|(L_0 \setminus N_{40}) = f_{40}^*|(L_0 \setminus N_{40}), \text{ and }$$
$$  f_{40}^n(\psi_j(p,t)) = (1-\varphi(t))\frac{1}{n}f_{40}(\psi_j(p,0)) 
+  \varphi(t) f_{40}^*(\psi_j(p,t)),$$
so that $ f_{40}^n = f_{40}^*$ on $\psi_j(T^2 \times [1,1/2])$.
%See figure \ref{figuretildef_40}.
\begin{figure}
\begin{center}
\includegraphics[width=4in]{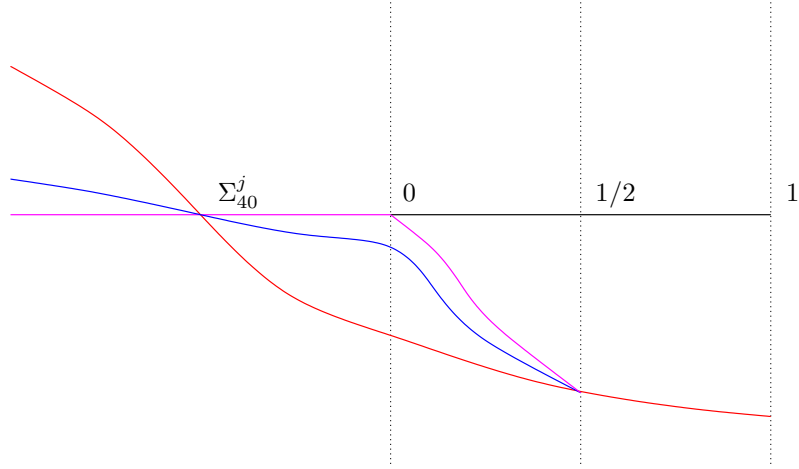} 
\put(-210,100){$\Sigma_{40}^j$}
\put(-140,100){$0$}
\put(-67,100){$1/2$}
\put(5,100){$1$}
%\put(-200,100){$(0,0)$}
\caption{Schematic of $f_{40}$ (red), $\widehat f_{40}$ (purple), $f_{40}^n$ (blue).}
\label{figuretildef_40}
\end{center}
\end{figure}
Now take smooth functions 
$$H_n, G_n: M \into \bbR \text{ so that } \phi_1^{H_n}(\Gamma(\widehat f_{40}))
= \Gamma(df_{40}^n), \text{ and }
\phi_1^{G_n}(L_0) = \Gamma(df_{40}^n).$$

We may assume $H_n \into 0$ in the $C^\infty$ norm,
 since $f_{40}^n \into \widehat f_{40}$. Similarly we may assume
$G_n \into G$ such that $\phi_1^G(L_0) = \Gamma(d \widehat f_{40})$.
To be more definite, assume that 
in coordinates $(x,y) \in T^*L_0$ 
in some small Weinstein neighborhood of $L_0$
we have $H_n(x,y) = -(f_{40}^*(x) -f_{40}^n(x))$ and
$G_n(x,y) = -f_{40}^n(x)$ (the signs in front are necessary because of our conventions in  \S \ref{conventions}).
\\
\newline Let $\widehat V_4$ and $V_4^n$ be the exact Lagrangian spheres obtained
by replacing the graph $\Gamma(df_{40})$ by $\Gamma(d \widehat f_{40})$ and 
 $\Gamma(df_{40}^n)$ respectively. 
 %The following lemma follows from a simple energy argument, and we will use it to localize our holomorphic strips to $D(T^*N_{40})$.

\begin{lemma} \label{stripcontainment} Take any $J \in \mathcal{J}([0,1],M,I)$. 
Assume $J_n  \in \mathcal{J}([0,1],M,I)$ satisfies $J_n \into J$ in the $C^\infty$ sense. 
Then for any $\epsilon >0$ there is an $N$ such that for $n \geq N$,  every 
$$u \in \mathcal M(\widehat V_4, V_0, H_n, J_n; M)$$
%$u: \bbR \times [0,1] \into M$ satisfying
%$$u(\{0\}\times\bbR) \subset \widehat V_4, u(\{1\}\times \bbR) \subset V_0$$
%$$\overline{\partial}_{(H_n, J_{40})}(u) =0,$$
%$$\int_\bbR \int_0^1 \omega( \partial_s u J_{40} \partial_s u) dt ds < \infty$$
has $Image (u) \subset \mathcal{N_\epsilon}(\widehat N_{40})= 
\{p \in M : dist(p,\widehat N_{40}) < \epsilon\} $.   % D_{\widehat \eta_{40}}(T^*N_{40})$.
\end{lemma}

\begin{proof} %Intuitively, as $n \into \infty$ a sequence $u_n \in \mathcal M(\widehat V_4, V_0, H_n, J; M)$ will converge to a point in $\widehat V_4 \cap V_0 = \widehat N_{40}$. 
%Hence, for large $n$, every solution should lie in a small neighborhood of $\widehat N_{40}$. This is the main idea, which we will now make precise.
%\\
%\newline 
Since $\phi_{H_n}(\widehat V_4)$ has Morse-Bott intersection with $V_0$ it 
follows from  condition \ref{convergence} that every $u \in \mathcal M(\widehat V_4, V_0, H_n, J_n; M)$ satisfies
$$\lim_{s \into \infty} u(s, \cdot) = y_0, \phantom{bb}\lim_{s \into -\infty} u(s, \cdot) = y_1$$
for some $y_0, y_1 \in C^\infty([0,1], M)$ such that, for $j=0,1$,
\begin{gather*}
y_j'(t) = X_{H_n}(y_j(t)), \, t\in [0,1], \phantom{bbb}
y_j(0) \in \widehat V_4, \, y_j(1) \in V_0.
\end{gather*} 
Such $y_j$ are in one-one correspondence with points of 
$$\phi_{H_n}(\widehat V_4) \cap V_0 = Crit(\frac{1}{n}f_{40})= Crit(f_{40}).$$ 
Since $dH_n(p) = 0$ for any $p \in Crit(\frac{1}{n}f_{40})$, the $y_j$ are constant.
\\
\newline Now suppose for a contradiction that the lemma is false. Then, there is some $\epsilon>0$ and a sequence $u_n \in  \mathcal M(\widehat V_4, V_0, H_n, J_n; M)$ such that for every $n$ the image of $u_n$ contains a point $x_n \notin \mathcal{N_\epsilon}(\widehat N_{40})$. 
Passing to subsequences, and reparameterizing each $u_n$ by a suitable $s_n \in \bbR$, we assume: $x_n \rightarrow x_0 \notin \mathcal{N_\epsilon}(\widehat N_{40})$; $u_n(t_n,0) = x_n$, for some $t_n \in [0,1]$, $t_n \rightarrow t_0$;  there are fixed critical manifolds $C_0, C_1$ of $\frac{1}{n}f_{40}$ 
such that for every $n$, $\lim_{s \rightarrow \infty} u_n(s, t) = y_0^n$, $\lim_{s \rightarrow -\infty} u(s,t) = y_1^n$, where $y_0^n \in C_0$, $y_1^n \in C_1$
$y_0^n \rightarrow y_0$, $y_1^n \rightarrow y_1$. 
\\
\newline Because $M$ and the Lagrangian submanifolds are exact, $\int_{\bbR \times [0,1]} u_n^*\omega$ is uniformly bounded, and there is no bubbling. Hence, by  Gromov compactness, there is a convergent subsequence $u_n \into u$ where %$u$ satisfies
\begin{gather}
\partial_su + J_t \partial_t u = 0, \label{Jholo}
\end{gather}
since $H_n \into 0$. We also have
$\lim_{s \rightarrow -\infty} u(s,t) = y_0$, $\lim_{s \rightarrow \infty} u(s, t) = y_1$, $u(t_0, 0) = x_0$.
Let $f_{\widehat V_4}\in C^\infty(\widehat V_4)$, $f_{V_0}\in C^\infty(V_0)$ be such that $\theta|\widehat V_4 = df_{\widehat V_4}$, $\theta|V_0 = df_{V_0}$.
We can assume $f_{\widehat V_4}$, $f_{V_0}$ to be equal to 0 on
$Crit(\frac{1}{n}f_{40})$, since $\widehat V_4 \cap \widehat N_{40} =\widehat N_{40}$
and $V_0 \cap  \widehat N_{40} =\widehat N_{40}$. This together with $H_n \into 0$
shows that $\int_{\bbR \times [0,1]} u_n^*\omega = A(y_1^n) -A(y_0^n) \into 0$.
It follows that $\int_{\bbR \times [0,1]} u^*\omega =0$, and so by (\ref{Jholo}), $u$ must be constant, and equal to  $u(t_0,0) = x_0 \notin \mathcal{N_\epsilon}(\widehat N_{40})$. On the other hand, since
$u_n(t_0, 0) \in \widehat V_4$, and $u_n(t_0, 1)\in V_0$ both converge to $x_0$, it follows that $x_0 \in \widehat V_4 \cap V_0 =  \widehat N_{40}$, contradiction.
\end{proof}

Note that for $n$ sufficiently large  we still have
$||\nabla_{\mu_{40}} f_{40}^n||_\infty < A(L_0,\mu_{40})$ (from the corresponding inequality for $f_{40}^*$).

\begin{prop}\label{correspondence}
For each $n$ there exist %$J_{20}^n$,  $J_{42}^n$, 
$J_{40}^n \in \mathcal{J}([0,1],M,I)$
such that $J_{40}^n$ converges to some  
$J_{40}^\infty \in \mathcal{J}([0,1],M,I)$ in the $C^\infty$ sense, and
%and similarly   $J_{42}^n \into J_{42}$, $J_{20}^n \into J_{20}$ 
$J_{40}^n$ is such that for all $n$ sufficiently large we have for each $p,q \in Crit(f_{40}) = Crit(\frac{1}{n}f_{40})$, a 1-1 correspondence between 
$$\{ \gamma\in C^{\infty}(\bbR, N_{40}) :\gamma'(s) = ( \frac{1}{n} 
\nabla_{\mu_{40}}f_{40})(\gamma(s)), \gamma(-\infty) = p, \gamma(\infty) = q \}$$
and
$\mathcal{M}(V_4^n, V_0, J_{40}^n,p,q; M).$
Namely the correspondence sends $\gamma$ to $u$, where
$u(s,t) = \phi_t(\gamma(s)),$  %\phi_t(\gamma\big(\frac{1}{n} s\big)).$$
and $\phi_t$, $t \in [0,1]$, is a certain  exact isotopy which fixes 
$Crit(f_{40})$ point-wise. Finally, $J_{40}^n$ is regular for 
the second moduli space  if and only if the 
first moduli space is regular.
\end{prop}
\begin{proof} Unfortunately Floer's proof does not work for the 
the case of a manifold with boundary. Instead, the idea is to appeal to Floer's theorem for $L= L_0, f = - f_{40}^n$, and then deduce from lemma \ref{stripcontainment} that, for large $n$, the corresponding result is 
 true for $N_{40}$ and $f_{40}^n|N_{40} : N_{40} \into \bbR$. 
(Since $f = - f_{40}^n$ we get a correspondence with \emph{positive} flow lines of $f_{40}^n$.)
\\
\newline Let $J_{\mu_{40}}^*$ denote any extension of $J_{\mu_{40}}$ from a Weinstein neighborhood $D(T^*L_0)$ to $M$ which is equal to $I$ near $\partial M$
(where $I$ makes $\partial M$ convex).
%Now let $J^n \in \mathcal{J}([0,1],M,I)$ be such that, for all $t \in [0,1]$,
Now set
$$J_t^n = (\phi^{H_n}_t)_*(\phi^{H_n}_1)^*(\phi^{G_n}_1)_* (\phi^{G_n}_t)^*(J_{\mu_{40}}^*).$$
Then $J_t^n \into (\phi^{G}_1)_* (\phi^{G}_t)^*(J_{\mu_{40}}^*)$ and
we may assume $G_n$ is zero near $\partial M$, so that
$J^n_t =I$ near $\partial M$. % \mathcal{J}([0,1],M,I)$. 
Note that $$(\phi^{H_n}_1)_*(\phi^{H_n}_t)^*J_t^n = 
(\phi^{G_n}_1)_* (\phi^{G_n}_t)^*(J_{\mu_{40}}) \text{ on } D(T^*N_{40}).$$
There are isomorphisms (see (\ref{isomorphism}) in \S  \ref{conventions}),
$$\mathcal M(\widehat V_4, V_0, H_n, \{J_t^n\} , M) \cong \mathcal M(\widehat V_4, (\phi^{H_n}_1)^{-1}(V_0), \{(\phi^{H_n}_t)^*J_t^n\}, M)$$
$$\cong \mathcal M(\phi^{H_n}_1 (\widehat V_4), V_0,\{ (\phi^{H_n}_1)_*(\phi^{H_n}_t)^*J_t^n\}, M),$$
%$$= \mathcal{M}(\phi^{H_n}_1 (\widehat V_4) , V_0, (\phi^{G_n}_1)_*(\phi^{G_n}_t)^*(-J_{g_0}), M)$$ 
where the isomorphisms are respectively 
$$u \mapsto \widetilde u, \text{ }  
\widetilde  u(s,t) = (\phi_t^{H_n})^{-1}(u(s,t)),\text{ and }\widetilde u \mapsto \phi_1^{H_n}(\widetilde u).$$
Now set 
$$J_{40}^n = (\phi^{H_n}_1)_*(\phi^{H_n}_t)^*J_t^n.$$
We will now prove this satisfies conditions of the lemma (keeping in mind  
$\phi^{H_n}_1 (\widehat V_4) = V_4^n$).
Let $\epsilon >0$ be such that  
$\mathcal{N}_{2\epsilon}(\widehat N_{40}) \subset D(T^*N_{40})$.
Lemma \ref{stripcontainment} implies that, for large $n$, each
$$u \in \mathcal M(\widehat V_4, V_0, H_n,\{ J^n_t\} , M)$$ 
is contained in $\mathcal N_{\epsilon}(\widehat N_{40})$. 
Since $H_n \into 0$ in $C^1$ we have, for large $n$, 
$$(\phi_1^{H_n}\circ (\phi_t^{H_n})^{-1})(\mathcal N_{\epsilon}(\widehat N_{40})) \subset \mathcal N_{2\epsilon}(\widehat N_{40}),$$
for all $t$. Each  
$$u \in \mathcal{M}(\phi^{H_n}_1 (\widehat V_4), V_0, 
\{(\phi^{H_n}_1)_*(\phi^{H_n}_t)^*J_t^n\}, M)$$ 
therefore lies in $\mathcal N_{2 \epsilon}(\widehat N_{40}) 
\subset D(T^*N_{40})$ for large $n$. 
Therefore it follows that each such $u$ satisfies 
$$u \in \mathcal{M}(\phi_1^{G_n}(L_0), L_0, (\phi_1^{G_n})_*(\phi_t^{G_n})^*(-J_{g_0}),  T^*L_0).$$
This is because $u$ lies in $D(T^*N_{40})$, and by construction the Lagrangians and complex structures are are equal in  $D(T^*N_{40})$.
Now Floer's Theorem \ref{Floer} applied to 
$$\mathcal{M}(L_0, L_0, G_n, J_{\mu_{40}}, T^*L_0)$$
implies that  $u$ is of the form 
$u(s,t) =  ((\phi_1^{G_n}) \circ (\phi_t^{G_n})^{-1})(\gamma(s))$
for some $\gamma \in C^{\infty}(\bbR, L_0)$ such that
$\gamma'(s) = (\nabla_{\mu_{40}} f_{40}^n )(\gamma(s))$
(recall $G_n(x,y) = -f_{40}^n(x)$).
Since $u$ lies in $D(T^*N_{40})$, we have
$\lim_{s \into \pm  \infty}u(s,t) = x_{\pm} \in N_{40}$. 
Therefore, since $\phi_t^{G_n}(x) = x$ for all $x \in Crit(f_{40}^n)$, 
$\gamma(\pm \infty) = x_{\pm} \in N_{40}.$
But the only gradient lines of $f_{40}^n$ joining points in $N_{40}$ 
are contained in $\widehat N_{40}$. Therefore, since
$f_{40}^n|\widehat N_{40}=  \frac{1}{n} f_{40}|\widehat N_{40}$
$\gamma$ in fact satisfies $\gamma \in C^{\infty}(\bbR, N_{40})$,
$\gamma'(s) = (\nabla_{g_0}(\frac{1}{n}f_{40})(\gamma(s)).$
\\
\newline To address the statement about regularity we note that
 correspondence in Floer's Theorem \ref{Floer} has a version at the 
linearized level which carries regular data into regular data 
(see  \cite[p. 86]{Poz}).
\end{proof}

One can prove a version of lemmas \ref{stripcontainment} and  
\ref{correspondence} for $(V_2^j,V_0)$ and $(V_4, V_2^j)$ in the same way. 
%One first defines auxiliary functions
%$\widehat f_{42}$, $\widehat f_{20}$ and corresponding $f_n$'s, etc. 
The signs $\gamma'(s) = -\frac{1}{n}(\nabla f_{42})(\gamma(s))$ and 
$\gamma'(s) = -\frac{1}{n}(\nabla f_{20})(\gamma(s))$ in the correspondence are determined in the course of the argument. Alternatively,
if $u \in \mathcal{M}(L_0,L_1,J)$, where $(L_0,L_1,J) =$ $(V_4,V_0,J_{40})$ or $(V_4,V_2^j,J_{42})$ or $(V_2^j,V_0,J_{20}^j)$, then 
the corresponding action functionals satisfy  
$$A_{20}(u(s,\cdot))= f_{20}^j(\pi(u(s,1))),\phantom{bbb}
A_{42}(u(s,\cdot))= -f_{42}^j(\pi(u(s,0))),$$
$$A_{40}(u(s,\cdot))= -f_{40}^j(\pi(u(s,0))),$$
and $\partial_s A(u(s,\cdot))< 0$ in all cases. Here $\pi: T^*N \into N$ is the projection for $N= N_{42}, N_{20},$ or  $N_{40}$. The term  in $A$ of the form $\int y^*\theta$ always vanishes because all  $u$ are such that, for any fixed $s$, $y(t) = u(s,t)$ lies in a fixed cotangent fiber for all $t$.
\section{Classification of holomorphic triangles in $D(T^*L_2^j)$} 
%in $D(T^*V_2^j)$}
\label{localtriangles} 

In this section we prove Proposition \ref{classificationII}, which corresponds here to Propositions
\ref{existence} (existence), \ref{uniqueness} (uniqueness), \ref{regular} (regularity), %which are treated  respectively 
in sections \ref{sectionexistence}, \ref{sectionuniqueness}, 
\ref{sectionregular}. (The arguments of this section are sketched in 
\S \ref{firsttriangle} and \S \ref{VCsketch}.)
\\
\newline Take $Z = \{ (z_1, \ldots, z_4) \in \bbC^4 : \Sigma z_j^2 =1\}$. 
This is exact symplectomorphic to $T^*S^3=\{(u,v) \in \bbR^4 \times \bbR^4: |u|=1, u\cdot v=0\}$, via 
$$\mu:Z \into  T^*S^3, \phantom{bb} \mu(x +iy) = (x/|x|, -|x| y), 
\label{mu} \text{ where } \mu^*( \Sigma_j -v_j du_j) = \Sigma_j \, y_j dx_j.$$
$Z$ has a complex structure $J_Z$, and we equip 
$T^*S^3$ with 
\begin{gather}
J_{\bbC} = \mu^*(J_Z).\label{J_{bbC}(b)}  
\end{gather}
The formula for $\mu^{-1}$ is 
\begin{equation}
\mu^{-1}(u,v)= f(|v|) u - i f(|v|)^{-1} v, \phantom{bbb} f(s) = \sqrt{  \frac{1+ \sqrt{1+4s^2 }    } {2}   }.
\label{muinverse}
\end{equation}
(This formula for $|x| = f(|v|)$ is obtained by combining the equations 
$|x|^2 -|y|^2 =1$ and $|x||y| = |v|$.)
%Note that  $f(s)$ is real analytic for all $s \in \bbR$.
\\
\newline 
As before we set $W_4^j = \widetilde V_4 \cap D(T^*L_2^j)$, 
$W_2^j = \widetilde V_2^j \cap D(T^*L_2^j)$,
$W_0^j = \widetilde V_0 \cap D(T^*L_2^j)$.  
For convenience of notation, we fix $j$ 
and consider $D(T^*L_2^j) = D(T^*S^3)$; we also drop the
$j$ from $W_4^j$, $W_2^j$, $W_0^j$, and from $K_{\pm}^j$.
Any 
$w \in \mathcal{M}(W_4, W_2, W_0, J_\bbC; D(T^*S^3))$
has a continuous extension  to $D^2$, by condition \ref{convergence},
and we denote it by $\overline w: D^2 \into D(T^*S^3).$

\subsection{Existence of holomorphic triangles}\label{sectionexistence}

For each $e \in K_+$, $f \in K_-$, let $Y_{ef} \subset Z$ 
denote the complex submanifold 
$Y_{ef} = (\bbC e \oplus \bbC f) \cap Z \subset \bbC^4.$ 

\begin{lemma}\label{biholo}$\mu(Y_{ef}) = T^*K_{ef}$, and
there is a biholomorphism $\rho_{ef} : \bbC^{\times} \into Y_{ef}.$
%and $\mu(Y_{ef}) = T^*K_{ef}$. 
\end{lemma}
\begin{proof} To see $\mu(Y_{ef}) = T^*K_{ef}$ we check that 
$$\mu^{-1}(T^*K_{ef}) \subset \{(z_1 e, z_2 f ) : z_1^2 + z_2^2 =1\}=Y_{ef},$$ 
using (\ref{muinverse}), and  
$$\mu(Y_{ef}) \subset 
[(\bbR e \oplus \bbR f) \oplus (\bbR e \oplus \bbR f)] \cap T^*S^3 =T^*K_{ef}.$$%Writing $Y_{ef} = \{(z_1 e, z_2 f ) : z_1^2 + z_2^2 =1\}$, 
We define
$\rho_{ef} : \bbC^{\times} \into Y_{ef}$
by $\rho_{ef}(\zeta) = (z_1 e, z_2 f)$, where $z_1 = \zeta + \zeta^{-1}$, $z_2 = -i(\zeta - \zeta^{-1})$.
Here, $z_1^2 + z_2^2 = (z_1 +iz_2)(z_1-iz_2) = 1,$ and $\zeta = z_1 +iz_2$.
\end{proof}
We will need $\rho_{ef}$ because $\sigma_{ef}: (\bbR /2\pi\bbZ) \times \bbR \into T^*K_{ef}$ is  unfortunately  not holomorphic. We have the 
following refinement of the Riemann mapping theorem
from \cite{Wen}, Chapter 2, Theorems 2.1, 2.2, 2.4.  

\begin{lemma}\label{Riemannmap}
Assume $\Omega \subset \bbC$ is homeomorphic to a closed disk,
%be compact and simply connected, and assume it is the closure of a simply connected open set in $\bbC$.
and $\partial \Omega$ is parametrized by a continuous, injective curve 
$\gamma: \bbR /2 \pi \bbZ \into \partial \Omega$, 
which is piece-wise real analytic. Let $t_0, t_1, t_2 \in \bbR /2 \pi \bbZ$ 
denote three distinct points
such that  $\gamma|[ (\bbR /2 \pi \bbZ) \setminus\{t_0, t_1, t_2\} ]$ 
is real analytic. Let $z_i= \gamma(t_i)$, and 
assume $z_0,z_1, z_2$  are labeled in counter-clockwise order. 
Let $l_0 = \gamma(t_0, t_1), l_1 = \gamma(t_1, t_2) , l_2 =\gamma(t_2, t_0)$. 
Then there is a unique biholomorphic map 
$$w: V \into \Omega \setminus\{z_0, z_1, z_2\}$$
such that 
$$w(I_i) = l_i, \text{ and } \lim_{\zeta \rightarrow \zeta_i} w(\zeta) = z_i, \, i=0, 1,2.\, \square$$  
\end{lemma} 

We will need the following stronger version of lemma \ref{Riemannmap} for the uniqueness part
%of lemma \ref{classificationII},
because our holomorphic triangles $w: V \into \Omega \setminus\{z_0, z_1, z_2\}$ would  not be assumed to be injective, nor surjective on the boundary.

\begin{lemma} \label{identity} 
Take $\Omega$ as in lemma \ref{Riemannmap}. 
Suppose  $w: V \into \Omega$ is holomorphic and 
$$\lim_{\zeta \rightarrow \zeta_i} w(\zeta) = z_i, \, i =0,1,2.$$
Taking $i, i+1$ modulo 3, assume that
$w(I_i) \subset  \gamma(t_i, t_{i+1}) \text{, } i =0,1,2.$
Then $w:V \into \Omega \setminus \{z_0, z_1, z_2\}$
is a biholomorphism, and it coincides with the map in lemma \ref{Riemannmap}.
\end{lemma}

\begin{proof} Let $\overline w : D^2 \into \Omega$ denote the continuous extension of $w$. Let $b: \bbR /2\pi \bbZ \into \partial D^2$ be 
$b(\theta) = e^{i\theta}$. 
Let $J_i \subset  \bbR /2\pi \bbZ$, $i=0,1,2$ denote the open subintervals 
 such that 
$b(J_i) = I_i \subset V$. Then $(\overline w \circ b) | \overline J_i$ is a continuous curve 
connecting the point $z_{i}$ to $z_{i+1}$, where we take $i, i+1$ modulo 3. Since $w(I_i) \subset  \gamma(t_i, t_{i+1})$, we conclude $(w \circ b) ( \overline Ji) = \gamma([t_i, t_{i+1}]).$
%(We remark that  $(w \circ \gamma) | \overline J_i$ is not necessarily injective but its image is equal to $\gamma([t_i, t_{i+1}])$, because of the boundary condition $\overline w(\overline I_i)  \gamma([t_i, t_{i+1}]$.)
In total $(\overline w \circ b) : \bbR /2\pi \bbZ \into \partial D^2$ 
winds once around the boundary of $\Omega$ in the counter-clockwise sense, 
since $z_0, z_1, z_2$ are so ordered.
%We rephrase this observation in terms of the argument principle as follows.
So, for any $w_0 \in \Int \Omega$, %since $(w \circ b)'(t)$ exists for all but finitely many $t$,
$$\frac{1}{2\pi i} \int_{0}^{2\pi} 
\frac{(w \circ b)'(t)} {(w \circ b)(t) - w_0}dt = 1,$$
where $(w \circ b)'(t)$ exists for all but finitely many $t$.
%since $w$ assumed to be holomorphic on $\partial V$.)
%On the other hand 
This integral is also the number of zeros of
$w(z) - w_0$ in $\Int D^2$ counted with multiplicity, so we conclude that 
%(Note that while $w(z) -w(z_0)$ is not holomorphic on $b$ at every point,
% one can 
%replace $b$ by $rb$ for some $0<r<1$ and then take the limit as $r \into 1$, 
%appealing to the Lebesgue dominated convergence theorem.)
$$w| \Int D^2 : \Int D^2 \into \Int \Omega$$ is a biholomorphic map. 
Let $\widetilde w$ denote the map from lemma \ref{Riemannmap}. 
Since $(\widetilde w^{-1}\circ w)|\Int D^2$ is a biholomorphism of $\Int(D^2)$ onto itself, 
it must be a linear fractional transformation, which extends to $\bbC$ and fixes
$\zeta_0$, $\zeta_1$ $\zeta_2$. Thus $w = \widetilde w$.
\end{proof}

Let $\Delta\subset \bbR/2\pi\bbZ \times [-r,r]$ denote any one of the  four triangles bounded by $\widetilde \Gamma_0$,  $\widetilde \Gamma_2$,  $\widetilde \Gamma_4$
in figure \ref{CorrectionGammaIIIbig}. Let $\delta_{42}, \delta_{20}, \delta_{40}$ denote the vertices of $\Delta$, where $\delta_{ij} \in\widetilde \Gamma_k \cap\widetilde \Gamma_i$, $k,i =0,2,4$. Set $\Delta^* = \Delta \setminus \{ \delta_{42}, \delta_{20}, \delta_{40} \}.$
\begin{prop} \label{existence}
For each $e \in K_+, f \in K_-$ there is a
$$w=  w_{ef} \in \mathcal{M}(W_4, W_2, W_0, J_\bbC; D(T^*S^3))$$
satisfying 
$$w_{ef}(V) = \sigma_{ef}(\Delta^*) \subset D(T^*K_{ef}),$$
$$\lim_{\zeta \rightarrow \zeta_0} w(\zeta) = \sigma_{ef}(\delta_{40}), \, \, 
\lim_{\zeta \rightarrow \zeta_1} w(\zeta) = \sigma_{ef}(\delta_{20}),
\, \, 
\lim_{\zeta \rightarrow \zeta_2} w(\zeta) = \sigma_{ef}(\delta_{42}).$$
\end{prop}

\begin{proof}
%Let 
%$$\tau^j:  (\bbR /2\pi\bbZ) \times \bbR \into  \bbR \times (\bbR /2\pi\bbZ)$$
%denote the map which swaps the coordinates. Regarding $ \bbR \times (\bbR /2\pi\bbZ)$ as a subset of 
%$\bbC$, there is the natural biholomorphism given by the exponential map
%$exp:  \bbR \times (\bbR /2\pi\bbZ) \into \bbC^{\times}$.
%Now $\tau^j(R)$ has boundary arcs in $\tau^j(\Gamma_0), \tau^j(\Gamma_2), \tau^j( \Gamma_4)$, but in 
%clockwise order. Two of the boundary arcs are straight lines and the thirs involves the function $\psi$ which was required to be real analytic, thus $\tau^j(R)$ has real analytic boundary.
%By lemma \ref{Riemannmap}
%Since our holomorphic triangle is supposed to have 
%$w(I_0) \subset W_4, \,  w(I_1) \subset W_2, \, w(I_0) \subset W_0,$ and 
%$I_0, I_1, I_2$ are in counterclockwise order, this is consistent with 
Let $\Delta' =  \rho_{ef}^{-1}(\mu^{-1}(\sigma_{ef}(\Delta))) \subset \bbC^{\times}.$
%$$\Delta'' = \rho_{ef}^{-1}(\Delta') \subset \bbC^{\times}.$$
%We will apply lemma \ref{Riemannmap} to produce a suitable holomorphic map 
%$\varphi: V \into \Delta'$. Composing this with 
%$\mu \circ \rho_{ef}: \bbC^{\times} \into T^*K_{ef}$ will yield $w_{ef}$.
%\\
%\newline 
First we check that the boundary arcs of $\Delta'$ are real analytic.
A calculation yields
\begin{equation*}
(\rho_{ef}^{-1} \circ \mu^{-1} \circ \sigma_{ef})(\theta, s) = 
(f(s) + \frac{s}{f(s)})e^{i\theta}.
\label{analyticbdry}
\end{equation*}
Recall from formula (\ref{muinverse}) that 
$$f(s) = \sqrt{  \frac{1+ \sqrt{1+4s^2 }    } {2}   } \geq 1.$$
Thus $s\mapsto f(s) + \frac{s}{f(s)}$  is real analytic on all of $\bbR$. 
%Note that $f(s) \geq 1$ so $\frac{s}{f(s)}$ is also real analytic on all of $\bbR$. 
By construction $\Delta$ (see \S \ref{sectionmonotonicity}) has real analytic boundary arcs, 
therefore $\Delta'$ does as well.
\\ 
\newline Note that $I_0,I_1,I_2$ label the boundary arcs of $V$ in counter-clockwise order,
whereas $\Gamma_0, \Gamma_2, \Gamma_4$ label the boundary arcs of $\Delta$ in clockwise order.
We claim that   
$$\rho_{ef}^{-1} \circ \mu^{-1} \circ \sigma_{ef}: \bbR/2\pi\bbZ \times \bbR \into \bbC^{\times} $$ 
is orientation reversing:
$(\rho_{ef}^{-1} \circ \mu^{-1} \circ \sigma_{ef})^*(d\tilde x \wedge d\tilde y) = -d\theta \wedge d\lambda,$
where $\tilde x +i\tilde y \in \bbC^{\times}$, $(\theta, \lambda) \in  \bbR/2\pi\bbZ \times \bbR$.
%where $\bbR/2\pi\bbZ \times \bbR$ has the orientation induced by 
%$d\theta \wedge d\lambda$, $(\theta, \lambda) \in  \bbR/2\pi\bbZ \times \bbR,$
%and $\bbC^{\times}$ has the orientation induced by $d\tilde x \wedge d\tilde y$, 
%$\tilde x +i\tilde y \in \bbC^{\times}$. 
To see this note that, since $\rho_{ef}^{-1}$ is a biholomorphism,
it respects the orientations of $(\Sigma_j \, dx_j\wedge dy_j)|Y_{ef})$ and
$d\tilde x \wedge d\tilde y$ on $\bbC^{\times}$,  where 
$(x_j+iy_j) \in \bbC,$ $j=1,\ldots,4$. 
Then one calculates:
$$(\mu^{-1})^*(\Sigma_j \, dx_j\wedge dy_j) =  (\Sigma_j \, dv_j \wedge du_j), 
\phantom{bbb}
\sigma_{ef}^*(\Sigma_j \, dv_j \wedge du_j) = - d\theta \wedge d\lambda, $$
where  $(u,v) \in \bbR^4 \times \bbR^4$ are the the usual coordinates for $T^*S^3$. 
%We leave the first equation to the reader; for the second, we have
%$$ v = \lambda(-\sin \theta e, \cos \theta f), u = (\cos\theta e, \sin \theta f)$$
%and for example
%$$  dv_1  = -\sin \theta e_1 d \lambda + \lambda \cos \theta e_1 d\theta $$
%$$  du_1 = 0 d\lambda + -\sin \theta e_1 d\theta$$
%$$ dv_1 \wedge du_1 =  sin^2 \theta e_1^2 d\lambda \wedge d\theta.$$
Therefore 
$$(\rho_{ef}^{-1} \circ \mu^{-1} \circ \sigma_{ef})(\Gamma_0), \, (\rho_{ef}^{-1} \circ \mu^{-1} \circ \sigma_{ef})(\Gamma_2), \, (\rho_{ef}^{-1} \circ \mu^{-1} \circ \sigma_{ef})(\Gamma_4)$$ 
label the boundary arcs of $\Delta'$ in counter-clockwise order. Then lemma \ref{Riemannmap} produces a biholomorphic map $\varphi: V \into \Delta'$ satisfying suitable boundary conditions, and we set $w_{ef} = \mu \circ \rho_{ef} \circ \varphi$.
\end{proof}

\subsection{Uniqueness of holomorphic triangles} \label{sectionuniqueness}
Define 
$$P:Z\into \bbC, \phantom{bbb} P(z_1, z_{2}, z_{3},  z_{4}) =  z^2_1 + z_{2}^2 - z_{3}^2 - z_{4}^2.$$
Set $Y = \{(w_1, w_2) \in \bbC^2 : w_1^2 + w_2^2 = 1\} \cong T^*S^1$.
%As motivation we note that $P$ is the complexification of a polynomial on 
%$S^3$ which is Morse-Bott, taking a  maximum at $K^+$ and a minimum 
%at $K^-$. $P$ maps $\nu^*K^+$ to $(-\infty, -1]$ and $S^3$ to $[-T-\delta, T+\sigma]$. 

\begin{lemma} \label{trivialization}
Let $U \subset \bbC \setminus \{\pm1\}$ be 
such that the map induced by the inclusion $\pi_1(U) \into \pi_1(\bbC \setminus \{\pm1\})$ is trivial.
Then there is a holomorphic trivialization of $P: Z \into \bbC$ over $U$ 
given by
 $$\Phi: P^{-1}(U) \into U \times Y \times Y, \phantom{bbb} \Phi (z_1, z_2,  z_{3}, z_{4}) = (\lambda,\alpha(\lambda) 
(z_1, z_2), \beta(\lambda) (z_3, z_{4}) ),$$ where 
 $\lambda = P (z_1, z_2,  z_{3}, z_{4})$, and 
$\alpha, \beta: U \into \bbC$ are the following holomorphic functions.
$$\alpha(\lambda) =  \sqrt{ \frac{2}{1+\lambda}}, \lambda \in U,  \phantom{bbb} \beta(\lambda) =  \sqrt{ \frac{2}{1-\lambda}}, \lambda \in U.$$
(Here we have chosen branches of the square root function.)
%$$\alpha(z-1,z_2) =  \sqrt{ \frac{2}{ z^2_1 + z_{2}^2 }$$
%$$\beta(z_3,z_4) =  \sqrt{ \frac{2}{ z^2_3 + z_{4}^2 }$$
\end{lemma}

\begin{proof}
%Note that $P$ is the complexification of a polynomial on $S^3$ which is 
%Morse-Bott, with no critical sets except for the  maximum 
%at $K^+$ and the minimum at $K^-$.
%Thus the critical values of $P$ are $-1$ and $1$.
Assume $z\in P^{-1}(\lambda)$ for some $\lambda \in \bbC$.
Then
$$z^2_1 + z_{2}^2 - z_{3}^2 - z_{4}^2 = \lambda, \, \, z_1^2 + z_3^2+z_4^2 + z_4^2 =1,$$ 
since the domain of $P$ is $Z$. Therefore
%Adding and subtracting these equations yields
$$ z_1^2 + z_2^2 = \frac{1+\lambda}{2}, \phantom{bbb} z_{3}^2 + z_{4}^2 = \frac{1-\lambda}{2}.$$ 
Now, by the assumptions on $U$,  $\{\frac{2}{\lambda+1}:\lambda \in U \}$ and 
$\{\frac{2}{\lambda-1}:\lambda \in U \}$ are subsets of $\bbC\setminus\{0\}$
and neither of them contains a nontrivial loop around 0. Therefore 
they both admit a branch of the square root function, and we can define $\alpha,\beta$ as
stated.
\end{proof}Let $B_1, B_2 \subset \bbR/2\pi\bbZ \times [-r, r]$ denote the two bigons 
adjacent to $\Delta$ 
bounded respectively by $\widetilde \Gamma_4$,  $\widetilde \Gamma_2$, 
and  $\widetilde \Gamma_2$,  $\widetilde \Gamma_0$ in figure \ref{CorrectionGammaIIbig}
(where each $B_i$ has one vertex in common with $\Delta$).
\begin{figure}
\begin{center}
\includegraphics[width=5in]{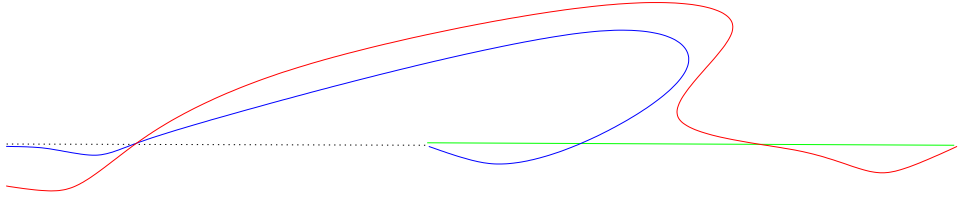} 
%\put(-200,100){$(0,0)$}
\caption{The curves $P(\mu^{-1}(W_4))$ (red), $P(\mu^{-1}(W_2))$ (green), 
$P(\mu^{-1}(W_0))$ (blue) in $\bbC$.} 
%$P(\Gamma_4)$, ...  }
\label{P(Gamma)}
\end{center}
\end{figure}
\begin{lemma} \label{Pimage} $P(\mu^{-1}(W_0)), P(\mu^{-1}(W_2)), P(\mu^{-1}(W_4)) \subset \bbC$ are embedded curves in $\bbC$ as in figure 
\ref{P(Gamma)}.
%as in figure \ref{??}. 
They enclose three compact regions 
$R$, $R_1$, $R_2$,  Here, $R_1, R_2  \subset \{ z\in \bbC : Im z \leq 0\}$ are bigons diffeomorphic to
$B_1$, $B_2$,
and $R$ is a triangular region, diffeomorphic to $\Delta$, which is contained in $$(\bbC\setminus \{-1,1\}) \cap \{ z\in \bbC : Im \,z \geq 0\}.$$
In particular $U= R$ satisfies the assumptions of the last lemma. 
%$R \cap R_1$ and $R\cap R_2$  each consist of a single point, which is a common vertex in both cases.
\end{lemma}
%See figure [localtriangles 2, Aug 9].
\begin{proof} %Let $p= P \circ \mu^{-1}\circ \sigma_{ef}:\bbR/2\pi\bbZ \times [-r,r]\into \bbC$.
A computation shows that %for $e \in K^+,f \in K^-$, $(\theta, \lambda) \in (\bbR /2\pi\bbZ) \times \bbR$
\begin{gather}
\label{ellipse}
 P(\mu^{-1}(\sigma_{ef}(\theta, \lambda)))=  
\sqrt{1+4\lambda^2 } \cos2\theta + i2\lambda\sin2 \theta.
 %\text{, } \lambda \in \bbR,  \theta \in \bbR/\pi \bbZ, 
\end{gather}
%where  $$a(\lambda) =\sqrt{1+4\lambda^2 }, \, b(\lambda)= 2\lambda$$
%In particular $P(\mu^{-1}(\sigma_{ef}(\theta, \lambda)) )$ is independent of $e,f$.
Indeed, $P(\mu^{-1}(\sigma_{ef}(\theta,\lambda)))$ is equal to
\begin{gather*}
%P(\mu^{-1}(\sigma_{ef}(\theta,\lambda))) \\
P\big(f(\lambda)(\cos(\theta)e, \sin(\theta)f) +i(-\frac{\lambda}{f(\lambda)}) 
(-\sin(\theta)e, \cos(\theta)f)\big)
\end{gather*}
%$$= P([f(\lambda) \cos(\theta) + i(\frac{\lambda}{f(\lambda)}) \sin(\theta)](e_1, e_2) 
%+ [f(\lambda)\sin(\theta) + i(-\frac{\lambda}{f(\lambda)}) \cos(\theta)](f_1,f_2))$$
%$$= [f(\lambda) \cos \theta + i f(\lambda)^{-1} \lambda  \sin \theta]^2(\Sigma e_i^2)-
%[(f(\lambda) \sin \theta - i f(\lambda)^{-1} \lambda  \cos \theta] ^2(\Sigma f_j^2) $$
%$$ = f(\lambda)^2 \cos^2 \theta  - f(\lambda)^{-2} \lambda^2  \sin^2 \theta
%+ 2 i  \lambda \sin \theta   \cos \theta  $$
%$$- f(\lambda
%)^2 \sin^2 \theta  + f(\lambda
%)^{-2} \lambda
% ^2  \cos^2 \theta
%+ 2 i  \lambda
%\sin \theta   \cos \theta   $$
$$= (f(\lambda)^2 + f(\lambda
)^{-2} \lambda
^2)  \cos2\theta+ 2i \lambda
 \sin2 \theta.$$
%To see that $f(\lambda)^2 + f(\lambda)^{-2} \lambda^2 =\sqrt{1+4\lambda^2}$,
Set 
$$a(\lambda) =f(\lambda)^2 + f(\lambda
)^{-2} \lambda
^2, \phantom{bbb} b(\lambda) =2 \lambda,$$ and  recall from (\ref{muinverse}) that $f(\lambda)$ satisfies
%$$f(\lambda) = \sqrt{  \frac{1+ \sqrt{1+4\lambda^2 }    } {2}   },$$
%and $f(\lambda)$ satisfies
$$(f(\lambda)^2)^2 -f(\lambda)^2 -\lambda^2 = 0.$$  
Hence
$$f(\lambda)^2a(\lambda) = (f(\lambda)^2)^2 + \lambda^2 = f(\lambda)^2  + 2\lambda^2\text{, and}$$ 
$$a(\lambda) = 1+ 2 \frac{\lambda^2}{f(\lambda)^2}=1+4 \frac{\lambda^2}{ 1+ \sqrt{1+4\lambda^2 }  }
=1+( \sqrt{1+4\lambda^2 } -1).$$
\\
Formula (\ref{ellipse}) implies 
$$p= P \circ \mu^{-1}\circ \sigma_{ef}:\bbR/2\pi\bbZ \times [-r,r]\into \bbC$$ 
is independent of $e,f$.
Therefore, it is enough to describe the images of $\widetilde \Gamma_4$, 
$\widetilde \Gamma_2$, $\widetilde \Gamma_0$ under $p$. 
Note also that the formula for $p$ is periodic, so it suffices to 
understand $p$ just near $\Delta$; the other three triangles are mapped 
onto $p(\Delta)$ as well. One can understand the map $p$ as follows. For fixed $\lambda > 0$, we have  $a(\lambda) >0$ and $b(\lambda) >0$; thus as $\theta$ ranges from  $0$ to $\pi$, 
$$\theta \mapsto p(\theta, \lambda) = a(\lambda)\cos2\theta + ib(\lambda) \sin2 \theta$$ 
parameterizes an ellipse in the counterclockwise direction, with axes $a(\lambda)$ and $b(\lambda)i$. Now as $\lambda>0$ varies, one gets a family of disjoint ellipses, since $a,b$ are strictly increasing functions. Similarly, for fixed $\lambda<0$, we have $a(\lambda) >0$, $b(\lambda) <0$, so one gets an ellipse parameterized in the clockwise direction. For $\lambda=0$, we get the map $\theta \mapsto \cos(2\theta)+0i$. %which maps $[0,\pi/2]$ diffeomorphically onto $[-1,1]$.
Using this description of $p$, we sketch $P(\mu^{-1}(W_0))= 
p(\widetilde \Gamma_0),$ $ P(\mu^{-1}(W_2)) = p(\widetilde  \Gamma_2),$ $ P(\mu^{-1}(W_4)) =p(\widetilde \Gamma_4)$ as in figure \ref{P(Gamma)}. 
%It follows that these curves bound compact regions $R= p(\Delta), R_1=p(B_1), R_2 = p(B_2)$, as in the statement of the lemma.
\end{proof}

\begin{lemma} \label{image}Let 
$w \in  \mathcal{M}(W_4, W_2, W_0, J_\bbC; D(T^*S^3))$ and denote its
continuous extension by $\overline w: D^2 \into D(T^*S^3))$.
Then $P(\mu^{-1}(\overline w(D^2)))= R.$
%and $P \circ w: V \into R \setminus \{0,1/2, \gamma(\pi/4)\} $ 
%coincides with the unique anti-biholomorphic map of Corollary \ref{base}.
\end{lemma} 

\begin{proof} For this proof we consider $\mu^{-1} \circ \overline w, \mu^{-1} \circ w : V \into Z$,
but we denote them again by $\overline w, w$.  The lemma follows basically from the maximum principle. Let $R, R_1, R_2$ be as in the last lemma (see figure \ref{P(Gamma)}). Let $p_{40}$, $p_{42}$, and $p_{20}$ denote the vertices of $R$, let 
$q_{42}$ and $p_{42}$ denote the two vertices of $R_1$, and let 
$q_{20}$ and $p_{20}$ denote those of $R_2$. 
Then, the boundary conditions on $w$, plus the description of
 $P(\mu^{-1}(W_4))$, $P(\mu^{-1}(W_2))$, $P(\mu^{-1}(W_0))$ 
in the last lemma  imply
$$P(\overline w(\zeta_0)) =  p_{40}, \, \, P(\overline w(\zeta_1)) =  p_{42} \text{ or } q_{42}, \, \,
P(\overline w(\zeta_2)) =  p_{20} \text{ or } q_{20}.$$         
 Let $D  = R_1 \cup R \cup R_2$ and  $D' = D \cup P(W_2) \cup P(W_4)$.
Assume for a contradiction there exists $\zeta \in D^2$ such that 
$(P \circ \overline w) (\zeta) \notin D'$. 
Let $C: [0, \infty) \into \bbC \setminus D'$
be a continuous curve %starting at $(P \circ \overline w) (\zeta)$ 
such that $C(0) = (P \circ \overline w) (\zeta)$ and $\lim_{t \rightarrow \infty} |C(t)| = \infty$.
%which goes to $\infty$, i.e. whose image is not contained in any compact subset.
Let $t_{max}$ be the largest element of the compact set 
$C^{-1}[(P \circ \overline w)(D^2)]$. 
Let $\zeta_{max} \in D^2$ be such that 
$(P \circ \overline w)(\zeta_{max}) = C(t_{max}) \in \bbC \setminus D'$.
Now $\zeta_{max}$ is necessarily in the interior of $D^2$ because of the boundary conditions on $w$. But then, since $(P\circ w)|\Int D^2$ is a nonconstant holomorphic function, there must be an open disk around 
$(P \circ w)(\zeta_{max})$ contained in $(P \circ  w)(\Int D^2)$. 
Then by continuity of $C$ there must exist 
$t>t_{max}$ such that $C(t) \in (P \circ \overline w)(D^2)$, contradiction. Therefore 
$(P \circ \overline w)(D^2) \subset D'$.
\\
\newline It follows easily that $(P\circ \overline w)(D^2) \subset D$, because if  
$\zeta \in D^2$ is such that $(P\circ \overline w)(\zeta) \in D' \setminus D$ then  by continuity there 
is a nearby point $\zeta' \in \Int(D^2)$  satisfying the same condition; then a small neighborhood of $\zeta'$ would map onto an open disk in $\bbC$, but this is impossible since $D' \setminus D$ is contains no open disks.
\\ 
\newline Now suppose for a contradiction that there exists $\zeta \in D^2$ such that 
$P(\overline w(\zeta)) \in R_1 \setminus \{p_{42}\}$. Let $\upsilon: [0,1] \into D^2$
be a path connecting $\zeta$ and $\zeta_0$ such that $\upsilon(0,1) \subset \Int(D^2)$. 
Then by connectedness there is $\zeta' \in \Int(D^2)$ 
satisfying $P(\overline w(\zeta')) = p_{42}.$ But there is no open disk contained around $p_{42}$ 
contained in $D$, so this is impossible. We conclude that $P(\overline w(D^2)) \subset R \cup R_2$
and a similar argument implies $P(\overline w(D^2))\subset R$. 
 The same type of argument, using the fact that $(P \circ  w)|\Int D^2$ is an open map, shows that $P(\mu^{-1}(\overline w(D^2)))= R$.
\end{proof}
Recall $Y= \{(w_1, w_2) \in \bbC^2 : w_1^2 + w_2^2 = 1\} \cong T^*S^1$ and set 
 $L = \bbR^2 \cap Y = S^1$.
Lemmas \ref{Pimage} and \ref{trivialization} imply there is a 
holomorphic trivialization
\begin{equation}
\Phi: P^{-1}(R) \into R \times Y \times Y \label{Phi}
\end{equation}
%where $Y= \{(w_1, w_2) \in \bbC^2 : w_1^2 + w_2^2 = 1\} \cong T^*S^1$. 
Denote the compositions of $\Phi$ with the two projections to $Y$ by
\begin{gather}
 \Phi_1, \Phi_2:  P^{-1}(R) \into Y. \label{Phi_1,Phi_2}
\end{gather}

\begin{lemma} \label{Lagrangian} For any $p \in \mu^{-1}(T^*K_{e'f'}) \cap P^{-1}(R)$, we have $\Phi_1(p) = \pm e$ and $\Phi_2(p) = \pm f$, 
where $e' = (e,0) \in K_+$, $f'= (0,f) \in K_-$. Consequently, 
$$\Phi_\ell(\mu^{-1}(W_i)\cap P^{-1}(R)) \subset L_i.$$
for any $\ell=1,2$, $i =0,2,4$
%For any fixed $p$ the precise signs %in the last lemma 
%are fixed by the 
%choice of branch of 
%the square root functions in the definitions of $\alpha, \beta :U \into \bbC$.
\end{lemma}

\begin{proof} By (\ref{muinverse}), every $p \in \mu^{-1}(T^*K_{e'f'})$ is of the form $p = (ae,bf),$ for some $a,b \in \bbC$. If $p \in \mu^{-1}(T^*K_{e'f'}) \cap \Phi^{-1}(R)$, then 
$$\Phi_1(p) = \alpha(\lambda)(ae) = \rho_1 e \in Y, \, \, \Phi_2(p) = \beta(\lambda)(bf) = \rho_2 f \in Y$$
for some $\rho_1,\rho_2 \in \bbC$. Since  $\rho_1 e, \rho_2 f \in Y$ we have
$\rho_1^2 (e_1^2 + e_2^2) = \rho_1^2 = 1$, and similarly $\rho_2^2 = 1$. Thus $\rho_1 = \pm 1, \rho_2 = \pm 1.$
%The precise signs are indeed fixed by the choice of branches in the 
%definitions of $\alpha, \beta$.
\end{proof}
\begin{prop} \label{uniqueness} Every $w \in \mathcal{M}(W_4, W_2, W_0, J_\bbC; D(T^*S^3))$
is equal to $w_{ef}$ for some $e,f$, as in Proposition \ref{existence}.
\end{prop}

\begin{proof} Let $\overline w: D^2 \into T^*S^3$ be the continuous extension of $w$.
Lemma \ref{image} implies $R =   P(\mu^{-1}(\overline w(D^2)))$, and therefore it makes sense to form
%We first prove $\overline w$ is constant. Then the choice of branches for $\alphmoduli spacea, \beta$ imply the stated values for $\lim_{\zeta \rightarrow \zeta_2} w(\zeta)$ and $\lim_{\zeta \rightarrow \zeta_0} w(\zeta)$.
$$\overline w_k = \Phi_k \circ \mu^{-1}\circ  \overline w : D^2 \into Y, \, k=1,2.$$
Set $w_k = \Phi_k \circ \mu^{-1}\circ  w.$ Then
$$w_k \in \mathcal{M}(L,L,L, J_{Y}; Y),$$
where $J_{Y}$ is the complex structure from $\bbC^2$.
Here, the boundary conditions follow from Lemma \ref{Lagrangian}, and $\int_V w_k^*\omega_{Y} < \infty$
follows from the fact that $w_k$ has a continuous extension $\overline w_k$.
Once we know $\int_V w_k^*\omega_{Y} < \infty$, we can apply Stokes' theorem to compute
$\int_V w_k^*\omega_{Y}$ in terms of $\theta_{Y}|L$; then, 
since $\theta_{Y}|L=0$,
$\int_V w_k^*\omega_{Y}$ is zero and $w_k$ must be constant, say $w_1 = e$, $w_2 =f$.
%The boundary conditions on $w$ imply $\mu^{-1}(w(\zeta)) \in \mu^{-1}(T^*K_{e'f'}) \cap P^{-1}(R)$
%for some $e',f'$, so Lemma \ref{Lagrangian} implies $\Phi_1 =  e$, $\Phi_2 = f$ for some $e,f$. 
\\
\newline Now, by lemma \ref{Pimage}, $\Omega =R$ satisfies the assumptions of lemma \ref{Riemannmap}. Further,
$P \circ \mu^{-1} \circ w: V \into R$ is holomorphic and maps $I_0$, $I_1$, $I_2$ into the corresponding boundary arcs of $R$ and $\zeta_0$, $\zeta_1$, $\zeta_2$ to the vertices of $R$ (see the proof of lemma \ref{Pimage}).
Hence, by lemma \ref{identity}, it must be the unique biholomorphism  $\psi:V \into R$ satisfying these conditions.  Choosing $e'=(\pm e, 0),f'=(0,\pm f)$ suitably, $w_{e'f'}$ will also satisfy 
$$\Phi_1 \circ \mu^{-1} \circ w_{e'f'} = e, \phantom{bbb}  \Phi_2 \circ \mu^{-1} \circ w_{e'f'} = f$$
(the sign of $\pm e,\pm f$ depends on the choice of square-roots for $\alpha, \beta$). And, again, $P  \circ \mu^{-1} \circ w_{e ' f' } = \psi$. 
Therefore 
$w(\zeta) = \Phi^{-1}(\psi(\zeta),\pm e, \pm f) =w_{e'f'}(\zeta).$ \end{proof}

\subsection{Regularity of the moduli space of holomorphic triangles}\label{sectionregular} For this section we will work with $Z$ 
rather than $T^*S^3$, and we will look at the corresponding moduli space
of $J$-holomorphic triangles, where $J = J_Z$, 
$$\mathcal{M} =\mathcal{M}( \widetilde W_4,\widetilde W_2,\widetilde W_0,J;Z),$$
where we set $\widetilde W_i = \mu^{-1}(W_j), \, i=0,2,4.$
As we explain in \S \ref{sectionFredholm},
$\mathcal{M}$ sits inside a certain Banach space 
$\mathcal{B}$: there is a Banach-bundle 
$\mathcal{E} \into \mathcal{B}$; the Cauchy-Riemann operator is a section
$\partial_J : \mathcal{B} \into \mathcal{E}$, and 
$\mathcal{M} = \partial_J^{-1}(0)$.
For each $w \in \mathcal{M}$ the linearization 
$ D(\overline{\partial}_J)_w: T_w \mathcal{B} \into \mathcal{E}_w$
is a Fredholm operator by lemma \ref{lemmaFredholm}.

\begin{prop}\label{regular} %For any finite area $J$-holomorphic triangle  $w:V \into Z$ with the usual boundary conditions
$J = J_Z$ is a regular almost complex structure. That is, 
for any $w \in \mathcal M$, the linearized operator  $D(\overline \partial_J)(w): T_w(\mathcal B) \into \mathcal{E}_w$ 
is surjective. %, i.e. $J = J_Z$ is a regular almost complex structure.
\end{prop}

This follows from lemmas \ref{kernal} and \ref{index} below, which show that 
$$\dim Ker(D(\overline \partial_J)(w)=2 \text{ and } \index D(\overline \partial_J)(w)=2.$$
 
\begin{remark} \label{general} 
In general, if $\dim  N = 2k$ and $\dim  M = 4k-2$, then the $2$ here will be $2k-2$.  
This is because $ L_0 \cong S^{2k-1}$,  $ K_j \cong S^{k-1} \cong K_-^j \cong K_+^j$, and then 
$2k-2 = \dim \, \partial \mathcal{N}_{S^{2k-1}}(K_j^{\pm})$. 
\end{remark}

For the proof of lemma \ref{kernal} we will need a special case  of 
\cite[lemma 11.5]{S08}, as stated in lemma \ref{winding} below. 
We first set up some notation. Take $E = \bbC \times V \into V$ to be the trivial Hermitian line bundle over $V$.
Let $F \into \partial V$ be a totally real sub-bundle of $E|\partial V$.
We define the winding number of $F$ as follows.
Recall from \S \ref{deftriangleproduct} that $\zeta_0$, $\zeta_1$, $\zeta_2$ label the three punctures in counter clockwise order and $I_0$, $I_1$, $I_2$ label the boundary components of
$\partial V$ in counter clockwise order, where $\zeta_0$, $\zeta_1$ lie on the boundary of $I_0$.
We have fixed (incoming) strip-like ends 
$$\epsilon_0, \epsilon_1: (-\infty, 0] \times [0,1] \into V$$
parameterizing a neighborhood of $\zeta_0$ and $\zeta_1$, and 
we have an (outgoing) strip-like end at $\zeta_2$:
$$\epsilon_2: [0,\infty) \times [0,1] \into V.$$
We assume $F$ is nondegenerate in the sense that $F_{\epsilon_{\zeta_j}(0,s)}$ and $F_{\epsilon_{\zeta_j}(1,s)}$ are transverse for $|s| >>0$ (where $s \in [0,\infty)$ if $j=2$ or $s \in (-\infty, 0]$ if $j=0,1$). Now, we define a homotopy between $F_{\epsilon_{\zeta_j}(0,s)}$ and $F_{\epsilon_{\zeta_j}(1,s)}$ as follows.
For convenience of notation, let us rotate  $\epsilon_j^*(E)$ by a constant  so that 
$$F_{\epsilon_{\zeta_j}(0,s)} = \bbR, \, F_{\epsilon_{\zeta_j}(1,s)} = e^{i \sigma_j}\bbR$$
for some $\sigma_j \in (-\pi, 0)$. In that notation we take the homotopy which goes through
\begin{gather} \label{homotopy}
 e^{it\sigma_j}, \, t \in [0,1]
\end{gather}
 Consider the compactification of $V$ to $\widehat V \cong D^2$
where we glue on a copy of the upper half plane  at each puncture and use the 
homotopy (\ref{homotopy}) near each puncture. Let $\rho: S^1 \into \bbR P^1$ denote the map we obtain in this way and let $\mu(\rho) \in \bbZ$ denote its degree; this is called the winding number of $F$.
Let $|V^-|$ denote the number of incoming punctures of $V$, which in our case is 2. 
Let $\overline \partial$ denote the usual Cauchy-Riemann operator on $V \subset \bbC$, 
so $\overline \partial = \partial_s + i\partial_t$ in the coordinates $s+it \in \bbC$.
We regard $\overline \partial$ as operating on sections of $E$ with boundary in $F$.

\begin{lemma} \label{winding}  If $\mu(\rho) - |V^-| = \mu(\rho) - 2 < 0$ then $Ker(\overline \partial) =\{0\}$.\, $\square$
\end{lemma}

For the proof see \cite[lemma 11.5]{S08}. 

\begin{lemma}\label{kernal} %For any finite area $J$-holomorphic triangle  $w:V \into Z$ with the usual boundary conditions
For any $w \in \mathcal M$,
%the linearized operator  $D(\overline \partial_J)(w): T_w(\mathcal B) \into \mathcal{E}_w$  has 
$dim(Ker D(\overline \partial_J)(w)) =2$. 
\end{lemma}

\begin{proof}Fix $w \in \mathcal M$. By condition \ref{convergence} $w$ has a continuous extension to $D^2$, which we denote $\overline w:D^2 \into Z$.
Since  $J$ is the restriction of the  usual complex structure on $\bbC^4$, 
we have $D(\overline \partial_{J})(w)(X) = \overline \partial X$, 
where $\overline \partial$ is the usual Cauchy-Riemann operator acting on 
$C^{\infty}(\bbC, \bbC^4).$
Thus, each  $X \in Ker D(\overline \partial_J)(w) $ is a holomorphic map
$X: V \into \bbC^4$. 
%a smooth map $X: V \into \bbC^4$ 
By definition of $T_w(\mathcal{B})$, it has a continuous extension $\overline X: D^2 \into \bbC^4$. This  satisfies %: D^2 \into \bbC^4$ 
%$D(\overline \partial_J)(w)(X) =0,$
$$\overline X(\zeta) \in T_{\overline w(\zeta)}Z= \{(z_1,z_2,z_3, z_4) \in \bbC^4: \Sigma_i z_i \overline w_i(\zeta) =0\}\text{ for all } \zeta \in D^2,$$
$$\overline X(\zeta) \in T_{\overline w(\zeta)}(\widetilde W_{2k}),  \text{ for all } \zeta \in \overline I_k \subset \partial D^2, k=0,1,2. $$%\, (j, k) =(0,4), (1,2), (2,0).$$ 

Since $w \in \mathcal{M}(\widetilde W_4, \widetilde W_2,\widetilde W_0,J_Z)$, 
%lemma \ref{image} implies
%$P(\overline w(D^2))=R$,
%which is a region where the holomorphic 
%trivialization 
%$$\Phi : P^{-1}(R) \into R\times Y_1 \times Y_2$$ 
% makes sense 
%and 
we have the holomorphic trivialization (\ref{Phi}) over $R =P(\overline w(D^2))$
$$\Phi: P^{-1}(R) \into R \times Y \times Y,$$
and projections (\ref{Phi_1,Phi_2}), 
$$\Phi_1,\Phi_2  : P^{-1}(R) \into  Y.$$
 %\, \overline w_2 = \Phi_2 \circ \overline w
%and $w_k = \Phi_k \circ w$. 
The rough idea for computing the kernal is to decompose a given $X$ using the trivialization $\Phi$:
$$X \mapsto (DP(X), (D \Phi_1)(X), (D \Phi_1)(X))$$ 
and then check that the first component is zero, while the last
two components are constant and each lie in a one dimensional linear space.
\\
\newline Define 
$$\overline \varphi: D^2 \into \bbC, \phantom{bbb} \overline \varphi(\zeta) = DP_{\overline w(\zeta)} (\overline X(\zeta)).$$ %, \zeta \in D^2$$
Then $\varphi= \overline \varphi|V: V \into \bbC$
is holomorphic and has a continuous extension to $D^2$, and 
%$$T_{\overline w(\zeta)}(Z) \into 
%T_{P(\overline w(\zeta))}(\bbC), \, \zeta \in D^2$$ is identically zero. 
%Let $$p = \overline p | V.$$
%Let us identify  $T_z(\bbC) = \bbC$ for all $z \in \bbC$. 
%For each $\zeta \in \overline I_k \subset \partial D^2$, $k=0,1,2,$   
\begin{gather} \label{boundaryconditions}
  \varphi (\zeta) \in T_{  \varphi(\zeta) }(P(W_{2k})), \text{ for each } \zeta \in I_k \subset \partial D^2,\, k=0,1,2.
\end{gather}

%$$(D_{\overline w(\zeta)}P) (\overline X(\zeta)) \in T_{(P\circ \overline w)(\zeta) }(P(W_j))$$ for some $j=0,1,2$. 
We will show $ \varphi$ is identically 0. 
\\
\newline
Recall $R = P(\overline w (D^2))$. The proof of lemma \ref{Pimage} shows that 
$$P(\widetilde W_{2k})\cap R= p(\widetilde \Gamma_{2k} \cap \Delta) \text{ where}$$
$$p(\theta,\lambda) = \sqrt{1+4\lambda^2}\cos2\theta + i2\lambda \sin2 \theta.$$
See figure \ref{P(Gamma)}. Now, let us regard 
$\varphi: V \into \bbC$ as a section of the trivial Hermitian line bundle $E = V \times \bbC$
satisfying the totally real boundary conditions $F \into \partial V$ given by (\ref{boundaryconditions}). Since $F_{\epsilon_j(s,0)}$ and $F_{\epsilon_j(s,1)}$ are transverse at each puncture for $|s| >>0$, we can apply lemma \ref{winding} above to $(E,F, \overline\partial)$.
Since $\varphi$ has a continuous extension to $D^2$, it has finite energy and therefore it lies in the domain of $\overline \partial$ (which is the Sobolev space $W^{1,p}(V,E,F)$, $p>2$). We claim that the winding number of $F$, i.e. $\mu(\rho) \in \bbZ$, is equal to zero. Thus by lemma \ref{winding}, $Ker(\overline \partial) =\{0\}$ and so $\varphi$ is identically zero. 
\\
\newline Indeed, $\mu(\rho)=0$ follows from  inspection of figure \ref{P(Gamma)}, where we use the homotopies at the punctures given by (\ref{homotopy}). Note that the strip-like ends label each puncture as follows: The  two boundary arcs at each of the two bottom vertices (i.e. the two punctures on the boundary of the arc $P(\mu^{-1}(W_2))$) will be labeled $1,0$ and $1,0$ in counter-clockwise order; the other vertex will be labeled $0,1$ in counter-clockwise order. (Roughly speaking, then, we have the following: The arc  $P(\mu^{-1}(W_0))$ contributes  $+1$ to the winding number;  the arc along $P(\mu^{-1}(W_2))$ of course contributes nothing to the winding number, but the two punctures between $P(\mu^{-1}(W_0))$, $P(\mu^{-1}(W_2))$ and between $P(\mu^{-1}(W_2))$, $P(\mu^{-1}(W_4))$ together contribute $-1$ to the winding number; finally, the arc $P(\mu^{-1}(W_4))$ and the puncture between $P(\mu^{-1}(W_4))$ and $P(\mu^{-1}(W_0))$ contribute nothing to the winding number.)
\\
\newline For the other two components, let
% and $w_k = \Phi_k \circ w :V \into Y, \,k=1,2.$$ For $\zeta \in D^2$, let 
$$\overline \varphi_1(\zeta) = (D_{\overline w(\zeta)}\Phi_1)(\overline X(\zeta)), \,\phantom{bb} \overline \varphi_2(\zeta) = (D_{\overline w(\zeta)}\Phi_2)(\overline X(\zeta)), \text{ for each } \zeta \in D^2.$$
Then 
$$\varphi_1 = \overline\varphi_1|V, \phantom{b} \varphi_2 = \overline\varphi_2|V: V \into \bbC^2$$ 
are holomorphic maps
%T(Y) \subset T(\bbC^2)= \bbC^2 \times \bbC^2$ 
with continuous extensions to $D^2$ and, setting $\overline w_j = \Phi_j\circ \overline w$, %: D^2 \into Y,$ 
we have  %$\overline \varphi_1(\zeta)$ lies in the complex subspace
\begin{gather*}
\overline \varphi_1(\zeta) \in T_{\overline w_1(\zeta)}(Y) \subset \bbC^2, \text{ for each }\zeta \in D^2, \text{ and}\\
%For each $\zeta \in \overline I_k \subset \partial D^2$, $k=0,1,2$,
%$\overline \varphi_1(\zeta)$ lies in the Lagrangian subspace 
\overline \varphi_1(\zeta) \in T_{\overline w_1(\zeta)}(\Phi_1(\widetilde W_{2k})) \subset T_{\overline w_1(\zeta))}(Y) \text{ for each }\zeta \in \overline I_k \subset \partial D^2,  k=0,1,2. 
\end{gather*}
Similarly for $\varphi_2$.
Now, the proof of lemma \ref{uniqueness} showed that 
$$\overline w_1 =e = (e_1,e_2) \in L, \overline w_2 = f =(f_1,f_2)\in L$$ 
are constant, where $L = Y \cap \bbR^2$. And Lemma \ref{Lagrangian} implies
$\Phi_1(\widetilde W_{2k}), \Phi_2(\widetilde W_{2k}) \subset L$, $k=0,1,2$.
Thus $\varphi_1$ is a holomorphic map from $V$ into 
$$T_e(Y) = \{(z_1,z_2) \in \bbC^2 : z_1e_1+ z_2e_2 =0\},\text{ with }$$
 $$\varphi_1(I_{2k}) \subset T_{e}(L) =  \{(x_1,x_2) \in \bbR^2 : x_1e_1+ x_2e_2 =0\}, k=0,1,2.$$
Further, $\int_V \varphi_1^*\omega < \infty$ because there are continuous extensions $\overline \varphi_k$.
Since $\Theta_{\bbC^2}| T_{e}(L)=0$, we conclude that $\varphi_1$ is constant, as in the proof of lemma \ref{uniqueness}. Similarly for $\varphi_2$. 
We conclude that the kernal of $D(\overline \partial)(w)$ is isomorphic 
to the the set of $X \in T_w B$ of the form 
%$$\{X \in T_w B : 
$$X(\zeta)  = (D\Phi_{w(\zeta)})^{-1}(0, (E_1, E_2), (F_1, F_2)), \zeta \in V$$
where  $E_1e_1+ E_2e_2=0,$ $F_1e_1+ F_2e_2=0$. In particular the dimension is 2.
%the the set of $X \in T_w B$ such that 
%$$D \Phi_1(X)  \in \{(E_1, E_2) : E_1e_1+ E_2e_2=0\} \text{ and}$$
%$$D \Phi_2(X)  \in \{(F_1, F_2) : F_1e_1+ F_2e_2=0\}.$$ 
%where $(\Phi_1 \circ w)(\zeta) = (e_1, e_2)$,   $(\Phi_2 \circ w)(\zeta) = (f_1, f_2)$, for all 
%$\zeta \in V$. %Given $(E_1, E_2)$, $(F_1,F_2)$ satisfying  $\Sigma_k E_k e_k =0$
%$\Sigma_k F_k f_k =0$, $X(\zeta) = 
%Any $(E_1, E_2), (F_1, F_2)$ can be realized because $D \Phi_j$ is inveritble.
%$D_{\overline  w(\zeta)}(\Phi_k)$ is invertible
%for any $\zeta$, so we can solve the equation 
%$$(D_{\overline w(\zeta)}\Phi_k)(\overline X(\zeta)) = (E_1,E_2) \text{ or } 
%(F_1,F_2) .$$
%In particular the kernal has dimension 2.
\end{proof}

\begin{lemma}\label{index} %For any finite area $J$-holomorphic triangle $w:V \into Z$ with the usual boundary conditions
For any $w \in \mathcal M$, %the linearized operator  $D(\overline \partial_J)(w): T_w(B) \into \mathcal{E}_w$ has 
$index (D(\overline \partial)(w)) =2$. 
\end{lemma}

\begin{proof}%[Proof of lemma \ref{index}] 
Let $w \in \mathcal{M}$.   
For this proof we work in $T^*S^3=\{(u,v) \in \bbR^4 \times \bbR^4: |u|=1, u \cdot v =0\}$ rather than $Z$. 
Hence, we compose $w: V \into Z$ with $\mu:Z \into T^*S^3$ and denote the result by $w:V \into T^*S^3$ as well.
The Lagrangians $\widetilde W_k \subset Z$ become $W_k \subset T^*S^3$.
Also, we sometimes write index formulas involving $n$, 
where $\dim  M = 2n$ (in our case $n=3$, of course).
\\
\newline Fix a symplectic bundle isomorphism of  $E = w^*(T(T^*S^3)) \into V$ with $V \times \bbC^3$.
In this proof $\bbC$ is always equipped with $dy \wedge dx$, $x+iy \in \bbC$; this is
%(We will choose a particular trivialization later on. Also, we take $dy \wedge dx$ on $\bbC$ because $Z \cong 
because $T^*S^3$ has symplectic structure $\Sigma_i dy_i \wedge dx_i$ in the local coordinates $(x,y) \mapsto \Sigma_i y_i dx_i$. We may assume that the Lagrangian boundary conditions $F_k = w^*(TW_{2k})\into I_k$ meet at angle $\pi/2$ by doing a small homotopy which does not affect the index. Let $H = \{z \in \bbC : Im\,  z \geq 0\}$
% $H$, which is biholomorphic to the disk with one puncture  on the boundary. %which is biholomorphic to the upper half plane. 
%Actually we will take two copies of $\overline H$ and one copy of $H$. The former have outgoing punctires and the latter has an incoming puncture. One can think of $\overline H$ as the lower half plane. 
and let $E_k = \bbC^3 \times H \into H$, $k=0,1,2$.
For each $k$, $F_k \subset \bbC^3 \times  I_k$ is asymptotic at each end of $I_k$ to a fixed Lagrangian subspace. At each vertex $\zeta_k$, where $k=0,1,2$ is modulo 3, set 
$$F_k^- = \lim_{\zeta\rightarrow \zeta_{k-1}} (F_{k-1})_\zeta, \, \, F_k^+ = \lim_{\zeta\rightarrow \zeta_{k}} (F_{k})_\zeta.$$
Now since $F_k^-$ and $F_k^+$ intersect in an $n-1 = 2$ dimensional
subspace say $T_k = F_k^- \cap F_k^+$, we can form the splittings
$$F_k^- = T_k \oplus L_k^-, \, F_k^+ = T_k \oplus L_k^+,$$
where $\dim  L_k^{\pm}= 1$. %is a one dimensional Lagrangian subspace of $\bbC^3$.
We equip $E_k \into H$ with 
the Lagrangian boundary condition $\widehat F^k \into \partial H = \bbR$
which on $[0,1]\subset \bbR$ is given by 
%the short rotational homotopy between $F_k^-, F_k^+ \subset \bbC^3$ given by 
$$\widehat F^k_t   = T_k \oplus e^{\frac{\pi i t}{2}}L_k^-, \, t \in [0,1]$$
and which is constant outside $[0,1]$. 
%(We will see later when we compute the Maslov index that this the correct direction to rotate.)
%Let $\overline \partial_H$ denote the standard $\overline \partial $
%operator on $H$,  acting on sections of $E_k \into H$,
%$k=0,1,2$. Similarly, denote by $\overline \partial_V$ the standard $\overline \partial $ operator on $V$, acting on sections of $\bbC^3 \times V \into V$.
Given $\Omega \subset \bbC$, let $\overline \partial_\Omega$ denote the the standard $\overline \partial $ operator on $\Omega$.
Since the asymptotic data of $(E,F,\overline \partial_V)$ at the punctures $\zeta_k$ matches up with that of $(E_k, \widehat F_k, \overline \partial _H)$, $k=0,1,2$, we can glue these bundles and operators together (using the strip like ends). The result is 
the trivial bundle $\bbC^3 \times D^2  \into D^2$, equipped with a certain Lagrangian boundary condition 
\begin{gather*}
F \subset \bbC^3 \times \partial D^2, \text{ and } 
\overline \partial_{V} \# \overline \partial_{H} \#  \overline 
\partial_{H} \# \overline \partial_{H} \cong \overline \partial_{D^2}. 
\end{gather*}
The standard formula for the index of the right-hand side is $n+ \mu =3 +\mu$, where $\mu$ is the Maslov index of $F$ (see \cite[lemma 11.7]{S08}). %, or \cite{AH} Appendix A5, for the case $M = \bbC^n$.)
This, together with the gluing formula (\ref{indexformula}), shows that 
$$\index(\overline \partial_{V}) + 3 \, \index(\overline\partial_{H}) -3(n-1) = n+ \mu.$$
%Note that $\index(\overline \partial_{H}) = \index(\overline \partial_{\overline H})$ because (?) [I think it's likely b/c must be $\pm$, but sum equals $n+ $ Maslov, which is typically nonzero; also Seidel seems to implicitly asume this in his proof of lemma 11.11].
Then, to compute $\index(\overline \partial_{H})$ we glue together two copies of 
$\overline \partial_{H}$ %and $\overline \partial_{\overline H}$ 
and apply the gluing formula again to get 
$$2\index(\overline \partial_{H}) -(n-1)  = n+ \mu'$$
where $\mu'$ denotes the Maslov index of the loop $F'$ of Lagrangian subspaces in $\bbC^3$  obtained by combining the path $\widehat F_k$ from $F_k^-$ to $F_k^+$ and 
from $F_k^+$ back to $F_k^-$. (This is the same for all $k$.)
To finish the proof we will compute $\mu$, $\mu'$ and show they are both $-1$; this will imply 
$$\index(\overline \partial_{V}) = n +  \mu = 3-1.$$
%To do this calculation we inspect the Lagrangians more closely in a suitable chart. 

To compute $\mu$, $\mu'$, we  trivialize $E=  w^*T(T^*S^3)\into V$ by taking a 
chart $\varphi: U \into S^3$ such that $T^*(\varphi(U))$ contains the image of $ w$, 
and for which the Lagrangians $W_k$ become particularly simple.
According to Proposition \ref{uniqueness}, $w= w_{e_0 f_0}$ for some 
$$e_0 = (\cos x_0, \sin x_0, 0,0)\in S^3,\phantom{bbb}  f_0 = (0,0,\cos y_0, \sin y_0) \in S^3.$$ 
Let 
$$\varphi(\theta, x,y) = (\cos \theta e, \sin \theta f),$$ 
where $e = (\cos x, \sin x), 
f = (\cos y, \sin y)$. Here $|x-x_0|< \epsilon$, $|y-y_0| < \epsilon$, $\theta \in (-\epsilon, \pi/2-\epsilon)$ where $0< \epsilon< \pi/2$ and we assume 
that $(\epsilon, \pi/2-\epsilon)$ contains the edge of $\Delta$ lying in $\Gamma_2$.
$U \subset \bbR^3$ is the set of such $(x,y,\theta)$ and 
$\varphi$ is an embedding because  $\theta \neq n \pi/2$ 
for any $n \in \bbZ$. Now
$\varphi$ gives rise to a symplectic chart
$$\widehat \varphi : T^*U \into T^*\widetilde U \subset T^*S^3,$$ where $\widetilde U = \varphi(U)$ and  we think of $T^*U$ as a subset of $\bbC^3$. %where $\bbC$ has the symplectic form $dy \wedge dx$, $x +iy \in \bbC$. 
Since $w$ takes values in $T^*\widetilde U \subset T^*S^3$, $E \into V$ has a trivialization  $E\cong \bbC^3 \times V$ coming from the trivialization of $T(T^*U) \subset T(\bbC^3)$.
\\
\newline %Now we analyze the Lagrangian boundary conditions
%$$F_k \subset E|I_k \into I_k$$ given by $F_k = ( w|I_k)^*(T W_k)$, $(k,k) =(0,2),(1,4),(2,0)$.
Consider $T^*K_{ef} \subset T^*S^3$ for any fixed  
$$e = (\cos x', \sin x', 0,0)\in S^3, \phantom{bbb} f = (0,0,\cos y', \sin y') \in S^3,$$ 
with $|x'-x_0| < \epsilon, |y'-y_0|< \epsilon$. Then $\widehat \varphi^{\,-1}(T^*K_{ef})$
is the set of 
\begin{gather*}
(x,y,\theta)+i(a,b,c) \in \bbC^3  \text{ satisfying}\\
x=x', y=y',  a=0,b=0,   \theta \in (\epsilon, \pi/2-\epsilon), c \in \bbR.
\end{gather*}
%Let us denote this set by $T^*U_{ef}$. There is an obvious identification
%Set $\alpha_{ef} = \widehat \varphi^{\,-1} \circ \sigma_{ef}: \bbR/2\pi\bbZ \times \bbR \into \widehat \varphi^{\,-1}(T^*K_{ef})$.
%namely $(\theta,\lambda) \mapsto ((x',y',\theta),\lambda d\theta)$.
And $\widehat \varphi^{\,-1}(\Gamma_k)$ is given by the additional constraint
$(\theta,c) \in \Gamma_k$. Now since $W_{2k} = \cup_{e\in K_+,f \in K_-} \sigma_{ef}(\Gamma_k)$, $k=0,1,2$, it follows that
%where $\Gamma_k \subset \bbR/2\pi\bbZ \times \bbR$, we have 
%is a certain curve (see \S \ref{vanishingcycles}), we have
$$\widehat \varphi^{-1}(W_k) =   (\bbR^2 \times \Gamma_k) \cap (T^*U).$$
%$$\{(x,y,\theta) + i(a,b,c) : x, y \in \bbR,  a=0,b=0, (\theta,c) \in \Gamma_k\}.$$
%Thus, regarding $T^*U \subset \bbC^3$, we have
%in terms of the identification of $T^*U$ with a subset of $\bbC^3$,
%$$\varphi^{-1}(W_k) = (\bbR^2 \times \Gamma_k) \cap (T^*U).$$ 
Therefore we have, for $\zeta \in I_k$, 
$$(F_k)_\zeta = \bbR^2 \times T_{\widetilde w(\zeta)}(\Gamma_{2k}),$$
where $\widetilde  w= \sigma_{ef}^{-1}\circ  w_{ef}: V \into \bbR/2\pi\bbZ\times \bbR$.
\\
\newline Thus we are reduced to computing the Maslov index $\mu$ of the loop
of Lagrangians given by $
T(\Gamma_0), T(\Gamma_2), T(\Gamma_4)$ in $T((-\epsilon, \pi/2-\epsilon) \times \bbR) \subset T(\bbC)$.
%where $\bbC = (\bbC,dy \wedge dx)$, $x+iy \in \bbC$. 
More precisely,  we must take the short homotopy between the two tangent spaces each vertex of $\Delta$. (This corresponds to gluing on $H$  with its Lagrangian boundary conditions.) 
To see in what direction the loop is traversed,
note that $\widetilde w(I_k) \subset \Gamma_{2k}$, $k=0,1,2$.
%\widetilde w(I_1) \subset \Gamma_2, \widetilde  w(I_2) \subset \Gamma_0,$ 
%because $ w$ has the corresponding boundary conditions
%$and $\alpha_{e_0f_0}^{-1} \circ \varphi^{-1}$ of course maps $W_k$ to $\Gamma_k$,$k=0,2,4$. 
This means that as we traverse $\partial V$ via $t \mapsto e^{2\pi i t}$ in the counter-clockwise direction, $t \mapsto (\widetilde w| \partial V)(e^{2\pi i t})$ traverses the boundary of $\Delta \subset  \bbR/2\pi\bbZ \times \bbR $ in the clockwise direction (see figure \ref{CorrectionGammaIIbig(mu2)}). 
Inspection of figure \ref{CorrectionGammaIIbig(mu2)} shows that the resulting loop of Lagrangians
in $\bbC$ is homotopic to the counter-clockwise loop $t \mapsto e^{2\pi i t}\bbR$. This has Maslov index 1
with respect to $dx \wedge dy$; %(see \cite{AH}, appendix A5); 
therefore $\mu=-1$ with respect to  $dy \wedge dx$.  To compute $\mu'$, for each vertex $v$, we homotope the two tangent spaces at $v$ slightly so they meet at angle $\pi/2$; then the short homotopy doubled up to give a loop in $\bbC$ from the first tangent space back to itself is also homotopic to $t \mapsto e^{2\pi i t}\bbR$ in $(\bbC, dy \wedge dx)$, so $\mu' =-1$ as well.\end{proof}

\section{Morse-Bott Lagrangian Floer homology}\label{MorseBott}

We define the Morse-Bott Floer homology groups in two 
special cases sufficient for $HF(V_4,V_0)$, $HF(V_4,V_2^j)$, 
$HF(V_2^j,V_0)$, the latter two being similar. In these special cases we 
also define the triangle product.  
After that we briefly explain
the above definitions  and the continuation map in the general case, and we make some remarks on  the analogous homology theory for Morse-Bott functions. 
In the last section we discuss the underlying Fredholm theory and give an index formula for gluing.
There are several technical results one would need to give a full 
treatment of the whole theory, notably exponential convergence of holomorphic strips at the ends; 
we give precise statements for some of these, but do not prove them here. 
\\
\newline
\emph{References.} The basic starting point for Morse-Bott Floer homology is 
\cite{Poz}. Our treatment is modeled on \cite{Fr}, which in turn uses the basic idea of
\cite{PSS}, \cite{Schwarz} to use holomorphic curves (or in his case gradient lines) with gradient lines attached as a way of encoding Morse cycles in the intersection components. 
 All the foundational issues regarding exponential convergence at the ends, regularity, gluing, and compactness are treated in \cite{BO} but in the case of holomorphic cylinders rather than strips. A similar model is used in \cite{B}, with similar results. 
%It is expected that all that goes through 
For the case of holomorphic strips, \cite{RS} proves exponential convergence in the transverse case, and  \cite{BC}  addresses all the other issues in a setup  essentially equivalent to the case $L_0 = L_1 =L$ (but their situation is more complicated because they work outside the exact setting). See also \cite[\S 8l]{S08}, Ch.II (8l), which sketches 
the theory in a TQFT context, and of course \cite{FOOO} treats  the Morse-Bott case as well.
% which treats this topic in full, but I confess that I was not able to penetrate that work.

\subsection{Conventions}\label{conventions}
%An almost complex structure 
An almost complex structure $J$ is compatible with the symplectic form $\omega$ if  $\omega(v,Jv) >0, \, v \neq 0.$
The  symplectic structure on $T^*N$ will be $\Sigma_j dy_j \wedge dx_j$
in the standard local coordinates $(x,y) \mapsto \Sigma_j y_j dx_j,$ and for its primitive 1-form
we will take $\Sigma_j y_j dx_j$. 
 Given $H: [0,1] \times M \into \bbR$, the Hamiltonian vector field $(X_H)_t$ is defined by
$\omega(v,(X_H)_t) = dH_t(v).$
Notice that if $H: T^*N \into \bbR$ is of the form $H(x,y) = h(x)$ for some $h \in C^\infty(N)$,
then $X_{-H}$ has flow $\phi_t^{-H}$ such that $\phi_1^{-H}(N) = \Gamma(dh)$. \label{-H}
%\newpage 

\subsection{Basic Floer theory notation} \label{notation} Let $(M,\omega, \theta)$ be an exact symplectic manifold, $\omega = d\theta$.
We say the boundary of $M$ is of contact type if $\theta|\partial M$ is a contact form for $\partial M$, and we say $\partial M$ convex if the Liouville vector field  $X_\theta$ defined by $\omega(X_\theta, \cdot) = \theta$
is such that $-X_\theta$ points strictly \emph{inwards} (towards $M$) along $\partial M$. 
We will say that an almost complex structure $I$ on $M$ is of contact type near the boundary 
if it is invariant under $-X_\theta$ and satisfies $J(R_\theta) = -X_\theta$, 
where $R_\theta$ is the Reeb vector field of $(\partial M, \theta)$ in a collar neighborhood
of $\partial M$ (see \cite[p. 94]{S08}). This implies that $I$ makes the boundary of $M$ $I-$convex, which means
any $I-$holomorphic map $w:\Sigma \into M$ cannot meet  $\partial M$ at an interior point of  $\Sigma$, unless it is constant (again, see \cite[p. 94]{S08}, or \cite[lemma 2.4]{M}. 
\\
\newline
Let $\mathcal{J}(M)$ denote the space of $\omega$-compatible
%\footnote{Our convention is: $\omega(v,Jv) >0$, for $v \neq 0$.} 
almost complex structures on 
$M$. Fix $I \in \mathcal{J}(M)$  which makes $\partial M$ 
$I$-convex.  Given any smooth manifold $\Sigma$, let $\mathcal{J}(\Sigma,M,I)$ 
denote the space of smooth families $J_\zeta \in \mathcal{J}(M)$, $\zeta \in \Sigma$, such that there is a neighborhood $U$ of  $\partial M$ with $J_\zeta(p) =I$ for all $p \in U$, $\zeta \in \Sigma$.
Set $\mathcal{H}= C^{\infty}([0,1] \times M, \bbR).$ Given $H \in \mathcal{H}$, let $(X_H)_t$, $t \in [0,1]$ be the Hamiltonian vector field of $H$, %\footnote{Our convention is: $\omega(\cdot, (X_H)_t(p)) = dH_t(p), p \in M$.}
and let $\phi^H_t$, $t \in [0,1]$, denote isotopy given by $X_H$.
Given two exact Lagrangian submanifolds $L_0, L_1$, let
$\mathcal{M}(L_0,L_1,J,H)$ %\label{mathcal{M}(L_0,L_1,J,H)} 
denote the space of
$u \in C^{\infty}(\bbR\times [0,1],M)$
which satisfy
$$u(\{0\}\times\bbR) \subset L_0, u(\{1\}\times \bbR) \subset L_1,$$
$$\overline{\partial}_{(H, J)}(u) = \partial_s u +J_t (\partial_t u - 
X^H_t(u)) = 0, \phantom{bbb}\int_\bbR \int_0^1 \omega( \partial_s u, J_t \partial_s u) dt ds < 
\infty.$$
We have the usual equivalence of moduli spaces:
\begin{equation} \label{isomorphism}
\mathcal{M}(L_0,L_1,H,J) \into 
\mathcal{M}(L_0,(\phi^H_1)^{-1}(L_1),0,\{(\phi^H_t)^*J_t\})
\end{equation}
given by $u \mapsto \widetilde u$, where
$\widetilde u(s,t)= (\phi^H_t)^{-1}(u(s,t))$. 
For $j=0,1$, let $f_{L_j}\in C^\infty(\widetilde L_j)$,  be such that $\theta|L_j = df_{L_j}$. The action functional for $(L_0,L_1,H)$ is 
\begin{gather}
A(y) = -\int y^*\theta +  \int_0^1 H(y(t))dt+ f_{L_1}(y(1)) -  f_{L_0}(y(0)), \label{action}\\
y \in C^\infty([0,1],M), \, y(0) \in L_0, \, y(1) \in L_1. \notag
\end{gather}
%For any $u \in C^{\infty}(\bbR\times [0,1],M)$ with $u(\{j\}\times \bbR) \subset L_j$, $j=0,1$
%and $\lim_{s \rightarrow -\infty} u(s,t) = y_0(t), \lim_{s \rightarrow -\infty} u(s,t) = y_1(t), \, t\in [0,1]$
For each $u \in \mathcal{M}(L_0,L_1,J,H)$, condition \ref{convergence} below, plus (\ref{isomorphism}),
implies that 
$$\lim_{s \rightarrow -\infty} u(s,t) = y_0(t), \lim_{s \rightarrow -\infty} u(s,t) = y_1(t)$$ exist uniformly in $t\in [0,1]$. The action functional satisfies
\begin{gather*}
\int_{\bbR} \int_{[0,1]} |\partial_s u|_{g_t}^2dtds = A(y_1) - A(y_0), 
\text{where }g_t(v,w) = \omega(v,J_tw)
%, \text{ where }\\
%\lim_{s \rightarrow -\infty} u(s,t) = y_0(t), \lim_{s \rightarrow -\infty} u(s,t) = y_1(t), \, t\in [0,1],
 \end{gather*}
\begin{comment}
To define the usual Floer homology one assumes in particular
that
$$\phi^H_1(L_0) \text{ intersects } L_1 \text{ transversely}.$$
The main point of Morse-Bott Floer homology is that we relax this 
assumption to
$$\phi^H_1(L_0) \text{ intersects } L_1 \text{ in a Morse-Bott 
fashion}.$$
Here Morse-Bott intersection for $L_0, L_1$ means
$$T_x(L_0 \cap L_1) = T_xL_0 \cap T_xL_1, x \in L_0\cap L_1.$$
If $L_1 = \Gamma(df) \subset T^*L_0$, then this is equivalent to $f$ 
being Morse-Bott. This condition is also called clean intersection.
\\
\newline 
\end{comment}
%Below we will explain the definition of Morse-Bott Floer homology in 
%two special cases. 
In the two special cases below, and in the general case, we always
pick $H$ so that  $\phi^H_1(L_0)$ and $L_1$ have Morse-Bott 
intersection. ($L_0,L_1$ have Morse-Bott (or clean) intersection if 
$T_x(L_0 \cap L_1) = T_xL_0 \cap T_xL_1, x \in L_0\cap L_1.$
If $L_1 = \Gamma(df) \subset T^*L_0$, this means $f$ 
is Morse-Bott.) Then we set
$$Y= L_0 \cap (\phi^H_1)^{-1}(L_1),$$
and choose
a Morse-Smale pair $(f,g)$ on $Y$; we denote the restriction to the 
component $C\subset Y$ by $(f_C, g_C)$.
 
\subsection{A convergence condition} Here we state without proof a basic convergence condition. %This is not known, but for references on similar results see\S \ref{sectionexponentialconvergence}.

\begin{condition}\label{convergence}
Assume $L_0 \cap L_1$ is Morse-Bott. Then, for any
$$u \in \mathcal{M}(L_0,L_1,J), \text{ we have }$$
$$\lim_{s \rightarrow \pm \infty} u(s,t) = p_{\pm} \text{ (uniformly in 
$t$)}$$
for some $p_{\pm} \in L_0 \cap L_1$.
\end{condition}

Let
$$ev_{\pm}: \mathcal{M}(L_0,L_1,J) \into L_0 \cap L_1$$
denote the evaluation maps $ev_{\pm}(u) = \lim_{s\rightarrow \pm \infty} 
u(s,t)$. 
Given components $C_-, C_+$ of $L_0 \cap L_1$, set
$$\mathcal{M}(L_0,L_1,J; C_{-}, C_{+}) = ev_{-}^{-1}(C_{-}) \cap 
ev_{+}^{-1}(C_+).$$ %\label{mathcal{M}(L_0,L_1,J; C_{-}, C_{+})}}
If $p_- \in C_-, p_+ \in C_+$, similar notation will denote $ev_{-}^{-1}(p_{-}) \cap 
ev_{+}^{-1}(p_+)$. 

\subsection{Definition of $\partial$ for $CF(V_4,V_2^j)$, $CF(V_2^j, V_0)$ (special case I)}\label{specialcaseI} 
Using the isomorphism (\ref{isomorphism}) we may reduce to the case $H=0$.
So assume $H=0$, and $L_0 \cap L_1$ is Morse-Bott. Further, Assume $L_0 \cap L_1$ consists 
of exactly two components 
$C_{-}, C_{+}$, where the action functional for $(L_0, L_1, H=0)$ 
satisfies
\begin{gather}\label{ActionI}
A(C_+) >A(C_-).
\end{gather}
This is sufficient to define $HF(V_4,V_2^j)$ and $HF(V_2^j,V_0)$; in those 
cases
$C_+ = \Sigma_{40}^j, C_- = K_+^j$ and $C_+ = \Sigma_{20}^j, C_- = K_+^j$, 
respectively. One can see $A$ satisfies $A(C_-) < A(C_+)$ in both cases 
because $A_{20} = f_{20}$ and $A_{42} = -f_{42}$  
(see the end of  \S \ref{localstrips}).
\\
\newline We have two evaluation maps
$$ev_{\pm}: \mathcal{M}(L_0,L_1,J) \into C_{\pm}.$$
There are two types of configurations we consider.
If $x,y \in C_{\pm}$,  we set
$$\mathcal{M}(x,y) = U(x) \cap S(y)$$
If $x \in Crit(f_{C_+}), y \in Crit(f_{C_-})$,   %we set
$$\mathcal{M}(x,y) = ev_{-}^{-1}(U(x)) \cap ev_{+}^{-1}(S(y)).$$
(If $x \in Crit(f_{C_-}), y \in Crit(f_{C_+})$ then $\mathcal{M}(x,y)=\emptyset$, because of (\ref{ActionI}).) See figure \ref{figureCascades}. 
\begin{figure}
\begin{center}
\includegraphics[width=3in]{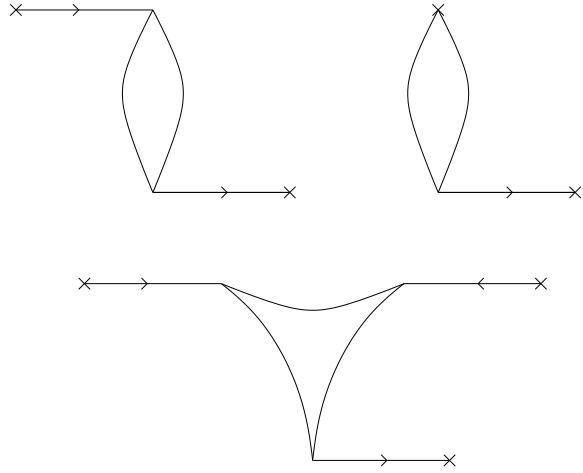} 
%\put(-175, 5){$-1$}
%\put(-5,5){$1$}
%\put(-200,100){$(0,0)$}
\caption{Schematic of Floer trajectories with gradient lines attached in special cases I, II (top left, right) 
and a holomorphic triangle with gradient lines attached.}
\label{figureCascades}
\end{center}
\end{figure}
%One can show
%that the linearized $\overline \partial$ operator is Fredholm by using Sobolev spaces involving exponential weights.
%(See lemma \ref{lemmaFredholm} for the case of holomorphic 
%triangles.)
We call $J$ regular if the linearized $\overline 
\partial$ operator %for $\mathcal{M}(L_0,L_1,J)$
is surjective. We call $((f,g); J)$ regular if 
$J$ is regular, and if the evaluation map
$$ev_{+} \times ev_{-}: \mathcal{M}(L_0,L_1,J) \into C_{-} \times C_{+}$$
is transverse to $U(x) \times S(y)$ for all $(x,y) \in Crit(f_{C_-}) 
\times Crit(f_{C_+})$.
Given fixed $(f,g)$ as above, we have that for generic $J$
the data $((f,g); J)$ is regular. (See \cite[pages~66-67]{Fr}, Appendix A.) 
For regular data $\mathcal{M}(x,y)$ is a smooth manifold. It has a free 
$\bbR$ action and we denote the quotient by
$\mathcal{M}^*(x,y)$.
 %and $\mathcal{M}^*(x,y)$ denotes the quotient by $\bbR$ if $x \neq y$; we 
Define
$\mathcal{M}^*(x,y) = \emptyset$ if $x =y$. For $d\geq 0$, let 
$\mathcal{M}^*_d(x,y)$ denote the union of the components of dimension $d$. 
We define the differential by
$$\partial x = \Sigma_y \# \mathcal{M}^*_0(x,y) \, y$$
Here $\#S$ denotes the number of elements mod 2 of a finite set $S$.
$\mathcal{M}^*_d(x,y)$ will be compact because of the usual compactness statements for holomorphic strips and gradient lines. One can see $\partial^2=0$  by looking at $\partial \mathcal{M}^*_1(x,y)$ (and applying gluing theorems): The holomorphic strip cannot break because there are only two intersection components, but either of the two gradient lines can break, or the holomorphic strip could collide with a critical point of $f$. (See also sketch of $\partial^2 =0$ in \S \ref{generalMorseBott}.)

\subsection{Definition of $\partial$ for $CF(V_4,V_0)$ (special case II)}
Next we discuss the case where $H=0$ and $L_0 \cap L_1$ is Morse-Bott and 
consists
of several isolated points $p_1, \ldots p_l$, together with several 
positive dimensional components $C_1, \ldots, C_k$ such that the 
action functional for $(L_0, L_1,H=0)$ satisfies
\begin{gather} \label{AII}
A(p_i) >  A(C_1) = \ldots = A(C_k),
\end{gather}
for all $i=1, \ldots, l$. This is sufficient to define $HF(V_4, V_0)$;  
in that case $C_j =\Sigma_{40}^j$ and the $p_i$ are the 
isolated critical points of $f_{40} : N_{40} \into \bbR.$
We know $A = A_{40}$  satisfies (\ref{AII})
because  $A_{40} = -f_{40}$ (see the end of  \S \ref{localstrips}).
\\
\newline There are three types of configurations we consider.
If $x =p_i$ for some $i$ and $y =p_j$ for some $j$
then we set
$$\mathcal{M}(x,y) = \mathcal{M}(L_0,L_1,J; p_i, p_j).$$
If $x,y \in C_j$ for some $j$, then we set
$$\mathcal{M}(x,y) = U(x) \cap S(y).$$
%$A(C_i) = A(C_j)$, and $A$ strictly decreases along nonconstant 
holomorphic strips. (If $x \in C_i$, $y \in 
C_j$, and $C_i \neq C_j$, then there can be no holomorphic strip 
asymptotic to $x$ and $y$ because the action functional satisfies $A(C_i) 
= A(C_j)$.)
If $x =p_i$ for some $i$ and $y \in C_j$ for some $j$
then we set
$$\mathcal{M}(x,y) = ev_{-}^{-1}(p_i) \cap ev_{+}^{-1}(S(y)).$$
(Note that  $ev_{-}^{-1}(U(y)) \cap ev_+^{-1}(p_i)$ is empty 
because of (\ref{AII}).)
 Regularity for $J$ and  $((f,g),J)$ are defined in the same  way
as the last section and the differential is defined by the same 
formula as before. For $\partial^2 =0$, the first two types of configurations are exactly as in the usual Floer 
homology and Morse homology, so $\mathcal{M}^*_1(x,y)$ compactifies in the familiar way in those cases. 
In the third type of configurations either the Holomorphic strip can break, or the gradient line can break, 
or the holomorphic strip can collide with a critical point of $f$. (See also the sketch of $\partial^2=0$ in \S \ref{generalMorseBott}.)

\subsection{Definition of the triangle product for $\widetilde V_4, \widetilde V_2^j, \widetilde V_0$} \label{deftriangleproduct}

Recall that $\widetilde V_4 \cap \widetilde V_2^j$, $\widetilde V_2^j \cap \widetilde V_0$, $\widetilde V_4 \cap \widetilde V_0$, $j=1, 
\ldots, k$ are all of Morse-Bott type, so we have set $H=0$ in each case. 
Assume we have picked regular data: 
\begin{gather}\label{data}
((h_{42}^j, g_{42}^j), \widetilde J_{42}^j), \, ((h_{20}^j, g_{20}^j), \widetilde J_{20}^j), \, 
((h_{40}, g_{40}),\widetilde J_{40})
\end{gather}
for $(\widetilde V_4, \widetilde V_2^j)$, $(\widetilde V_2^j, \widetilde V_0)$, $(\widetilde V_4, \widetilde V_0)$, respectively. In fact we assume that
$$\widetilde J_{42}^1 = \ldots = \widetilde J_{42}^k, \text{ and } \widetilde J_{20}^1 = \ldots = \widetilde J_{20}^k,$$
and we drop the $j$'s from the notation and denote these $\widetilde J_{42}$,
$\widetilde J_{20}$. Now we want to define, for each $j$, the triangle product
$$\mu_2: HF(\widetilde V_4, \widetilde V_2^j) \otimes HF(\widetilde V_2^j,\widetilde V_0) \into HF(\widetilde V_4, \widetilde V_0).$$
For that, let $D^2$ be the closed unit disk in $\bbC$ with the standard 
complex structure and let $\zeta_0, \zeta_1,\zeta_2$ denote three distinct 
boundary points
labeled in counter-clockwise order, and set
$$V= D^2 \setminus \{ \zeta_0, \zeta_1,\zeta_2 \}.\label{V}$$
Denote by $j_V$ the restriction of the standard complex structure.
Let $I_0$,$I_1$,$I_2$ be the three
boundary components of $V$ labeled in counter-clockwise order with
$\zeta_0$ and $\zeta_1$ on the boundary of $I_0$. We equip $V$ with 
strip-like ends
near puncture, which means we fix proper holomorphic embeddings of the 
half-strip
parameterizing disjoint neighborhoods of each puncture. We regard 
$\zeta_1, \zeta_2$ as  ''incoming'' points and $\zeta_0$ as an ''outgoing''
point. This means that the strip-like ends for $\zeta_0,\zeta_1, \zeta_2$ 
are respectively maps defined on the negative and positive half strips as 
follows
$$\epsilon_0: Z_{+} = [0,+\infty) \times [0,1] \into V, \, \epsilon_1, \epsilon_2: Z_{-} = (-\infty,0]\times [0,1] \into V.$$
Let
 $J \in \mathcal{J}(V,M,I)$. Assume that $J$ is compatible with $\widetilde J_{40}, \widetilde J_{42}, \widetilde J_{20}$
in the sense that
\begin{align} \label{Jcompatible}
 J_{\epsilon_0{(s,t)}} = (\widetilde J_{40})_t, \,  % \nonumber
J_{\epsilon_1{(-s,t)}} = (\widetilde J_{42})_t, \,   
 J_{\epsilon_2{(-s,t)}} = (\widetilde J_{20})_t. %\nonumber
\end{align}
for all $j$, and all $s>0$, $t \in [0,1]$. 
Let $\mathcal{M}(\widetilde V_4, \widetilde V_2^j, \widetilde V_0, J)$ 
%\label{mathcal{M}(\widetilde V_4, \widetilde V_2^j, \widetilde V_0, J)} 
denote the space of
$w\in C^{\infty}(V,M)$ satisfying
$$w(I_0) \subset  \widetilde V_0, w(I_1) \subset \widetilde V_2^j, w(I_2) \subset \widetilde V_4,$$
%!! changed this Aug 19, 2009. Look at HF_*(\widetilde V_4,\widetilde V_0) strips: \widetilde V_4 is on bottom, 
%\widetilde V_0 on top; both on the right atteh output (this is consistent with proof 
% of existence of trialngles too: start at "Note that I_0, are in counter 
%clockwise order but Gamma'a label
%in clockwise order.."; perhaps the old convention was from
%back when I was doing HF^* (coho)
$$Dw(\zeta) j_V = J_\zeta Dw(\zeta), \text{ for all } \zeta \in V, \, \text{ and }\int_V w^* \omega < \infty.$$
By assumption \ref{convergence}, we have evaluation maps
$$ev_j: \mathcal{M}(\widetilde V_4, \widetilde V_2^j, \widetilde V_0, J) \into Y_j,\, j =0,1,2, \text{ where}$$
$$Y_0= \widetilde V_4 \cap \widetilde V_0, \, Y_1 = \widetilde V_4 \cap \widetilde V_2^j, \, Y_2 = \widetilde V_2^j \cap \widetilde V_0,
\text{ and } ev_j(w) = \lim_{\zeta \rightarrow \zeta_j} w(\zeta).$$
Given components $C_j \subset Y_j$, we set
$$ \mathcal{M}(\widetilde V_4, \widetilde V_2^j, \widetilde V_0, J; C_0, C_1, C_2) = ev_0^{-1}(C_0)\cap 
ev_1^{-1}(C_1) \cap ev_2^{-1}(C_2).$$
Now fix $j$ and set
$$C_0 = \Sigma_{40}, C_1 = \Sigma_{42}^j, C_2 = \Sigma_{20}^j.$$
We now specialize further and assume that $J_\zeta|D(T^*\widetilde V_2^j) = 
J_{\bbC}$, for all $\zeta \in V$, $j=1, \ldots, k$. (Recall that 
$D(T^*\widetilde V_2^j) \subset M_0 \subset M$, and 
there is positive distance from $D(T^*\widetilde V_2^j)$ to $\partial M$, so we can also have $J_\zeta = I$ near $\partial M$.)
The condition $J_\zeta|D_r^{g_2}(T^*\widetilde V_2^j) = J_{\bbC}$  means 
that
$\widetilde J_{40}, \widetilde J_{42},\widetilde J_{20}$ have to satisfy a related condition, but let 
us assume that is so. Then, Proposition \ref{containment} implies that 
every $w \in \mathcal{M}(
\widetilde V_4, \widetilde V_2^j,\widetilde V_0; J)$ necessarily satisfies
$$ev_j(w) \in C_j, \, j=0,1,2.$$
Given $x_0 \in Crit(f_{40}|C_0)$, $x_1\in Crit(f_{42}^j|C_2)$, $x_2 \in 
Crit(f_{20}^j|C_0)$, we let
$$\mathcal{M}(x_0, x_1,x_2) = ev_0^{-1}(S(x_0)) \cap ev_1^{-1}(U(x_1)) \cap 
ev_2^{-1}(U(x_2)).$$
We call $J$ \emph{regular} if the linearized $\overline \partial$ operator is 
surjective. Letting $\star$ stand for the data (\ref{data}), we call $(J,\star)$ regular if 
$J$ is regular, the data (\ref{data}) is regular, and 
$$ev_0 \times ev_1 \times ev_1:
\mathcal{M}(\widetilde V_4, \widetilde V_2^j, \widetilde V_0, J) \into Y_0 \times Y_1 \times Y_1$$ is 
transverse to
$$U(x_2) \times U(x_1) \times S(x_0)$$
for all $x_0 \in Crit(f_{40}|C_0)$, $x_1\in Crit(f_{42}^j|C_2)$, $x_2 \in 
Crit(f_{20}^j|C_0)$. %Given $J$ is regular for a  generic choice; %(but we do not actually use this fact).
In that case $\mathcal{M}(x_0,x_1,x_2)$ is a smooth manifold and we denote 
the union of the components of dimension $d \geq 0$ by 
$\mathcal{M}_d(x_0,x_1,x_2)$.
We define
$$\mu_2(x_1 \otimes x_2) = \Sigma_{x_0} \, \#\mathcal{M}_0(x_0,x_1,x_2)x_0 $$
One can see that $\mu_2$ descends to homology in the usual way: In this 
particular case, a sequence $w_n \in \mathcal{M}_1(\widetilde V_4,\widetilde 
V_2^j, \widetilde V_0; J)$ 
necessarily breaks at one of the gradient trajectories in $C_0, C_1,$ or 
$C_2$; a holomorphic strip cannot break off
for reasons we explain in the next paragraph. 
%Associativity follows as usual by 
%considering
%the one dimensional space of holomorphic 4-gons, with analogous 
%gradient trajectories attached at each vertex, and with appropriate 
%boundary conditions.
\\
\newline In general $\mu_2$ is more complicated: it involves finite 
combinations of gradient lines and holomorphic strips attached at 
each vertex. (See \S \ref{generalMorseBott} for the 
precise definition in the case of the differential.) This is 
because a one parameter family of holomorphic triangles can in 
principle split up in that way (without reaching the boundary of the moduli space). This, however, does 
not come up for the particular Lagrangians $(\widetilde V_4, \widetilde V_2^j, \widetilde V_0)$
because, due to the way the holomorphic strips arise by splitting 
off, they must be attached at the vertices in a way which respects the
ingoing or outgoing nature of the vertex: the $-\infty$ (resp. $+\infty$) 
end of the strip must attach to an outgoing (resp. incoming) vertex.
But in our case there is no %holomorphic strip
$u \in \mathcal{M}(\widetilde V_4, \widetilde V_2^j, \widetilde J_{42})$
which satisfies
$ev_{-}(u) \in \Sigma_{42}^j$, and  $ev_{+}(u) \in K_+^j,$
because $A(\Sigma_{42}^j) > A(K_-^j)$, and similarly at the other two vertices.
%indeed, all such $u$ instead satisfy
%$ev_{+}(u) \in \Sigma_{42}^j, \text{ and } ev_{-}(u) \in K_+^j.$
%One can see this by inspecting the action functional and the definition of 
%$f_{42}^j$.
%A similar statement holds for all $u \in \mathcal{M}(\widetilde V_2^j,\widetilde V_0, 
%\widetilde J_{20}^j),$
%and similarly there are no nonconstant
%$u \in \mathcal{M}(V_4,V_0, \widetilde J_{40}),$
%which satisfy $ev_{+}(u) \in \Sigma_{40}^j$ for some $j=1,\ldots, k$.

\subsection{A sketch of the general case and the continuation map}\label{generalMorseBott}

In this section we sketch the definition of the differential, the triangle 
product and the continuation map in the general case. 
In particular we explain 
why the Morse-Bott theory is isomorphic to the usual theory. Such an 
isomorphism is given by the continuation map, where in the target one 
chooses $H$ so that the Lagrangians are transverse, and $(f,g)$ are 
constant on the resulting finite union of points. Roughly speaking, the main new feature in the general case is that one has to count
configurations of holomorphic strips and gradient lines of arbitrary 
finite length; these show up in all three definitions.
%, and the moduli spaces of configurations of different lengths between 
%two fixed critical points all fit together to form one smooth 
%modulispace.
\\
\newline %As before we pick $H$ so that $\phi^H_1(L_0)$ intersects $L_1$ in 
%a Morse-Bott fashion; we set
%$$Y= L_0 \cap (\phi^H_1)^{-1}(L_1),$$
%and choose
%a Morse-Smale pair $(f,g)$ on $Y$. %$L_0 \cap (\phi^H_1)^{-1}(L_1)$; we
%we denote the restriction to the component $C\subset Y$
%by $(f_C, g_C)$.
We have evaluation maps
$$ev_{\pm}: \mathcal{M}(L_0,L_1,J,H) \into L_0 \cap 
(\phi^H_1)^{-1}(L_1),\phantom{bbb}ev_{\pm}(u) = \lim_{s\rightarrow \pm \infty} \widetilde u(s,t),$$
where we have used the isomorphism $u \mapsto
\widetilde u$ given by (\ref{isomorphism}).
\\
\newline Given critical points $x,y \in Crit(f)$ and $k\geq 0$, we define 
$\mathcal{M}(x,y;k)$, the space of \emph{Floer trajectories with $k$ cascades} from $x$ to $y$, as follows.
For $k=0$, $\mathcal{M}(x,y;0) = U(x) \cap (y) \subset Y$.
%is defined to be the space of 
%$ \gamma \in C^{\infty}(\bbR, Y)$ satisfying 
%$$\gamma'(s) = 
%-(\nabla_g f)(\gamma(s)), \gamma(-\infty) = x, \, \gamma(+\infty) = y $$
For $k\geq 1$, let $\mathcal{M}(x,y;k)$%\footnote{There is a maximum value $k$ for which
%$\mathcal{M}(x,y;k)$ is nonempty because the usual action functional 
%strictly decreases along each nonconstant $\widetilde u \in 
%\mathcal{M}(L_0,(\phi^H_1)^{-1}(L_1),(\phi^H)^*J)$.} 
be the space of nonconstant tuples
$$(u_1, \ldots, u_k) \in \mathcal{M}(L_0,L_1,J,H)^k$$
such that
$$ev_{-}(u_1) \in U(x), ev_{+}( u_n) \in S(y),$$
and such that for each $j$, $1\leq j \leq k-1$, there exists a 
nonconstant
$$\gamma_j \in C^{\infty}(\bbR, Y) \text{ with }
\gamma_j'(s) = -(\nabla_g f)(\gamma_j(s))$$
such that
$$ev_{+}(u_j) = \gamma_j(0), \, ev_{-}(u_{j+1}) = 
\gamma_j(t_j)$$
for some $t_j \geq 0$. 
%Here
%$$\widetilde u_j \in \mathcal{M}(L_0,(\phi^H_1)^{-1}(L_1),(\phi^H)^*J)$$
%as in (\ref{isomorphism}).
\\
\newline The reason we need such a complicated moduli space is the 
following. If one starts with a one parameter family of Floer trajectories 
with exactly one cascade $u$, then $u$ can break at several 
intersection components of $L_0 \cap (\phi_1^H)^{-1}(L_1)$ producing 
several new cascades $u_1, \ldots u_k$; then these $u_j$ can slide
apart
along gradient lines of $f$ and the result is  a typical element of the 
moduli space. All this can happen in a single smooth family: When $u$ breaks and slides, one does not 
reach the boundary of the moduli space; the only way that happens is if
one of the $u_j$ runs into a critical point of $f$, or a gradient line breaks.
\\
\newline $(J,H)$ is called regular if the linearized $\overline \partial_{(J,H)}$ operator is
surjective. %o that $$\mathcal{M}(L_0,L_1,J,H) \cong 
%\mathcal{M}(L_0,(\phi^H_1)^{-1}(L_1),(\phi^H)^*J)$$
%is a smooth manifold.. 
$((f,g); (J,H))$ is called regular if $(J,H)$ is regular and if the product of evaluation maps
%Now one can reformulate the definition $\mathcal{M}_k(x,y)$ in terms of 
%the product of evaluation maps
$$EV = ev^1_- \times (ev^2_- \times ev^2_+) \times \ldots (ev^{k-1}_- 
\times ev^{k-1}_+) \times ev^k_+,$$
%which is a map 
$$EV: \mathcal{M}(L_0,L_1,J,H)^k \into Y \times (Y\times 
Y)^{k-1} \times Y,$$
is transverse to a suitable subset of the target involving gradient lines of $(f,g)$. 
If $H,(f,g)$ are fixed, $((f,g); (J,H))$ is regular for a generic choice of $J$.
\\
\newline As we just described, we actually expect all of these 
spaces ${\mathcal{M}}(x,y;k)$ to fit together into a single smooth 
manifold (with corners)
$${\mathcal{M}}(x,y)= \cup_{k\geq 0} {\mathcal{M}}(x,y;k).$$
To see this one proves suitable gluing theorems in the case where $ev_{+}(u_j) = ev_{-}(u_{j+1})$
for some subset of $j$'s. 
\\
\newline Note that $\mathcal{M}(x,y;0)$ has a free $\bbR$
action if $x \neq y$; we denote the quotient by
$ {\mathcal{M}}^*(x,y;0)$ and we set ${\mathcal{M}}^*(x,y;0) 
= \emptyset$ if $x=y$. Meanwhile, for $k\geq 1$, $\mathcal{M}(x,y;k)$ has 
a free $\bbR^k$ action for $k\geq 1$ and we denote the quotient by 
$\mathcal{M}^*(x,y;k)$. Set
$$ {\mathcal{M}}^*(x,y) = \cup_{k\geq 0} {\mathcal{M}}^*(x,y;k).$$
 We denote the union of the components of dimension $d$ by $ {\mathcal{M}}^*_d(x,y)$ 
Given regular data $((H,J);(f,g))$, we define the ungraded Morse-Bott 
Floer chain complex $CF(L_0, L_1, (H,J);(f,g) )$ with $\bbZ/2$ coefficients to be freely generated 
by the critical points of $f$, and we define the differential by
$$\partial x = \Sigma_y \# {\mathcal{M}}^*_0(x,y) y.$$
To sketch why $\partial^2=0$, we briefly describe the compactification of ${\mathcal{M}}^*_1(x,y;k)$. 
%As we mentioned earlier, in a one parameter family of Floer trajectories with cascades, the number of cascades can vary. 
The family can break in one of three ways:
First, one of the holomorphic strips can collide with a critical point of $f$;
%so that in the limit  we have $ev_{-}(u_{j+1}) = \gamma_j(\infty) = p \in Crit (f)$, and $ ev_{+}(u_j) \in  S(p)$. 
%(Similarly, if $\gamma'(s) = -\nabla_g f(\gamma(s))$ and $x = \gamma(-\infty)$, then in the limit we could have
% $ev_-(u_1) = \gamma(\infty) = p \in Crit(f)$.
%And similarly for the $\gamma$ at $y$.)
second, 
a gradient line $\gamma_j$ can break at some critical point; 
%so that in the limit we have two trajectories 
%$\gamma_-^j$ and $\gamma_+^j$ with $\gamma_-^j(\infty) =  p = \gamma_+^j(-\infty)$ and $ev_+(u_j) \in S(p)$,  $ev_-(u_{j+1}) \in U(p)$. 
%(Similarly, if $\gamma$ is such that $\gamma'(s) = -\nabla_g f(\gamma(s))$ and $x = \gamma(-\infty)$,
%then in the limit we can have $\gamma^-$ and $\gamma^+$ with $\gamma^-(-\infty) =  x$,
%$\gamma^-(\infty) =  p = \gamma_+^j(-\infty)$ and $ev_-(u_1) \in U(p)$.
%And similarly for the $\gamma$ at $y$.) 
third, a holomorphic strip $u_j$ can break right at a critical point of $f$
 (i.e. in the limit we have two strips $u_-^j, u_+^j$ with $ev_+(u_-^j) = ev_-(u_{j}^+) = p \in Crit(f)$).
\\
\newline The definition of the triangle product is the same as before except that at each vertex of the holomorphic triangle, instead of just a single gradient line, one should have a Floer trajectory with cascades of arbitrary finite length. This shows up because in a one parameter family of triangles, one has at each vertex the same behavior as for holomorphic strips.
\\
\newline %Similarly the definition of the continuation map is the same as before except that at each end of the $s$-dependent holomorphic strip, instead of just a single gradient line, one should have a Floer trajectory with cascades of arbitrary finite length. Another new feature for $\phi^{\alpha\beta}$ is that  we no longer have $H=0$, and we do not keep $(f,g)$ fixed.  Namely, 
We define the continuation map in two scenarios. First, if we want to change 
the Morse-Bott $(h,g)$ (and keep $(J,H)$ fixed), 
we use the usual $s$-dependent 
gradient trajectories as in Morse homology. Here the continuation map $\phi$
is defined by counting holomorphic strips with cascades, except one of the gradient lines is allowed to be $s$ dependent.
Second, if we want to change $(J,H)$ we now have
two regular  pieces of data: $((J^\alpha, H^{\alpha}); (f^{\alpha},g^{\alpha}))$
and  $((J^\beta, H^{\beta}); (f^{\beta},g^{\beta}))$ (for fixed $L_0, L_1$). 
We choose  a  homotopy  $(H^{\alpha \beta}, J^{\alpha \beta})$ from $(J^\alpha, H^{\alpha})$ to 
$(J^\beta, H^{\beta})$. Then the main moduli space defining $\phi^{\alpha \beta}$,  denoted
$\mathcal{M}^{\alpha\beta}(x,y)$, consists of configurations of the form:
a $((J^\alpha, H^{\alpha}); (f^{\alpha},g^{\alpha}))$-Floer trajectory with  cascades, followed by an $s$-dependent $(J^{\alpha,\beta}, H^{\alpha,\beta})$-holomorphic strip, followed by   $((J^\beta, H^{\beta}); (f^{\beta},g^{\beta}))$-Floer trajectory with cascades. Then $\phi^{\alpha \beta}(x) = \Sigma_y \#\mathcal{M}_0(x,y)y$.
\\
\newline We sketch proofs of the basic properties of $\phi^{\alpha \beta}$. 
First it is easy to check that $\partial^\alpha \phi^{\alpha\beta} - \phi^{\alpha\beta} \partial^\beta =0$ by looking at $\partial \mathcal{M}_1(x,y)$.
Also, to see that $\phi^{\alpha \beta}$ is independent of the chosen homotopies, the 
usual proof suffices: One chooses a homotopy of homotopies and uses this to define a moduli space which gives a chain homotopy between the two continuation maps. (See \cite[lemma 6.3]{SZ}.)
%$(J^{\alpha,\beta}_\lambda, H^{\alpha,\beta}_\lambda)$, $(f^{\alpha\beta}_\lambda, g^{\alpha\beta}_\lambda)$, $\lambda \in [0,1]$. Then set
%$$ \mathcal{M}(x,y) = \{(\lambda,u) : u \in \mathcal{M}(x,y;  (J^{\alpha,\beta}_\lambda, 
%H^{\alpha,\beta}_\lambda), (f^{\alpha\beta}_\lambda, g^{\alpha\beta}_\lambda) \}$$
To see that $\phi^{\alpha\alpha} = id$ one chooses constant homotopies. Then, 
the $s$-dependent holomorphic strips (or the $s$-dependent gradient lines) must be constant, because otherwise they are not isolated. 
After that one is left with a Floer trajectory with cascades from say $x$ to $y$. %The dimenion of the moduli space is zero. 
Using a suitable notion of Maslov index $\mu$, one can express the dimension as $\mu(y)-\mu(x) =0$, but a Floer trajectory with cascades will strictly increase the index (as in Morse theory), unless it is constant.
To see that  $\phi^{\alpha\beta} \circ \phi^{\beta\gamma} = \phi^{\alpha\gamma}$ in homology,
we make the following argument. We suppress all the data except the $J$'s.
For $R\geq 0$,  define $J^{\alpha\gamma}_R$ by gluing together $J^{\alpha\beta}$ and $J^{\beta\gamma}$ so that  
\begin{gather*}
J^{\alpha\gamma}_R(s,t) = J^{\alpha\beta}(s-R,t)\text{ for }s \leq -1 \\
J^{\alpha\gamma}_R(s,t) = J^{\alpha\beta}(s+R,t) \text{ for }s \geq 1.
\end{gather*}
Then set
$$ \widehat {\mathcal{M}}(x,y) = \{(R,u) :R\geq 0, u \in \mathcal{M}(x,y;J^{\alpha,\gamma}_R) \}.$$
Define $G(x) = \Sigma_y \# \widehat {\mathcal{M}}_0(x,y)y$. $G$ will be a  chain homotopy between 
$\phi^{\alpha\gamma}$ and $\phi^{\gamma\beta} \circ \phi^{\beta \alpha}$.
The boundary of the one dimensional part of $\widehat {\mathcal{M}}(x,y)$ has points arising in three ways.
First $(R_n,u_n)$ can be such that $R_n \into0$. Then $u_n \into u \in \mathcal{M}(x,y;J^{\alpha,\gamma}_0)$,
which contributes to $\langle \phi^{\alpha\gamma}(x),y\rangle $ 
(here, if $z = \Sigma_i c_i y_i$ then $\langle z, y_i \rangle = c_i$).
Second $R_n$ can converge to some finite number $0<R_0<\infty$, and then the $\alpha$ or $\beta$ cascades in the configurations $u_n$ can break. This results in a contribution to a term of the form 
$\langle (\partial^\alpha \circ G + G \circ \partial^\beta)(x), y\rangle$.
Finally if $R_n \into \infty$ then $u_n \into u$, where $u$ can be obtained uniquely by gluing 
together a pair 
$$(u^{\alpha\beta}, u^{\beta\gamma}) \in \mathcal{M}(x,z;J^{\alpha\beta}) \times 
\mathcal{M}(z,y;J^{\beta\gamma})$$ 
for some $z$. Such $u$ contribute terms of the form 
$\langle \phi^{\gamma\beta} \circ \phi^{\beta \alpha}(x), y \rangle $.
In total we conclude that
$$ \phi^{\alpha\gamma} - \phi^{\gamma\beta} \circ \phi^{\beta \alpha} =
 \partial^\alpha \circ G + G \circ \partial^\beta. $$

\subsection{The Morse-Bott complex} \label{MorseBotthomology} Let $f: N \into \bbR$ be  a Morse-Bott function and let $g$ be Riemannian metric. Fix a Morse-Smale pair $(h,g_0)$ on $Crit(f)$. For generic $g$ one can define a chain complex whose differential 
counts flow lines of $f$ with cascades; these are defined in the same way as above, except that instead of holomorphic strips one has negative gradient lines of $f$. 
%In fact  everything above can be done in this setting.
We call the resulting complex the Morse-Bott complex. 
By using the continuation map one can see that its homology is isomorphic to the usual Morse-homology. 
(See \cite{Fr} Appendix A for details.) 
%(The linerized gradient equation for $(f,g)$ must be surjective, 
%and the product of certain evaluation maps  must be transverse to the product of certain unstable and stable manifolds in the critical components.)
\\
\newline \emph{A Regularity criterion.} The definition of regularity for $((f,g);(h, g_0))$ is the same as above. If $f$ is Morse it is well known that the linearized operator 
$$X \mapsto \frac{ dX}{dt} + (Hess f)(X)$$
is surjective along at a solution to the negative gradient equation if and only if the corresponding stable and unstable manifolds are transverse in $N$. The same argument in the  Morse-Bott case shows that regularity for $(f,g)$ is equivalent to 
$S(C_+)$ and $U(C_-)$ being transverse in $N$, where $C_\pm$ are critical components of $f$ and we are linearizing at a gradient line $\gamma: \bbR \into N$ satisfying $\gamma(\pm \infty) \in C_{\pm}$.  
%The other regularity condition concerning the evaluation maps also has a simple geometric meaning; see the proofs of lemmas \ref{MorsehomologyI}, \ref{MorsehomologyII} for that. 

\subsection{An exponential convergence condition} \label{sectionexponentialconvergence} Because of nondegenerate asymptotics, the Fredholm theory in  \S \ref{sectionFredholm}  requires
%In order for the linearized $\overline{\partial}$ operator to be Fredholm 
that we work with Sobolev spaces with exponential weights at the 
punctures. Therefore we need the convergence in condition \ref{convergence} to be exponential in suitable coordinates provided by Proposition 3.4.1 from \cite{Poz}. This states that 
%if $L_0, L_1$ have Morse-Bott intersection then 
near each $p \in L_0 
\cap L_1$ there is a symplectic chart
$\varphi: U \rightarrow B \subset  \bbC^n$, where $B$ is a ball centered at 0, 
such that %if we split
%$\bbC^n$ as $\bbR^k \times \bbR^{n-k} + i(\bbR^k \times \bbR^{n-k})$ then
$\varphi(L_0 \cap U)= \bbR^n \cap B$ and $\varphi(L_1 \cap U) =( (\bbR^k \times 0) + i(0 \times \bbR^{n-k})) \cap B$. Now suppose $u$ and $p_{+}$ are as in condition \ref{convergence}. Choose 
a chart $\varphi$ as above for $p=p_+$. For $s>0$ sufficiently large we write
$$\varphi(u(s,t)) = (x_1(s,t),x_2(s,t)) + i (y_1(s,t),y_2(s,t)).$$
Say $\varphi(p) = (x_1^0, 0) +i(0,0).$ Then the exponential decay condition is:

\begin{condition}\label{exponentialconvergence} There is an $r>0$ 
satisfying the following. For any $u$ as in condition \ref{convergence},  and for any multi-index $I$, $|I|\geq 0$,
there is a constant $C_I>0$  such that, for $s>0$ sufficiently large,
$$|\partial^I[(x_1(s,t),x_2(s,t)) + i (y_1(s,t),y_2(s,t)) -(x_1^0, 0) 
-i(0,0))] |  \leq C_I e^{-rs},$$
where $\partial^I = \partial_t^a\partial_s^b$, if $I = (a,b)$, $a,b \geq 0$. 
For $\lim_{s\rightarrow -\infty} u(s,t) = p_-$ we have a corresponding statement where, for $s<0$, the above expression is bounded by $C_I e^{rs}$.
\end{condition}

\subsection{Fredholm theory and a gluing formula} \label{sectionFredholm}

In this section %we discuss the functional analytic framework in which the 
%Morse-Bott Floer cohomology is formulated. 
we specify suitable 
exponentially weighted Sobolev spaces for which the linearized 
$\overline \partial$ operator is Fredholm; this makes sense in view of 
condition \ref{exponentialconvergence}. (Our approach follows \cite{B} section
5.1.) Then we state a gluing formula for the Fredholm index. There we follow to some extent \cite[\S 8h, 8i, 11c]{S08}. Our formula is similar to (11.17) in 11c.
%That treatment, however, expresses things in terms of connections. 
\subsubsection{Fredholm theory for holomorphic triangles and strips}  
We focus on the case of holomorphic triangles;
the whole discussion applies equally well to holomorphic strips. 
We use notation from  \S \ref{deftriangleproduct}. 
%and we fix strip like ends for $V$ as in that section.
Let $V$ be a disk with three boundary punctures; we fix strip like ends $\epsilon_j$, $j=0,1,2$ for $V$ as 
in \S \ref{deftriangleproduct}. Let $C_0, C_1,C_2$ be intersection components of 
$\widetilde V_4\cap \widetilde V_0$, $\widetilde V_2^j\cap \widetilde V_0$, $\widetilde V_2^j\cap \widetilde V_4$ respectively.
Fix $p>2$ and $0<d<r$, where $r$ is from condition \ref{exponentialconvergence}. 
For $k\geq 1$, let
$$\mathcal{B}= \mathcal{B}_k^{p,d}(\widetilde V_4, \widetilde V_2^j, \widetilde V_0; C_0, C_1, C_2)$$
denote the Banach space of maps $w: V \into M$ which are locally in 
$L_k^p \subset C^0$ and which satisfy $\lim_{\zeta \rightarrow \zeta_j} w(\zeta) = p_j$ for some $p_j \in C_j$ and
$$(w_j(s,t) -x_j) \in L^{p,d}_k = \{f(s,t) : f(s,t) e^{d|s|/p} \in L^p_k\}$$
for $|s|$ sufficiently large.  Here, $w_j(s,t) = (\varphi_j \circ w \circ \epsilon_j)(s,t)$, and 
$x_j = \varphi_j(p_j)$, where $\varphi_j$ is a chart near $p_j$, as in condition \ref{exponentialconvergence}.
Fix $J \in \mathcal{J}(V,M,I)$ satisfying (\ref{Jcompatible}). Condition \ref{exponentialconvergence} implies 
$$\mathcal{M}= \mathcal{M}(J, \widetilde V_4, \widetilde V_2^j, \widetilde V_0; C_0,C_1,C_2) \subset \mathcal{B}^{p,d}_k$$ 
for all $p>2$, $k\geq 1$, $0<d<r$.
Let $\mathcal{E}\into \mathcal{B}$ be the Banach bundle which has fiber
$$\mathcal{E}_w = L^{p,d}_{k-1}(V,\Lambda^{0,1}(w^*TM)),$$ 
where this has the same definition as for $\mathcal{B}$.
%which are $L^{p,d}_{k-1}$ $(0,1)$-forms on $V$ with values in $u^*TM$. 
%We also have $L^{p,d}_k(u^*TM)$, which is 
%$L^{p,d}_k$ sections of
%$u^*TM$. 
%Fix $J \in \mathal{J}(V,M,I)$ satisfying (\ref{Jcompatible}). 
Let $\overline{\partial}_J: \mathcal{B} \into \mathcal{E}$ denote the Cauchy-Riemann 
operator corresponding to $J$. By elliptic regularity $\overline{\partial}_J^{\, -1}(0) \subset \mathcal{B}$
consists of smooth solutions, and coincides with $\mathcal{M}$. 
At each $w \in \mathcal{M}$ we have the linearization 
$$ D(\overline{\partial}_J)_w: T_w \mathcal{B} \into \mathcal{E}_w.$$
We identify 
$$ T_w \mathcal{B} \cong \bbR^N \oplus 
L^{p,d}_k(V, w^*TM,F),$$ 
$N = \Sigma_j \dim C_j$, and % $F \into \partial V$ is a Lagrangian sub-bundle of $E|\partial V$, and 
$L^{p,d}_k(V,w^*TM,F)$ denotes the Banach space (defined as before) of 
sections $X: V \into w^*TM$ with Lagrangian boundary conditions $X(I_k) \subset F|I_k$, $k=0,2,4$, where
$F|I_0 = (w|I_0)^*(T\widetilde V_{0})$, $F|I_1 = (w|I_1)^*(T\widetilde V_{2}^j)$, $F|I_2 = (w|I_2)^*(T\widetilde V_4)$.
Here, the $\bbR^N$ factor corresponds to choosing 
a basis of solutions to
$$i\partial_t v + S_j(t)v =0, \, v(0) \in \bbR^n, \, v(1) \in (\bbR^k \times\{0\}) \oplus i(\{0\} \times \bbR^{n-k})$$
on $[0,1]$, say $v^j_k,\,  k=1, \ldots, \dim  C_j$.
Take  a monotone bump function $\rho_j(s,t)= \rho_j(s)$ defined on each strip like end which is 0 near $s=0$ and equal to 1 for large $|s|$. Then, 
$\rho_j(s)v^j_k(t)$ span the subspace of $T_u \mathcal{B}$ corresponding to $\bbR^N$ above.
Near a puncture $\zeta_j$,  $D(\overline{\partial}_J)_w$  has the form
$$\partial_s + i \partial_t + S_j(s,t)$$
where $s>0$ or $s<0$ depending on $j$, and  
$S_j(s,t)$ is a smooth family of  matrices in $\bbC^{n \times n}$, with
$$S_j(t) = \lim_{s \into \pm \infty} S_j(s,t)$$
symmetric, see \cite[p. 10]{RS}. %(Here we have $\pm$ depending on $j$.) 
%We may as well assume $S_j(s,t) 
%=S_j(t)$ for all $s$ since that changes the operator by a compact 
%perturbation only, so the question of being Fredholm and the index are 
%left invariant.
\begin{comment} This Cauchy-Riemann type operator is degenerate because if we let 
$\Psi(t)$, $t\in [0,1]$
be the family of symplectic matrices satisfying
$$\Psi'(t) = J_{\bbC} S_j(t)\Psi(t), \Psi(0) = I$$
then $\Psi(1)-I$ will have kernal with dimension $\dim C_j$.
\end{comment}

\begin{lemma}\label{lemmaFredholm} The operator
$$D(\overline{\partial}_J)_u: \bbR^N \oplus L^{p,d}_k(u^*TM) \into 
\mathcal{E}_u$$
is Fredholm,
\end{lemma}
\begin{proof} First, it suffices to prove that the restriction to
$L^{p,d}_k(u^*TM)$ is Fredholm, because we are throwing away only a 
finite dimensional space. Next, consider the isomorphisms
$$\varphi: L^{p,d}_{k}(u^*TM) \into L^{p}_{k}(u^*TM)$$
$$\varphi': L^{p,d}_{k-1}(\Lambda^{0,1}(u^*TM)) \into 
L^{p}_{k-1}(\Lambda^{0,1}(u^*TM))$$
given by multiplication by $e^{d|s|/k}$ on the strip-like ends and by 1 elsewhere. (Note that 
for $s=0$, $e^{d|s|/k} =1$, so this makes sense.) Now the 
restriction $D' = D(\overline{\partial}_J)_u|L^{p,d}_k(u^*TM)$ has the 
form $\partial_s + i\partial_t + S_j(s,t)$ near each puncture, so
$\varphi \circ D' \circ \varphi^{-1}: L^{p}_{k}(u^*TM) \into 
L^{p}_{k-1}(\Lambda^{0,1}(u^*TM))$ has the form
$$\partial_s + i \partial_t + S_j(s,t) \pm d/p.$$
Because of the perturbation $\pm d/p$, this is a standard Cauchy-Riemann type operator with 
nondegenerate asymptotics, hence this operator is Fredholm, and so the 
original operator is too.
\end{proof}

\subsubsection{A gluing formula} 
Take two Riemann surfaces $S_1, S_2$, which we suppose for simplicity are of the form 
$D^2$ minus some boundary punctures. We assume $S_1,S_2$ are  equipped with strip-like ends.
Suppose that  $E_j\into S_j$ is  a symplectic vector bundle and $F_j \into \partial S_j$
is a Lagrangian sub-bundle of $E_j|\partial S_j$.
We assume that, in a suitable symplectic trivialization of $E_j$,  $F_j$ is asymptotic at each puncture $\zeta$ to fixed Lagrangian subspaces $F_\zeta^-, F_\zeta^+ \subset \bbC^n$, where  $F_\zeta^-, F_\zeta^+$ is the counter-clockwise order along $\partial S_j$. Suppose now that each $S_j$ is equipped with a Cauchy-Riemann type operator
$\overline \partial_j$, which is of the form 
$$\partial_s + i \partial_t + S_\zeta^j(s,t)$$
in the trivialization restricted to the strip-like end.
Let $\zeta_+$ be an out going puncture of $S_1$, and $\zeta_-$ an incoming 
puncture of $S_2$. Given $l> 0$ can glue together $S_1 \setminus 
\epsilon_+([0,1]\times(l,\infty))$ and $S_2\setminus 
\epsilon_+([0,1]\times(-\infty, l))$ by identifying $ 
\epsilon_+([0,1]\times[0,l]) $ with $ \epsilon_-([0,1]\times[-l,0]) $, via 
$\epsilon_+(t,s)
\mapsto \epsilon_-(t,-s)$; we denote the result by $S_1 \#_l S_2$.
Assume that the asymptotic Lagrangian boundary conditions match up, $(F_1)_{\zeta}^{\pm} = (F_2)_{\zeta}^{\mp}$,  and the Cauchy-Riemann operators agree asymptotically:
$$ S_{\zeta^+}^1(t)= \lim_{s \rightarrow \infty} S_{\zeta^+}^1(s,t) =\lim_{s \rightarrow -\infty} S_{\zeta^-}^2(s,t)= S_{\zeta^-}^1(t).$$
Then one can glue all the data over $S_1 \#_l S_2$ for $l>>0$ and get
$E_1 \# E_2$, $F_1 \#F_2$, and $\overline \partial_1 \# \overline\partial_2$.
(To do this we should assume that $F_1$, $F_2$, $S_{\zeta^+}^1(s,t)$, $S_{\zeta^-}^2(s,t)$
are asymptotically constant, which can be arranged by a compact perturbation.)
In the transverse case (i.e. if $(F_j)_{\zeta}^{-} \cap (F_j)_{\zeta}^{+}$ is transverse), there is a gluing formula 
$$\index(\overline \partial_1 \# \overline \partial_2) = \index(\overline 
\partial_1) + \index(\overline \partial_2).$$
To prove it, one first reduces  to the case where the operators are surjective
by a finite dimensional stabilization argument.
Then one proves an isomorphism
$$Ker(\overline \partial_1 \# \overline \partial_2) \cong Ker(\overline 
\partial_1) \oplus Ker(\overline \partial_2)$$
by taking a pair $(X_1,X_2)$ on the right hand side, patching them 
together to get an approximate solution on the left hand side and then 
orthogonally projecting to yield an exact solution. 
(See \cite{Schwarzthesis} for example.)
\\
\newline In the case where the Lagrangians intersect in a Morse-Bott 
fashion,  
we can try the same type of argument, but in this case if $(X_1,X_2) \in 
Ker(\overline \partial_1) \oplus Ker(\overline \partial_2)$ then
$(X_1,X_2)$ must agree asymptotically 
in order for the patched together element to be an approximate 
solution
to $(\overline \partial_1 \# \overline \partial_2)(X) = 0$.
Therefore one gets an isomorphism of the following type.
%Let $\widetilde V_+ \subset T_u(\mathcal{B$ $\widetilde V_{-}$ denote the span of the 
%solutions that we did be
Let $T = (F_1)_{\zeta^-}^{-} \cap (F_1)_{\zeta^-}^{+}= (F_1)_{\zeta^+}^{-} \cap (F_1)_{\zeta^+}^{+}$,
and let $\bbR^k \subset \bbR^{N_+}$, $\bbR^k \subset \bbR^{N_-}$ denote the corresponding subspaces, where $\bbR^{N_{\pm}}$ is as in lemma \ref{lemmaFredholm}.
Let $\Delta \subset \bbR^{N_+} \times \bbR^{N_-}$ denote the codimension $k$
subspace of elements $(x,y)$ whose components $(x',y')$ in $\bbR^k \times \bbR^k \subset \bbR^{N_+} \times \bbR^{N_-}$ are equal.
%diagonal subspace, and let $Delta \oplus \bbR^\ell \subset \bbR^N_- \oplus \bbR^N_+$
The patching argument shows that
$$Ker(\overline \partial_1 \# \overline \partial_2) \cong Ker(\overline 
\partial_1 \oplus \overline \partial_2)|(\Delta \oplus L^{p,d}_k \oplus 
L^{p,d}_k)$$
and from this one concludes that
\begin{equation}
\label{indexformula} \index (\overline \partial_1 \# \overline \partial_2) 
= \index(\overline \partial_1) + \index( \overline \partial_2) - k.
\end{equation}

\vspace{0.5 cm}

\end{document}